\documentclass[12pt]{iopartcustom}

\usepackage{iopams}
  \expandafter\let\csname equation*\endcsname\relax
  \expandafter\let\csname endequation*\endcsname\relax
  \usepackage{amsmath}
    \usepackage{amsthm}
\usepackage[dvips]{graphicx}
\usepackage{pst-all} 
\usepackage[all]{xy} 

\numberwithin{equation}{section}
\newtheorem{thm}{Theorem}[section]
\newtheorem{prop}[thm]{Proposition}
\newtheorem{conj}[thm]{Conjecture}
\newtheorem{cor}[thm]{Corollary}
\newtheorem{lem}[thm]{Lemma}

\theoremstyle{definition}

\newtheorem{rem}[thm]{Remark}
\newtheorem{ex}[thm]{Example}
\newtheorem{defn}[thm]{Definition}

\newtheorem{ass}[thm]{Assumption}
\newtheorem{cond}[thm]{Condition}

\def\bbP{\mathbb{P}}
\def\bbZ{\mathbb{Z}}
\def\bbQ{\mathbb{Q}}
\def\bbC{\mathbb{C}}
\def\bbV{\mathbb{V}}
\def\bfS{\mathbf{S}}

\def\bfA{\mathbf{A}}

\def\bfM{\mathbf{M}}
\def\mathcalP{\mathcal{P}}

\def\ve{\varepsilon}

\newcommand{\notch}{\scriptstyle\bowtie} 

\begin{document}
\bibliographystyle{amsalpha}

\title[Exact WKB analysis and cluster algebras]{Exact WKB analysis
 and cluster algebras\\
\it \small To the memory of Kentaro Nagao}
 

\author{Kohei Iwaki}
\address{\noindent Research Institute for Mathematical Sciences, Kyoto University, 
Kyoto,
657-8501, Japan}
\ead{iwaki@kurims.kyoto-u.ac.jp}

\author{Tomoki Nakanishi}
\address{\noindent Graduate School of Mathematics, Nagoya University, 
Chikusa-ku, Nagoya,
464-8602, Japan}
\ead{nakanisi@math.nagoya-u.ac.jp}


\begin{abstract}
We develop the mutation theory
in the exact WKB analysis using the framework of cluster algebras.
Under a continuous deformation of the potential of the
Schr\"odinger equation on a compact Riemann surface,
 the Stokes graph may change the topology.
 We call this phenomenon the mutation of  Stokes graphs.
Along the mutation of  Stokes graphs,
 the Voros symbols, which are
monodromy data of the equation,
also mutate due to the Stokes phenomenon.
We show that the Voros symbols mutate as
variables of a cluster algebra with surface realization.
As an application, we obtain the identities of  Stokes automorphisms
associated with periods of cluster algebras.
The paper also includes an extensive introduction of
the exact WKB analysis and the surface realization of cluster algebras
for nonexperts.
\end{abstract}



\tableofcontents

\section{Introduction}
\label{sec:introduction}

In this paper we start to develop the mutation theory
in the {\em exact WKB analysis\/} using the framework of {\em cluster algebras}.

The  {\em WKB method\/} was originally initiated by
Wentzel, Kramers, and Brillouin in 1926
as
 the  method
for
obtaining approximate solutions of the {\em Schr\"odinger equation\/} in the semiclassical limit in  quantum mechanics.
Voros reformulated the theory based on the  Borel resummation method  \cite{Voros83},
and this new formulation has been further developed by \cite{Aoki91}, \cite{Delabaere93}, etc.,
and it is called the {\em exact} WKB analysis.
See the monograph \cite{Kawai05} for the introduction of the subject.
On the other hand, cluster algebras were introduced by Fomin and Zelevinsky around 2000
\cite{Fomin02} to study the coordinate rings of certain algebraic varieties
and subsequently developed in a series of the papers \cite{Fomin03a,Berenstein05,Fomin07};
it was also developed independently by Fock and Goncharov \cite{Fock03b,Fock03} from the viewpoint of
  higher Teichm\"uller theory.
 It turned out that cluster algebras are ``unexpectedly" related with several branches of mathematics beyond the original scope,
 for example, representation theories of quivers and quantum groups, 
 triangulated categories,  hyperbolic geometry,
 integrable systems, $T$-systems and $Y$-systems,
 the classical and quantum dilogarithms, Donaldson-Thomas theory, 
and so on.
 See the excellent surveys \cite{Keller08,Keller11} for the introduction of the subject.

Let us quickly explain the intrinsic reason why the above seemingly
unrelated two subjects are closely related.
Let us consider the Schr\"odinger equation on a compact
Riemann surface $\Sigma$
\begin{align}
\label{eq:Schro1}
\left(\frac{d^2}{dz^2}-\eta^2 Q(z,\eta)\right)
\psi(z,\eta)
=0,
\end{align}
where $z$ is a local complex coordinate of $\Sigma$,
$\eta=\hbar^{-1}$ is a large parameter,
and the potential $Q(z,\eta)$ is a function of both $z$ and $\eta$.
The principal part  $Q_0(z)$ of $Q(z,\eta)$  in the power series expansion 
in $\eta^{-1}$ defines a meromorphic {\em quadratic differential\/} $\phi$ on $\Sigma$.
The trajectories of the quadratic differential $\phi$ determine a graph $G$ 
on $\Sigma$ called
the {\em Stokes graph\/} of the equation \eqref{eq:Schro1},
 which plays the central role
in the exact WKB analysis.
On the other hand,  the Stokes graph $G$ can be  translated into a {\em triangulation\/} $T$ of the surface 
$\Sigma$ (with holes and punctures) \cite{Kawai05,Gaiotto09,Bridgeland13}.
Due to the works by Gekhtman, Shapiro, and Vainshtein
\cite{Gekhtman05},
Chekhov, Fock, and Goncharov (\cite{Fock03b}, \cite{Fock05} for a review),
and Fomin, Shapiro, and Thurston \cite{Fomin08,Fomin08b},
 the triangulation $T$ is further identified with a {\em seed\/}  $(B,x,y)$ of a certain cluster algebra,
which is the main object in  cluster algebra theory.

Our main purpose is to develop 
the {\em mutation theory\/} in the exact WKB analysis.
Under a continuous deformation of the potential $Q(z,\eta)$,
 the Stokes graph may change its topology.
 We call this phenomenon the 
 {\em mutation\/} of  Stokes graphs,
since they correspond to the mutation of triangulations through
 the above correspondence.
Along the mutation of  Stokes graphs,
the monodromy data of the equation \eqref{eq:Schro1}
called the {\em Voros symbols},
also mutate \cite{Delabaere93,Delabaere99}.
It turns out that  this  precisely coincides with
the mutation of seeds of the corresponding cluster algebra.
In short, this is the main result of the paper.

Before going into further detail of the results,
let us mention previous works
closely related to this work.
Our results have remarkable overlaps and resemblance with the wall-crossing
formula of the {\em Donaldson-Thomas invariants}
and {\em quantum dilogarithm identities}, since they are also related with (quantum) cluster algebras \cite{Fock07b,Kontsevich08,Kontsevich09,Nagao10,Keller11,Nagao11b,Kashaev11}.
To understand the {\em BPS spectrum\/} of the $d=4$, $\mathcal{N}=2$ field theories,
Gaiotto, Neitzke, and Moore \cite{Gaiotto09} studied the WKB approximation
for the flat connections of the {\em Hitchin system\/} on a Riemann surface, and
its mutation theory.
The Stokes graph naturally appeared also in their study,
and, in particular, they clarified that
there are two types of  ``elementary mutations'' of  Stokes graphs,
namely, {\em flips\/} and {\em pops\/}.
They also identify certain quantities for the Hitchin system
as the $y$-variables (the ``Fock-Goncharov coordinate'' therein)
 in cluster algebras.
 See  \cite{Xie12,Cirafici13}, for example, for a recent development.
 The mutation aspect of  Stokes graphs was further developed by Bridgeland and Smith \cite{Bridgeland13};
 their aim was
the construction of the {\em stability condition\/} in the 3-Calabi-Yau categories associated with
surface triangulations
based on the work of Labardini-Fragoso \cite{Labardini12}.
The connection between such 3-Calabi-Yau categories and cluster algebras
were studied by \cite{Kontsevich08,Nagao10}.
In this paper
we will rely on the result of \cite{Bridgeland13} for the mutation property
of  Stokes graphs.
We are also motivated by
Kontsevich and Soibelman's
observation that ``There is a striking similarity between our [their] wall-crossing
formula and identities for the Stokes automorphisms in the theory of WKB asymptotics...'' \cite[Section 7.5]{Kontsevich09}.
(See \eqref{eq:pentagon1} and \eqref{eq:dilog1} below.)
Our result  provides an understanding of this similarity at the level of cluster algebras.
We summarize the relation of previous and this works
schematically in Figure \ref{fig:outline}.

\begin{figure}
\begin{center}
\begin{pspicture}(0,0)(14.5,4.2)
\psset{linewidth=0.5pt}
\psline{->}(2.5,2.8)(3.5,2.4)
\psline{->}(2.5,1.2)(3.5,1.6)
\psline{->}(6.5,2)(7.5,2)
\psline{->}(10.5,2.4)(11.5,2.8)
\psline{->}(10.5,1.6)(11.5,1.2)
\psline{->}(13,2.5)(13,1.5)
\psline[linestyle=dashed]{->}(2.5,0.5)(11.5,0.5)
\rput[c]{0}(1,3.5){\makebox(2.3,1.2){
\begin{tabular}{c}
 Hitchin \\system \\
 \end{tabular}}}
 \rput[c]{0}(1,0.5){\makebox(2.3,1.2){
\begin{tabular}{c}
Schr\"odinger \\
equation\\
 \end{tabular}}}
\rput[c]{0}(5,2){\makebox(2.3,1.2){
\begin{tabular}{c}quadratic\\
 differentials \\
 \end{tabular}}}
\rput[c]{0}(9,2){\makebox(2.3,1.2){
\begin{tabular}{c}triangulations\\
 of surfaces \\
 \end{tabular}}}
\rput[c]{0}(13,3.5){\makebox(2.3,1.2){
\begin{tabular}{c}3-CY\\
 categories \\
 \end{tabular}}}
\rput[c]{0}(13,0.5){\makebox(2.3,1.2){
\begin{tabular}{c}cluster\\
 algebras \\
 \end{tabular}}}
 \rput[l](2.5,3.1){\small \cite{Gaiotto09}}
 \rput[l](2.5,0.9){\small \cite{Kawai05}}
 \rput[c](7,3.1){\small \cite{Kawai05,Gaiotto09,Bridgeland13}}
\rput[r](11.5,3.1){\small \cite{Labardini12,Bridgeland13}}
\rput[r](11.5,0.9){\small \cite{Gekhtman05,Fock03b,Fomin08}}
\rput[l](13.2,1.7){\small \cite{Nagao10}}
\rput[l](13.2,2.3){\small \cite{Kontsevich08}}
\rput[c](7,0.2){\small this work}
 \end{pspicture}
\end{center}
\caption{Outlines of previous and this works}
\label{fig:outline}
\end{figure}
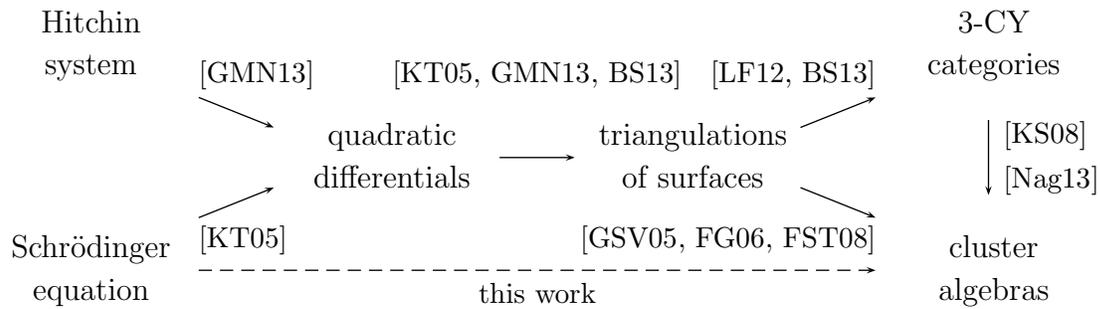

For those who are familiar with the subject,
let us give a  little more extended summary of our results
and also present some keywords without definitions.
The readers can safely skip this summary,
and come back later when the materials are discussed.
In any case Figure \ref{fig:outline} and
Table \ref{tab:dic} should be a useful guide to read through the paper.

{\em (a). Signed flips and signed pops.}
The mutation property of Stokes graphs itself is purely geometrical.
Here, we consider two kinds of elementary mutations, flips and pops.
To be precise, there are {\em two\/} ways to do flips and pops, namely,
to do them clockwise and anticlockwise.
We call them {\em signed flips} and {\em signed pops}.
Accordingly, we   need to extend the usual notions of tagged triangulations 
(or equivalently,  signed triangulations) and  seeds
to  what we call {\em Stokes triangulations} and {\em extended seeds}.
Then, we  define the signed flips (signed mutations for seeds) and signed pops
for Stokes triangulations and extended seeds.

{\em (b). Local result: Mutations of  simple paths, simple cycles, and Voros symbols.}
Let $\hat{\Sigma}$ be the covering of the surface $\Sigma$ to make
the square root of the quadratic differential $\phi$ single valued.
We introduce
the {\em simple paths} and the {\em simple cycles},
which are certain elements of the relative homology and
the homology of $\hat{\Sigma}$.
Under the mutation of Stokes graphs, they transform (= mutate)
as {\em monomial $x$-variables\/} and {\em monomial $y$-variables},
which are ingredients in our extended seeds (Proposition \ref{prop:cycle1}).
We consider the {\em Voros symbols\/}  associated with
the simple paths and the simple cycles.
As formal series in the parameter $\eta^{-1}$,
they mutate according to the mutations of the simple paths and the simple cycles.
In addition, by the {\em Borel resummation\/} the Voros symbols 
suffer nontrivial jumps along flips and pops of Stokes graphs
due to the Stokes phenomenon.
The jump formula  was known for flips (Theorem \ref{thm:DDP-analytic}) earlier by \cite{Delabaere93,Delabaere99},
and we call it the {\em Delabaere-Dillinger-Pham (DDP) formula}.
An analogous formula for pops  (Theorem \ref{thm:loop-type-degeneration-analytic}) are recently given by
\cite{Aoki14} in conjunction with this work.
Combining these geometric and analytic results, we conclude that the Voros symbols
for the simple paths mutate as {\em $x$-variables\/} in our extended seeds,
while the Voros symbols
for the simple cycles mutate as {\em $\hat{y}$-variables\/} therein (Theorem 
\ref{thm:localmutation}).
This is our first main result.
The correspondence between the data in the exact WKB analysis and
cluster algebras are summarized in Table
\ref{tab:dic}.
We note that much of our efforts are spent to work on {\em pops}.
In particular, if we concentrate on flips,
the setting becomes much lighter.

\begin{table}
\begin{center}
\begin{tabular}{ll}
\hline
exact WKB analysis & cluster algebra\\
\hline
signed flip  of Stokes graph & signed mutation\\
signed pop  of Stokes graph & signed pop (local rescaling)\\
simple path & monomial $x$-variable\\
simple cycle & monomial $y$-variable\\
Voros symbol for simple path &  $x$-variable in extended seed\\
Voros symbol for simple cycle & $\hat{y}$-variable in extended seed\\
\hline
\end{tabular}
\end{center}
\caption{Dictionary between exact WKB analysis and cluster algebras.}
\label{tab:dic}
\end{table}

{\em (c). Global result: Identities of  Stokes automorphisms.}
According to \cite{Delabaere93}, the mutation formula of
the Voros coefficients in (b) can be rephrased
in terms of the {\em Stokes automorphisms\/} acting on the field generated by the Voros symbols.
It is known that cluster algebras have a rich periodicity property.
Thanks to our result (b),
a periodicity in cluster algebras implies an identity of
Stokes automorphisms (Theorem \ref{thm:sid1}).
As the simplest example, if we apply it for the celebrated periodicity of flips of
triangulations of a pentagon with period 5 (Figure \ref{fig:pentagon1}),
we have the identity in \cite{Delabaere93}:
\begin{align}
\label{eq:pentagon1}
\mathfrak{S}_{\gamma_2}
\mathfrak{S}_{\gamma_1}
=
\mathfrak{S}_{\gamma_1}
\mathfrak{S}_{\gamma_1 + \gamma_2}
\mathfrak{S}_{\gamma_2},
\end{align}
where $\mathfrak{S}_{\gamma}$ is the Stokes automorphism for a cycle $\gamma$.
Our identities give a vast generalization of the identity \eqref{eq:pentagon1}.
This is our second main result.
We note that a  {\em quantum dilogarithm identity\/} is also associated with the same periodicity of the
cluster algebra \cite{Keller11,Nagao11b,Kashaev11}.
For example, 
the quantum dilogarithm identity associated with the
same period of a pentagon gives
the celebrated {\em pentagon identity\/} by
  \cite{Faddeev94},
and it looks as follows:
\begin{align}
\label{eq:dilog1}
\Psi_q(U_2)\Psi_q(U_1)
=
\Psi_q(U_1)
\Psi_q(q^{-1}U_2U_1)
\Psi_q(U_2),
\end{align}
where $\Psi_q(x)=\prod_{k=0}^{\infty}
(1+xq^{2k+1})$ is the quantum dilogarithm,
and $U_2U_1=q^2 U_1U_2$.
This is also interpreted as the simplest example of the wall-crossing formula
of the Donaldson-Thomas invariant in  \cite{Kontsevich08,Kontsevich09}.
The  similarity
between the identities 
\eqref{eq:pentagon1} and 
\eqref{eq:dilog1}
is the one observed by
\cite{Kontsevich09}.
Our derivation of \eqref{eq:pentagon1}
based on a periodicity of a cluster algebra
naturally explains the similarity.
It is desirable to understand
the similarity at a deeper level,
and we leave it as a future problem.

Let us explain the organization of the paper.
We anticipate that  most of the readers are  
 unfamiliar with at least one of two main subjects,
the exact WKB analysis or cluster algebras and their
surface realization.
So we provide an extensive introduction of both subjects
through Sections 2--5,
while setting up the formulation we will use.
In Section 2 we review the theory of  the exact WKB analysis,
mainly following \cite{Kawai05}.
Furthermore, we extend the method to a general compact
Riemann surface.
In Section 3 we introduce an important notion 
in  the exact WKB analysis, 
called the Voros symbols. We discuss the jump property 
of the Voros symbols caused by the Stokes phenomenon 
relevant to the appearance of saddle trajectories 
in the Stokes graph.
In Section 4 we introduce the basic notions and properties in cluster algebras
which we will use later.
In Section 5 the surface realization of cluster algebras by
 \cite{Gekhtman05,Fock03b,Fomin08,Fomin08b} is reviewed.
 Since careful treatment of mutations involving a self-folded triangle
 is crucial throughout the paper, we explain in detail how there are related
 to tagged triangulations and signed triangulations.
 The extended seeds and their signed mutations and pops
are also defined.

Then, we start 
 to integrate these two methods from Section 6.
In Section 6 we study the mutation of Stokes graphs,
which is purely geometric.
We  introduce Stokes triangulations, and their signed flips and pops.
They effectively control the
mutation of Stokes graphs.
We introduce the simple paths and the simple cycles of a Stokes graph,
and give their mutation formulas.
In Section 7 we combine the analytic and geometric results in Sections 3 and 6
and show that the Voros symbols for the simple paths and the simple cycles
mutate exactly as $x$-variables and $\hat{y}$-variables in our extended seeds.
In Section 8 by combining all results in the previous sections
we derive the identities of  Stokes automorphisms associated with
periods of seeds in cluster algebras.

\medskip
{\em Acknowledgements.}
We are grateful to Tatsuya Koike and Reinhard Sch\"afke to sharing their result
before publication.
We thank Takashi Aoki, Yuuki Hirako, Kazuo Hosomichi,
Akishi Ikeda, Takahiro Kawai,
Alastair King,
Hirokazu Maruhashi,
Andrew Neitzke, Michael Shapiro,
  Ivan Smith, Toshinori Takahashi, Yoshitsugu Takei,
  and Dylan Thurston for
useful discussions and communications.
The first author is supported by Research Fellowships of Japan Society
for the Promotion for Young Scientists.
We dedicate the paper to the memory of Kentaro Nagao,
who inspired us by his beautiful papers, talks, and private
conversations at various occasions.

\section{Exact WKB analysis}
\label{sec:exact}

In this section we review the theory of 
the exact WKB analysis (\cite{Voros83}). 
Most of our notations are consistent 
with those of \cite{Kawai05}. 
Usually, in the exact WKB analysis 
the Schr{\"o}dinger equation is studied on 
the Riemann sphere ${\mathbb P}^1$. 
Here, we extend the method to general 
compact Riemann surfaces.

\subsection{Schr{\"o}dinger equations  
and associated quadratic differentials}
Let $\Sigma$ be a compact Riemann surface,
by which we mean a compact, connected, and
oriented Riemann surface throughout the paper.

Consider a differential equation 
${\mathcal L} : L \psi = 0$ for a function 
$\psi$ on $\Sigma$. Here $L = L(z,d/dz,\eta)$ 
is a second order linear differential operator 
with meromorphic coefficients and 
containing a large parameter $\eta$. 
We usually regard $\eta$ as a real (positive) large parameter, 
but sometimes regard it as a complex large parameter. 
Assume that, in a local complex coordinate $z$ of $\Sigma$, 
${\mathcal L}$ is represented as follows:  
\begin{equation} \label{eq:Sch}
{\mathcal L} : 
L\varphi =\left( \frac{d^{2}}{dz^{2}} - 
\eta^{2}Q(z,\eta) \right) \psi(z,\eta) = 0,
\end{equation}
where 
\begin{equation}
Q(z,\eta)=Q_{0}(z)+\eta^{-1}Q_{1}(z) + \eta^{-2} Q_{2}(z)+ \cdots
\end{equation}
is a polynomial in $\eta^{-1}$ 
(i.e., $Q_n(z)=0$ for $n \gg 1$)
whose coefficients $\{Q_n(z) \}_{n \ge 0}$
are meromorphic functions on $\Sigma$. 
We remark that any ordinary differential equation of the form
\begin{equation} \label{eq:Sch-on-Sigma}
\left( \frac{d^{2}}{dz^{2}} + \eta p(z,\eta) \frac{d}{dz}
+ \eta^{2}q(z,\eta) \right) \varphi(z,\eta) = 0
\end{equation}
can be reduced to the form \eqref{eq:Sch} by a certain gauge 
transformation.
The equation \eqref{eq:Sch} is nothing but a one-dimensional 
stationary {\em Schr\"{o}dinger equation}, where $\eta^{-1}$ 
corresponds to the Planck constant $\hbar$, with the potential 
function $Q(z,\eta)$ whose principal term is given by $Q_0(z)$. 
We will impose some assumptions on the potential $Q(z,\eta)$
in subsequent subsections. 

We call \eqref{eq:Sch} {\em the Schr\"{o}dinger form
(of ${\mathcal L}$) in the local coordinate $z$}, since 
the potential function $Q(z,\eta)$ depends on the choice of 
the local coordinate. 
If we take a coordinate transformation $z=z(\tilde{z})$ 
and a gauge transformation, the Schr\"{o}dinger form 
in the local coordinate $\tilde{z}$ becomes
\begin{equation} \label{eq:Sch-w}
\left( \frac{d^{2}}{d\tilde{z}^{2}} - 
\eta^{2} \tilde{Q}(\tilde{z},\eta) \right) 
\tilde{\psi}(\tilde{z},\eta)= 0, \quad 
\tilde{\psi}(\tilde{z},\eta) 
= \psi\bigl(z(\tilde{z}),\eta\bigr)
\left(\frac{dz(\tilde{z})}{d\tilde{z}}\right)^{-1/2},
\end{equation}
\begin{equation} \label{eq:transformation-of-potentail}
\tilde{Q}(\tilde{z},\eta) = 
Q\bigl(z(\tilde{z}),\eta\bigr)
\left(\frac{dz(\tilde{z})}{d\tilde{z}}\right)^{2}-
\frac{1}{2}\eta^{-2}\{ z(\tilde{z}); \tilde{z}\},
\end{equation}
where $\{ z(\tilde{z}); \tilde{z}\}$ is 
the Schwarzian derivative
\[
\{ z(\tilde{z}); \tilde{z}\} = 
\left({\frac{d^{3}z(\tilde{z})}{d\tilde{z}^{3}}}\biggl/
{\frac{dz(\tilde{z})}{d\tilde{z}}}\right)-\frac{3}{2}
\left({\displaystyle 
\frac{d^{2}z(\tilde{z})}{d\tilde{z}^{2}}}\biggl/
{\displaystyle \frac{dz(\tilde{z})}{d\tilde{z}}}\right)^{2}.
\]
In particular, the transformation law 
\begin{equation} \label{eq:transformation-Q0}
\tilde{Q}_{0}(\tilde{z}) = Q_{0}\bigl(z(\tilde{z})\bigr)
\left(\frac{dz}{d\tilde{z}}\right)^{2}
\end{equation}
of the principal terms of the potential functions of 
the Schr\"{o}dinger 
form coincides with that of a 
{\em meromorphic quadratic differential}, 
that is, a meromorphic section of the line bundle 
$\omega_{\Sigma}^{\otimes2}$. Here $\omega_{\Sigma}$ 
is the holomorphic 
cotangent bundle on $\Sigma$. 
\begin{defn}
{\em The quadratic differential associated with ${\mathcal L}$}
is the meromorphic quadratic differential on $\Sigma$ 
which is locally given by
\begin{equation}
\phi = Q_{0}(z) dz^{\otimes2}
\end{equation}
in a local coordinate $z$. Here $Q_{0}(z)$ is 
the principal term of the potential function $Q(z,\eta)$ 
of the Schr\"{o}dinger form of ${\mathcal L}$ 
in the local coordinate $z$. 
\end{defn}

Geometry of zeros, poles, and trajectories of $\phi$ 
are important in  the exact WKB analysis. 
They relate to properties of solutions 
of ${\mathcal L}$ deeply. 


\subsection{Turning points and singular points}

The poles of the associated quadratic differential $\phi$ 
are singular points of the differential equation 
\eqref{eq:Sch}. In the exact WKB analysis 
the zeros of $\phi$ are also important.

\begin{defn}
A zero (resp., simple zero) of $\phi$ is called a {\it turning point} 
(resp., {\em simple turning point}) of ${\mathcal L}$.
\end{defn}

Let $P_0$ and $P_{\infty}$ be the set of the zeros and the poles of 
$\phi$, respectively, and set $P=P_{0} \cup P_{\infty}$. 
In this paper we always impose the following assumption.
\begin{ass} \label{assumption:zeros and poles} 
Let $\phi$ be the quadratic differential associated with 
${\mathcal L}$. We assume 
\begin{itemize} 
\item $\phi$ has at least one zero, and at least one pole,
\item all zeros of $\phi$ are simple,
\item the order of any pole of $\phi$ 
is more than or equal to  2.
\end{itemize}
\end{ass}

\begin{rem} {
The behavior of the WKB solutions around a simple pole was
studied by \cite{Koike00}, and it requires special attention
in our problem.
We will treat the simple pole case in a separate publication.}
\end{rem}
The quadratic differentials satisfying 
the above assumption are called 
{\em complete Gaiotto-Moore-Neitzke (GMN) differentials} 
in \cite[Section 2.2]{Bridgeland13}. This assumption makes treatment of 
trajectories easier. The assumption that all turning points 
are simple is also reasonable in  
the exact WKB analysis. For example, 
Theorem \ref{thm:Voros-formula} below can not 
be applied for higher order turning points. 

In addition to 
Assumption \ref{assumption:zeros and poles}, 
we also impose the following assumption 
for $Q_n(z)$ with $n \ge 1$. 
\begin{ass} \label{assumption:summability}
\begin{itemize}
\item[(i).] %
If a point $p \in \Sigma$ is a pole of $Q_n(z)$ 
for some $n\ge1$, then $p \in P_{\infty}$.
\item[(ii).] %
If $\phi$ has a pole $p$ of order $m \ge 3$, then 
the following condition holds.
\begin{equation} \label{eq:pole-order-ge3}
\text{(order of $Q_{n}(z)$ at $p$)}~<~
1+\frac{m}{2} \quad \text{for all $n\ge1$}.
\end{equation}
\item[(iii).] %
If $\phi$ has a pole $p$ of order $m = 2$, then 
the following conditions hold.
\begin{itemize}
\item[$\bullet$] %
$Q_n(z)$ has an at most simple pole at $p$ for all 
$n \ge 1$ except for $n=2$. 
\item[$\bullet$] %
$Q_2(z)$ has a double pole at $p$ and satisfies 
\begin{equation} \label{eq:assumption-Q2}
Q_2(z) = -\frac{1}{4z^2} (1+O(z)) \quad 
\text{as $z \rightarrow 0$},
\end{equation}
where $z$ is a local coordinate of $\Sigma$ near 
$p$ satisfying $z(p) = 0$. 
\end{itemize}
\end{itemize}
\end{ass}

Note that the conditions \eqref{eq:pole-order-ge3} and 
\eqref{eq:assumption-Q2} are independent 
of the choice of the local coordinate due to 
the transformation law 
\eqref{eq:transformation-of-potentail} 
of Schr\"{o}dinger forms. These assumptions will be 
necessary to define an integral of a certain 1-form 
from a point $p \in P_{\infty}$ 
(see Proposition \ref{proposition:integrability}). 
Moreover, Assumption \ref{assumption:summability}
is also used in the proof of the Borel summability 
of the WKB  solutions 
(see Theorem \ref{thm:summability-P1}). 
Let us give examples satisfying 
Assumption \ref{assumption:summability}.

\begin{ex} \label{example:some-examples}
(a). %
Let $\Sigma = \bbP^1$, and consider
the potential 
$Q(z,\eta) = Q_0(z)$ which
is independent of $\eta$ and a polynomial in $z$
of degree $m \ge 1$. Then,
the quadratic differential  $\phi$ has only one pole 
of order $m + 4$ at $\infty$.
This is the case that \cite{Voros83} and 
\cite{Delabaere93} considered. %
\par
(b). Let $\Sigma = \bbP^1$, and consider 
the following differential equation: 
\[
\left(\frac{d^2}{dz^2} - \eta^2 Q(z,\eta) \right)\psi=0, 
\quad Q(z,\eta) = Q_0(z)+\eta^{-2}Q_{2}(z), 
\]
\[
Q_0(z)=\frac{(\alpha-\beta)^2z^2+
2(2\alpha\beta-\alpha\gamma-\beta\gamma)z+\gamma^2}
{4z^2(z-1)^2}, \quad Q_2(z)=-\frac{z^2-z+1}{4z^2(z-1)^2}. 
\]
Here $\alpha$, $\beta$ and $\gamma$ are complex parameters. 
This equation is equivalent to Gauss' hypergeometric 
equation and studied in \cite{Aoki12}.
Under a generic condition for the parameters $\alpha$, $\beta$ 
and $\gamma$, the quadratic differential $\phi$ has two simple zeros and three poles 
of order 2 at $0,1,\infty$. 
We can easily check that \eqref{eq:assumption-Q2} 
is satisfied at each pole.
\end{ex}

\subsection{Riccati equation}
\label{secton:Riccati-equation}

Following \cite[Section 2]{Kawai05}, 
to construct the WKB solutions of \eqref{eq:Sch}, 
we consider the following auxiliary equation, which is called 
the {\em Riccati equation} associated with \eqref{eq:Sch}:
\begin{equation} \label{eq:Riccati}
\frac{dS}{dz}+S^{2} = \eta^{2}Q(z,\eta).
\end{equation}
A solution of \eqref{eq:Sch} and that of 
\eqref{eq:Riccati} are related as  
\begin{equation} \label{eq:WKBsol1}
\psi(z,\eta)=\exp\left(\int^{z} S(z,\eta)dz\right).
\end{equation}
We can construct a formal (series) solution of \eqref{eq:Riccati} 
in the following form:
\[
S(z,\eta) = \sum_{n=-1}^{\infty}\eta^{-n}S_{n}(z) = 
\eta \hspace{+.1em} S_{-1}(z)+S_{0}(z)+
\eta^{-1} S_{1}(z)+\cdots.  
\]
Here ``formal series'' means formal 
Laurent series in $\eta^{-1}$. 
The family of functions $\{ S_{n}(z) \}_{n \ge -1}$
 must satisfy 
the following recursion relation
\begin{align}
\begin{cases}
\label{eq:recursion}
S_{-1}^{2} = Q_{0}(z), \\[+.2em]
\displaystyle 2S_{-1}S_{n+1} + 
\sum_{
{ n_{1}+n_{2} = n \atop
0 \le n_{j} \le n}
}
S_{n_{1}}S_{n_{2}} + \frac{dS_{n}}{dz} = 
Q_{n+2}(z) \quad(n\ge-1).
\end{cases}
\end{align} 
We obtain two families of functions 
$\{ S^{(+)}_{n}(z) \}_{n \ge -1}$ and 
$\{ S^{(-)}_{n}(z) \}_{n \ge -1}$ which satisfy the 
recursion relation \eqref{eq:recursion},
 depending on the 
choice of the root  $S_{-1}=\pm\sqrt{Q_0(z)}$
for the initial condition in 
\eqref{eq:recursion}. Thus we have two formal solutions
\begin{equation} \label{eq:S-plus-minus}
S^{(\pm)}(z,\eta) = \sum_{n=-1}^{\infty}\eta^{-n} 
S^{(\pm)}_{n}(z) = 
\pm \eta \sqrt{Q_{0}(z)} + \cdots 
\end{equation} 
of the Riccati equation \eqref{eq:Riccati}. 
The functions $\{ S^{(\pm)}_{n}(z) \}_{n\ge-1}$ 
are singular on $P$, and multi-valued and holomorphic 
on $\Sigma\setminus{P}$. 

Following \cite[Remark 2.2]{Kawai05}, we define 
the {\em odd part} and 
the {\em even part} of $S(z,\eta)$ by 
\begin{equation} \label{eq:Sodd-and-Seven}
S_{\rm odd}(z,\eta) = \frac{1}{2}\left( 
S^{(+)}(z,\eta) - S^{(-)}(z,\eta) \right),\quad
S_{\rm even}(z,\eta) = \frac{1}{2}
\left(S^{(+)}(z,\eta) + S^{(-)}(z,\eta) \right).
\end{equation}
These quantities have the following properties. 
\begin{prop} \label{prop:Sodd} 
(a).%
The equality 
\begin{align}
S^{(\pm)}(z,\eta) = 
\pm S_{\rm odd}(z,\eta) + S_{\rm even}(z,\eta)
\end{align}
holds,
 and the even part is given by the logarithmic 
derivative of the odd part: 
\begin{equation}
S_{\rm even}(z,\eta) = -\frac{1}{2 S_{\rm odd}(z,\eta)}
\frac{d S_{\rm odd}(z,\eta)}{dz}.
\end{equation}
\par
(b).
The (formal series valued) 1-form 
$S_{\rm odd}(z,\eta)dz$ is invariant under coordinate 
transformations. That is, the odd part 
$\tilde{S}_{\rm odd}(\tilde{z},\eta)$ of a formal solution 
of the Riccati equation associated with \eqref{eq:Sch-w} 
is given by 
\begin{equation} \label{eq:Sodd-is-coordinate-free}
\tilde{S}_{\rm odd}(\tilde{z},\eta) = S_{\rm odd}
\bigl( z(\tilde{z}),\eta \bigr) 
\frac{dz(\tilde{z})}{d\tilde{z}}
\end{equation}
if we choose the square root in \eqref{eq:S-plus-minus} 
so that the following equality holds 
(cf.~\eqref{eq:transformation-Q0}):
\begin{equation} \label{eq:relation-of-S-1}
\sqrt{\tilde{Q}_{0}(\tilde{z})} = 
\sqrt{Q_{0}\bigl(z(\tilde{z})\bigr)}
\frac{dz}{d\tilde{z}}.
\end{equation}
\end{prop}
\begin{proof} 
The claims (a) and (b) are proved by the same argument 
in \cite[Remark 2.2]{Kawai05} and 
\cite[Corollary 2.17]{Kawai05}, respectively. 
\end{proof}

Proposition \ref{prop:Sodd} implies that the 1-form 
$S_{\rm odd}(z,\eta)dz$ is globally defined 
(but multi-valued) on $\Sigma\setminus{P}$. 
This is not integrable at a point in $P_{\infty}$
because the principal term $\eta \sqrt{Q_0(z)}dz$ 
is singular. However, under Assumption 
\ref{assumption:summability}, 
we can show the following fact. 

\begin{prop}\label{proposition:integrability} 
For any point $p \in P_{\infty}$ and any local coordinate 
$z$ of $\Sigma$ around $p$ such that $z = 0$ at $p$, 
the formal power series valued 1-form defined by
\begin{equation}\label{eq:Soddreg}
S_{\rm odd}^{\rm reg}(z,\eta)~dz = 
\left(S_{\rm odd}(z,\eta)-\eta \sqrt{Q_{0}(z)}\right)dz,
\end{equation}
is integrable at $z=0$. Namely, for any $n \ge 0$, 
there exists a real number $\ell > -1$ such that 
\begin{equation} \label{eq:integrability-estimate}
S_{{\rm odd},n}(z) = O(z^{\ell}) \quad 
\text{as $z \rightarrow 0$}.
\end{equation} 
Here $S_{{\rm odd},n}(z)$ is the coefficient of 
$\eta^{-n}$ in the formal series $S_{\rm odd}(z,\eta)$.
Especially, all coefficients of 
$S_{\rm odd}^{\rm reg}(z,\eta)$ are holomorphic at $p$ 
if it is an even order pole of $\phi$.
\end{prop}
\begin{proof}
Fix any local coordinate $z$ around $p$ as above. 
It follows from the recursion relation 
\eqref{eq:recursion} and the definition 
\eqref{eq:Sodd-and-Seven} of $S_{\rm odd}(z,\eta)$ 
that $S^{(\pm)}_0(z)$ and $S_{{\rm odd},0}(z)$ 
are given by 
\begin{equation} \label{eq:S0-and-Sodd0}
S^{(\pm)}_{0}(z) =  - \frac{1}{4Q_0(z)}
\frac{d Q_0(z)}{dz} \pm \frac{Q_1(z)}{2\sqrt{Q_0(z)}}, 
\quad S_{{\rm odd},0}(z) = \frac{Q_1(z)}{2\sqrt{Q_0(z)}}.
\end{equation}
Then, although $S^{(\pm)}_{0}(z) = O(z^{-1})$ as 
$z \rightarrow 0$, we can show that 
\eqref{eq:integrability-estimate} 
holds for $n=0$ due to Assumption 
\ref{assumption:summability}.  
Similarly, $S^{(\pm)}_1(z)$ is given by 
\begin{equation} \label{eq:Splusminus1}
S^{(\pm)}_{1}(z) = \frac{\pm1}{2 \sqrt{Q_0(z)}} 
\left( Q_2(z) - S^{(\pm)}_{0}(z)^2 - 
\frac{d S^{(\pm)}_{0}(z)}{dz} \right).
\end{equation}
Denote by $m$ the pole order of $\phi$ at $p$.
If $m \ge 3$, we can verify that 
$S^{(\pm)}_{1}(z) = O(z^{\ell})$ 
for some $\ell > - 1$ 
since $\sqrt{Q_0(z)} = O(z^{-m/2})$,  
$S^{(\pm)}_{0}(z) = O(z^{-1})$ and we 
have Assumption \ref{assumption:summability} (ii).
Hence we have \eqref{eq:integrability-estimate} for $n=1$.
On the other hand, the situation is different when $m=2$.
In view of \eqref{eq:Splusminus1}, 
$S^{(\pm)}_{1}(z)$ may have a simple pole at $p$ 
since $\sqrt{Q_0(z)} = O(z^{-1})$ when $m=2$.
However, with the aid of Assumption 
\ref{assumption:summability} (iii), 
we can show that $S^{(\pm)}_{1}(z)$ 
becomes holomorphic because 
\begin{equation}
Q_2(z) - S^{(\pm)}_{0}(z)^2 - 
\frac{d S^{(\pm)}_{0}(z)}{dz} = O(z^{-1})
\end{equation}
holds by \eqref{eq:assumption-Q2}  
and \eqref{eq:S0-and-Sodd0}. Therefore, we also 
have \eqref{eq:integrability-estimate} for $n=1$ 
in the case $m=2$.
The estimate \eqref{eq:integrability-estimate} 
for $n \ge 2$ can be shown by the induction  
from the recursion relation \eqref{eq:recursion} 
and Assumption \ref{assumption:summability}. 
Furthermore, since $\sqrt{Q_0(z)}$ is single-valued 
around $p$ when it is an even order pole of $\phi$, 
the recursion relation \eqref{eq:recursion} also implies 
that $S^{(\pm)}_n(z)$ and $S_{{\rm odd},n}(z)$ are 
single-valued around $p$ for all $n \ge 0$. Thus,
$S_{{\rm odd},n}(z)$ becomes holomorphic at $p$ 
for all $n \ge 0$ due to \eqref{eq:integrability-estimate}. 
\end{proof}

We call $S_{\rm odd}^{\rm reg}(z,\eta)$ in \eqref{eq:Soddreg} 
{\em the regular part of $S_{\rm odd}(z,\eta)$}. 
$S_{\rm odd}^{\rm reg}(z,\eta)$ 
is a formal {\em power} series in $\eta^{-1}$ since 
the principal term of $S_{\rm odd}(z,\eta)$ is eliminated. 
Integrals of $S_{\rm odd}(z,\eta)dz$ and 
$S_{\rm odd}^{\rm reg}(z,\eta)dz$ on $\Sigma$ are 
important in  the exact WKB analysis.

\subsection{WKB solutions}
\label{subsec:WKB}

Using the relation \eqref{eq:WKBsol1} between the solutions of 
\eqref{eq:Sch} and \eqref{eq:Riccati}, and the property 
(a) in Proposition \ref{prop:Sodd}, we obtain the following 
two formal solutions of \eqref{eq:Sch}:
\begin{equation} \label{eq:WKBsol2}
\psi_{\pm}(z,\eta) = \frac{1}{\sqrt{S_{\rm odd}(z,\eta)}}
\exp\left(\pm\int^{z}S_{\rm odd}(z,\eta)dz\right).
\end{equation}
\begin{defn}
The formal solutions \eqref{eq:WKBsol2} 
are called the {\em WKB solutions} of \eqref{eq:Sch}. 
\end{defn}

The integral of $S_{\rm odd}(z,\eta)dz$ is defined as a 
term-wise integral for the coefficient 
of each power of $\eta$. 
The lower end-point of the integral \eqref{eq:WKBsol2} 
will be discussed later. Since the coefficients of 
$S_{\rm odd}(z,\eta)dz$ are 
multi-valued on $\Sigma\setminus{P}$, 
the path of integral in \eqref{eq:WKBsol2} should be 
considered in the Riemann surface $\hat{\Sigma}$ of 
the multi-valued 1-form $\sqrt{Q_0(z)}dz$. 
To be more explicit, $\hat{\Sigma}$ is given 
by a section of the cotangent bundle of $\Sigma$ as
$\hat{\Sigma} = \{(z,\nu)~|~\nu^2=\phi \} 
\subset \omega_{\Sigma}$. Then the coefficients of the 
1-form $S_{\rm odd}(z,\eta)dz$ are single-valued 
on $\hat{\Sigma}$. 
The projection $\pi : \hat{\Sigma} \rightarrow \Sigma$ 
is a double cover branching at the simple zeros and the odd order 
poles of $\phi$.

To visualize $\hat{\Sigma}$, and to determine the 
branch of the square root in \eqref{eq:S-plus-minus}, 
we usually take {\em branch cuts} on $\Sigma$. 
A branch cut must connect 
two branch points of the covering map $\pi$, and each branch 
point must be an end-point of a branch cut. Such a collection 
of branch cuts together with a choice of a point 
$\hat{z} \in \hat{\Sigma}$ give an embedding 
$\iota:\Sigma\rightarrow\hat{\Sigma}$, which is a piecewise 
continuous and has a discontinuity on the branch cut, and 
contains $\hat{z}$ in its image. We call the image of $\Sigma$ 
by $\iota$ the {\em first sheet}, while the complement of the 
first sheet in $\hat{\Sigma}$ the {\em second sheet}. 
We may regard a point on $\Sigma$ as a point on $\hat{\Sigma}$ 
by such an embedding $\iota$ for a fixed appropriate branch cut, 
and use the same symbol $z$ for a coordinate of the first sheet, 
and use $z^{*} = \tau(z)$ for that of the second sheet. 
Here $\tau : \hat{\Sigma} \rightarrow \hat{\Sigma}$ is the 
covering involution which exchanges the first  and the 
second sheet, and it commutes with the projection $\pi$. 
Then, the action of $\tau$ for $S_{\rm odd}(z,\eta)$ 
and $S_{\rm odd}^{\rm reg}(z,\eta)$ are given by 
\begin{equation} \label{eq:covering-involution}
S_{\rm odd}(z^{*},\eta) = - S_{\rm odd}(z,\eta), \quad
S_{\rm odd}^{\rm reg}(z^{*},\eta) = 
- S_{\rm odd}^{\rm reg}(z,\eta)
\end{equation}  
since the involution $\tau$ exchanges the sign in 
\eqref{eq:S-plus-minus}.

Here we give two well-normalized expressions 
of the WKB solutions which will be considered in this paper.
\begin{itemize}
\item {\em normalized at a turning point} $a \in P_{0}$:
\begin{equation} \label{eq:WKBsol-TP}
\psi_{\pm}(z,\eta) = \frac{1}{\sqrt{S_{\rm odd}(z,\eta)}}
\exp\left(\pm\int_{a}^{z}S_{\rm odd}(z,\eta)dz\right).
\end{equation}
Although the coefficients of $S_{\rm odd}(z,\eta)$ 
have a singularity at $a$, the integral
\eqref{eq:WKBsol-TP} can be defined 
with the aid of the anti-invariant property 
\eqref{eq:covering-involution} 
of $S_{\rm odd}(z,\eta)$. Namely, it is defined by 
the half of the contour integral 
\begin{equation}
\int_{a}^{z}S_{\rm odd}(z,\eta)dz = 
\frac{1}{2}\int_{\gamma_{z}}S_{\rm odd}(z,\eta)dz
\end{equation}
along a path $\gamma_{z}$ as in Figure 
\ref{fig:normalized-at-TP}. Here the wiggly line 
designates a branch cut, and the solid part 
(resp., the dotted part) belongs to the first sheet
(resp., the second sheet). In this paper 
integrals of $S_{\rm odd}(z,\eta)$ and 
$S_{\rm odd}^{\rm reg}(z,\eta)$ from a simple turning 
point are always defined in this manner. 

\item {\em normalized at a pole} $p \in P_{\infty}$:
\begin{equation} \label{eq:WKBsol-P}
\psi_{\pm}(z,\eta) = \frac{1}{\sqrt{S_{\rm odd}(z,\eta)}}
\exp\left\{\pm\left(\eta\int_{a}^{z}
\sqrt{Q_{0}(z)}dz+\int_{p}^{z}
S_{\rm odd}^{\rm reg}(z,\eta) dz\right)\right\}.
\end{equation}
Here $a$ is any turning point independent of $p$. 
Note that, the integral of $S_{\rm odd}^{\rm reg}(z,\eta)$ 
from a pole $p$ is well-defined by Proposition 
\ref{proposition:integrability}.
\end{itemize}

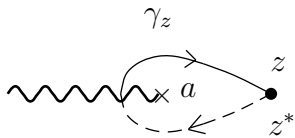
\begin{figure}
\begin{center}
\begin{pspicture}(0,0)(4,2)
%
\psset{fillstyle=solid, fillcolor=black}
\pscircle(3.5,0.5){0.08} 
\psset{fillstyle=none}
\rput[c]{0}(2.05,0.48){\small $\times$}
\rput[c]{0}(2,1.5){$\gamma_z$}
\rput[c]{0}(3.6,0.9){$z$}
\rput[c]{0}(3.65,0.1){$z^{*}$}
\rput[c]{0}(2.4,0.55){$a$}
\psset{linewidth=1pt}
\pscurve(2,0.5)(1.9,0.4)(1.8,0.5)(1.7,0.6)
(1.6,0.5)(1.5,0.4)(1.4,0.5)(1.3,0.6)(1.2,0.5)
(1.1,0.4)(1,0.5)(0.9,0.6)(0.8,0.5)(0.7,0.4)(0.6,0.5)
(0.5,0.6)(0.4,0.5)(0.3,0.4)(0.2,0.5)(0.1,0.6)(0,0.5)
%
%
%
\psset{linewidth=0.5pt}
%
\psecurve(1.45,0.2)(1.5,0.5)(1.6,0.8)(2,1)(2.6,0.9)(3.5,0.5)
\pscurve(2.6,0.9)(3,0.75)(3.5,0.5)
\psline(2.5,0.95)(2.37,1.08)
\psline(2.5,0.95)(2.35,0.85)
\psline(2.4,0.025)(2.57,0.22)
\psline(2.4,0.025)(2.63,-0.03)
\psset{linewidth=0.5pt, linestyle=dashed}
\psecurve(1.45,0.8)(1.5,0.5)(1.6,0.2)(2,0)(2.6,0.1)(3.5,0.5)
\pscurve(2.6,0.1)(3,0.25)(3.5,0.5)
\end{pspicture}
\end{center}
\caption{Normalization at a simple turning point.} 
\label{fig:normalized-at-TP}
\end{figure}

\subsection{Borel resummation method and Stokes phenomenon}
\label{section:Borel-resummation}

Let us expand \eqref{eq:WKBsol2}
in the following formal series: 
\begin{equation} \label{eq:WKBsol3}
\psi_{\pm}(z,\eta)=\exp\left(\pm\eta\int^{z}
\sqrt{Q_{0}(z)}dz\right) \eta^{-1/2}
\sum_{k=0}^{\infty} \eta^{-k} \psi_{\pm, k}(z). 
\end{equation}
It is known that, the series \eqref{eq:WKBsol3} is divergent 
in general, and its ``principal term'' (see
\eqref{eq:S0-and-Sodd0})
\[
\psi_{\pm}(z,\eta) 
= \frac{\eta^{-1/2}}{Q_{0}(z)^{1/4}}
\exp\left(\pm
\int^{z}
\left\{
\eta
\sqrt{Q_{0}(z)}
+\frac{Q_1(z)}{2\sqrt{Q_0(z)}}
\right\}dz
\right)(1+O(\eta^{-1}))
\]
is known as the Wentzel-Kramers-Brillouin 
approximation (the {\it WKB approximation}) of the solutions 
of the Schr{\"o}dinger equation \eqref{eq:Sch}. 
(Usually $Q_1(z)=0$.)
In the framework of the {\em exact} 
WKB analysis we take the {\em Borel resummation} of the WKB solutions 
to obtain analytic results. For the convenience of readers, 
we give an explanation of the Borel resummation method 
for formal series in $\eta^{-1}$. 
See \cite{Costin08} for further explanation.  
\begin{defn} 
\label{def:Borel-summability}
\begin{itemize}
\item 
A formal power series
$f(\eta) = \sum_{n=0}^{\infty}\eta^{-n}f_{n}$ 
in $\eta^{-1}$ is said to be 
{\em Borel summable} 
if the formal power series 
\begin{equation} \label{eq:Borel-transform}
f_{B}(y) = \sum_{n=1}^{\infty}f_{n}\hspace{+.1em}
\frac{y^{n-1}}{(n-1) !}
\end{equation}
converges near $y=0$ and can be analytically continued 
to a domain $\Omega$ containing the half line 
$\{y \in {\mathbb C}~|~{\rm Re}~y \ge 0,~{\rm Im}~y =0 \}$,
and satisfies 
\begin{equation} \label{eq:exponential-type}
\left|f_B(y) \right|
\le C_1 e^{C_2|y|}
\quad (y\in \Omega)
\end{equation}
with positive constants $C_1, C_2 > 0$. 
The function $f_B(y)$ is called the {\it Borel transform} 
of $f(\eta)$. %
\item 
For a Borel summable formal power series 
$f(\eta) = \sum_{n=0}^{\infty}\eta^{-n}f_{n}$,  
define the {\em Borel sum of $f(\eta)$} 
by the following Laplace integral:
\begin{equation} \label{eq:Borel sum}
{\mathcal S}[f](\eta)=f_0+\int_{0}^{\infty}
e^{-\eta \hspace{+.1em} y}f_{B}(y) dy.
\end{equation}
Here the path of the integral is taken along the 
positive real axis. Due to \eqref{eq:exponential-type}, 
the Laplace integral \eqref{eq:Borel sum} converges 
and gives an analytic function of $\eta$ on 
$\{\eta\in{\mathbb R}~|~\eta\gg1 \}$.%
\item 
Let $f(\eta) = e^{\eta \hspace{+.1em} s}\eta^{-\rho}
\sum_{n=0}^{\infty}\eta^{-n}f_{n}$ be a formal series 
with an exponential factor $e^{\eta \hspace{+.1em} s}$ 
for some  $\rho\in{\mathbb C}$ and $s \in {\mathbb C}$. 
$f(\eta)$ is said to be {\em Borel summable} 
if the formal power series 
$g(\eta)=\sum_{n=0}^{\infty}\eta^{-n}f_{n}$ 
is Borel summable. 
The {\em Borel sum of $f(\eta)$}
is defined by ${\mathcal S}[f](\eta) 
= e^{\eta \hspace{+.1em} s}
\eta^{-\rho}{\mathcal S}[g](\eta)$, 
where ${\mathcal S}[g]$ is the Borel sum of $g(\eta)$.
\end{itemize}
\end{defn}

For the simplest example, let us consider the monomial 
$f(\eta) = \eta^{-n}$ $(n \ge 1)$. Then we have 
$f_B(y) = {y^{n-1}}/{(n-1)!}$ and hence the Borel sum 
${\mathcal S}[f](\eta) = \eta^{-n}$ coincides with the 
original monomial. In general, it is known that, if the 
formal power series $f(\eta)$ converges and defines 
a holomorphic function near $\eta = \infty$, then $f(\eta)$ 
is Borel summable and the Borel sum coincides 
with the original function $f(\eta)$. 

The map ${\mathcal S}$ from a set of Borel summable formal 
series to a set of analytic functions of $\eta$ 
is called the {\em Borel resummation operator}.
The following properties are well-known 
(e.g., \cite[Section 4]{Costin08}).
\begin{prop} \label{prop:property-of-Borel-sum}
(a). The operator ${\mathcal S}$ 
commutes with addition and multiplication. 
That is, for formal power series 
$f(\eta)$ and $g(\eta)$ which are 
Borel summable, we have
${\mathcal S}[f+g] = 
{\mathcal S}[f]+{\mathcal S}[g]$, 
${\mathcal S}[f \cdot g] = 
{\mathcal S}[f]\cdot{\mathcal S}[g]$. %
\par
(b). If a formal power series 
$f(\eta)$ is Borel summable, then 
${\mathcal S}[f](\eta)$ is 
asymptotically expanded to $f(\eta)$ when 
$\eta \rightarrow +\infty$.
\par
(c). Let $A(t)=\sum_{k=0}^{\infty}A_k t^k$ 
be a convergent series defined near the origin $t=0$. 
If a formal power series 
$f(\eta) = \sum_{n=1}^{\infty}\eta^{-n}f_n$ without 
a constant term is Borel summable, 
then the formal power series 
$A(f(\eta))=\sum_{k=0}^{\infty}A_k (f(\eta))^k$ is 
also Borel summable. 
Moreover, the Borel sum 
is given by ${\mathcal S}[A(f(\eta))] = 
A({\mathcal S}[f](\eta))$.
\end{prop}

Even if a formal power series $f(\eta)$ is divergent, its 
Borel sum ${\mathcal S}[f](\eta)$ becomes analytic and the 
original $f(\eta)$ is recovered as an asymptotic expansion 
of the Borel sum, if $f(\eta)$ is Borel summable. 
In this sense the Borel resummation method is 
a natural resummation procedure of divergent series. 

However, when the Borel transform $f_B(y)$ of $f(\eta)$ 
has a singular point $y=y_0$ on the positive real axis 
(i.e., $f(\eta)$ is {\em not} Borel summable), 
then the Laplace integral \eqref{eq:Borel sum} can not 
be defined and we can not find an analytic function 
of $\eta$ having the above asymptotic property 
by the ``usual" Borel resummation method. 

In such a case, to obtain an analytic function 
which has $f(\eta)$ as its asymptotic expansion 
when $\eta \rightarrow +\infty$, 
we regard $\eta$ as a complex large parameter 
with a certain phase $\arg\eta=\theta \in {\mathbb R}$ 
and consider the following Borel resummation 
{\em in the direction $\theta$}:
\begin{equation} \label{eq:Borel-sum-theta}
{\mathcal S}_{\theta}[f](\eta)=f_0+\int_{0}^{\infty e^{-i\theta}}
e^{-\eta \hspace{+.1em} y}f_{B}(y) dy.
\end{equation}
Here the path of integral in \eqref{eq:Borel-sum-theta} 
is taken along the half line 
$\{y = r e^{-i \theta} \in {\mathbb C}~|~ r \ge 0 \}$  
so that the singular point $y_0$ of $f_B(y)$ does not 
lie on the path. If the Laplace integral 
\eqref{eq:Borel-sum-theta} is 
well-defined in a similar sense of 
Definition \ref{def:Borel-summability}, 
then $f(\eta)$ is said to be Borel summable 
{\em in the direction $\theta$}, and 
${\mathcal S}_{\theta}$ is called the Borel resummation 
operator {\em in the direction $\theta$}. 
Then, the analytic continuation of the Borel sum 
\eqref{eq:Borel-sum-theta} becomes an analytic function 
of $\eta$ in a sector $\{\eta\in{\mathbb C}~|~
|\arg\eta-\theta|<\pi/2, |\eta|\gg1 \}$.
Especially, if $f(\eta)$ is Borel summable in the direction 
$\delta$ for a sufficiently small $\delta > 0$, 
then ${\mathcal S}_{\delta}[f](\eta)$
is analytic on $\{\eta\in{\mathbb R}~|~ \eta \gg1 \}$ 
and having $f(\eta)$ as its asymptotic expansion when  
$\eta \rightarrow+\infty$. That is, 
${\mathcal S}_{\delta}[f](\eta)$ has the desired 
asymptotic property for large $\eta > 0$.

However, there is an ambiguity in analytic functions 
which are asymptotically expanded to $f(\eta)$ 
as $\eta \rightarrow +\infty$. Suppose that 
$f(\eta)$ is Borel summable in the both directions 
$+\delta$ and $-\delta$ for a sufficiently small number 
$\delta > 0$. Then,  both of the Borel sums 
${\mathcal S}_{\pm\delta}[f](\eta)$ have the same asymptotic 
expansion $f(\eta)$ when $\eta \rightarrow +\infty$.
But these functions do {\em not} coincide in general; 
if $f_B(y)$ has a singular point $y_0$ on the 
positive real axis, the Borel sums 
${\mathcal S}_{+\delta}[f](\eta)$ and 
${\mathcal S}_{-\delta}[f](\eta)$ may be different 
since the path of Laplace integrals are not 
homotopic due to the singular point $y_0$.

This is the so-called {\em Stokes phenomenon} 
for the formal series $f(\eta)$. 
Here the Stokes phenomenon means a phenomenon that, 
the analytic function which has $f(\eta)$ with 
its asymptotic expansion when 
$|\eta| \rightarrow + \infty$ depends on the 
direction of an approach to $\eta = \infty$,
and the analytic functions may be different 
for different directions in general. 
Similarly to Proposition \ref{prop:property-of-Borel-sum} 
(b), ${\mathcal S}_{\pm\delta}[f](\eta)$ is asymptotic 
to $f(\eta)$ when $|\eta|\rightarrow+\infty$
with $\arg\eta=\pm\delta$.  
Therefore, the fact that the Borel sums 
${\mathcal S}_{+\delta}[f](\eta)$ and 
${\mathcal S}_{-\delta}[f](\eta)$ are different implies that 
the Stokes phenomenon occurs to $f(\eta)$. 
This is the formulation of the Stokes phenomenon 
in terms of Borel resummation method. 
Moreover, the difference of the Borel sums 
${\mathcal S}_{\pm\delta}[f](\eta)$ are exponentially small 
when $\eta \rightarrow +\infty$ since they have the 
same asymptotic expansion. 

If the formal power series $f(\eta)$ 
is Borel summable in any direction $\theta$ satisfying 
$-\delta \le \theta \le + \delta$ with a sufficiently 
small number $\delta > 0$, then $f(\eta)$ does not 
enjoy the Stokes phenomenon; 
that is, the Borel sums satisfies 
\begin{equation}
{\mathcal S}_{-\delta}[f](\eta) = 
{\mathcal S}_{+\delta}[f](\eta) = 
{\mathcal S}[f](\eta)
\end{equation}
as analytic functions of $\eta$ on 
$\{\eta \in {\mathbb R}~|~ \eta \gg 1 \}$. 
This is because the Borel transform $f_B(y)$ 
does not have singular points in a domain 
containing the sector 
$\{y=r e^{-i\theta}~|~ r \ge 0,~ 
-\delta \le \theta \le +\delta \}$ and 
the Laplace integrals \eqref{eq:Borel-sum-theta}
give the same analytic function.
Thus the singular points of the Borel transform $f_B(y)$
are closely related to the Stokes phenomenon 
for the formal series $f(\eta)$. 

The following lemma will be used 
in the subsequent discussions. 
\begin{lem} \label{lemma:0-sum-and-theta-sum}
{
Let $f(\eta) = \sum_{n=0}^{\infty} f_n \eta^{-n}$ 
be a formal power series and $\theta$ be a real number. 
Then, $f(\eta)$ is Borel summable in the direction $\theta$ 
if and only if the formal power series 
$f^{(\theta)}(\eta) = \sum_{n=0}^{\infty} 
f_n e^{-i n \theta} \eta^{-n} (=f(e^{i\theta}\eta))$ 
is Borel summable in the usual sense 
(i.e., Borel summable in the direction $0$).
} 
\end{lem}
Lemma \ref{lemma:0-sum-and-theta-sum} follows immediately 
from the equality
\begin{equation} \label{eq:fB-and-ftheta-B}
f^{(\theta)}_{B}(y) = 
e^{-i\theta} f_{B}(e^{-i\theta}y)
\end{equation}
that can be shown by a straightforward computation. 

When we apply the Borel resummation method to the WKB 
solutions \eqref{eq:WKBsol3}, we fix the independent variable 
$z$ and regard them as formal series in $\eta^{-1}$ with 
exponential factors $\exp(\pm\eta\int^{z}\sqrt{Q_{0}(z)}dz)$. 
Therefore, the condition that 
``the Borel sum is well-defined" gives a constraint for $z$. 
The condition can be checked by looking 
the {\em Stokes graph\/} defined in the next subsection.

\subsection{Trajectories, Stokes curves, and Stokes graphs}
\label{section:Stokes-graphs}

Let $\phi$ be the quadratic differential associated 
with ${\mathcal L}$. 
This subsection is devoted to the description
of properties of {\em trajectories} of $\phi$.
Here a trajectory of $\phi$ is a leaf 
of the foliation on $\Sigma\setminus{P}$ defined 
by the equation 
\begin{equation} \label{eq:theta-foliation}
{\rm Im} \int^{z}\sqrt{Q_{0}(z)}dz = {\rm constant}.
\end{equation}
Every point of $\Sigma\setminus P$ 
lies on a unique trajectory, and any two 
trajectories are either disjoint or coincide. 
The foliation structure by the trajectories of $\phi$ has been 
well studied in  Teichm\"uller theory \cite{Strebel84}. 
The relationship between the geometry of trajectories and 
the asymptotic property of WKB solutions is studied by 
Fedoryuk \cite{Fedoryuk93}.
The geometry of trajectories is also important in the exact 
WKB analysis since we can read off a lot of properties 
of the WKB solutions, such as the Borel summability 
(i.e., well-definedness of 
the Borel sum \eqref{eq:Borel sum}), 
from the geometry of the trajectories of $\phi$. 

\begin{defn}[{\cite[Definition 2.6]{Kawai05}}]
A {\em Stokes curve of ${\mathcal L}$} 
is a trajectory of $\phi$ 
whose one of the end-points 
is a turning point of ${\mathcal L}$. 
Namely, in a local coordinate $z$ of $\Sigma$,
the Stokes curves 
emanating from a turning point 
$a \in P_0$ are defined as
\begin{equation} \label{eq:StokesCurves}
{\rm Im}\int_{a}^{z}\sqrt{Q_{0}(z)}dz=0.
\end{equation}
\end{defn}

Note that the Stokes curves are determined from 
the principal term $Q_0(z)$ of the 
potential function $Q(z,\eta)$ of \eqref{eq:Sch}. 
Figure \ref{fig:examples-of-Stokes-graphs} depicts 
examples of the Stokes curves for several 
rational functions $Q_0(z)$  on $
\mathbb{C}\subset \Sigma = \bbP^1$. 
Here we use the symbol ${\times}$ 
for a point in $P_{0}$ (i.e., a turning point)
and $\bullet$ for a point in $P_{\infty}$ 
(i.e., a pole of $\phi$) in the figures.
The  quadratic differentials $\phi$ on $\Sigma$ in these examples have a pole also at $z=\infty$, which is omitted in the figures.

  \begin{figure}
  \begin{center}
\begin{pspicture}(0,0)(0,0)
\psset{linewidth=0.5pt}
\end{pspicture}
\end{center} \hspace{-2.em}
  \begin{minipage}{0.33\hsize}
  \begin{center} 
  \includegraphics[width=40mm]
  {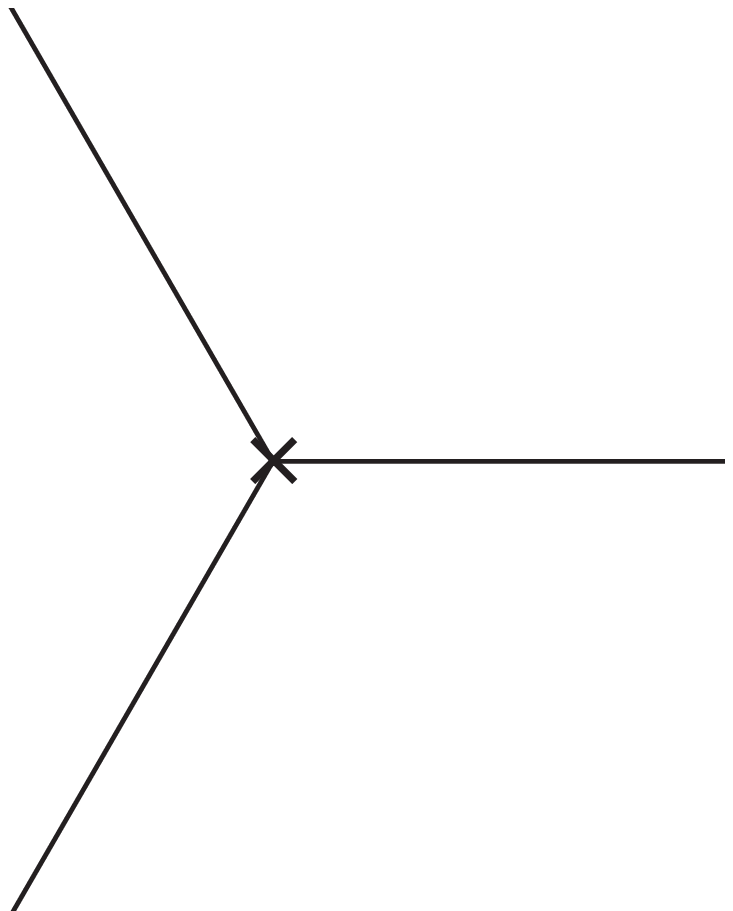} \\
  {(a) $z$.} 
  \end{center}
  \label{fig:Airy}
  \end{minipage} 
  \begin{minipage}{0.33\hsize}
  \begin{center}
  \includegraphics[width=40mm]
  {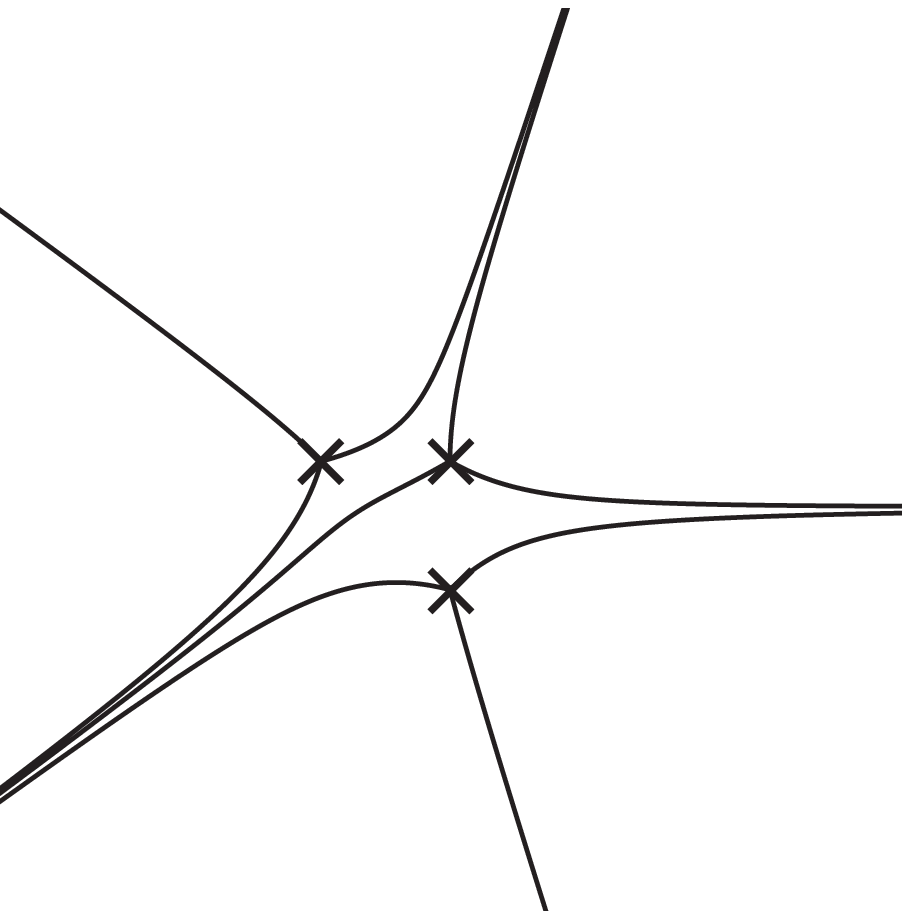} \\
  {(b) $z(z+1)(z+i)$.} 
  \end{center}
  \label{fig:cubic}
  \end{minipage}  
  \begin{minipage}{0.33\hsize}
  \begin{center}
  \includegraphics[width=40mm]
  {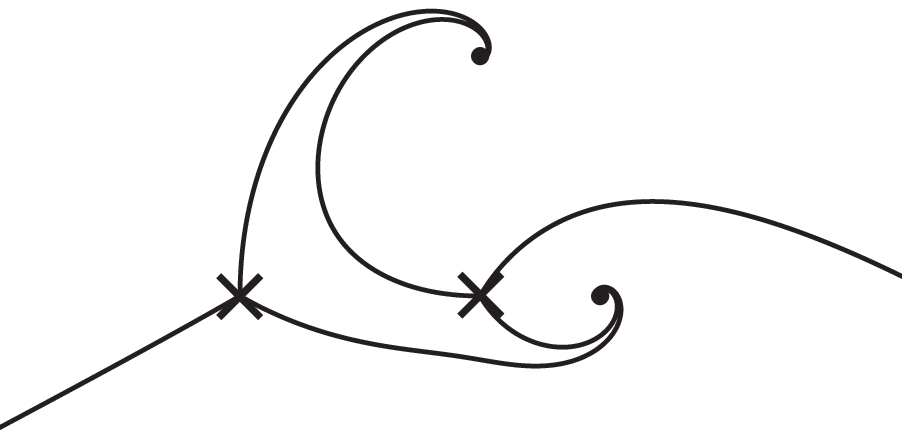} \\[+.5em]
  {(c) $\displaystyle \frac{z(z+2)}
  {(z-1)^{2}(z-2i)^{2}}$.} 
  \end{center}
  \label{fig:Weber-generic}
  \end{minipage}  
  \begin{minipage}{0.33\hsize}
  \begin{center}
  \includegraphics[width=40mm]
  {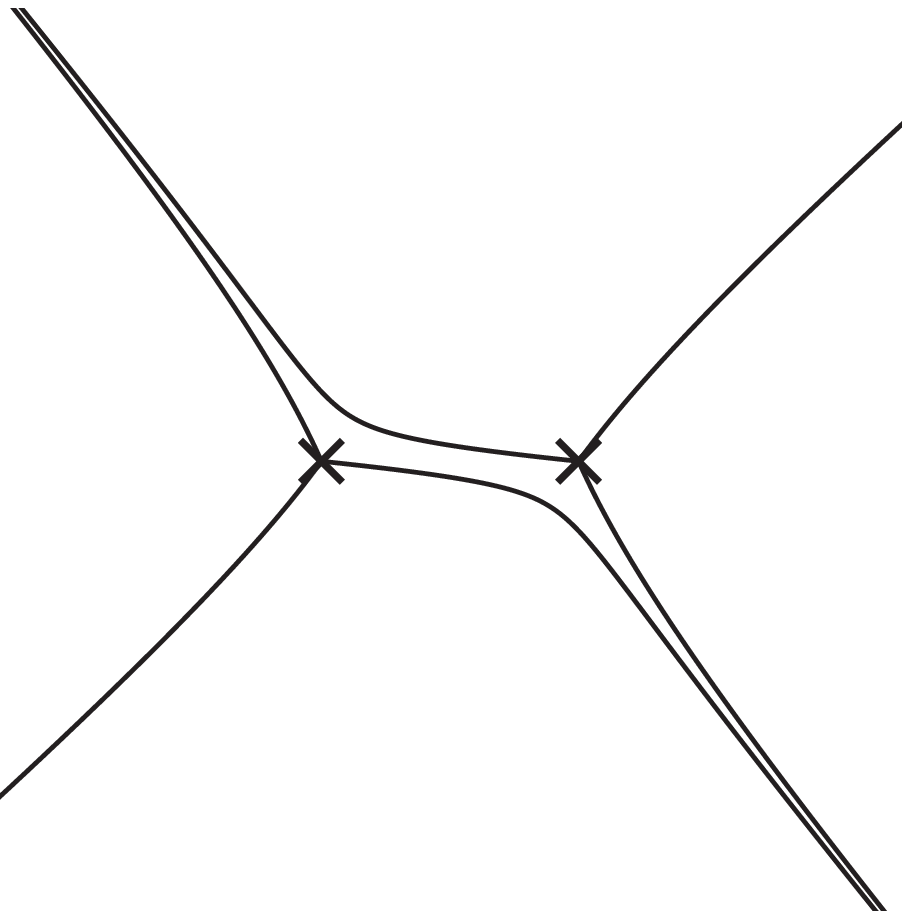} \\
  {(d) $e^{+{\pi i}/{10}}(1-z^{2})$.} 
  \end{center}
  \label{fig:Weber-minus}
  \end{minipage} 
  \begin{minipage}{0.33\hsize}
  \begin{center}
  \includegraphics[width=40mm]
  {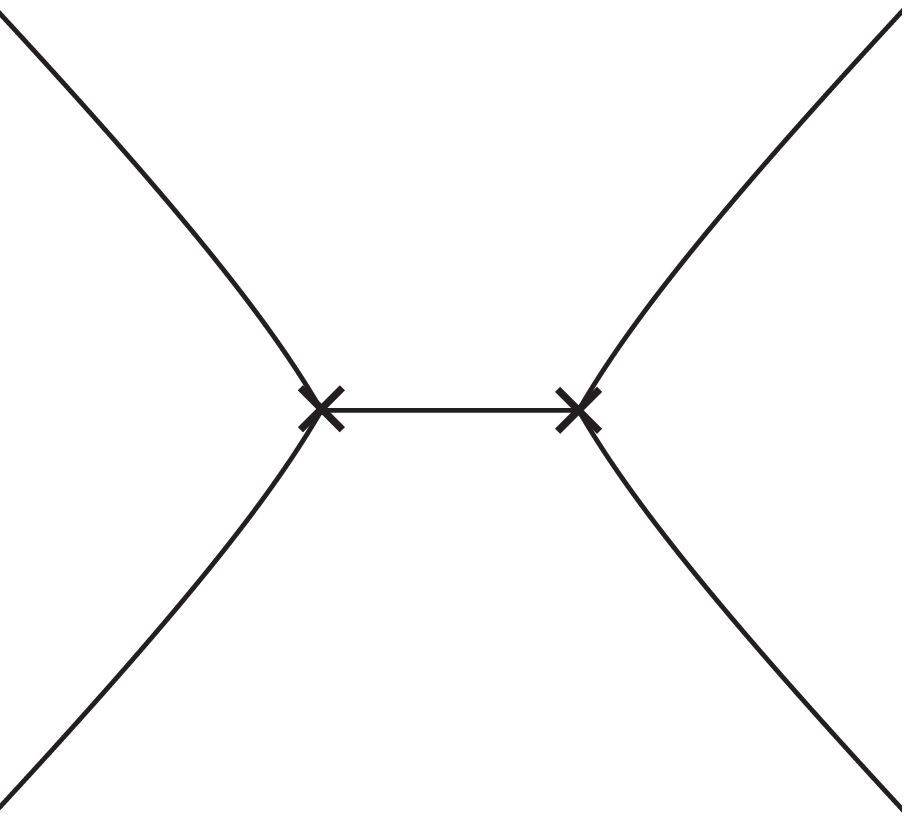} \\
  {(e) $1-z^{2}$.} 
  \end{center}
  \label{fig:Weber-0+.eps}
  \end{minipage} 
  \begin{minipage}{0.33\hsize}
  \begin{center}
  \includegraphics[width=40mm]
  {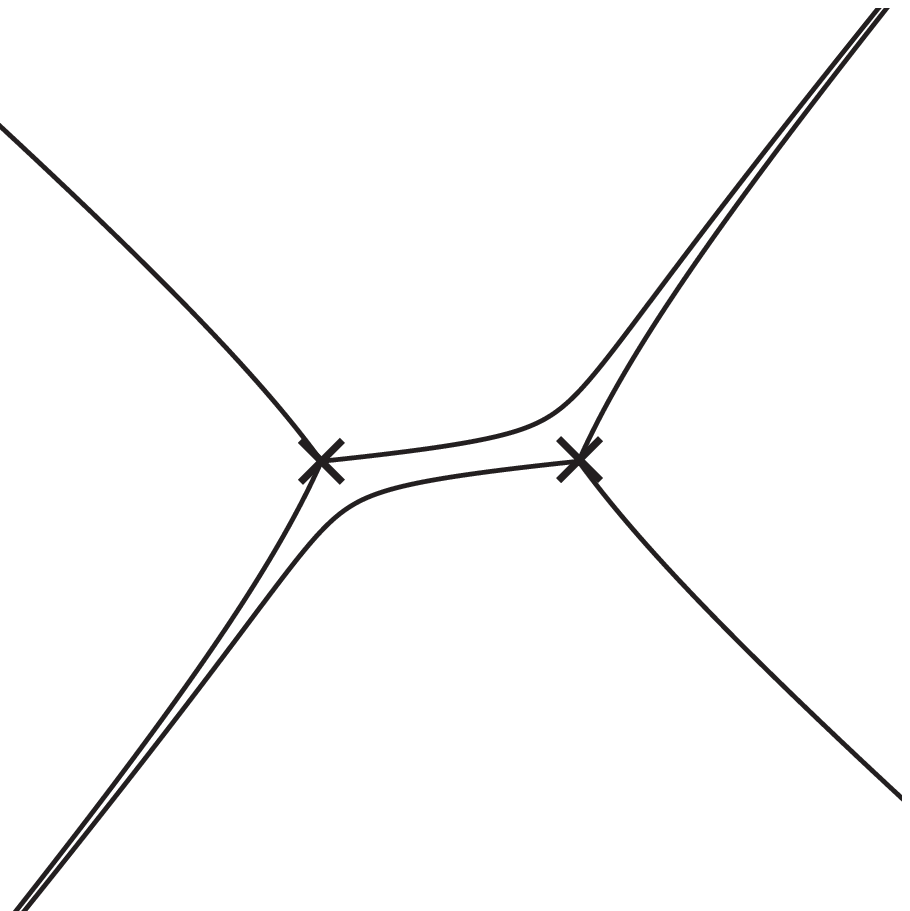} \\
  {(f) $e^{-\pi i/10}(1-z^{2})$.}
  \end{center}
  \label{fig:Weber-plus}
  \end{minipage}   
  \begin{minipage}{0.33\hsize}
  \begin{center}
  \includegraphics[width=40mm]{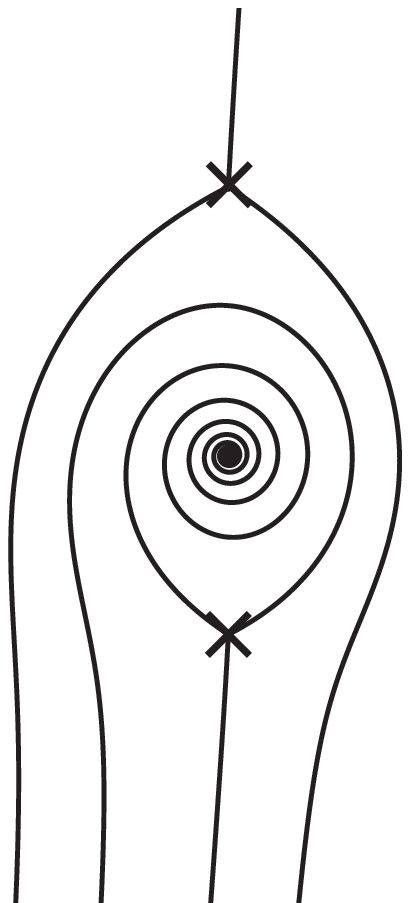} \\
  {(g) $\displaystyle -e^{+{\pi i}/{20}}
  \frac{(z+2i)(z-3i)}{z^{2}}$.}
  \end{center}
  \label{fig:loop-plus}
  \end{minipage} 
  \begin{minipage}{0.33\hsize}
  \begin{center}
  \includegraphics[width=40mm]{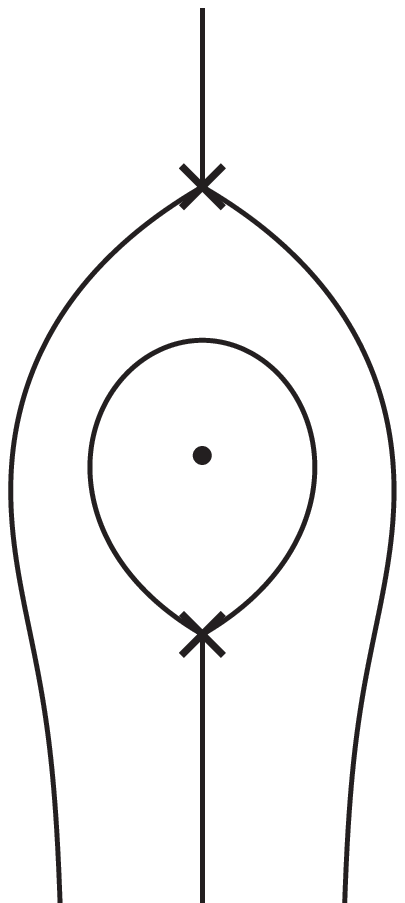} \\
  {(h)  $\displaystyle -\frac{(z+2i)(z-3i)}{z^{2}}$.} 
  \end{center}
  \label{fig:loop-0}
  \end{minipage} 
  \begin{minipage}{0.33\hsize}
  \begin{center}
  \includegraphics[width=40mm]{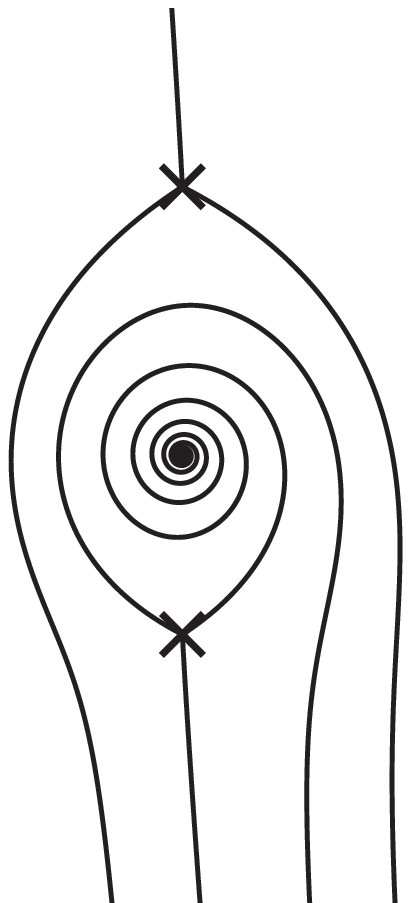} \\
  {(i)  $\displaystyle -e^{-{\pi i}/{20}}
  \frac{(z+2i)(z-3i)}{z^{2}}$.}
  \end{center}
  \label{fig:loop-minus}
  \end{minipage}   
  \caption{Examples of Stokes graphs. 
  The rational functions represent the function $Q_0(z)$.}  
  \label{fig:examples-of-Stokes-graphs}
 \end{figure}

Here we recall some basic properties of the trajectories 
of $\phi$ from \cite{Strebel84}. 
See also \cite[Section 3]{Bridgeland13} for comprehensible expositions. 
Firstly, the local foliation structure 
around simple zeros and poles of order $m\geq 2$ 
are given below and depicted in Figures 
\ref{fig:foliation1}--\ref{fig:foliation3}.
For a simple zero $a$, there are {\em exactly three\/} 
trajectories entering $a$ which are the Stokes curves 
(Figure \ref{fig:foliation1}).
For a double pole $p$, there are three cases 
depending on the residue 
$r_p = \operatornamewithlimits{Res}_{z=p} 
\sqrt{Q_0(z)}\, dz$ 
(Figure \ref{fig:foliation2}).
  \begin{itemize}
\item[(a).]
Clockwise or counterclockwise logarithmic spirals 
wrap onto $p$. This occurs when 
$r_p \notin {\mathbb R} \cup i{\mathbb R}$. 
\item[(b).] Radial arcs entering $p$. 
This occurs when $r_p \in {\mathbb R}$. 
\item[(c).] Closed trajectories surround $p$. 
This occurs when $r_p \in i{\mathbb R}$. 
  \end{itemize}
For a pole $p$ of order $m\geq 3$, there are 
{\em exactly $m-2$} asymptotic tangent directions 
for the trajectories entering $p$
  (Figure \ref{fig:foliation3}).

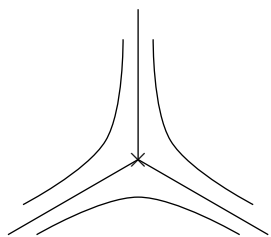
\begin{figure}
\begin{center}
\begin{pspicture}(-1.73,-1)(1.73,2)
\psset{linewidth=0.5pt}
%
\psline(0,0)(0,2)
\psline(0,0)(1.73,-1)
\psline(0,0)(-1.73,-1)
%
\pscurve(1.35,-1)(0,-0.5)(-1.35,-1)
\pscurve(0.2,1.6)(0.4325,0.25)(1.53,-0.6)
\pscurve(-0.2,1.6)(-0.4325,0.25)(-1.53,-0.6)
\rput[c]{0}(0,0){\small $\times$}
\end{pspicture}
\end{center}
\caption{Foliation around a simple zero.}
\label{fig:foliation1}
\end{figure}

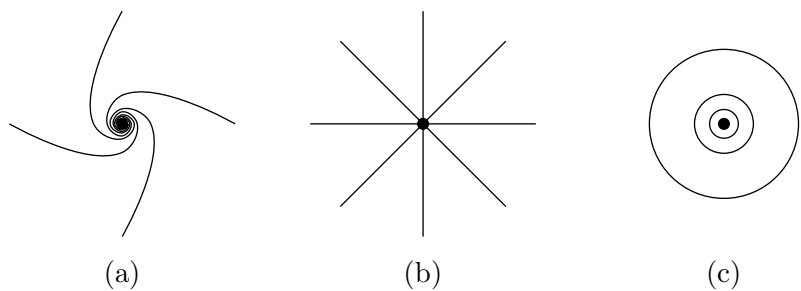
\begin{figure}
\begin{center}
\begin{pspicture}(-1.5,-2.2)(9.5,1.5)
\psset{linewidth=0.5pt}
\psset{fillstyle=solid, fillcolor=black}
\pscircle(0,0){0.08} 
\psset{fillstyle=none}
%
%
\psecurve(1.5,0)(1.5,0)(0,0.38)(-0.2,0)
(0,-0.16)(0.12,0)(0,0.1)(-0.08,0)(0,-0.06)
\psecurve(-1.5,0)(-1.5,0)(0,-0.38)(0.2,0)
(0,0.16)(-0.12,0)(0,-0.1)(0.08,0)(0,0.06)
\psecurve(0,-1.5)(0,-1.5)(0.38,0)(0,0.2)
(-0.16,0)(0,-0.12)(0.1,0)(0,0.08)(-0.06,0)
\psecurve(0,1.5)(0,1.5)(-0.38,0)(0,-0.2)
(0.16,0)(0,0.12)(-0.1,0)(0,-0.08)(0.06,0)
\rput[c]{0}(0,-2){\small (a)}

\psset{fillstyle=solid, fillcolor=black}
\pscircle(4,0){0.08} 
\psset{fillstyle=none}
\psline(5.5,0)(2.5,0)
\psline(4,1.5)(4,-1.5)
\psline(5.1,1.1)(2.9,-1.1)
\psline(5.1,-1.1)(2.9,1.1)
\rput[c]{0}(4,-2){\small (b)}
\psset{fillstyle=solid, fillcolor=black}
\pscircle(8,0){0.08} 
\psset{fillstyle=none}
\pscircle(8,0){0.2}
\pscircle(8,0){0.4}
\pscircle(8,0){1}
\rput[c]{0}(8,-2){\small (c)}
\end{pspicture}
\end{center}
\caption{Patterns of foliation around a double pole.}
\label{fig:foliation2}
\end{figure}

\begin{figure}
\begin{center}
\begin{pspicture}(-1.5,-1.5)(1.5,1.5)
\psset{linewidth=0.5pt}
\psset{fillstyle=solid, fillcolor=black}
\pscircle(0,0){0.08} 
\psset{fillstyle=none}
\psline(1.5,0)(-1.5,0)
\psline(0,1.5)(0,-1.5)
%
\pscurve(0,0)(0.08,0.64)(0.4,0.86)
(0.76,0.76)(0.86,0.4)(0.64,0.08)(0,0)
\pscurve(0,0)(0.04,0.32)(0.2,0.43)
(0.38,0.38)(0.43,0.2)(0.32,0.04)(0,0)
\pscurve(0,0)(0.08,-0.64)(0.4,-0.86)
(0.76,-0.76)(0.86,-0.4)(0.64,-0.08)(0,0)
\pscurve(0,0)(0.04,-0.32)(0.2,-0.43)
(0.38,-0.38)(0.43,-0.2)(0.32,-0.04)(0,0)
\pscurve(0,0)(-0.08,0.64)(-0.4,0.86)
(-0.76,0.76)(-0.86,0.4)(-0.64,0.08)(0,0)
\pscurve(0,0)(-0.04,0.32)(-0.2,0.43)
(-0.38,0.38)(-0.43,0.2)(-0.32,0.04)(0,0)
\pscurve(0,0)(-0.08,-0.64)(-0.4,-0.86)
(-0.76,-0.76)(-0.86,-0.4)(-0.64,-0.08)(0,0)
\pscurve(0,0)(-0.04,-0.32)(-0.2,-0.43)
(-0.38,-0.38)(-0.43,-0.2)(-0.32,-0.04)(0,0)
\end{pspicture}
\end{center}
\caption{Foliation  around a pole of order 
  $m\geq 3$. The case $m=6$ is shown.}
\label{fig:foliation3}
\end{figure}
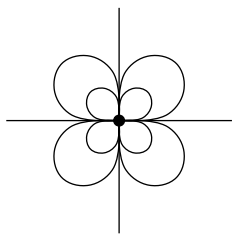

Secondly, we focus on global properties of the
trajectories of $\phi$. It is known that every 
trajectories fall into exactly 
one of the following five types 
(\cite[Section 3.4]{Bridgeland13}): 
\begin{itemize}
\item[(a).] %
A {\em saddle trajectory} flows into points 
in $P_0$ at both ends. 
\item[(b).] %
A {\em separating trajectory} flows into a point in $P_0$ 
at one end, and a point in $P_{\infty}$ at the other end.
\item[(c).] %
A {\em generic trajectory} flows into points 
in $P_{\infty}$ at both ends.
\item[(d).] %
A {\em closed trajectory} is a simple closed curve 
in $\Sigma \setminus P$.
\item[(e).] %
A {\em divergent trajectory\/} has the limit set 
consisting of more than one point
in at least one direction.
\end{itemize}
A Stokes curve is one of a saddle trajectory, a separating trajectory,
and a divergent trajectory.  

Saddle trajectories play important roles in this article. 
Typically, there are two kinds of saddle trajectories: 
\begin{itemize}
\item[(a).] %
A {\em regular saddle trajectory} connects 
two different points in $P_0$. 
An example appears in 
Figure \ref{fig:examples-of-Stokes-graphs} (e).
\item[(b).] %
A {\em degenerate saddle trajectory} 
forms a loop around a double pole $p \in P_{\infty}$. 
An example appears in 
Figure \ref{fig:examples-of-Stokes-graphs} (h).
\end{itemize} 
In addition to degenerate saddle trajectories, 
other kinds of loop-type saddle trajectories may appear. 
For example, a Stokes curve emanating 
from $a \in P_0$ may return to the same point $a$ 
after encircling  several points in $P$. 
Such an example is discussed in 
\cite[Section 10]{Gaiotto09}, 
but we will not consider these cases. 
In this paper we will concentrate on the following cases:
\begin{ass} 
\label{assumption:at-most-one-saddle}
The number of the saddle trajectories of $\phi$ is at most one.
\end{ass} 

%
%
%
%
%
%

Under Assumptions \ref{assumption:zeros and poles} 
and \ref{assumption:at-most-one-saddle}, 
a saddle trajectory must be either 
a regular or a degenerate saddle trajectory 
(see \cite[Proposition 10.4]{Bridgeland13}). 
Moreover, we can show that divergent 
trajectories never appear in this case.
\begin{lem} \label{lemma:no-divergent-trajectory}
Under Assumptions \ref{assumption:zeros and poles} and 
\ref{assumption:at-most-one-saddle}, 
$\phi$ has no divergent trajectories. 
\end{lem}
\begin{proof}
If $\phi$ does not have any saddle trajectory, 
then the statement is proved in 
\cite[Lemma 3.1]{Bridgeland13}. 
Assume that $\phi$ has a unique saddle trajectory. 
If a divergent trajectory appears, the interior of 
the closure of the divergent trajectory gives a domain 
called a ``spiral domain". It is known that 
the boundary of such a spiral domain must consist 
of a number of saddle trajectories
(see \cite[Section 3.4]{Bridgeland13}). 
Since we have assumed that the number of saddle trajectories 
is exactly one, a domain whose boundary consists of 
saddle trajectories must be a ``degenerate ring domain" 
(see \cite[Section 3.4]{Bridgeland13} or below). 
Then we have a contradiction because any trajectory 
in a degenerate ring domain must be a closed trajectory, 
which is not a divergent trajectory.  
\end{proof}

Therefore, under Assumptions 
\ref{assumption:zeros and poles} and 
\ref{assumption:at-most-one-saddle}, 
a Stokes curve must be a saddle trajectory or 
a separating trajectory. In other words, 
a Stokes curve emanating from a 
turning point must flows into a point in $P$, 
and these objects define a graph on $\Sigma$.
\begin{defn}[{\cite[Definition 2.10]{Kawai05}}]
\begin{itemize}
\item %
The {\it Stokes graph of ${\mathcal L}$} 
is a graph in $\Sigma$ whose vertices 
are the points in $P$, and whose edges are the Stokes curves 
of ${\mathcal L}$. The Stokes graph is denoted by $G$.
\item %
The interior of each face of the Stokes graph $G$
is called a {\em Stokes region} of $G$. 
\end{itemize}
\end{defn}

We sometimes write $G = G(\phi)$ for the Stokes graph 
and call it the {\em Stokes graph of $\phi$}
when we want to emphasize the dependence on $\phi$. 
If the Stokes graph $G$ does not 
have any saddle trajectory, $G$ 
is said to be {\em saddle-free}, and then
$\phi$ is also said to be {saddle-free}, 
following \cite[Section 3.5]{Bridgeland13}.
Under Assumptions \ref{assumption:zeros and poles} 
and \ref{assumption:at-most-one-saddle},
the Stokes regions of $G$ are classified as follows
(\cite[Section 3.4]{Bridgeland13}):
\begin{itemize}
\item[(a).] %
A {\em horizontal strip} is equivalent to a region 
\[
\{ w \in {\mathbb C}~|~ a < {\rm Im}(w) < b \} 
\quad(a, b \in {\mathbb R})
\] 
equipped with the differential $dw^{\otimes2}$.
It is swept out by generic trajectories which connect 
two (not necessarily distinct) 
poles of arbitrary order $m \ge 2$.
\item[(b).] %
A {\em half plane} is equivalent to the upper half plane 
\[
\{ w \in {\mathbb C}~|~ 0 < {\rm Im}(w)  \} 
\]
equipped with the differential $dw^{\otimes2}$.
It is swept out by generic trajectories which connect 
a fixed pole of order $m \ge 3$ to itself.
\item[(c).] %
A {\em degenerate ring domain} is equivalent to a region 
\[
\{ w \in {\mathbb C}~|~ 0 < |w| < a \} 
\quad(a \in {\mathbb R})
\]
equipped with the differential 
$r dw^{\otimes2}/w^2$ 
for some $r \in {\mathbb R}_{<0}$.
It is swept out by closed trajectories, and its 
boundary consists of a degenerate saddle trajectory 
and the double pole lying inside of the 
degenerate saddle trajectory. 
\end{itemize}
For example,  all three Stokes regions in 
Figure \ref{fig:examples-of-Stokes-graphs} (a) are 
half planes. On the other hand,  all three Stokes 
regions in Figure \ref{fig:examples-of-Stokes-graphs} (c) 
are horizontal strips. 
In Figure \ref{fig:examples-of-Stokes-graphs} (b) 
there are five half planes near $z=\infty$ and 
two horizontal strips. An example of a degenerate 
ring domain can be found in Figure 
\ref{fig:examples-of-Stokes-graphs} (h).

In Section \ref{section:Borel-summability} 
we will explain the relationship between 
the geometry of the trajectories of $\phi$ and the 
Borel summability of the WKB solutions. 

In the subsequent discussions we will consider 
not only the usual Borel resummation 
but also the Borel resummation in a direction 
$\theta \in {\mathbb R}$ as explained 
in Section \ref{section:Borel-resummation}. 
{ 
Lemma \ref{lemma:0-sum-and-theta-sum} shows that 
the Borel summability of the formal power 
series $S^{\rm reg}_{\rm odd}(z,\eta)$ in the direction 
$\theta$ is equivalent to the Borel summability of 
$S^{\rm reg}_{\rm odd}(z,e^{i\theta}\eta)$.
Actually, $S^{\rm reg}_{\rm odd}(z,e^{i\theta}\eta)$ 
coincides with the formal power series 
\eqref{eq:Soddreg} defined from 
the Schr{\"o}dinger equation 
\begin{equation} 
\left( \frac{d^2}{dz^2} - \eta^2 
e^{2 i \theta}Q(z,e^{i\theta}\eta) \right) \psi = 0.
\end{equation}
(See Lemma \ref{lemma:Sodd-theta-and-Sodd} below.)
Therefore, the Borel summability of 
$S^{\rm reg}_{\rm odd}(z,\eta)$
(and of the WKB solutions) 
in the direction $\theta$ is relevant to 
the geometry of trajectories of 
the quadratic differential 
\begin{equation} \label{eq:phi-theta}
\phi_\theta = e^{2 i \theta} \phi.
\end{equation} 
}
Here $\phi$ is the original 
quadratic differential associated with ${\mathcal L}$.
Since the quadratic differential 
$\phi_{\theta}$ also satisfies 
Assumption \ref{assumption:zeros and poles}, 
trajectories of $\phi_{\theta}$ have the same 
properties explained in this subsection. 
Define the {\em Stokes curves in the direction $\theta$} 
emanating from a turning point $a \in P_0$ by 
\begin{equation}
{\rm Im}\left(e^{i\theta} 
\int_a^z \sqrt{Q_0(z)}dz \right) = 0,
\end{equation}
and also define the
{\em Stokes graph in the direction $\theta$} 
by the graph consists of the Stokes curves in 
the direction $\theta$ and the points in $P$. 
The Stokes graph in the direction $\theta$ is 
denoted by $G_{\theta}$ $(= G(\phi_{\theta}))$.

If we vary the direction $\theta$ continuously, 
the topology of the Stokes graph $G_{\theta}$ 
changes
when a saddle trajectory appears. 
Let us explain the phenomenon for an example 
$\phi_{\theta} = e^{2i\theta}(1-z^2)dz^{\otimes2}$ 
defined on ${\mathbb P}^1$
(see Figure \ref{fig:examples-of-Stokes-graphs} (d)--(f)). 
If $\theta \ne 0$ and $|\theta|$ is sufficiently small, 
there are five Stokes regions; one is a horizontal strip 
and the other four are half planes 
(see in Figure \ref{fig:examples-of-Stokes-graphs} (d), (f)). 
As we vary $\theta$ continuously, the Stokes graph 
deforms continuously  as long as $\theta \ne 0$. 
However, when $\theta = 0$, the horizontal strip 
disappears from the Stokes graph and the number of 
Stokes regions becomes four; all  Stokes regions 
are half planes as shown in Figure 
\ref{fig:examples-of-Stokes-graphs} (e).  
Moreover, the topologies of the Stokes graphs $G_{\theta}$ 
for $\theta > 0$ and $\theta < 0$ are different. 
A similar change of the topology  is also 
observed when a degenerate saddle trajectory appears 
(see Figure \ref{fig:examples-of-Stokes-graphs} (g)--(i)).    
These are typical examples of
the phenomenon which we call the {\em mutation} of Stokes graphs. 
The mutation of Stokes graphs is the theme of the paper.

\subsection{Orientation of trajectories}
\label{section:orientation}


The inverse image of 
the foliation \eqref{eq:theta-foliation} 
in $\Sigma\setminus{P}$ by the projection $\pi$ 
defines a foliation on 
$\hat{\Sigma}\setminus\pi^{-1}{P}$.
For a trajectory $\beta$ in 
$\Sigma$,
we  call each lift of  $\beta$ in $\hat{\Sigma}$
by $\pi$
a {\em trajectory in $\hat{\Sigma}$}.
Since the 1-form defined by $\sqrt{Q_0(z)}dz$ 
is single-valued on $\hat{\Sigma}$, trajectories in 
$\hat{\Sigma}\setminus\pi^{-1}{P}$ has the 
{\em orientation} defined by the following rule; 
the real part of the function 
$\int^{z} \sqrt{Q_{0}(z)}dz$ increases 
along the trajectory in the positive direction. 
Since the covering involution $\tau$ reverses 
the sign of $\sqrt{Q_0(z)}$, the orientation 
of a trajectories in $\hat{\Sigma}$ is also 
reversed by $\tau$. Figure \ref{fig:orientation} 
depicts examples of the orientation in the first sheet, 
projected to $\Sigma$ by $\pi$. The orientation 
is well-defined on $\hat{\Sigma}$, but its 
projection has a discontinuity on the branch cut. 
When we discuss the orientation, we assign the symbols 
$\oplus$ and $\ominus$ to the asymptotic directions 
of trajectories entering points of $P_{\infty}$ so that 
the trajectories with positive directions flows from 
$\ominus$ to $\oplus$. These signs depend on the choice 
of the branch cuts and embedding $\iota$, and the covering 
involution $\tau$ exchanges all signs simultaneously. 

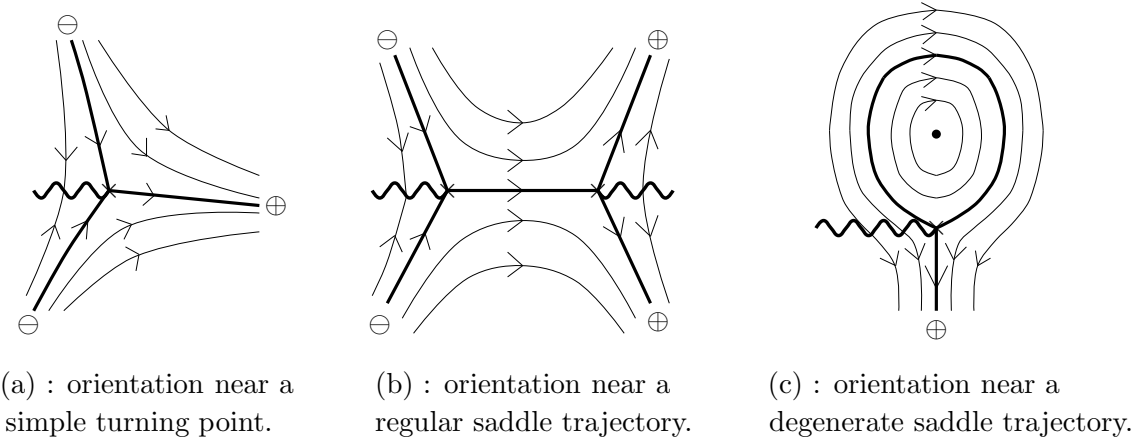
\begin{figure}
\begin{center}
\begin{pspicture}(0.5,-3)(16,3)
%
%
\psset{fillstyle=solid, fillcolor=black}
\pscircle(13,0.75){0.06}
\psset{fillstyle=none}
\rput[c]{0}(2,0){\small $\times$}
\rput[c]{0}(4.22,-0.2){\small $\oplus$}
\rput[c]{0}(1.45,2.2){\small $\ominus$}
\rput[c]{0}(0.93,-1.78){\small $\ominus$}
\rput[c]{0}(2.5,-2.6){\small (a) : orientation near a} 
\rput[c]{0}(2.4,-3.1){\small simple turning point.}
\rput[c]{0}(6.5,0){\small $\times$}
\rput[c]{0}(8.5,0){\small $\times$}
\rput[c]{0}(5.7,2.0){\small $\ominus$}
\rput[c]{0}(5.6,-1.78){\small $\ominus$}
\rput[c]{0}(9.3,-1.75){\small $\oplus$}
\rput[c]{0}(9.3,+2.0){\small $\oplus$}
\rput[c]{0}(7.5,-2.6){\small (b) : orientation near a} 
\rput[c]{0}(7.65,-3.1){\small regular saddle trajectory.}
\rput[c]{0}(13,-0.5){\small $\times$}
\rput[c]{0}(13,-1.85){\small $\oplus$}
\rput[c]{0}(12.7,-2.6){\small (c) : orientation near a} 
\rput[c]{0}(13.2,-3.1){\small degenerate saddle trajectory.}
\psset{linewidth=1.3pt}
\pscurve(2,0)(1.9,-0.1)(1.8,0)(1.7,0.1)
(1.6,0)(1.5,-0.1)(1.4,0)(1.3,0.1)(1.2,0)
(1.1,-0.1)(1,0)
\pscurve(6.5,0)(6.4,-0.1)(6.3,0)(6.2,0.1)
(6.1,0)(6.0,-0.1)(5.9,0)(5.8,0.1)(5.7,0)
(5.6,-0.1)(5.5,0)
\pscurve(8.5,0)(8.6,-0.1)(8.7,0)(8.8,0.1)
(8.9,0)(9.0,-0.1)(9.1,0)(9.2,0.1)(9.3,0)
(9.4,-0.1)(9.5,0)
\pscurve(13,-0.5)(12.9,-0.6)(12.8,-0.5)
(12.7,-0.4)(12.6,-0.5)(12.5,-0.6)(12.4,-0.5)
(12.3,-0.4)(12.2,-0.5)(12.1,-0.6)(12,-0.5)
(11.9,-0.4)(11.8,-0.5)(11.7,-0.6)(11.6,-0.5)
(11.5,-0.4)(11.4,-0.5)
\psset{linewidth=1.2pt}
\pscurve(2,0)(3,-0.10)(4,-0.2)
\pscurve(2,0)(1.65,1.5)(1.5,2)
\pscurve(2,0)(1.45,-0.8)(1,-1.6)
\psline(6.5,0)(8.5,0)
\psline(6.5,0)(5.8,1.8)
\psline(6.5,0)(5.7,-1.5)
\psline(8.5,0)(9.2,1.8)
\psline(8.5,0)(9.2,-1.5)
\psline(13,-0.5)(13,-1.6)
\pscurve(13,-0.5)(13.2,-0.4)(13.5,-0.2)(13.6,-0.1)
(13.8,0.2)(13.9,0.7)(13.9,0.9)(13.8,1.3)(13.7,1.5)
\pscurve(13.7,1.5)(13.6,1.6)(13.5,1.66)(13.4,1.71)
(13.3,1.75)(13,1.8)
\pscurve(13,-0.5)(12.8,-0.4)(12.5,-0.2)(12.4,-0.1)
(12.2,0.2)(12.1,0.7)(12.1,0.9)(12.2,1.3)(12.3,1.5)
\pscurve(12.3,1.5)(12.4,1.6)(12.5,1.66)(12.6,1.71)
(12.7,1.75)(13,1.8)
\psset{linewidth=0.2pt}
\pscurve(1.65,2)(2.4,0.55)(4,-0.1)
\pscurve(1.86,2)(2.8,0.8)(4,0.2)
\pscurve(1.3,2)(1.4,0)(0.9,-1.4)
\pscurve(1.2,-1.6)(2.2,-0.5)(4,-0.3)
\pscurve(1.4,-1.6)(2.3,-0.9)(4,-0.55)
\psline(2.6,-0.07)(2.5,0.1)\psline(2.6,-0.07)(2.45,-0.2)
\psline(1.88,0.6)(1.97,0.8)\psline(1.88,0.6)(1.68,0.72)
\psline(1.72,-0.4)(1.7,-0.65)\psline(1.72,-0.4)(1.5,-0.5)
\psline(2.5,0.45)(2.5,0.7)\psline(2.5,0.46)(2.3,0.49)
\psline(2.8,0.8)(2.76,1)\psline(2.8,0.8)(2.62,0.82)
\psline(2.3,-0.45)(2.1,-0.38)\psline(2.3,-0.45)(2.2,-0.65)
\psline(2.4,-0.85)(2.22,-0.75)\psline(2.4,-0.85)(2.36,-1.07)
\psline(1.28,-0.46)(1.1,-0.6)\psline(1.28,-0.46)(1.35,-0.65)
\psline(1.43,0.4)(1.58,0.58)\psline(1.43,0.4)(1.28,0.57)
\pscurve(5.5,-1.4)(5.68,-1)(5.95,-0.1)
\pscurve(5.6,1.7)(5.85,0.8)(5.955,-0.1)
\pscurve(6.0,1.9)(6.7,0.65)(7.5,0.4)(8.2,0.6)(9,1.9)
\pscurve(6.3,1.9)(6.9,1.1)(7.5,0.9)(8.0,1.1)(8.7,1.9)
\pscurve(5.9,-1.7)(6.7,-0.65)(7.5,-0.4)(8.2,-0.6)(9,-1.7)
\pscurve(6.05,-1.9)(6.8,-1.15)(7.5,-1)(8.2,-1.2)(8.85,-1.9)
\pscurve(9.45,-1.5)(9.10,0)(9.45,1.75)
\psline(6.2,0.75)(6.27,1.0)\psline(6.2,0.78)(6.,0.85)
\psline(6.2,-0.56)(6.27,-0.79)\psline(6.2,-0.57)(6.,-0.65)
\psline(5.9,0.52)(5.7,0.7)\psline(5.9,0.52)(5.98,0.77)
\psline(5.8,-0.6)(5.6,-0.75)\psline(5.8,-0.6)(5.9,-0.82)
\psline(9.2,0.8)(9.0,0.6)\psline(9.2,0.8)(9.3,0.55)
\psline(9.2,-0.61)(9.0,-0.42)\psline(9.2,-0.61)(9.3,-0.33)
\psline(7.5,0)(7.3,0.15)\psline(7.5,0)(7.3,-0.15)
\psline(7.5,0.4)(7.3,0.55)\psline(7.5,0.4)(7.3,0.25)
\psline(7.5,-0.4)(7.3,-0.55)\psline(7.5,-0.4)(7.3,-0.25)
\psline(7.5,0.9)(7.3,1.07)\psline(7.5,0.9)(7.3,0.75)
\psline(7.5,-1)(7.3,-1.15)\psline(7.5,-1)(7.3,-0.82)
\psline(8.8,0.8)(8.85,0.50)\psline(8.8,0.8)(8.57,0.6)
\psline(8.8,-0.62)(8.85,-0.35)\psline(8.8,-0.62)(8.57,-0.48)
\pscurve(13,-0.1)(13.2,-0.05)(13.48,0.2)(13.52,0.3)
(13.6,0.75)(13.5,1.2)(13.4,1.38)(13,1.5)
\pscurve(13,-0.1)(12.8,-0.05)(12.52,0.2)(12.48,0.3)
(12.4,0.75)(12.5,1.2)(12.6,1.38)(13,1.5)
\pscurve(13,0.2)(13.2,0.28)(13.28,0.4)(13.35,0.8)
(13.3,1)(13.2,1.15)(13,1.2)
\pscurve(13,0.2)(12.8,0.28)(12.72,0.4)(12.65,0.8)
(12.7,1)(12.8,1.15)(13,1.2)
\pscurve(13.2,-1.6)(13.3,-0.7)(13.5,-0.5)(13.7,-0.33)
(13.9,-0.13)(14,0)(14.1,0.3)(14.15,0.8)
\pscurve(14.15,0.8)(14.12,1.2)(14.0,1.5)(13.85,1.7)
(13.65,1.9)(13.3,2.05)(13,2.1)
\pscurve(12.8,-1.6)(12.7,-0.7)(12.5,-0.5)(12.3,-0.33)
(12.1,-0.13)(12,0)(11.9,0.3)(11.85,0.8)
\pscurve(11.85,0.8)(11.88,1.2)(12.0,1.5)(12.15,1.7)
(12.35,1.9)(12.7,2.05)(13,2.1)
\pscurve(13.5,-1.6)(13.55,-1)(13.8,-0.65)(13.9,-0.55)
(14,-0.45)(14.2,-0.2)(14.25,-0.1)(14.4,0.5)(14.42,0.8)
\pscurve(14.42,0.8)(14.27,1.6)(13.8,2.21)(13,2.4)
\pscurve(12.5,-1.6)(12.45,-1)(12.2,-0.65)(12.1,-0.55)
(12,-0.45)(11.8,-0.2)(11.75,-0.1)(11.6,0.5)(11.58,0.8)
\pscurve(11.58,0.8)(11.73,1.6)(12.2,2.21)(13,2.4)
\psline(13,-1.3)(13.15,-1)\psline(13,-1.3)(12.85,-1)
\psline(12.46,-1)(12.25,-0.9)\psline(12.46,-1)(12.50,-0.8)
\psline(12.76,-0.92)(12.58,-0.77)\psline(12.76,-0.92)(12.85,-0.7)
\psline(13.23,-0.92)(13.12,-0.7)\psline(13.23,-0.92)(13.4,-0.75)
\psline(13.55,-1)(13.52,-0.8)\psline(13.55,-1)(13.72,-0.9)
\psline(13,1.8)(12.83,1.67)\psline(13,1.8)(12.8,1.9)
\psline(13,2.1)(12.82,1.95)\psline(13,2.1)(12.8,2.2)
\psline(13,2.4)(12.82,2.27)\psline(13,2.4)(12.8,2.5)
\psline(13,1.5)(12.83,1.37)\psline(13,1.5)(12.8,1.57)
\psline(13,1.2)(12.85,1.06)\psline(13,1.2)(12.8,1.25)
\end{pspicture}
\end{center}
\caption{Examples of orientation of trajectories.} 
\label{fig:orientation}
\end{figure}

\subsection{Borel summability of WKB solutions}
\label{section:Borel-summability} 

Now we claim an important result concerning with the 
Borel summability of the WKB solutions for a  
{\em fixed\/} direction $\theta \in {\mathbb R}$. 
Note that, setting $\theta = 0$ in the following claims, 
we obtain results for the ``usual" Borel summability 
(Definition \ref{def:Borel-summability}).

Let $\phi$ be the quadratic differential associated with 
$\mathcal{L}$, and assume that 
$\phi_{\theta}=e^{2i\theta}\phi$ has at most one
saddle trajectory.
Let $G_{\theta}$ be the Stokes graph 
in the direction $\theta$ 
in Section \ref{section:Stokes-graphs}.
Take any Stokes region $D$ of $G_{\theta}$.  
Recall that $D$ must be one of a horizontal strip, 
a half plane or a degenerate ring domain. 
Fix a local coordinate $z$ of $\Sigma$ 
whose domain contains the Stokes region $D$. 
Recently, Koike and Sch\"afke proved the following 
statement which ensures the Borel summability 
of the formal power series 
$S_{\rm odd}^{\rm reg}(z,\eta)$ on each Stokes region
when $\Sigma$ is the Riemann sphere ${\mathbb P}^1$.
\begin{thm}[\cite{Koike13}]
\label{thm:summability-P1}
Assume that $\Sigma$ is the Riemann sphere $\bbP^1$, 
and the coefficients $\{ Q_{i}(z) \}_{i=0}^n$ 
of the potential function of \eqref{eq:Sch}
are meromorphic functions satisfying 
Assumption \ref{assumption:summability}. 
Let $G_{\theta}$ and $D$ be as above.
\begin{itemize}
\item[(a).] 
For any fixed $z \in D$, the formal power series 
$S_{\rm odd}^{\rm reg}(z,\eta)$ is Borel summable 
in the direction $\theta$ as a formal power series 
in $\eta^{-1}$. The Borel sum of 
$S_{\rm odd}^{\rm reg}(z,\eta)$ becomes holomorphic 
function of $z$ around the point in question 
(and also analytic in $\eta$ on 
$\{\eta \in {\mathbb C}~|~ |\arg\eta-\theta| < \pi/2,~
|\eta| \gg 1 \}$).
\item[(b).] 
Let $p \in P_{\infty}$ be any pole lying on 
the boundary of $D$. Then, for any fixed $z \in D$, 
the formal power series defined by the integral
\begin{equation} \label{eq:integral-from-p-to-z}
\int_{p}^z S_{\rm odd}^{\rm reg}(z,\eta)dz
\end{equation}
is Borel summable in the direction $\theta$ 
as a formal power series 
in $\eta^{-1}$ if 
the path of the integral \eqref{eq:integral-from-p-to-z} 
is contained in $D \cup \{ p \}$. 
The Borel sum becomes holomorphic function of $z$ 
around the point in question (and also analytic in $\eta$ on 
$\{\eta \in {\mathbb C}~|~ |\arg\eta-\theta| < \pi/2,~
|\eta| \gg 1 \}$). 
\end{itemize}
\end{thm}  

Actually, the above claim follows from the results 
of \cite{Koike13} and the fact that 
$S_{\rm odd}^{\rm reg}(z,\eta)$ is integrable at 
each pole (see Proposition \ref{proposition:integrability}). 
{ 
In \cite{Koike13} the above claim is proved 
in the case $\theta = 0$. The statement for general 
$\theta$ follows from the result of \cite{Koike13} 
together with Lemma \ref{lemma:0-sum-and-theta-sum}
(see Section \ref{section:Stokes-graphs}).
}
Although Theorem \ref{thm:summability-P1} is proved 
when $\Sigma=\bbP^1$ in \cite{Koike13}, 
their proof is also applicable to 
the case when $\Sigma$ is a compact Riemann 
surface $\Sigma$ since their proof uses only local properties of 
$\{ Q_{i}(z) \}_{i=0}^n$ in each Stokes region and the 
orders of poles lying on the boundary of $D$. 
Therefore, we can extend it to the following theorem.
\begin{thm} \label{thm:summability}
Theorem \ref{thm:summability-P1} also holds for any 
compact Riemann surface $\Sigma$.  
\end{thm}

Since Stokes regions are independent of the choice 
of the local coordinate, the notion of Borel summability 
is also independent of the choice. 
If $z$ lies on a Stokes curve in the direction $\theta$, 
the trajectory of $\phi_{\theta}$ passing 
through $z$ flows into a turning point at one end. 
The proof of Theorem \ref{thm:summability-P1} by 
\cite{Koike13} is not applicable to such a situation.

Next we discuss the Borel summability of the WKB solutions 
in a fixed direction $\theta \in {\mathbb R}$. 
Since the WKB solutions are defined 
by integrating $S_{\rm odd}(z,\eta)dz$ along a path 
on the Riemann surface $\hat{\Sigma}$, 
the Borel summability of the WKB solutions 
is more delicate than that of 
$S_{\rm odd}^{\rm reg}(z,\eta)$ explained above. 
To state the criterion of the Borel summability 
of the WKB solutions proposed by Koike and Sch\"afke, 
we introduce the notion of an admissible path. 
Set $\hat{P}_{0} = \pi^{-1}(P_{0})$, 
$\hat{P}_{\infty} = \pi^{-1}(P_{\infty})$ and 
$\hat{P} = \hat{P}_0 \cup \hat{P}_{\infty}$.
\begin{defn}
A path $\beta$ on $\hat{\Sigma}\setminus{\hat{P}_{0}}$ 
is said to be {\em admissible in the direction $\theta$}
if the projection of $\beta$ to $\Sigma$ by $\pi$ 
either never intersects with the Stokes graph 
$G_{\theta}$, or intersects with $G_{\theta}$ 
only at points in ${P}_{\infty}$.
\end{defn}
Especially, any generic trajectory and any 
closed trajectory of $\phi_{\theta}$ are admissible 
in the direction $\theta$. 
For a given path on $\hat{\Sigma}$ which is not 
admissible, we may find a decomposition of the path
into a finite number of admissible paths as follows.
\begin{lem}
\label{lemma:admissible-decomposition}
Let $\beta$ be a path on 
$\hat{\Sigma}\setminus{\hat{P}_0}$ with end-points
$\hat{z}_1, \hat{z}_2 \in \hat{\Sigma}\setminus{\hat{P}_0}$
satisfying the following conditions: 
\begin{itemize}
\item %
The end-point $\hat{z}_{1}$ either does not lie on the 
Stokes graph $G_{\theta}$ or a point in $\hat{P}_{\infty}$.
The other end-point $\hat{z}_{2}$ also satisfies 
the same condition.
\item %
$\beta$ never intersects with a saddle trajectory
of $\phi_{\theta}$.
\end{itemize}
Then, $\beta$ has a decomposition 
into a finite number of paths 
$\beta = \beta_{1} + \cdots + \beta_{N}$ 
in the relative homology group 
$H_{1}(\hat{\Sigma}\setminus{\hat{P}_0},\hat{P}_{\infty}\cup
\{\hat{z}_{1},\hat{z}_{2}\};\bbZ)$ and each summand 
$\beta_{i}$ ($1 \le i \le N$) is admissible 
in the direction $\theta$. 
\end{lem}
\begin{proof}
In the proof we regard a Stokes region as 
one of its lift in $\hat{\Sigma}$ by the projection $\pi$. 
Although two Stokes regions in $\hat{\Sigma}$
have the same projection, we distinguish them if they lie 
on different sheets of $\hat{\Sigma}$. 
Moreover, since we only consider the Stokes graph 
for a fixed $\theta$, we omit 
``in the direction $\theta$" for simplicity.

Since any point in a Stokes region and a point 
in $\hat{P}_{\infty}$ which lies on the boundary 
of the Stokes region can be connected 
by an admissible path, in the proof we may assume that 
$\beta$ never passes through a point in $\hat{P}_{\infty}$
(i.e., $\beta$ is contained 
in $\hat{\Sigma}\setminus{\hat{P}}$) 
without loss of generality.  
Especially, we may assume that
$\hat{z}_1, \hat{z}_2 \notin \hat{P}_{\infty}$.
If $\hat{z}_1$, $\hat{z}_2$ and the path $\beta$ 
are contained in the same Stokes region, 
$\beta$ is admissible by definition. 
Therefore, it suffices to consider the case that 
$\hat{z}_1$ and $\hat{z}_2$ are contained in different 
Stokes regions and the path $\beta$ connects them 
crossing finitely many Stokes curves 
which are not saddle trajectories.  
We may also assume that the Stokes regions 
containing $\hat{z}_1$ or $\hat{z}_2$
are not degenerate ring domains
(otherwise a path $\beta$ satisfying 
the assumption never exists). 

Let us consider the case that $\beta$ intersect with 
a Stokes curve just once. Since the Stokes curve is not 
a saddle trajectory, it must be a separating trajectory 
by Lemma \ref{lemma:no-divergent-trajectory}.
That is, the Stokes curve connects a point $a \in \hat{P}_0$ 
and a point $p \in \hat{P}_{\infty}$. 
Therefore, we can decompose $\beta$ into a sum of two paths 
$\beta_1 + \beta_2$ in the relative homology group, 
where $\beta_1$ (resp., $\beta_2$) connects 
$\hat{z}_1$ (resp., $\hat{z}_2$) and $p$ as 
indicated in Figure \ref{fig:example-decomposition}. 
Here we can take the path $\beta_1$ (resp., $\beta_2$)
to be admissible since the point $p$ lies on the boundary 
of the Stokes region containing $\hat{z}_1$ 
(resp., $\hat{z}_2$). 

Any path $\beta$ in $\hat{\Sigma}\setminus\hat{P}$ 
can be written by the sum of a finite number of paths 
whose each summand intersect with the Stokes curves
just once. Therefore, applying the decomposition 
as in Figure \ref{fig:example-decomposition} 
to each summand, we can find a desired decomposition
of $\beta$ by admissible paths.
\end{proof}

\begin{figure}
\begin{center}
\begin{pspicture}(-2.2,-1.5)(6,1.5)
\psset{linewidth=0.5pt}
\psset{fillstyle=none}
\psline(0,-1.2)(0,+0.6)
%
\pscurve(-1.0,-0.3)(-0.5,-0.4)(0.5,-0.2)(1.0,-0.3)
\psline(0.5,-0.2)(0.35,-0.05)
\psline(0.5,-0.2)(0.35,-0.35)
\rput[c]{0}(0,-1.2){\small $\times$}
\rput[c]{0}(0,+0.6){$\bullet$}
\rput[c]{0}(-1.0,-0.3){\tiny $\bullet$}
\rput[c]{0}(+1.0,-0.3){\tiny $\bullet$}
\rput[c]{0}(0.4,+0.7){\small $p$}
\rput[c]{0}(0.4,-1.3){\small $a$}
\rput[c]{0}(-0.5,-0.7){\small $\beta$}
\rput[c]{0}(-1.2,+0){\small $\hat{z}_1$}
\rput[c]{0}(1.2,0){\small $\hat{z}_2$}
\rput[c]{0}(2,-0.4){$=$}
\psline(4,-1.2)(4,+0.6)
%
\pscurve(3.0,-0.3)(3.5,-0.1)(4,0.6)
\pscurve(5.0,-0.3)(4.5,-0.1)(4,0.6)
\psline(3.6,0)(3.55,-0.25)
\psline(3.6,0)(3.35,-0.02)
\psline(4.6,-0.17)(4.55,0.06)
\psline(4.6,-0.17)(4.4,-0.17)
\rput[c]{0}(4,-1.2){\small $\times$}
\rput[c]{0}(4,+0.6){$\bullet$}
\rput[c]{0}(3.0,-0.3){\tiny $\bullet$}
\rput[c]{0}(+5.0,-0.3){\tiny $\bullet$}
\rput[c]{0}(4.4,+0.7){\small $p$}
\rput[c]{0}(4.4,-1.3){\small $a$}
\rput[c]{0}(3.5,-0.5){\small $\beta_1$}
\rput[c]{0}(4.5,-0.5){\small $\beta_2$}
\rput[c]{0}(2.8,+0){\small $\hat{z}_1$} 
\rput[c]{0}(5.2,0){\small $\hat{z}_2$} 
\end{pspicture}
\end{center}
\caption{An example of decomposition of a path 
which intersects with the Stokes curves just once.} 
\label{fig:example-decomposition}
\end{figure}
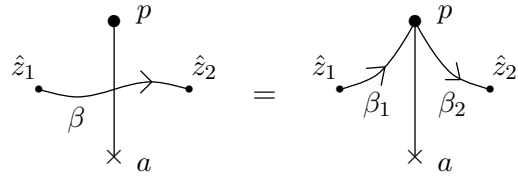

Then, a criterion of the Borel summability of 
the WKB solutions proposed by Koike and Sch\"afke
is stated as follows. 
\begin{cor} [\cite{Koike13}]
\label{cor:summability} 
\begin{itemize}
\item[(a).] %
Let $\beta$ be a path on $\hat{\Sigma}\setminus{\hat{P}_{0}}$ 
with end-points $\hat{z}_{1},\hat{z}_{2} 
\in\hat{\Sigma}\setminus{\hat{P}_{0}}$ satisfying the 
same assumption in Lemma \ref{lemma:admissible-decomposition}.
Then, the formal power series 
$\int_{\beta}S^{\rm reg}_{\rm odd}(z,\eta)~dz$ 
is Borel summable in the direction $\theta$. %
\item[(b).] %
If the Stokes graph $G_{\theta}$ is saddle-free, 
then the WKB solutions which are normalized 
as \eqref{eq:WKBsol-TP} and \eqref{eq:WKBsol-P} 
are Borel summable in the direction $\theta$ 
at any point $z$ in each Stokes region. 
The Borel sums of the WKB solutions give 
analytic solutions of \eqref{eq:Sch} on 
each Stokes region (which is also analytic in $\eta$ 
on a domain 
$\{\eta \in {\mathbb C}~|~ |\arg\eta-\theta| < \pi/2,~
|\eta| \gg 1 \}$).  
\end{itemize}
\end{cor}
\begin{proof}
Theorem \ref{thm:summability-P1} (b) ensures that 
a formal power series defined by integrating 
$S_{\rm odd}^{\rm reg}(z,\eta)$ along an admissible path 
is Borel summable in the direction $\theta$.
Therefore, the first claim (a) follows 
from Lemma \ref{lemma:admissible-decomposition}. 

Let us show the claim (b). When the Stokes graph is 
saddle-free, any path $\beta$ on 
$\hat{\Sigma}\setminus\hat{P}_0$ 
can be decomposed into admissible paths by 
Lemma \ref{lemma:admissible-decomposition}. 
For example, if $z$ lies on a Stokes region, then
the path $\gamma_z$ (see Figure \ref{fig:normalized-at-TP})
which determines the WKB solutions \eqref{eq:WKBsol-TP} 
is decomposed into admissible paths as depicted in 
Figure \ref{fig:example-summability}. 
Therefore the integral 
$\int_{\gamma_z}S^{\rm reg}_{\rm odd}(z,\eta)dz$ 
is Borel summable in the direction $\theta$
by Theorem \ref{thm:summability-P1} (b).
The Borel summability of the WKB 
solutions \eqref{eq:WKBsol-TP}
follows from (c) in Proposition 
\ref{prop:property-of-Borel-sum}. 
The Borel summability of the WKB solutions 
\eqref{eq:WKBsol-P} can be shown similarly. 
\end{proof}

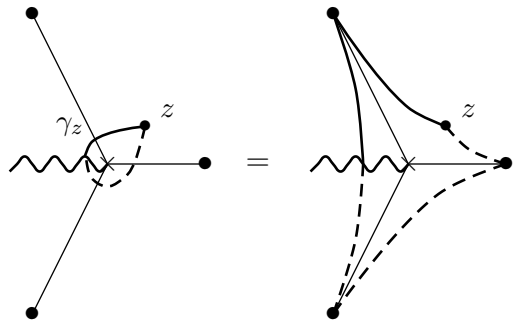
\begin{figure}
\begin{center}
\begin{pspicture}(-0.5,-2)(8,2)
%
%
\rput[c]{0}(2,0){\small $\times$}
\rput[c]{0}(3.3,0){$\bullet$}
\rput[c]{0}(1.0,2){$\bullet$}
\rput[c]{0}(1.0,-2){$\bullet$}
\rput[c]{0}(2.5,0.5){\footnotesize $\bullet$}
\rput[c]{0}(2.8,0.7){\small $z$}
\rput[c]{0}(1.5,0.5){\small $\gamma_z$}
\rput[c]{0}(4.0,0){$=$}
\rput[c]{0}(6,0){\small $\times$}
\rput[c]{0}(7.3,0){$\bullet$}
\rput[c]{0}(5.0,2){$\bullet$}
\rput[c]{0}(5.0,-2){$\bullet$}
\rput[c]{0}(6.5,0.5){\footnotesize $\bullet$}
\rput[c]{0}(6.8,0.7){\small $z$}
\psset{linewidth=1pt}
\pscurve(2,0)(1.9,-0.1)(1.8,0)(1.7,0.1)
(1.6,0)(1.5,-0.1)(1.4,0)(1.3,0.1)(1.2,0)
(1.1,-0.1)(1,0)(0.9,0.1)(0.8,0)(0.7,-0.1)
\pscurve(6,0)(5.9,-0.1)(5.8,0)(5.7,0.1)
(5.6,0)(5.5,-0.1)(5.4,0)(5.3,0.1)(5.2,0)
(5.1,-0.1)(5,0)(4.9,0.1)(4.8,0)(4.7,-0.1)
\psset{linewidth=0.5pt}
\psline(2,0)(3.3,0)
\psline(2,0)(1,2)\psline(2,0)(1,-2)
\psline(6,0)(7.3,0)
\psline(6,0)(5,2)\psline(6,0)(5,-2)
\psset{linewidth=1pt}
\pscurve(2.5,0.5)(1.8,0.3)(1.7,0.1)
\pscurve(6.5,0.5)(6,0.75)(5,2)
\pscurve(5,2)(5.3,1)(5.4,0)
\psset{linewidth=1pt, linestyle=dashed}
\pscurve(2.5,0.5)(2.45,0.3)(2.3,-0.1)(2.2,-0.2)(2.0,-0.3)
(1.8,-0.2)(1.75,-0.1)(1.72,0.05)
\pscurve(5,-2)(5.3,-1)(5.4,0)
\pscurve(5,-2)(6.4,-0.4)(7.3,0)
\pscurve(7.3,0)(6.8,0.2)(6.5,0.5)
\end{pspicture}
\end{center}
\caption{Decomposition of $\gamma_z$ 
into admissible paths.} 
\label{fig:example-summability}
\end{figure}

\begin{rem}
Suppose that the Stokes graph has a saddle trajectory. 
Even if the point $z$ does not lie on the Stokes graph, 
the path $\gamma_z$ in Figure \ref{fig:normalized-at-TP} 
can not be decomposed into admissible paths when 
$\gamma_z$ intersects with the saddle trajectory.
Therefore, we can not expect the Borel summability for 
the WKB solutions in general when a saddle trajectory appears 
in the Stokes graph. 
\end{rem}

The above statements guarantee the Borel summability 
of $S_{\rm odd}^{\rm reg}(z,\eta)$ and the WKB solutions 
in a fixed direction $\theta$. As is explained in 
Section \ref{section:Borel-resummation}, 
the rotation of the direction $\theta$ 
may break the Borel summability of the WKB solutions. 
The following claim gives an criterion for 
the invariance of the Borel sum under a 
rotation of $\theta$.

\begin{prop} \label{prop:S1-action-and-summability} 
Let $\beta$ be a path on 
$\hat{\Sigma}\setminus{\hat{P}_{0}}$ 
with end-points $\hat{z}_{1},\hat{z}_{2} 
\in\hat{\Sigma}\setminus{\hat{P}_{0}}$.
Suppose that there exist real numbers 
$\theta_1, \theta_2$ with $\theta_1 < \theta_2$ 
such that the following conditions hold.
\begin{itemize}
\item
The quadratic differential  $\phi_{\theta}$ has at most
one saddle trajectory for any $\theta_1\leq \theta \leq \theta_2$.
\item %
The end-point $\hat{z}_{1}$ either does not 
lie on the Stokes graphs $G_{\theta}$ for any
 $\theta_1 \le \theta \le \theta_2$ 
or is a point in $\hat{P}_{\infty}$. The other end-point 
$\hat{z}_{2}$ also satisfies the same condition.
\item %
The path $\beta$ never touches with a saddle trajectory
of $\phi_{\theta}$
for any $\theta_1 \le \theta \le \theta_2$.
\end{itemize}
Then, the Borel sums of the formal power series 
$f(\eta)=\int_{\beta}S_{\rm odd}^{\rm reg}(z,\eta)dz$ 
in the direction $\theta_1$ and $\theta_2$ coincide. 
That is, the following equality holds 
as analytic functions of $\eta$ 
defined on a domain containing 
$\{\eta \in {\mathbb C}~|~ 
\theta_1-\pi/2 < \arg\eta < \theta_2 + \pi/2,~|\eta| \gg 1 \}$:
\begin{equation}\label{eq:no-PSP-formula}
{\mathcal S}_{\theta_1}[f](\eta) = {\mathcal S}_{\theta_2}[f](\eta).
\end{equation}
\end{prop}
\begin{proof}
Since $\beta$ satisfies the assumption of Corollary 
\ref{cor:summability} (a) for any $\theta$ satisfying 
$\theta_1 \le \theta \le \theta_2$, 
the formal power series $f(\eta)$ is Borel summable 
in all directions $\theta_1 \le \theta \le \theta_2$. 
That means that the Borel transform $f_B(y)$ of $f(\eta)$
does not have singular points in a domain 
containing the sector 
$\{y=re^{-i \theta}~|~r\ge0,\theta_1\le\theta\le\theta_2\}$,
and has an exponential growth near $y=\infty$ 
(see Definition \ref{def:Borel-summability}).
Hence, the Laplace integrals \eqref{eq:Borel sum} 
give the same analytic function of $\eta$ 
for all $\theta_1 \le \theta \le \theta_2$. 
Thus we obtain \eqref{eq:no-PSP-formula}.
\end{proof}


\begin{rem}
The ``{\em resurgence property}" 
(i.e., the endlessly continuability of the Borel transform),
which is stronger than the Borel summability,
was claimed by Ecalle in \cite{Ecalle84}, 
but not all details are clear 
(see \cite[Comment in Section 1.2]{Delabaere99}).
Still currently there are many contributions to warrant 
the resurgence property of the WKB solutions; 
e.g.,  \cite{Getmanenko09, Getmanenko11,Getmanenko11b}. 
The case where the potential function is entire on 
${\mathbb C}$ is discussed in these works. 
\end{rem}

\subsection{Connection formula for WKB solutions}
\label{section:connection-formula}

Corollary \ref{cor:summability} (b) ensures that, 
if the Stokes graph $G_{\theta}$ in 
a fixed direction $\theta$ is saddle-free, 
then the WKB solutions are Borel summable in the direction 
$\theta$ on each Stokes region of $G_{\theta}$. 
Here we show an explicit and simple connection 
formula between the Borel sums of the WKB solutions defined 
on adjacent Stokes regions found by Voros \cite{Voros83}. 
(In this subsection we do not consider the rotation of 
the direction $\theta$. The following statements 
hold for any fixed $\theta$, if the Stokes graph 
$G_{\theta}$ is saddle-free.)  

Here we specify the situation to state the connection 
formula. Assume that the Stokes graph $G_{\theta}$ 
is saddle-free. 
Let $a \in P_{0}$ be a simple turning point, and suppose 
that two Stokes regions $D_{1}$ and $D_{2}$ have a common 
boundary $C$ which is a Stokes curve emanating from $a$, 
and $D_{2}$ comes next to $D_{1}$ in the counter-clockwise 
direction with the reference point $a$. Take appropriate 
branch cuts so that $C$ does not cross any branch cut. 
Then we have two possibilities (a) and (b) shown in 
Figure \ref{fig:signed-Stokes-curves} for the sign of 
the other end-point of $C$ than $a$. For each case, 
the connection formula is formulated as follows. 
\begin{figure}
\begin{center}
\begin{pspicture}(0,-3)(12,3)
%
%
\psset{fillstyle=solid, fillcolor=black}
\psset{fillstyle=none}
\rput[c]{0}(2,0){\small $\times$}
\rput[c]{0}(3.5,0.15){\small $C$}
\rput[c]{0}(1.5,0.5){$a$}
\rput[c]{0}(3.7,-1.4){\small $D_{1}$}
\rput[c]{0}(3.7,1.5){\small $D_{2}$}
\rput[c]{0}(4.9,0.05){$\oplus$}
\rput[c]{0}(1.45,2.2){$\ominus$}
\rput[c]{0}(0.93,-1.78){$\ominus$}
\rput[c]{0}(2.4,-2.5)
{$(a) : \mathrm{Re}(e^{i\theta}
\int_{a}^{z}\sqrt{Q_{0}(z)}dz) \ge 0$ on $C$}
\rput[c]{0}(8.5,0){\small $\times$}
\rput[c]{0}(10,0.15){\small $C$}
\rput[c]{0}(8,0.5){$a$}
\rput[c]{0}(10.2,-1.4){\small $D_{1}$}
\rput[c]{0}(10.2,1.5){\small $D_{2}$}
\rput[c]{0}(11.4,0.05){$\ominus$}
\rput[c]{0}(7.95,2.2){$\oplus$}
\rput[c]{0}(7.43,-1.78){$\oplus$}
\rput[c]{0}(9.6,-2.5)
{$(b) : \mathrm{Re}(e^{i\theta}
\int_{a}^{z}\sqrt{Q_{0}(z)}dz) \le 0$ on $C$}
\psset{linewidth=1pt}
\pscurve(2,0)(1.9,-0.1)(1.8,0)(1.7,0.1)(1.6,0)(1.5,-0.1)(1.4,0)
(1.3,0.1)(1.2,0)(1.1,-0.1)(1,0)(0.9,0.1)(0.8,0)(0.7,-0.1)(0.6,0)
(0.5,0.1)(0.4,0)(0.3,-0.1)(0.2,0)(0.1,0.1)(0,0)
\pscurve(8.5,0)(8.4,-0.1)(8.3,0)(8.2,0.1)(8.1,0)(8,-0.1)(7.9,0)
(7.8,0.1)(7.7,0)(7.6,-0.1)(7.5,0)(7.4,0.1)(7.3,0)(7.2,-0.1)(7.1,0)
(7,0.1)(6.9,0)(6.8,-0.1)(6.7,0)(6.6,0.1)(6.5,0)
\psset{linewidth=0.5pt}
\pscurve(2,0)(3.5,-0.15)(4.7,0)
\pscurve(2,0)(1.7,1.5)(1.5,2)
\pscurve(2,0)(1.4,-0.8)(1,-1.6)
\pscurve(8.5,0)(10,-0.15)(11.2,0)
\pscurve(8.5,0)(8.2,1.5)(8,2)
\pscurve(8.5,0)(7.9,-0.8)(7.5,-1.6)
\end{pspicture}
\end{center}
\caption{Two possibilities of assignment of sign.} 
\label{fig:signed-Stokes-curves}
\end{figure}
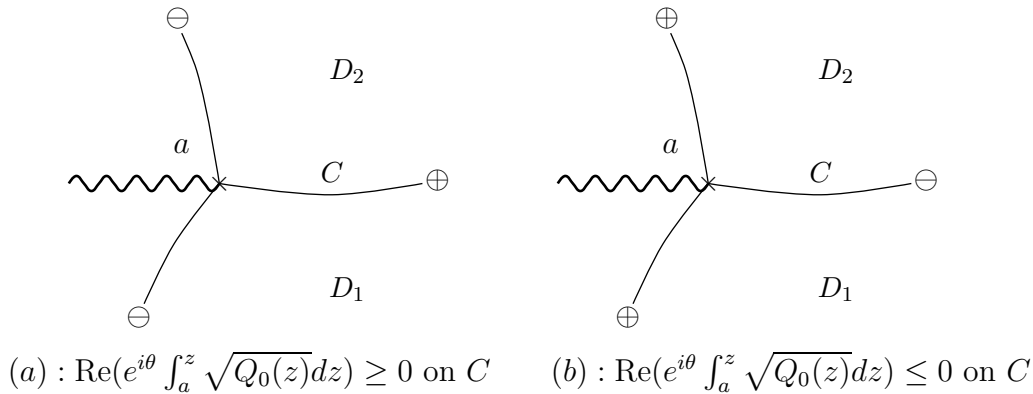

\begin{thm}
[
{\cite[Section 6]{Voros83}, \cite[Section 2]{Aoki91}}]
\label{thm:Voros-formula}
Suppose that the Stokes graph $G_{\theta}$ is saddle-free. 
Let $a \in P_0$, $C$, $D_1$ and $D_2$ be as above and 
\[
\psi_{\pm}(z,\eta) = \frac{1}{\sqrt{S_{\rm odd}(z,\eta)}}
\exp\left(\pm\int_{a}^{z} S_{\rm odd}(z,\eta)dz\right)
\] 
be the WKB solutions normalized at the turning point $a$ 
as defined in \eqref{eq:WKBsol-TP}. Denote by 
$\Psi^{D_j}_{\pm}$ ($j=1,2$) the Borel sum of $\psi_{\pm}$ 
on the Stokes region $D_j$ ($j=1,2$). 
Then, the analytic continuation 
of $\Psi^{D_1}_{\pm}$ to $D_{2}$ across the Stokes curve 
$C$ satisfy the following equalities:
\begin{eqnarray}
\begin{cases} \label{eq:Voros-formula-1}
\Psi^{D_1}_{+} = \Psi^{D_2}_{+} + i \Psi^{D_2}_{-}, \\ 
\Psi^{D_1}_{-} = \Psi^{D_2}_{-},
\end{cases}
~\text{for Figure \ref{fig:signed-Stokes-curves} (a).}
\\
\begin{cases} \label{eq:Voros-formula-2}
\Psi^{D_1}_{+} = \Psi^{D_2}_{+},  \\ 
\Psi^{D_1}_{-} = \Psi^{D_2}_{-} + i \Psi^{D_2}_{+},
\end{cases}
~\text{for Figure \ref{fig:signed-Stokes-curves} (b).}
\end{eqnarray}
Here $i$ appearing in the formula is 
the imaginary unit $\sqrt{-1}$.
\end{thm}

\begin{rem}
Theorem \ref{thm:Voros-formula} is proved by \cite{Voros83}
and \cite{Aoki91} in the case that $\Sigma = \bbP^1$. 
Since the proof of \cite[Appendix A.2]{Aoki91} is based only on 
local properties of the WKB solutions near a simple turning point, 
the same discussion is applicable to
 a general compact Riemann surface $\Sigma$.
  Therefore, together 
with Theorem \ref{thm:summability}, Theorem 
\ref{thm:Voros-formula} is valid when $\Sigma$ is 
a general compact  Riemann surface.
\end{rem}

\begin{rem}\label{remark:Voros-formula-is-Stokes}
Theorem \ref{thm:Voros-formula} gives connection formulas 
which describe the analytic continuation of the Borel sums 
of WKB solutions on the $z$-plane. 
On the other hand, the formulas indeed describe the Stokes phenomenon 
(for a large $\eta$) occurring to WKB solutions since they are 
derived through the analysis of the singularity of the Borel transforms
of WKB solutions (see \cite{Kawai05}). 
That is, we can reformulate Theorem \ref{thm:Voros-formula}
in an alternative manner as follows: 
Take any point $z \in \hat{\Sigma}\setminus \hat{P}$, and
suppose that there exists a direction 
$\theta_0$ and a sufficiently small 
number $\varepsilon > 0$ satisfying 
\begin{itemize}
\item Stokes graphs in any direction $\theta$ 
satisfying $\theta_0-\varepsilon \le \theta \le 
\theta_0+\varepsilon$ is saddle-free, 
\item the point $z$ lies on a Stokes curve $C$ 
in the direction $\theta_0$, and does not lie on 
the Stokes graph in any direction $\theta$ satisfying 
$\theta_0-\varepsilon \le \theta < \theta_0$ or 
$\theta_0 < \theta \le \theta_0+\varepsilon$.
\end{itemize}
Let $\psi_{\pm}$ be the WKB solution normalized at the 
turning point from where the Stokes curve $C$ emanates, 
and denote by 
$\Psi^{(\theta_0-\varepsilon)}_{\pm}$
(resp., $\Psi^{(\theta_0+\varepsilon)}_{\pm}$) 
the Borel sum of $\psi_{\pm}$ in the direction 
$\theta_0-\varepsilon$ (resp., $\theta_0+\varepsilon$) 
Then, the following relation holds in a neighborhood of $z$.
\begin{eqnarray}
\hspace{-2.0em}
\begin{cases} \label{eq:Voros-formula-Stokes-1}
\Psi^{(\theta_0-\varepsilon)}_{+} = 
\Psi^{(\theta_0+\varepsilon)}_{+} + 
i \Psi^{(\theta_0+\varepsilon)}_{-}, \\ 
\Psi^{(\theta_0-\varepsilon)}_{-} = 
\Psi^{(\theta_0+\varepsilon)}_{-},
\end{cases}
\text{if the sign of the end point of $C$ is $\oplus$,}
\\
\hspace{-2.0em}
\begin{cases} \label{eq:Voros-formula-Stokes-2}
\Psi^{(\theta_0-\varepsilon)}_{+} = 
\Psi^{(\theta_0+\varepsilon)}_{+},  \\ 
\Psi^{(\theta_0-\varepsilon)}_{-} = 
\Psi^{(\theta_0+\varepsilon)}_{-} + 
i \Psi^{(\theta_0+\varepsilon)}_{+},
\end{cases}
\text{if the sign of the end point of $C$ is $\ominus$.}
\end{eqnarray}
Thus, the Stokes phenomenon occurring to $\psi_{\pm}$
is described by the completely same formulas 
in Theorem \ref{thm:Voros-formula}. We use this point of view 
in Appendix \ref{section:proof-of-AIT}. 
\end{rem}

The connection formula in Theorem \ref{thm:Voros-formula} 
is quite effective for the global problems of differential 
equations.
For example, if $\Sigma = \bbP^{1}$ and 
the equation ${\mathcal L}$ is Fuchsian 
(i.e., all poles of $\phi$ are order 2) 
with a saddle-free Stokes graph, then 
the monodromy group of ${\mathcal L}$ can 
be expressed by the following  quantities 
(\cite[Theorem 3.5]{Kawai05}):
\begin{itemize}
\item[(i)] characteristic exponents 
at regular singular points, 
\item[(ii)] the Borel sum of contour integrals of 
$S_{\rm odd}(z,\eta)dz$ along 
cycles in $\hat{\Sigma}\setminus\hat{P}$.
\end{itemize} 
In \cite[Section 3.1]{Kawai05} a recipe to obtain an explicit 
expression of the monodromy group is given. 
Contour integrals of $S_{\rm odd}(z,\eta)dz$ appear
when we use the connection formulas 
\eqref{eq:Voros-formula-1} and 
\eqref{eq:Voros-formula-2} iteratively. 

\section{Voros symbols and Stokes automorphisms}
\label{section:Voros-coefficients-and-DDP-formula}

{
In this section we introduce an important notion 
in  the exact WKB analysis, 
called the {\em Voros symbols}. We discuss the jump property 
of the Voros symbols caused by the Stokes phenomenon 
relevant to the appearance of saddle trajectories 
in the Stokes graph.
}

\subsection{Homology groups and Voros symbols}
\label{example:connection-multiplier}

Let us consider the homology group 
$H_1(\hat{\Sigma}\setminus\hat{P})=
H_1(\hat{\Sigma}\setminus\hat{P};\bbZ)$ 
and the relative homology group
$H_1(\hat{\Sigma}\setminus\hat{P}_0,\hat{P}_{\infty})=
H_1(\hat{\Sigma}\setminus\hat{P}_0,\hat{P}_{\infty};\bbZ)$.
In what follows we call elements of 
$H_1(\hat{\Sigma}\setminus\hat{P})$ and 
$H_1(\hat{\Sigma}\setminus{P}_{0},\hat{P}_{\infty})$ 
as {\em cycles} and {\em paths}, respectively, 
to distinguish them. By the {\em Lefschetz duality} 
there exists a bilinear form
\begin{equation}
\label{eq:braket1}
\langle~,~\rangle: H_1(\hat{\Sigma}\setminus\hat{P}) \times 
H_1(\hat{\Sigma}\setminus\hat{P}_0,\hat{P}_{\infty}) 
\rightarrow \bbZ
\end{equation}
on these homology groups given by the intersection number of 
cycles and paths. The intersection number depends on 
the orientations of cycles and paths, and we normalize the bilinear 
form as $\langle${$x$-axis}, {$y$-axis}$\rangle = +1$. 
It also induces a bilinear form 
\begin{equation}
\label{eq:intersection1}
(~,~): H_1(\hat{\Sigma}\setminus\hat{P}) \times 
H_1(\hat{\Sigma}\setminus\hat{P}) \rightarrow \bbZ.
\end{equation}
We call both these bilinear forms  
{\em intersection forms}. 

Here we introduce the notion of the {\em Voros symbols},
which are the main objects in this paper.
\begin{defn} [{\cite[Section 1.2]{Delabaere93}}]
\begin{itemize}
\item %
Let $\beta \in H_1(\hat{\Sigma}\setminus\hat{P}_0,
\hat{P}_{\infty})$ be a path. The formal power series 
$e^{W_{\beta}}$ is called 
the {\em Voros symbol for the path $\beta$}.
Here $W_{\beta} = W_{\beta}(\eta)$ is the formal power
series defined by the integral 
\begin{equation} \label{eq:Voros-coeff-W}
{W}_{\beta}(\eta)=\int_{\beta}
S_{\rm odd}^{\rm reg}(z,\eta)dz.
\end{equation} 
\item %
Let $\gamma \in H_1(\hat{\Sigma}\setminus\hat{P})$ 
be a cycle. The formal series $e^{V_{\gamma}}$ 
is called the {\em Voros symbol for the cycle $\gamma$}.
Here $V_{\gamma} = V_{\gamma}(\eta)$ is the formal series
defined by the integral 
\begin{equation} \label{eq:Voros-coeff-V}
{V}_{\gamma}(\eta) = \oint_{\gamma}S_{\rm odd}(z,\eta)dz.
\end{equation}
\end{itemize}
\end{defn}

\begin{rem}
The formal series $W_{\beta}$ in 
\eqref{eq:Voros-coeff-W} (resp., $V_{\gamma}$ in 
\eqref{eq:Voros-coeff-V}) is called 
the {\em Voros coefficient} for the path $\beta$ 
(resp., for the cycle $\gamma$). 
The Voros coefficients for paths in 
$H_1(\hat{\Sigma}\setminus\hat{P}_0,\hat{P}_{\infty})$ 
attract attention recently 
(e.g., \cite{Takei08}, \cite{Aoki12}). 
\end{rem}

The Voros symbols $e^{{W}_{\beta}}$ for 
$\beta \in H_1(\hat{\Sigma}\setminus\hat{P}_0,
\hat{P}_{\infty})$ are formal power series without a
exponential factor since $S_{\rm odd}^{\rm reg}(z,\eta)$ 
is a formal power series. On the other hand, 
the Voros symbols $e^{{V}_{\gamma}}$ for 
$\gamma \in H_1(\hat{\Sigma}\setminus\hat{P})$ 
are formal series with the exponential factors
$\exp(\eta \hspace{+.1em} v_{\gamma})$, where
\begin{equation} \label{eq:central-charge}
v_{\gamma} = \oint_{\gamma} \sqrt{Q_0(z)}dz.
\end{equation}

As mentioned in Section \ref{section:connection-formula},
the Voros symbols 
appear in the expression of monodromy 
group of the equation \eqref{eq:Sch} 
(see Section \ref{section:connection-formula}). 
They are Borel summable 
(in the direction $\theta=0$) 
if the paths of 
the integrals in \eqref{eq:Voros-coeff-W} and 
\eqref{eq:Voros-coeff-V} do not intersect 
with a saddle trajectory of $\phi$ 
by Corollary \ref{cor:summability}. 
The appearance of a saddle trajectory breaks 
the Borel summability, and cause the Stokes phenomenon
as explained in Section \ref{section:Borel-resummation}. 
That is, if a saddle trajectory appears, 
the Borel sums of a Voros symbol in the directions 
$\pm \delta$ are different in general for a
sufficiently small $\delta>0$.  
As noted in 
 Section \ref{section:Stokes-graphs},
the Stokes graph mutates when a saddle 
trajectory appears. The rest of this section is 
devoted to analyze the Stokes phenomenon occurring to 
the Voros symbols under the mutation of Stokes graphs.

\subsection{Saddle class associated with saddle trajectory}
\label{subsec:saddleclass}

Suppose that the Stokes graph $G_0 = G(\phi)$
has a regular or degenerate saddle trajectory $\ell_0$. 
Recall that a regular saddle 
trajectory connects two different zeros of 
$\phi$, while a degenerate saddle trajectory 
forms a closed loop around a double pole of $\phi$ 
(see Section \ref{section:Stokes-graphs}). 
Then, there exists a cycle 
$\gamma_{0} \in H_1(\hat{\Sigma}\setminus\hat{P})$
whose projection on $\Sigma$ by $\pi$ surrounds $\ell_0$ 
as in Figure \ref{fig:saddle-class}, 
and its orientation is given so that 
\begin{equation} \label{eq:saddle-orientation}
v_{\gamma_{0}} = 
\oint_{\gamma_{0}}\sqrt{Q_0(z)}dz < 0.
\end{equation}
(See Section \ref{section:orientation} 
for the rule of the assignment of signs.)
Note that, if a cycle $\gamma_0$ satisfies 
the above conditions, then the cycle $-\gamma_0^{\ast}$ 
also satisfies the same conditions. 
(Here $\gamma_0^{\ast}$ is the image of $\gamma_0$
by the covering involution $\tau$.)
We choose any of the two cycles, 
and call the resulting homology class 
$\gamma_{0} \in H_1(\hat{\Sigma}\setminus\hat{P})$ 
the {\em saddle class} associated with the 
saddle trajectory $\ell_0$, 
following \cite[Section 3.6]{Bridgeland13}.
Note that ``the Voros symbol for the saddle class" 
is well-defined because 
\begin{equation}
\oint_{\gamma}S_{\rm odd}(z,\eta)dz = 
\oint_{-\gamma^{\ast}}S_{\rm odd}(z,\eta)dz
\end{equation}
holds for any cycle $\gamma$ due to the anti-invariant 
property \eqref{eq:covering-involution} 
of $S_{\rm odd}(z,\eta)$.

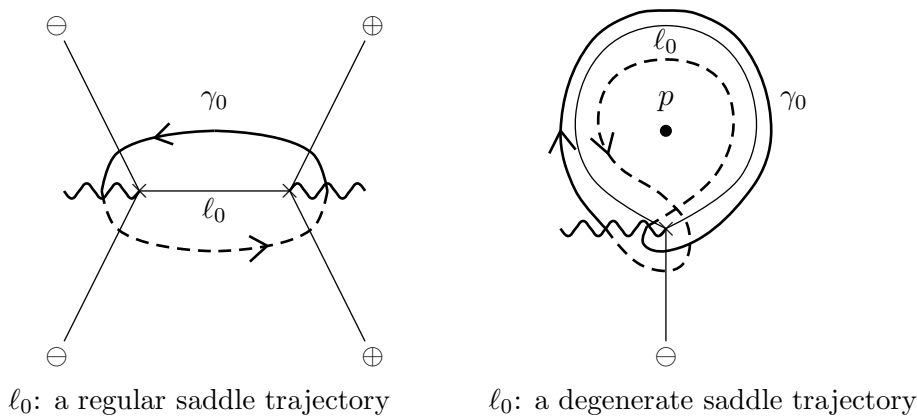
\begin{figure}[h]
\begin{center}
\begin{pspicture}(5.5,-2.5)(16.5,2.5)
%
%
\psset{fillstyle=solid, fillcolor=black}
\pscircle(14,0.8){0.08} 
\psset{fillstyle=none}
\psset{linewidth=1pt}
\rput[c]{0}(10.1,2.2){\small $\oplus$}
\rput[c]{0}(10.1,-2.2){\small $\oplus$}
\rput[c]{0}(5.9,2.2){\small $\ominus$}
\rput[c]{0}(5.9,-2.2){\small $\ominus$}
\rput[c]{0}(7,0){\small $\times$}
\rput[c]{0}(9,0){\small $\times$}
\rput[c]{0}(8,-0.25){\small$\ell_0$} 
\rput[c]{0}(7.8,-2.8)
{\small $\ell_0$: a regular saddle trajectory} 
\rput[c]{0}(14.0,-2.2){\small $\ominus$}
\rput[c]{0}(14,-0.5){\small $\times$}
\rput[c]{0}(14,2.0){\small $\ell_0$} 
\rput[c]{0}(14,1.2){$p$} 
\rput[c]{0}(14.5,-2.8)
{\small $\ell_0$: a degenerate saddle trajectory} 
\rput[c]{0}(8,1.2){\small $\gamma_{0}$} 
\psset{linewidth=1pt}
\pscurve(6.5,0)(6.7,0.5)(8,0.8)
\pscurve(9.5,0)(9.3,0.5)(8,0.8)
\psline(7.2,0.7)(7.4,0.9) \psline(7.2,0.71)(7.45,0.62)
\psline(8.7,-0.73)(8.47,-0.94) 
\psline(8.7,-0.73)(8.46,-0.60)
\psset{linewidth=1pt, linestyle=dashed}
\pscurve(6.5,0)(6.7,-0.5)(7.95,-0.8)
\pscurve(9.5,0)(9.3,-0.5)(8.05,-0.8)
\rput[c]{0}(15.7,1.2){\small ${\gamma}_{0}$} 
\psset{linewidth=1pt,linestyle=solid}
\pscurve(13.75,-0.47)(13.7,-0.7)(14,-0.8)(15.15,0.05)(15.4,0.8)
(15,1.98)(14.4,2.36)(14,2.4)
(13.6,2.36)(13,1.98)(12.6,0.8)(12.8,0)(13.2,-0.5)
\psline(12.6,0.8)(12.45,0.50)
\psline(12.6,0.8)(12.8,0.55)
\psline(13.0,0.70)(13.23,0.50)
\psline(13.23,0.50)(13.30,0.78)
\psset{linewidth=1pt, linestyle=dashed}
\pscurve(13.75,-0.47)(14.65,0.2)(14.9,0.9)(14.7,1.5)(14,1.75)
(13.3,1.5)(13.1,0.9)
(13.3,0.4)(14,-0.1)(14.3,-0.5)(14.2,-1)(13.7,-1)(13.2,-0.5)
\psset{linewidth=1pt,linestyle=solid}
\pscurve(7,0)(6.9,-0.1)(6.8,0)(6.7,0.1)(6.6,0)
(6.5,-0.1)(6.4,0)(6.3,0.1)(6.2,0)(6.1,-0.1)(6,0)
\pscurve(9,0)(9.1,0.1)(9.2,0)(9.3,-0.1)(9.4,0)
(9.5,0.1)(9.6,0)(9.7,-0.1)(9.8,0)(9.9,0.1)(10,0)
\pscurve(14,-0.5)(13.9,-0.6)(13.8,-0.5)(13.7,-0.4)
(13.6,-0.5)(13.5,-0.6)(13.4,-0.5)(13.3,-0.4)
(13.2,-0.5)(13.1,-0.6)(13,-0.5)(12.9,-0.4)(12.8,-0.5)
(12.7,-0.6)(12.6,-0.5)
\psset{linewidth=0.5pt}
%
\pscurve(7,0)(8,0)(9,0)
\pscurve(7,0)(6.5,1)(6,2)
\pscurve(7,0)(6.5,-1)(6,-2)
\pscurve(9,0)(9.5,1)(10,2)
\pscurve(9,0)(9.5,-1)(10,-2)
\psline(14,-0.5)(14,-2)
\pscurve(14,-0.5)(15,0.2)(15.2,1)(14.8,1.9)(14,2.2)
(13.2,1.9)(12.8,1)(13,0.2)(14,-0.5)
\end{pspicture}
\end{center}
\caption{Saddle trajectories and the associated saddle classes.} 
\label{fig:saddle-class}
\end{figure}

\subsection{Saddle reduction}
\label{section:saddle-reduction}
 
\begin{figure}[t]
\begin{center}
\begin{pspicture}(0.5,-9)(16,2.3)
%
%
\psset{fillstyle=solid, fillcolor=black}
\psset{linewidth=0.5pt}
\pscircle(5,2){0.08} \pscircle(5,-2){0.08} 
\pscircle(1,2){0.08} \pscircle(1,-2){0.08} 
\pscircle(3,-5.5){0.08} \pscircle(3,-8.5){0.08}
\pscircle(6,2){0.08} \pscircle(6,-2){0.08} 
\pscircle(10,2){0.08} \pscircle(10,-2){0.08} 
\pscircle(8,-5.5){0.08} \pscircle(8,-8.5){0.08}
\pscircle(11,2){0.08} \pscircle(11,-2){0.08} 
\pscircle(15,2){0.08} \pscircle(15,-2){0.08} 
\pscircle(13,-5.5){0.08} \pscircle(13,-8.5){0.08}
\psset{fillstyle=none}
\rput[c]{0}(2,0){\small $\times$}
\rput[c]{0}(4,0){\small $\times$}
\rput[c]{0}(3.2,1.8){$G_{+\delta}$} 
\rput[c]{0}(3,-6.98){\small $\times$}
\rput[c]{0}(3.1,-3.8){$G_{+\delta}$} 
\rput[c]{0}(7,0){\small $\times$}
\rput[c]{0}(9,0){\small $\times$}
\rput[c]{0}(8.1,1.8){$G_0$} 
\rput[c]{0}(8,-0.5){$\ell_{0}$} 
\rput[c]{0}(8,-2.8){\small 
(a) Saddle reduction of 
regular saddle trajectory and flip.} 
\rput[c]{0}(8,-6.98){\small $\times$}
\rput[c]{0}(8.1,-3.8){$G_0$} 
\rput[c]{0}(8,-4.75){$\ell_{0}$} 
\rput[c]{0}(8,-9.1){\small 
(b) Saddle reduction of 
degenerate saddle trajectory and pop.} 
\rput[c]{0}(12,0){\small $\times$}
\rput[c]{0}(14,0){\small $\times$}
\rput[c]{0}(13.2,1.8){$G_{-\delta}$} 
\rput[c]{0}(13,-6.98){\small $\times$}
\rput[c]{0}(13.1,-3.8){$G_{-\delta}$} 
\psset{linewidth=0.5pt}
\pscurve(2,0)(4,-1)(5,-2)
\pscurve(2,0)(1.5,1)(1,2)
\pscurve(2,0)(1.5,-1)(1,-2)
\pscurve(4,0)(2,1)(1,2)
\pscurve(4,0)(4.5,1)(5,2)
\pscurve(4,0)(4.5,-1)(5,-2)
\psline(3,-7.0)(3,-8.5)
\pscurve(3,-7.0)(2.4,-6.1)(2.6,-5.0)
(3.0,-4.85)(3.3,-5.0)(3.3,-5.6)(3.0,-5.7)(2.87,-5.5)(3.0,-5.38)(3.1,-5.5)(3.0,-5.55)
\pscurve(3.0,-7.0)(3.9,-6.2)(4.0,-5.3)(3.4,-4.5)(2.7,-4.4)
(2.0,-4.9)(1.8,-6.4)(2.0,-7.0)(3.0,-8.5)
\pscurve(7,0)(8,0)(9,0)
\pscurve(7,0)(6.5,1)(6,2)
\pscurve(7,0)(6.5,-1)(6,-2)
\pscurve(9,0)(9.5,1)(10,2)
\pscurve(9,0)(9.5,-1)(10,-2)
\psline(8,-7.0)(8,-8.5)
\pscurve(8,-7.0)(9,-6.3)
(9.2,-5.5)(8.9,-4.65)(8,-4.3)(7.1,-4.65)(6.8,-5.5)(7,-6.3)(8,-7.0)
\pscurve(12,0)(14,1)(15,2)
\pscurve(12,0)(11.5,1)(11,2)
\pscurve(12,0)(11.5,-1)(11,-2)
\pscurve(14,0)(12,-1)(11,-2)
\pscurve(14,0)(14.5,1)(15,2)
\pscurve(14,0)(14.5,-1)(15,-2)
\psline(13,-7.0)(13,-8.5)
\pscurve(13,-7.0)(13.6,-6.1)(13.4,-5.0)
(13.0,-4.85)(12.7,-5.0)(12.7,-5.6)(13.0,-5.7)(13.13,-5.5)(13,-5.38)(12.9,-5.5)(13.0,-5.55)
\pscurve(13.0,-7.0)(12.1,-6.2)(12.0,-5.3)
(12.6,-4.5)(13.3,-4.4)(14.0,-4.9)(14.2,-6.4)(14.0,-7.0)(13.0,-8.5)
\end{pspicture}
\end{center}
\caption{Reduction of saddle trajectories. 
Figures describe a part of Stokes graphs.} 
\label{fig:Stokes-auto}
\end{figure}

Suppose that the Stokes graph $G_0 = G(\phi)$ has a unique 
regular or degenerate saddle trajectory $\ell_0$. 
Then, as in \cite[Section 5 and Section 10.3]{Bridgeland13}, 
there exists $r>0$ such that for all $0 < \delta \le r$ 
the quadratic differentials 
$\phi_{\pm \delta} = e^{\pm 2 i \delta}\phi$ 
are saddle-free. We call $G_{\pm\delta}$ 
{\em saddle reductions} of $G_0$. 
The topology of the Stokes 
graph $G_{+ \delta} = G(\phi_{+\delta})$ 
(resp.,  $G_{- \delta} = G(\phi_{-\delta})$)
does not change as long as $0 < \delta \le r$
since $\phi_{+ \delta}$ (resp., $\phi_{- \delta}$) 
is saddle-free for all $0 < \delta \le r$. 
However, varying $\delta$ across $0$, 
the topology of the Stokes graph changes
as explained in Section \ref{section:Stokes-graphs}.
We say that $G_{-\delta}$ and $G_{+\delta}$
are related by a {\em flip} (resp., {\em pop}) 
if they give saddle reductions of a regular 
(resp., degenerate) saddle trajectory 
(see Figure \ref{fig:Stokes-auto}).

Since the Stokes graphs $G_{\pm\delta}$ 
are saddle-free, the Voros symbols 
are Borel summable in any direction 
$\pm\delta$ with $0 < \delta \le r$ 
by Corollary \ref{cor:summability}. 
Furthermore, we can show the following.
\begin{lem}\label{lemma:saddle-reduction-of-Voros}
Suppose that the Stokes graph $G_0$
has a unique saddle trajectory. 
Then, there exists a sufficiently small $r > 0$
such that the following equalities hold 
as analytic functions of $\eta$ 
for any $0<\delta_1,\delta_2 \le r$:
\begin{eqnarray} \label{eq:saddle-reduction-of-Voros}
{\mathcal S}_{+\delta_1}[e^{W_{\beta}}](\eta) & = &
{\mathcal S}_{+\delta_2}[e^{W_{\beta}}](\eta),~~
{\mathcal S}_{+\delta_1}[e^{V_{\gamma}}](\eta) =
{\mathcal S}_{+\delta_2}[e^{V_{\gamma}}](\eta), \\
{\mathcal S}_{-\delta_1}[e^{W_{\beta}}](\eta) & = &
{\mathcal S}_{-\delta_2}[e^{W_{\beta}}](\eta),~~
{\mathcal S}_{-\delta_1}[e^{V_{\gamma}}](\eta) =
{\mathcal S}_{-\delta_2}[e^{V_{\gamma}}](\eta).
\label{eq:saddle-reduction-of-Voros-minus}
\end{eqnarray}
Here $\beta \in H_1(\hat{\Sigma}\setminus\hat{P}_0,
\hat{P}_{\infty})$ and 
$\gamma \in H_1(\hat{\Sigma}\setminus\hat{P})$
are any path and cycle, respectively.
\end{lem}
\begin{proof}
Note that any path $\beta \in H_1(\hat{\Sigma}\setminus
\hat{P}_0,\hat{P}_{\infty})$ or any cycle 
$\gamma \in H_1(\hat{\Sigma}\setminus\hat{P})$ 
is decomposed into the sum of a finite number of paths 
whose end-points are contained in $\hat{P}_{\infty}$
in the relative homology group 
$H_1(\hat{\Sigma}\setminus\hat{P}_0,\hat{P}_{\infty})$ 
(see Figure \ref{fig:gamma-is-writtwn-by-beta}). 
Therefore, it suffices to show the equalities 
\eqref{eq:saddle-reduction-of-Voros} and 
\eqref{eq:saddle-reduction-of-Voros-minus} 
for any $\beta \in H_1(\hat{\Sigma}\setminus\hat{P}_0,
\hat{P}_{\infty})$ whose end-points are contained in 
$\hat{P}_{\infty}$. Take any such a path $\beta$, 
and fix a sufficiently small $r > 0$ so that 
the Stokes graphs $G_{\pm\delta}$ 
are saddle-free for all $0 < \delta \le r$. 
Then, the path $\beta \in H_1(\hat{\Sigma}
\setminus\hat{P}_0,\hat{P}_{\infty})$ never touches with 
saddle trajectory of $\phi_{\pm\delta}$ 
for all $0 < \delta \le r$. 
Therefore, since $\beta$ satisfies the 
assumption of Proposition 
\ref{prop:S1-action-and-summability}, 
the equalities \eqref{eq:saddle-reduction-of-Voros} and 
\eqref{eq:saddle-reduction-of-Voros-minus} 
follows form \eqref{eq:no-PSP-formula}. 
\end{proof}

\begin{figure}[h]
\begin{center}
\begin{pspicture}(5,-2.3)(18,2.3)
%
%
\psset{fillstyle=solid, fillcolor=black}
\psset{fillstyle=none}
\rput[c]{0}(5.93,2.1){$\bullet$} 
\rput[c]{0}(5.93,-2.1){$\bullet$} 
\rput[c]{0}(10.07,2.1){$\bullet$} 
\rput[c]{0}(10.03,-2.1){$\bullet$} 
\rput[c]{0}(7,0){\small $\times$}
\rput[c]{0}(9,0){\small $\times$}
\rput[c]{0}(11.5,0){$=$}
\psset{linewidth=0.5pt}
\pscurve(7,0)(9,-1)(10,-2.1)
\pscurve(7,0)(6.5,1)(5.9,2.1)
\pscurve(7,0)(6.5,-1)(5.9,-2.1)
\pscurve(9,0)(7,1)(6,2.1)
\pscurve(9,0)(9.5,1)(10.05,2.1)
\pscurve(9,0)(9.5,-1)(10,-2)
\psset{linewidth=1pt}
\pscurve(6.5,0)(6.7,0.3)(8,0.5)
\pscurve(9.5,0)(9.3,0.3)(8,0.5)
\psset{linewidth=1.2pt, linestyle=dashed}
\pscurve(6.5,0)(6.7,-0.3)(7.95,-0.5)
\pscurve(9.5,0)(9.3,-0.3)(8.05,-0.5)
\psset{linewidth=1pt, linestyle=solid}
\psline(7.6,0.5)(7.8,0.7)
\psline(7.6,0.5)(7.8,0.3)
\psline(7.8,-0.5)(7.6,-0.7)
\psline(7.8,-0.5)(7.6,-0.3)
\rput[c]{0}(12.93,2.1){$\bullet$} 
\rput[c]{0}(12.93,-2.1){$\bullet$} 
\rput[c]{0}(17.07,2.1){$\bullet$}
\rput[c]{0}(17.03,-2.1){$\bullet$} 
\rput[c]{0}(14,0){\small $\times$}
\rput[c]{0}(16,0){\small $\times$}
\psset{linewidth=0.5pt}
\pscurve(14,0)(16,-1)(17,-2.1)
\pscurve(14,0)(13.5,1)(12.95,2.1)
\pscurve(14,0)(13.5,-1)(12.95,-2.1)
\pscurve(16,0)(14,1)(12.98,2.1)
\pscurve(16,0)(16.5,1)(17.05,2.1)
\pscurve(16,0)(16.5,-1)(17,-2)
\psset{linewidth=1pt}
\pscurve(13.2,0)(13.15,1.0)(12.9,2.1)
\pscurve(16.8,0)(16.85,1.0)(17.1,2.1)
\pscurve(13.0,2.1)(15.7,1.0)(17,2.1)
\psset{linewidth=1.2pt,linestyle=dashed}
\pscurve(13.2,0)(13.15,-1.0)(12.9,-2.1)
\pscurve(16.8,0)(16.85,-1.0)(17.1,-2.1)
\pscurve(13.0,-2.1)(14.3,-1.0)(17,-2.1)
\psset{linewidth=1pt, linestyle=solid}
\psline(15.3,1.0)(15.5,1.2)
\psline(15.3,1.0)(15.5,0.8)
\psline(14.7,-1.0)(14.5,-1.2)
\psline(14.7,-1.0)(14.5,-0.8)
\psset{linewidth=1pt,linestyle=solid}
\pscurve(7,0)(6.9,-0.1)(6.8,0)(6.7,0.1)
(6.6,0)(6.5,-0.1)(6.4,0)
(6.3,0.1)(6.2,0)(6.1,-0.1)(6,0)
(5.9,0.1)(5.8,0)(5.7,-0.1)(5.6,0)
\pscurve(9,0)(9.1,0.1)(9.2,0)(9.3,-0.1)(9.4,0)
(9.5,0.1)(9.6,0)(9.7,-0.1)(9.8,0)(9.9,0.1)(10,0)
(10.1,-0.1)(10.2,0)(10.3,0.1)(10.4,0)
\pscurve(14,0)(13.9,-0.1)(13.8,0)(13.7,0.1)
(13.6,0)(13.5,-0.1)(13.4,0)
(13.3,0.1)(13.2,0)(13.1,-0.1)(13,0)
(12.9,0.1)(12.8,0)(12.7,-0.1)(12.6,0)
\pscurve(16,0)(16.1,0.1)(16.2,0)(16.3,-0.1)(16.4,0)
(16.5,0.1)(16.6,0)(16.7,-0.1)(16.8,0)(16.9,0.1)(17,0)
(17.1,-0.1)(17.2,0)(17.3,0.1)(17.4,0)
\end{pspicture}
\end{center}
\caption{An example of decomposition of a cycles 
into the sum of paths whose end-points are contained 
in $\hat{P}_{\infty}$.} 
\label{fig:gamma-is-writtwn-by-beta}
\end{figure}
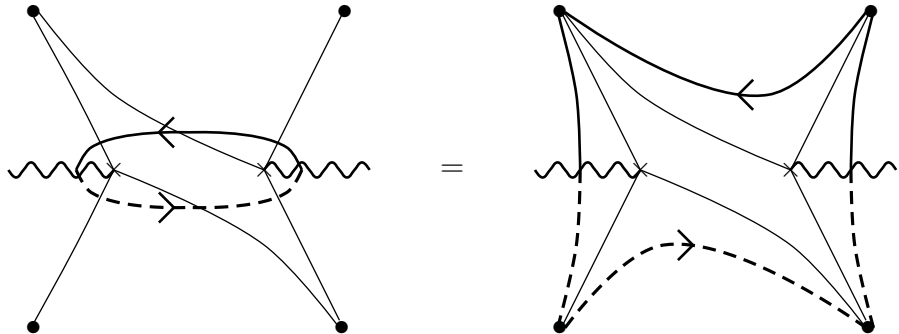

Define ${\mathcal S}_{\pm}[e^{W_\beta}] = 
{\mathcal S}_{\pm}[e^{W_\beta}](\eta)$ 
(resp., ${\mathcal S}_{\pm}[e^{V_\gamma}] = 
{\mathcal S}_{\pm}[e^{V_\gamma}](\eta)$) 
by the Borel sum ${\mathcal S}_{\pm\delta}[e^{W_\beta}](\eta)$
(resp., ${\mathcal S}_{\pm\delta}[e^{V_\gamma}](\eta)$) 
of the Voros symbol for a path $\beta \in H_1(\hat{\Sigma}
\setminus\hat{P}_0,\hat{P}_{\infty})$ 
(resp., for a cycle 
$\gamma  \in H_1(\hat{\Sigma}\setminus\hat{P})$) 
for a sufficiently small $\delta > 0$.
Due to Lemma \ref{lemma:saddle-reduction-of-Voros}, 
${\mathcal S}_{\pm}[e^{W_\beta}]$ and 
${\mathcal S}_{\pm}[e^{V_\gamma}]$ are well-defined. 
As explained in Section \ref{section:Borel-resummation}, 
the Borel sums ${\mathcal S}_{\pm}[e^{W_\beta}]$ and 
${\mathcal S}_{\pm}[e^{V_\gamma}]$ are analytic in $\eta$ 
on a domain containing $\{\eta\in{\mathbb R}~|~\eta\gg1 \}$.
In the rest of this section we will describe 
the relationship between ${\mathcal S}_{+}[e^{W_\beta}]$ 
(resp., ${\mathcal S}_{+}[e^{V_\gamma}]$)
and ${\mathcal S}_{-}[e^{W_\beta}]$ 
(resp., ${\mathcal S}_{-}[e^{V_\gamma}]$); 
that is, the formulas describing 
the Stokes phenomenon occurring to the Voros symbols.

\subsection{Jump formula and Stokes automorphism for regular saddle 
trajectory}

Here we specify the situation to state Theorem 
\ref{thm:DDP-analytic} below. Suppose that the Stokes graph 
$G_0 = G(\phi)$ has a unique {\em regular} saddle trajectory 
$\ell_0$ with the associated saddle class 
$\gamma_0 \in H_1(\hat{\Sigma}\setminus\hat{P})$. 
Let $G_{\pm\delta} = G(\phi_{\pm\delta})$ be 
saddle reductions of $G_0$ for a sufficiently small 
$\delta>0$ as in (a) of Figure \ref{fig:Stokes-auto} 
(i.e., $G_{-\delta}$ and $G_{+\delta}$ is 
related by a flip). 
Then, the Stokes phenomenon occurring to the 
Voros symbols are described explicitly by
the following ``jump formula''.
\begin{thm}[{\cite[Section 3]{Delabaere93}}] \label{thm:DDP-analytic}
The Borel sums ${\mathcal S}_{\pm}[e^{{W}_{\beta}}]$ and 
${\mathcal S}_{\pm}[e^{{V}_{\gamma}}]$ for any 
$\beta \in H_1(\hat{\Sigma}\setminus\hat{P}_0,
\hat{P}_{\infty})$ and any $\gamma \in H_1(\hat{\Sigma}
\setminus\hat{P})$ satisfy the following equalities 
as analytic functions of $\eta$ on a domain 
containing $\{\eta\in{\mathbb R}~|~\eta\gg1 \}$:
\begin{align} \label{eq:DDP-analytic}
\begin{split}
{\mathcal S}_{-}[e^{{W}_{\beta}}] &= 
{\mathcal S}_{+}[e^{{W}_{\beta}}](1 + {\mathcal S}_{+}
[e^{{V}_{\gamma_0}}])^{-\langle\gamma_0,\beta\rangle},\\
{\mathcal S}_{-}[e^{{V}_{\gamma}}] &= {\mathcal S}_{+}
[e^{{V}_{\gamma}}](1 + {\mathcal S}_{+}
[e^{{V}_{\gamma_0}}])^{-(\gamma_0,\gamma)}.
\end{split}
\end{align}
\end{thm}

\begin{rem}
Originally, Theorem \ref{thm:DDP-analytic} is proved 
in \cite[Section 3]{Delabaere93} for the case that 
the potential $Q(z,\eta) = Q_0(z)$ is independent of $\eta$ 
and is a polynomial in $z$. 
Since the Borel summability of the WKB solutions 
are established in \cite{Koike13}
(see Theorem \ref{thm:summability-P1} and 
\ref{thm:summability}), 
the proof of \cite{Delabaere93} is also valid 
for general cases. For a convenience of readers, 
we briefly recall the sketch of the proof of 
Theorem \ref{thm:DDP-analytic} in
Appendix \ref{section:proof-of-DDP}.
\end{rem}

The formula \eqref{eq:DDP-analytic} 
in fact describes the Stokes phenomenon for the 
Voros symbols relevant to the flip of the Stokes graph. 
The exponentially small difference between the Borel 
sums of Voros symbols are explicitly given in 
\eqref{eq:DDP-analytic}. Note that the Borel sum 
${\mathcal S}_{\pm}[e^{{V}_{\gamma_0}}]$ is exponentially 
small for a sufficiently large $\eta \gg 1$ 
because it is asymptotically expanded to the formal series
$e^{{V}_{\gamma_0}}$ as $\eta\rightarrow+\infty$ 
whose exponential factor $e^{\eta v_{\gamma_0}}$ 
is exponentially small due to the orientation 
\eqref{eq:saddle-orientation} of the saddle class $\gamma_0$.

In \cite{Delabaere93} the formula 
\eqref{eq:DDP-analytic} 
is stated in a different manner.
Let ${\mathbb V}={\mathbb V}(Q(z,\eta))$
 be the field of the rational functions  
generated by the Voros symbols $e^{{W}_{\beta}}$ 
and $e^{{V}_{\gamma}}$, which we call the {\em Voros field
for a potential $Q(z,\eta)$}. 
Define a field automorphism ${\mathfrak S}_{\gamma_{0}}:
{\mathbb V} \rightarrow {\mathbb V}$ by 
\begin{eqnarray} \label{eq:DDP-formal}
{\mathfrak S}_{\gamma_{0}} : 
\begin{cases}
\displaystyle e^{{W}_{\beta}} \mapsto 
e^{{W}_{\beta}} 
(1 + e^{{V}_{\gamma_{0}}})^{-\langle\gamma_{0},
\beta\rangle} & 
(\beta \in H_1(\hat{\Sigma}\setminus\hat{P}_0,
\hat{P}_{\infty})), \\[+.3em]%
\displaystyle e^{{V}_{\gamma}} ~ \mapsto  
e^{{V}_{\gamma}} \hspace{+.3em}
(1 + e^{{V}_{\gamma_{0}}})^{-(\gamma_0,\gamma)} & 
(\gamma \in 
H_1(\hat{\Sigma}\setminus\hat{P})).
\end{cases}
\end{eqnarray}
The equalities \eqref{eq:DDP-analytic} implies that 
${\mathfrak S}_{\gamma_{0}}$ satisfies 
\begin{equation} \label{eq:Stokes-automorphism-def}
{\mathcal S}_{-} = {\mathcal S}_{+}\circ {\mathfrak S}_{\gamma_{0}}.
\end{equation}
Here ${\mathcal S}_{\pm}$ is the Borel summation operator 
${\mathcal S}_{\pm\delta}$ for a sufficiently small $\delta>0$.
To be precise, the map ${\mathcal S}_{\pm}$
in Definition \ref{def:Borel-summability}
 is not defined for 
sums of Voros symbols with different exponential factors.
Here 
we extend it to the map from ${\mathbb V}$ to 
a space of analytic functions of $\eta$
so that ${\mathcal S}_{\pm}$ commutes with the operations 
addition, multiplication, and division. 
In view of \eqref{eq:Stokes-automorphism-def},
the map ${\mathfrak S}_{\gamma_0}$ measures the 
difference between the Borel sums of Voros symbols
for different directions. 
The map ${\mathfrak S}_{\gamma_0}$
is called the {\em Stokes automorphism} 
for the saddle class $\gamma_0$ associated with 
a regular saddle trajectory $\ell_0$ 
(see \cite[Section 0.4]{Delabaere99}). 

We call the formulas \eqref{eq:DDP-analytic} 
and \eqref{eq:DDP-formal}
the {\em DDP (Delabaere-Dillinger-Pham) formula}. 
Later in Section \ref{sec:mut_voros} we will reformulate the DDP formula 
in view of cluster algebras theory. 
Furthermore, we apply this formulation to 
study identities of Stokes automorphisms. 

\begin{rem}
The DDP formula  resembles to the 
{\em Kontsevich-Soibelman transformation}
in \cite{Gaiotto09}, where the counterpart of 
the Voros symbols are  the {\em Fock-Goncharov coordinates} 
of the moduli space of the flat connections associated with 
a Hitchin system of rank 2. In their context, a quadratic 
differential appears as the image of the Hitchin fibration, 
and its saddle trajectories capture {\em BPS states} 
in a four dimensional field theory.
\end{rem}

\subsection{Jump formula and Stokes automorphism
for degenerate saddle trajectory}

Similarly to regular saddle trajectories, degenerate saddle 
trajectories also cause the Stokes phenomenon for 
the Voros symbols. 
This subsection is devoted to the description of the 
formula for the Voros symbols describing the Stokes phenomenon.
Suppose that the Stokes graph $G_0 = G(\phi)$ has a unique 
{\em degenerate} saddle trajectory $\ell_0$ with 
the associated saddle class 
${\gamma}_0 \in H_1(\hat{\Sigma}\setminus\hat{P})$. 
Let $G_{\pm\delta} = G(\phi_{\pm\delta})$ 
be saddle reductions 
of $G_0$ for a sufficiently small $\delta>0$ 
as in (b) of Figure \ref{fig:Stokes-auto} 
(i.e., $G_{-\delta}$ and $G_{+\delta}$ is related by a pop). 
Then, the Stokes phenomenon occurring to the Voros symbols 
are described explicitly the following jump formula.
\begin{thm}[\cite{Aoki14}] 
\label{thm:loop-type-degeneration-analytic}
The Borel sums ${\mathcal S}_{\pm}[e^{{W}_{\beta}}]$ and 
${\mathcal S}_{\pm}[e^{{V}_{\gamma}}]$ for any 
$\beta \in H_1(\hat{\Sigma}\setminus\hat{P}_0,
\hat{P}_{\infty})$ and any $\gamma \in H_1(\hat{\Sigma}
\setminus\hat{P})$ satisfy the following equalities 
as analytic functions of $\eta$ on a domain 
containing $\{\eta\in{\mathbb R}~|~\eta\gg1 \}$:
\begin{align} \label{eq:loop-type-degeneration-analytic}
\begin{split}
{\mathcal S}_{-}[e^{{W}_{\beta}}] &= 
{\mathcal S}_{+}[e^{{W}_{\beta}}]
(1 - {\mathcal S}_{+}[e^{{V}_{{\gamma}_0}}]
)^{\langle{\gamma}_0,\beta\rangle},\\
{\mathcal S}_{-}[e^{{V}_{\gamma}}] &= 
{\mathcal S}_{+}[e^{{V}_{\gamma}}].
\end{split}
\end{align}
\end{thm}

The formula \eqref{eq:loop-type-degeneration-analytic} 
is derived as a corollary of the main result of 
the forthcoming paper \cite{Aoki14}. 
To make the paper self-contained,
we will give an alternative proof of 
\eqref{eq:loop-type-degeneration-analytic} 
in Appendix \ref{section:proof-of-AIT}. 
Note that the Borel sum of the Voros symbol $e^{{V}_{\gamma}}$ do not jump for any
$\gamma \in H_1(\hat{\Sigma}\setminus\hat{P})$.
This is a consequence of the first equality of 
\eqref{eq:loop-type-degeneration-analytic} and the fact
$({\gamma}_0,\gamma)=0$ for any 
$\gamma \in H_1(\hat{\Sigma}
\setminus\hat{P})$.

Moreover, we have 
\begin{align}
\label{eq:Vint1}
\begin{split}
V_{{\gamma}_0}(\eta) &= 
\oint_{{\gamma}_0}\left( \eta \sqrt{Q_0(z)} 
+ S_{\rm odd}^{\rm reg}(z,\eta) \right)dz\\
&=\oint_{{\gamma}_0}\left( \eta \sqrt{Q_0(z)} 
 \right)dz
\end{split}
\end{align}
since 
$S_{\rm odd}^{\rm reg}(z,\eta)dz$ is
holomorphic at the double pole $p$ 
by Proposition \ref{proposition:integrability}. 
This implies that the Voros symbol 
$e^{V_{{\gamma}_0}}$ 
for the saddle class ${\gamma}_0$ associated with 
a degenerate saddle trajectory is not a formal 
series but a scalar. 

Similarly to \eqref{eq:DDP-formal},
we also define a field automorphism 
${\mathfrak K}_{{\gamma}_0} : 
{\mathbb V} \rightarrow {\mathbb V}$ by 
\begin{eqnarray} \label{eq:loop-type-degeneration-formal}
{\mathfrak K}_{{\gamma}_{0}} : 
\begin{cases}
\displaystyle e^{{W}_{\beta}} \mapsto e^{{W}_{\beta}} 
(1 - e^{{V}_{{\gamma}_{0}}})^{
\langle{\gamma}_{0},\beta\rangle} 
& (\beta \in H_1(\hat{\Sigma}\setminus\hat{P}_0,
\hat{P}_{\infty})), \\[+.3em]%
\displaystyle e^{{V}_{\gamma}} ~ \mapsto  
e^{{V}_{\gamma}} & (\gamma \in 
H_1(\hat{\Sigma}\setminus\hat{P})). 
\end{cases}
\end{eqnarray}
Then the map ${\mathfrak K}_{{\gamma_0}}$ satisfies
\begin{equation}
{\mathcal S}_{-} = {\mathcal S}_{+}
\circ {\mathfrak K}_{{\gamma}_{0}}.
\end{equation}
The map ${\mathfrak K}_{{\gamma}_0}$ is called 
the {\em Stokes automorphism} for the saddle class 
${\gamma}_0$ associated with 
a degenerate saddle trajectory $\ell_0$.
 
\subsection{$S^1$-action on potential and jump formulas}
\label{subsec:S1action}
Let us give an alternative interpretation 
of the jump formulas \eqref{eq:DDP-analytic} and 
\eqref{eq:loop-type-degeneration-analytic} in view of 
the deformation of the potential $Q(z,\eta)$. 
We consider a particular deformation realized 
by an action of the unit circle 
$S^1 = \{e^{i\theta} 
~|~\theta \in {\mathbb R} \}$,
which we call the
{\em $S^1$-action on the potential $Q(z,\eta)$}.

Suppose that the Stokes graph $G_0 = G(\phi)$ has a 
unique regular or degenerate saddle trajectory $\ell_0$. 
Take a number $r > 0$ and consider the family 
of Schr{\"o}dinger equations
\begin{equation} \label{eq:Sch-theta}
\left( \frac{d^{2}}{dz^{2}} - \eta^{2} 
Q^{(\theta)}(z,\eta) \right) 
\psi^{(\theta)}(z,\eta)= 0 
\quad(-r\le\theta\le+r),
\end{equation}
\[
Q^{(\theta)}(z,\eta)=
Q^{(\theta)}_{0}(z)+\eta^{-1}Q^{(\theta)}_{1}(z) 
+ \eta^{-2} Q^{(\theta)}_{2}(z) 
+ \cdots.
\]
Here the family of potentials 
$\{Q^{(\theta)}(z,\eta)~|~-r\le\theta\le+r \}$
is defined by 
\begin{equation} \label{eq:Qtheta}
Q^{(\theta)}(z,\eta) = e^{2i\theta}Q(z,e^{i\theta}\eta),
\end{equation}
where $Q(z,\eta)$ is the original potential of \eqref{eq:Sch}.
{ 
We call this family the 
{\em $S^1$-family for the potential $Q(z,\eta)$}.
} 
Note that \eqref{eq:Qtheta} satisfies 
Assumptions \ref{assumption:zeros and poles} 
and \ref{assumption:summability} 
for all $\theta \in [-r,+r]$. 
Taking $r>0$ sufficiently small, we may  
assume that the Stokes graph defined from 
$Q^{(\theta)}(z,\eta)$ is saddle-free if 
$\theta \ne 0$, $\theta \in [-r,+r]$.
Since the principal terms of potentials satisfy
\begin{equation} \label{eq:Qtheta0}
Q^{(\theta)}_0(z) = e^{2i\theta} Q_0(z),
\end{equation} 
the quadratic differential
associated with \eqref{eq:Sch-theta} 
is noting but $\phi_{\theta}$ 
defined in \eqref{eq:phi-theta}. 
The Stokes graph for $Q^{(0)}(z,\eta) = Q(z,\eta)$ 
coincides with the original Stokes graph $G_0$ 
containing the saddle trajectory $\ell_0$. 

For any fixed $\theta$, 
let $S_{\rm odd}^{(\theta)}(z,\eta)$ (resp.,  
$S^{\rm reg (\theta)}_{\rm odd}(z,\eta)$) be the 
formal power series defined in the same manner as 
\eqref{eq:Sodd-and-Seven} (resp., \eqref{eq:Soddreg})
from the Schr{\"o}dinger equation \eqref{eq:Sch-theta}. 
The following statement immediately follows from 
the uniqueness of formal solutions of the 
Riccati equation associated with \eqref{eq:Sch-theta}
(see Section \ref{secton:Riccati-equation}).
\begin{lem} \label{lemma:Sodd-theta-and-Sodd}
The following identities holds:
\begin{equation} \label{eq:Sodd-theta}
S_{\rm odd}^{(\theta)}(z,\eta) = 
S_{\rm odd}(z,e^{i\theta}\eta), \quad 
S^{\rm reg (\theta)}_{\rm odd}(z,\eta) =
S^{\rm reg}_{\rm odd}(z,e^{i\theta}\eta).
\end{equation}
\end{lem}

We define the Voros symbols 
$e^{W_{\beta}^{(\theta)}}$ and 
$e^{V_{\gamma}^{(\theta)}}$ 
($\beta \in H_1(\hat{\Sigma}\setminus\hat{P}_0;
\hat{P}_{\infty})$, 
$\gamma \in H_1(\hat{\Sigma}\setminus\hat{P}_0)$) 
of the Schr{\"o}dinger equation \eqref{eq:Sch-theta} by
\begin{equation}
\label{eq:WVtheta1}
W_{\beta}^{(\theta)}(\eta) = \int_{\beta} 
S_{\rm odd}^{\rm reg (\theta)}(z,\eta)dz, \quad
V_{\gamma}^{(\theta)}(\eta) = \oint_{\gamma} 
S_{\rm odd}^{(\theta)}(z,\eta)dz.
\end{equation}
Note that,
by \eqref{eq:Qtheta0},
the Riemann surface $\hat{\Sigma}$
defined from \eqref{eq:Sch-theta} does not depend on $\theta$.
Thus, 
the homology groups $H_1(\hat{\Sigma}\setminus\hat{P}_0)$ 
and $H_1(\hat{\Sigma}\setminus\hat{P}_0;\hat{P}_{\infty})$
for the Schr{\"o}dinger equations \eqref{eq:Sch-theta} 
also do not depend on $\theta$.
Thus, the equality \eqref{eq:Sodd-theta} implies that 
\begin{equation} \label{eq:Voros-theta}
W_{\beta}^{(\theta)}(\eta) = 
W_{\beta}(e^{i\theta}\eta),\quad
V_{\gamma}^{(\theta)}(\eta) = 
V_{\gamma}(e^{i\theta}\eta)
\end{equation}
hold as formal series.
\begin{lem}
For any $\theta \ne 0$ satisfying 
$-r \le \theta \le + r$, 
the formal series $e^{W_{\beta}^{(\theta)}}$ and 
$e^{V_{\gamma}^{(\theta)}}$ are Borel summable 
(in the direction $0$), and the equalities
\begin{equation} \label{eq:Borel-sum-theta-and-0}
{\mathcal S}[e^{W_{\beta}^{(\theta)}}](\eta) = 
{\mathcal S}_{\theta}[e^{W_{\beta}}](e^{i\theta}\eta), \quad
{\mathcal S}[e^{V_{\gamma}^{(\theta)}}](\eta) = 
{\mathcal S}_{\theta}[e^{V_{\gamma}}](e^{i\theta}\eta)
\end{equation}
hold as analytic functions of $\eta$ on  
$\{\eta\in{\mathbb R}~|~\eta\gg1 \}$.
\end{lem}
\begin{proof}
Since the argument is the same,
let us concentrate on the case of $W_{\beta}^{(\theta)}$.
{ 
It follows from \eqref{eq:fB-and-ftheta-B}
that the equality  
\[
W^{(\theta)}_{\beta, B}(y) = 
e^{-i\theta} W_{\beta, B}(e^{-i\theta}y)
\]
holds near $y=0$. 
} 
Here $W_{\beta, B}(y)$ and 
$W^{(\theta)}_{\beta, B}(y)$ are the Borel transform 
of $W_{\beta}(\eta)$ and $W^{(\theta)}_{\beta}(\eta)$, 
respectively. Since the quadratic differential 
$\phi_{\theta}$ is saddle-free, $W_{\beta}(\eta)$ is 
Borel summable in the direction $\theta$ 
by Corollary \ref{cor:summability}. 
{ 
Then, Lemma \ref{lemma:0-sum-and-theta-sum} implies 
that $W^{(\theta)}_{\beta}(\eta)$ is Borel summable 
in the direction $0$ and we have the equality
\begin{equation} \label{eq:Borel-sum-theta-and-0-2}
{\mathcal S}[W^{(\theta)}_{\beta}(\eta)] = 
{\mathcal S}_{\theta}[W_{\beta}(e^{i\theta}\eta)] 
\end{equation}
by the definition \eqref{eq:Borel-sum-theta} 
of the Borel sum in the direction $\theta$.
} 
Then the desired equality
\eqref{eq:Borel-sum-theta-and-0} 
follow from \eqref{eq:Borel-sum-theta-and-0-2}. 
\end{proof}

The equality \eqref{eq:Borel-sum-theta-and-0} 
and Lemma \ref{lemma:saddle-reduction-of-Voros} 
imply that the limit $\delta \rightarrow +0$ of 
the function ${\mathcal S}[e^{W_{\beta}^{(\pm\delta)}}](\eta)$ 
exists and coincides with 
${\mathcal S}_{\pm}[e^{W_{\beta}}](\eta)$ defined in 
Section \ref{section:saddle-reduction}. That is, 
\begin{equation}
\lim_{\delta \rightarrow +0}
{\mathcal S}[e^{W_{\beta}^{(\pm\delta)}}](\eta) = 
{\mathcal S}_{\pm}[e^{W_{\beta}}](\eta)
\end{equation}
holds 
on $\{\eta \in {\mathbb R}~|~\eta\gg1 \}$. 
Similarly, we also have 
\begin{equation}
\lim_{\delta \rightarrow +0}
{\mathcal S}[e^{V_{\gamma}^{(\pm\delta)}}](\eta) = 
{\mathcal S}_{\pm}[e^{V_{\gamma}}](\eta).
\end{equation}
Therefore, we obtain the following jump formulas  
for the $S^1$-action
on the potential from 
Theorem \ref{thm:DDP-analytic} and 
\ref{thm:loop-type-degeneration-analytic}.

\begin{thm}
\label{thm:jump1}
\begin{itemize}
\item[(a).] 
Suppose that $\ell_0$ is a regular saddle trajectory
with the associated saddle class $\gamma_0$. 
Then we have
\begin{eqnarray}\hspace{-3.em}
\label{eq:jump1}
\begin{cases}
\displaystyle%
\lim_{\delta \rightarrow +0} 
{\mathcal S}[e^{W_{\beta}^{(-\delta)}}](\eta) = 
\lim_{\delta \rightarrow +0} \left(
{{\mathcal S}[e^{W_{\beta}^{(+\delta)}}](\eta) 
\bigl(1+ {\mathcal S}[e^{V_{\gamma_0}^{(+\delta)}}](\eta) 
\bigr)^{-\langle \gamma_0,\beta \rangle}} \right), \\[+.5em]
\displaystyle%
\lim_{\delta \rightarrow +0} 
{\mathcal S}[e^{V_{\gamma}^{(-\delta)}}](\eta) = 
\lim_{\delta \rightarrow +0} \left(
{{\mathcal S}[e^{V_{\gamma}^{(+\delta)}}](\eta) 
\bigl(1+ {\mathcal S}[e^{V_{\gamma_0}^{(+\delta)}}](\eta) 
\bigr)^{-(\gamma_0,\gamma)}} \right),
\end{cases}
\end{eqnarray}
for any $\beta \in H_1(\hat{\Sigma}\setminus\hat{P}_0,
\hat{P}_{\infty})$ and any $\gamma \in H_1(\hat{\Sigma}
\setminus\hat{P})$.
\item[(b).]
Suppose that $\ell_0$ is a degenerate saddle trajectory
with the associated saddle class ${\gamma}_0$. 
Then we have
\begin{eqnarray}\hspace{-3.em}
\begin{cases}
\displaystyle%
\lim_{\delta \rightarrow +0} 
{\mathcal S}[e^{W_{\beta}^{(-\delta)}}](\eta) = 
\lim_{\delta \rightarrow +0} \left(
{{\mathcal S}[e^{W_{\beta}^{(+\delta)}}](\eta) 
\bigl(1 - {\mathcal S}[e^{V_{{\gamma}_0}^{(+\delta)}}](\eta) 
\bigr)^{\langle {\gamma}_0, \beta \rangle}} \right), \\
\displaystyle%
\lim_{\delta \rightarrow +0} 
{\mathcal S}[e^{V_{\gamma}^{(-\delta)}}](\eta) =
\lim_{\delta \rightarrow +0} 
{{\mathcal S}[e^{V_{\gamma}^{(+\delta)}}](\eta)},
\end{cases}
\end{eqnarray}%
for any $\beta \in H_1(\hat{\Sigma}\setminus\hat{P}_0,
\hat{P}_{\infty})$ and any $\gamma \in H_1(\hat{\Sigma}
\setminus\hat{P})$.
\end{itemize}
\end{thm}

This concludes the exposition of the materials from the exact WKB analysis.
We revisit the jump formulas and Stokes automorphisms in this section
in view of cluster algebra theory later in Sections \ref{sec:mut_voros} and 
\ref{sec:identities}.

\section{Cluster algebras with coefficients}

In this section we summarize the basic notions and properties in cluster algebras
which we will use in this paper.
We also introduce
the notion of signed mutations of seeds in Section \ref{subsec:monomial}
 to accommodate the forthcoming results of this paper.
We ask the readers
to consult \cite{Fomin07,Nakanishi11c}, for example, for further explanations
if necessary,
though it is our hope that the readers will not be bogged down in the 
cluster algebras machinery presented in this and the next sections,
 and smoothly proceed to Section 
\ref{sec:mutationofStokes} where our real work starts.

\subsection{Semifields} 
Let us start from the notion of semifields, where ``coefficients'' of
cluster algebras live.

\begin{defn}
A {\em semifield\/} $\bbP$ is a  multiplicative abelian group endowed with an addition
denoted by $\oplus$,
which is commutative, associative, and distributive with respect to the multiplication.
\end{defn}

To say it plainly, a semifield is almost a field, but without zero and subtraction.
In this paper we mainly use the following  examples.

\begin{ex}
\label{ex:semifield1}
Let  $u=(u_i)_{i=1}^n$ be an $n$-tuple of formal variables.
\par
 (a) {\em The universal semifield $\mathbb{Q}_+(u)$ of $u$}.
This is the semifield of all nonzero rational functions of $u$
which have subtraction-free expressions,
where the multiplication $\cdot$ and the addition $\oplus$ 
are defined by the usual  $\times$ and $+$ in 
the rational function field $\mathbb{Q}(u)$ of $u$.
As a standard example, the polynomial $u_1^2-u_1u_2 +u_2^2$ does not seem to belong to
$\mathbb{Q}_+(u)$; however, it actually does, since
\begin{align}
u_1^2-u_1u_2 + u_2^2
=
\frac{
u_1^3+u_2^3
}
{
u_1 + u_2
}
=
\frac{
u_1^3\oplus u_2^3
}
{
u_1 \oplus  u_2
}
\in
\mathbb{Q}_+(u).
\end{align}
\par
(b) {\em The tropical semifield\/ $\mathrm{Trop}(u)$ of $u$}.
This is the  multiplicative free abelian group
generated by $u$,
endowed with the {\em tropical sum} $\oplus$ defined by
\begin{align}
\label{eq:pi1}
\prod_{i=1}^n u_i^{a_i}
\oplus
\prod_{i=1}^n u_i^{b_i}
:=
\prod_{i=1}^n u_i^{\min(a_i,b_i)}.
\end{align}
It is called so because it is essentially the exponential and multivariable version of
the {\em tropical semiring\/} (also known as
the {\em min-plus algebra\/}) $a\oplus b:=\min(a,b)$, $a\otimes b := a+b$.
The tropical semiring
  is the central object of the {\em tropical mathematics\/}
well studied since 90's.
See \cite{Speyer04} for the subject and the explanation for this peculiar terminology.
\par
(c) {\em The tropicalization map}. There is the natural semifield homomorphism
\begin{align}
\label{eq:trop1}
\begin{matrix}
\pi_{\mathrm{trop}}: & \mathbb{Q}_+(u) &\rightarrow & \mathrm{Trop}(u)\\
& u_i & \mapsto & u_i\\
& c & \mapsto & 1 & (c\in \mathbb{Q}_+).
\end{matrix}
\end{align}
For example,
\begin{align}
\label{eq:pi2}
\pi_{\mathrm{trop}}:
\frac{2u_1^2\oplus u_1u_2}
{u_1^2u_2\oplus 2u_1u_2^2}
=
\frac{u_1 (2u_1\oplus u_2)}
{u_1u_2(u_1\oplus 2 u_2)}
\mapsto
\frac{u_1}
{u_1u_2}=
u_2^{-1},
\end{align}
where we used \eqref{eq:pi1}.
So, plainly speaking, the tropicalization map
$\pi_{\mathrm{trop}}$ is the operation of
taking the ``principal term" (in the  sense of \eqref{eq:pi2}) of a rational
function in $\mathbb{Q}_+(u) $.
\end{ex}

For a given semifield $\bbP$,
let $\bbZ\bbP$ denote the group ring of $\bbP$ over $\bbZ$.
Namely, $\bbZ\bbP$ is the commutative ring of all formal finite sums
$\sum_{r=1}^m n_r p_r$ ($n_r\in \mathbb{Z}, p_r\in \bbP$).
Note that we have two kinds of additions, $\oplus$ for $\bbP$ and the addition $+$
for the group ring.
For $\oplus$, there is no subtraction, but for $+$, there is the usual subtraction $-$
in $\bbZ\bbP$.
It is known that $\bbZ\bbP$ is a domain \cite[Section 1.2]{Fomin02};
namely, it has no zero divisors.
Thus, the field of fractions  $\bbQ\bbP$ of  the ring $\bbZ\bbP$ is well-defined.

\subsection{Mutation of seeds and cluster algebra with coefficients}

Let us recall the notions of mutations of seeds
and cluster algebras,
following \cite{Fomin03a,Fomin07}.

To introduce a cluster algebra (with coefficients),
let us first fix a positive integer $n$ called the {\em rank},
 and a semifield $\bbP$
called the {\em coefficient semifield}.
We choose an
$n$-tuple of formal variables, say,
 $w=(w_1,\dots,w_n)$,
 and consider
 the field of the rational functions in $w$
 with coefficients in $\bbQ\bbP$,
 which denoted by $\bbQ\bbP(w)$.

A {\em (labeled) seed $(B,x,y)$ with coefficients in $\bbP$} is a triplet with the following data:
\begin{itemize}
\item
an {\em exchange matrix\/} $B=(b_{ij})_{i,j=1}^n$,
which is a  skew-symmetric integer matrix,
\item a {\em cluster\/} $x=(x_i)_{i=1}^n$,
which is
 an $n$-tuple  of algebraically independent
elements in $\bbQ\bbP(w)$ over $\bbQ\bbP$,
\item  a {\em coefficient tuple\/} $y=(y_i)_{i=1}^n$,
which is an $n$-tuple  of elements in $\bbP$.
\end{itemize}
Each $x_i$ and $y_i$ are called
 a {\em cluster variable} and  {\em coefficient},
 respectively.
In this paper we call them,  a little casually, an {\em $x$-variable\/} and a {\em $y$-variable},
respectively.
(They correspond to an {\em $\mathcal{A}$-coordinate\/} and an {\em $\mathcal{X}$-coordinate\/}
in \cite{Fock03}, respectively.)

For any seed $(B,x,y)$ and any $k=1,\dots,n$,
we define
another seed 
$(B',x',y')$,
called the {\em mutation of $(B,x,y)$ at $k$}
and denoted by $\mu_k(B,x,y)$,
by the following relations:
\begin{align}
\label{eq:bmut}
b'_{ij}&=
\begin{cases}
-b_{ij}& \mbox{$i=k$ or $j=k$}\\
b_{ij}+[- b_{ik}]_+ b_{kj} + b_{ik}[ b_{kj}]_+
& \mbox{$i,j\neq k$,}\\
\end{cases}
\\
\label{eq:ymut}
y'_i&=
\begin{cases}
y_k{}^{-1} & i= k\\
\displaystyle
y_i \frac{(1\oplus y_k)^{[-b_{ki}]_+}}
{(1\oplus y_k{}^{-1})^{[b_{ki}]_+}}
& i\neq k,\\
\end{cases}
\\
\label{eq:xmut}
x'_i&=
\begin{cases}
\displaystyle
{x_k}^{-1}\left(
\frac{1}{1\oplus y_k^{-1}}
\prod_{j=1}^n x_j{}^{[b_{jk}]_+}
+ 
\frac{1}{1\oplus y_k}
\prod_{j=1}^n x_j{}^{[-b_{jk}]_+}
\right)
& i=k\\
x_i & i\neq k.\\
\end{cases}
\end{align}
Here, for any integer $a$, we set $[a]_+:=\max(a,0)$.
The above relations
are called the {\em exchange relations}.
The involution property $\mu_k^2=\mathrm{id}$ holds.

\begin{defn}
Let us fix an arbitrary seed $(B^0,x^0,y^0)$ with coefficients in $\bbP$,
and call it the {\em initial seed}.
Then, repeat mutations from the initial seed to all directions.
Let $\mathrm{Seed}(B^0,x^0,y^0;\bbP)$ denote the set of all  so obtained seeds,
including the initial one;
namely,
\begin{align}
\mathrm{Seed}(B^0,x^0,y^0;\bbP)
=\left\{\mu_{k_N}\cdots \mu_{k_1}(B^0,x^0,y^0)
\mid N \geq 0;\ k_1,\dots,k_N\in \{1,\dots,n\}
\right\}.
\end{align}
The {\em cluster algebra 
$\mathcal{A}(B^0,x^0,y^0;\bbP)$ with coefficients in $\bbP$} is
the ${\bbZ\bbP}$-subalgebra of $\bbQ\bbP(w)$
generated by all $x$-variables
belonging to  seeds in $\mathrm{Seed}(B^0,x^0,y^0;\bbP)$.
A seed in $\mathrm{Seed}(B^0,x^0,y^0;\bbP)$ is called
a {\em seed of $\mathcal{A}(B^0,x^0,y^0;\bbP)$}.
\end{defn}

What is important in our application is not
the algebra $\mathcal{A}(B^0,x^0,y^0;\bbP)$ itself  but the exchange relations
 \eqref{eq:bmut}--\eqref{eq:xmut}.
They are the abstraction of  relations occurring in Lie theory due to
Fomin and Zelevinsky. For example, the relation \eqref{eq:xmut} typically
appears as relations in the coordinate rings of certain algebraic varieties related to Lie group,
e.g., Grassmannians, $SL(n)$, etc.
The relation \eqref{eq:ymut} is a sort of the ``dual'' of
the relation \eqref{eq:xmut} as we see below.
As explained at the beginning of Section \ref{sec:introduction},
it is nowadays known that the cluster algebra structure,
(i.e., the exchange relations
 \eqref{eq:bmut}--\eqref{eq:xmut})
serves
a common
 underlying algebraic/combinatorial structure
 in several branches of mathematics
 --- from algebra, geometry, analysis to combinatorics.
 Such  ubiquity as a common structure reminds us of {\em root systems}.
 Indeed, cluster algebra theory may be regarded as an extended theory of root systems in several aspects.
An explicit example of the exchange relations
 \eqref{eq:bmut}--\eqref{eq:xmut}
will be exhibited in Example \ref{ex:pentagon1} later.

For each seed $(B,x,y)$ with coefficients in $\bbP$, we define  {\em $\hat{y}$-variables\/}
$\hat{y_1},\dots, \hat{y}_n \in \bbQ\bbP(w)$  by
\begin{align}
\label{eq:yhat}
\hat{y}_i=
y_i \prod_{j=1}^n x_j{}^{b_{ji}}.
\end{align}
It is easy to verify the following property by  using \eqref{eq:bmut}--\eqref{eq:xmut}.
\begin{prop}[{\cite[Prop. 3.9]{Fomin07}}] 
\label{prop:yhat1}
Under the mutation $(B',x',y')=\mu_k(B,x,y)$ the following relation holds:
\begin{align}
\label{eq:ymut4}
\hat{y}'_i&=
\begin{cases}
\hat{y}_k{}^{-1} & i= k\\
\displaystyle
\hat{y}_i \frac{(1+ \hat{y}_k)^{[-b_{ki}]_+}}
{(1+ \hat{y}_k{}^{-1})^{[b_{ki}]_+}}
& i\neq k.\\
\end{cases}
\end{align}
\end{prop}
In other words, $\hat{y}$-variables mutate
just in the same way as $y$-variables.

\subsection{Quivers}

It is often convenient to represent a skew-symmetric matrix $B=(b_{ij})_{i,j=1}^n$
by a (labeled) quiver $Q$ whose vertices are labeled by $1,\dots,n$.
In our convention, we write $b_{ij}$ arrows from  vertex $i$ to  vertex $j$ if and only if
$b_{ij}>0$. This gives a one-to-one correspondence between skew-symmetric matrices
and quivers without any loops (1-cycles) and  oriented 2-cycles.
Here is an example:
\begin{align}
\label{eq:BQ1}
B=
\begin{pmatrix}
0 & -1 & 2\\
1 & 0 & -1\\
-2 & 1 & 0
\end{pmatrix}
\quad
\longleftrightarrow
\quad
Q = 
\raisebox{-14pt}{
\begin{xy}
(0,0)*\cir<2pt>{},
(5,8.7)*\cir<2pt>{},
(10,0)*\cir<2pt>{},
(-4,0)*{1},
(14,0)*{3},
(5,12)*{2},
\ar (4,7);(1,1.7)
\ar (9,1.7);(6,7)
\ar (2,0.7);(8,0.7)
\ar (2,-0.7);(8,-0.7)
\end{xy}
}.
\end{align}

In terms of quivers,
the exchange relation \eqref{eq:bmut} for the mutation at $k$ 
is translated as follows.
\begin{itemize}
\item[] Step 1. For each pair of an  arrow from $i$ to $k$ and an 
arrow from $k$ to $j$, add an arrow from $i$ to $j$.

\item[]
Step 2. Reverse all  arrows incident with $k$.

\item[]
Step 3. Remove the arrows in a maximal set of pairwise disjoint 2-cycles.
\end{itemize}

For example, for the quiver $Q$ in \eqref{eq:BQ1}, the mutation at $1$ is
done in the following manner.
\begin{align}
\raisebox{-14pt}{
\begin{xy}
(0,0)*\cir<2pt>{},
(5,8.7)*\cir<2pt>{},
(10,0)*\cir<2pt>{},
(-4,0)*{1},
(14,0)*{3},
(5,12)*{2},
\ar (4,7);(1,1.7)
\ar (9,1.7);(6,7)
\ar (2,0.7);(8,0.7)
\ar (2,-0.7);(8,-0.7)
\end{xy}
}
\buildrel
{\text{Step 1}}
\over
{
\Longrightarrow
}
\raisebox{-14pt}{
\begin{xy}
(0,0)*\cir<2pt>{},
(5,8.7)*\cir<2pt>{},
(10,0)*\cir<2pt>{},
(-4,0)*{1},
(14,0)*{3},
(5,12)*{2},
\ar (4,7);(1,1.7)
\ar (9,1.7);(6,7)
\ar (7,7.5);(10,2.2)
\ar (8,8);(11,2.7)
\ar (2,0.7);(8,0.7)
\ar (2,-0.7);(8,-0.7)
\end{xy}
}
\buildrel
{\text{Step 2}}
\over
{
\Longrightarrow
}
\raisebox{-14pt}{
\begin{xy}
(0,0)*\cir<2pt>{},
(5,8.7)*\cir<2pt>{},
(10,0)*\cir<2pt>{},
(-4,0)*{1},
(14,0)*{3},
(5,12)*{2},
\ar (1,1.7);(4,7)
\ar (9,1.7);(6,7)
\ar (7,7.5);(10,2.2)
\ar (8,8);(11,2.7)
\ar (8,0.7);(2,0.7)
\ar (8,-0.7);(2,-0.7)
\end{xy}
}
\buildrel
{\text{Step 3}}
\over
{
\Longrightarrow
}
\raisebox{-14pt}{
\begin{xy}
(0,0)*\cir<2pt>{},
(5,8.7)*\cir<2pt>{},
(10,0)*\cir<2pt>{},
(-4,0)*{1},
(14,0)*{3},
(5,12)*{2},
\ar (1,1.7);(4,7)
\ar (6,7);(9,1.7)
\ar (8,0.7);(2,0.7)
\ar (8,-0.7);(2,-0.7)
\end{xy}
}
\end{align}

\subsection{Tropicalization of $y$-variables and tropical sign}

In this paper we mainly use the following  two choices of  $y$-variables.
(See Example \ref{ex:semifield1}.)
\begin{itemize}
\item[(a).]
We set 
the coefficient semifield 
as $\bbQ_+(y^0)$ with $y^0=(y^0_1,\dots,y^0_n)$;
furthermore, we set  the initial $y$-variables
as $y^0$.
We call the $y$-variables of $\mathcal{A}(B^0,x^0,y^0; \bbQ_+(y^0))$
the  {\em universal $y$-variables}.
\item[(b).]
We set 
the coefficient semifield
as $\mathrm{Trop}(y^0)$ with $y^0=(y^0_1,\dots,y^0_n)$;
furthermore, we set  the initial $y$-variables
as $y^0$.
We call  the $y$-variables of $\mathcal{A}(B^0,x^0,y^0; \mathrm{Trop}(y^0))$
the {\em tropical $y$-variables}.
(It is more standard to call them  the {\em principal coefficients\/} \cite{Fomin07},
but here we emphasize their tropical nature.)
\end{itemize}

The tropical $y$-variables are obtained from the universal $y$-variables
by applying  the tropicalization map
 in Example \ref{ex:semifield1},
\begin{align}
\label{eq:trop2}
\pi_{\mathrm{trop}}:  \mathbb{Q}_+(y^0) \rightarrow  \mathrm{Trop}(y^0).
\end{align}
Namely, for any universal $y$-variable $y_i\in \bbQ_+(y^0)$,
let $[{y}_i]:=
\pi_{\mathrm{trop}}(y_i)\in \mathrm{Trop}(y^0)$.
Since $\pi_{\mathrm{trop}}$ is a semifield homomorphism,
it preserves the exchange relation \eqref{eq:ymut};
therefore, it commutes with mutations.
Thus, $[{y}_i]$ is  a tropical $y$-variable.
{}From now on we conveniently use this expression for  tropical $y$-variables.

By definition,  a tropical $y$-variable $[{y}_i]$  is a Laurent monomial of the initial tropical $y$-variables $y^0$;
namely, it is written in the form
\begin{align}
\label{eq:cvec}
[{y}_i]=
\prod_{j=1}^n (y_j^0) ^{c_j},
\end{align}
where $c=c(y_i)=(c_j)_{j=1}^n$ is an integer vector depending on ${y}_i$.
The vector $c(y_i)$ is introduced in  \cite{Fomin07} and 
  called the {\em $c$-vector of $y_i$}.
 
We say that an integer vector is
{\em positive\/} (resp., {\em negative}) if it
 is a nonzero vector and its components are all
nonnegative (resp., nonpositive).
We have the following important property of $c$-vectors.
\begin{thm} [{Sign coherence of $c$-vectors (\cite[Prop. 5.7]{Fomin07},
\cite[Theorem 1.7]{Derksen10})}]
Any $c$-vector  is either a positive vector or a negative vector.
\end{thm}

Thanks to the theorem,
we have the notion of  the tropical sign.

\begin{defn}  Let $\mathcal{A}(B^0,x^0,y^0;\bbQ_+(y^0))$ be
a cluster algebra with  universal $y$-variables,
and let $(B,x,y)$ be its seed.
Then, to each  component $y_i$ of $y$ we
assign the {\em tropical sign\/} $\ve(y_i)$
as $\ve(y_i)=+$ (resp., $\ve(y_i)=-$)
 if the $c$-vector $c(y_i)$ is positive (resp., negative).
\end{defn}

Throughout the paper we conveniently identify 
the signs $\ve=\pm$
and 
the numbers $\ve=\pm 1$.

It is important that the tropical sign is a relative concept
depending on $(B,x,y)$ {\em and\/} the initial seed $(B^0,x^0,y^0)$.
By the definition of the tropical sign, we have
 \begin{align}
\label{eq:tropsign1}
 1\oplus [y_i]^{\varepsilon(y_i)}=1
 \end{align}
in $\mathrm{Trop}(y^0)$.

\subsection{$\ve$-expression of exchange relations}
Let us focus on some fine property of the exchange relations
 \eqref{eq:ymut} and  \eqref{eq:xmut}.
It is easy to check that \eqref{eq:ymut} and \eqref{eq:xmut} can be expressed alternatively as
\cite{Keller11, Nakanishi11c}
\begin{align}
\label{eq:ymut3}
y'_i&=
\begin{cases}
y_k{}^{-1} & i= k\\
\displaystyle
y_i y_{k}{}^{[\varepsilon b_{ki}]_+}
(1\oplus y_k{}^{\varepsilon})^{-b_{ki}}
& i\neq k,\\
\end{cases}
\allowdisplaybreaks
\\
\label{eq:xmut3}
x'_i&=
\begin{cases}
\displaystyle
x_k{}^{-1} 
\left(
\prod_{j=1}^n x_{j}{}^{[-\varepsilon b_{jk}]_+}
\right)
\frac{
1+ \hat{y}_k{}^{\varepsilon}}
{
1\oplus y_k{}^{\varepsilon}
}
& i= k\\
\displaystyle
x_i 
& i\neq k,\\
\end{cases}
\end{align}
where $\varepsilon \in \{+,-\}=\{1,-1\}$,
and the right hand sides  are {\em independent of the choice of $\varepsilon$}.
We call them the {\em $\varepsilon$-expression\/} of the exchange relations.

Let us specialize the $y$-variables $y_i$ in
\eqref{eq:ymut3} and \eqref{eq:xmut3}
to  tropical $y$-variables $[y_i]$;
furthermore, let us specialize $\varepsilon$ therein
 to the {\em tropical sign\/} $\varepsilon(y_k)$.
Then, by \eqref{eq:tropsign1},
the relations \eqref{eq:ymut3} and \eqref{eq:xmut3} reduce to the following
ones:
 \begin{align}
\label{eq:ymut5}
[y'_i]&=
\begin{cases}
[y_k]^{-1} & i= k\\
\displaystyle
[y_i][ y_{k}]^{[\varepsilon (y_k)b_{ki}]_+}
& i\neq k,\\
\end{cases}
\\
\label{eq:xmut5}
x'_i&=
\begin{cases}
\displaystyle
x_k{}^{-1} 
\left(
\prod_{j=1}^n x_{j}{}^{[-\varepsilon(y_k) b_{jk}]_+}
\right)
(1+ \hat{y}_k{}^{\varepsilon(y_k)})
& i= k\\
\displaystyle
x_i 
& i\neq k,\\
\end{cases}
\end{align}
where $\hat{y}_k = [y_k] \prod_{j=1}^n x_j^{b_{ji}}$.
These are an alternative expressions of
the exchange relations for  tropical $y$-variables
and  $x$-variables with  tropical $y$-variables as coefficients.

\subsection{Signed mutations}
\label{subsec:monomial}

We introduce some new notions in cluster algebras,
motivated by the forthcoming results in this paper.

For any seed $(B,x,y)$ with coefficients in  $\mathrm{Trop}(y^0)$,
any $k=1,\dots,n$, and any sign $\varepsilon\in \{+,-\}$,
we introduce the {\em signed monomial mutation\/} $(B',x',y')=m^{(\ve)}_k(B,x,y)$
by the following exchange relation,
where $B'$ is defined as usual:
\begin{align}
\label{eq:ymut6}
y'_i&=
\begin{cases}
y_k^{-1} & i= k\\
\displaystyle
y_i  y_{k}{}^{[\varepsilon b_{ki}]_+}
& i\neq k,\\
\end{cases}
\\
\label{eq:xmut6}
x'_i&=
\begin{cases}
\displaystyle
x_k{}^{-1} 
\prod_{j=1}^n x_{j}{}^{[-\varepsilon b_{jk}]_+}
& i= k\\
\displaystyle
x_i 
& i\neq k.\\
\end{cases}
\end{align}
Unlike \eqref{eq:ymut5} and \eqref{eq:xmut5},
they {\em depend on $\ve$}.
If we  set $\ve$
to the tropical sign $\ve(y_k)$ for the universal $y$-variables,
the relation \eqref{eq:ymut6} reduces to the exchange relation \eqref{eq:ymut5}
of the tropical $y$-variables.
Starting from a given initial seed $(B^0,x^0,y^0)$,
we obtain a family of seeds by repeating the above mutations
to any direction $k$ and any sign $\varepsilon$.
We call  so obtained $x_i$'s and $y_i$'s 
the {\em monomial $x$-variables\/} and the
{\em monomial $y$-variables}, respectively.


In the same setting of seeds, we   also consider another kind of mutation,
the {\em signed mutation\/} $(B',x',y')=\mu^{(\ve)}_k(B,x,y)$,
by keeping \eqref{eq:ymut6} and
replacing  \eqref{eq:xmut6} by the following:
\begin{align}
\label{eq:xmut7}
x'_i&=
\begin{cases}
\displaystyle
x_k{}^{-1} 
\left(
\prod_{j=1}^n x_{j}{}^{[-\varepsilon b_{jk}]_+}
\right)
(1+ \hat{y}_k{}^{\varepsilon})
& i= k\\
\displaystyle
x_i 
& i\neq k,\\
\end{cases}
\end{align}
where
$\hat{y}_k = y_k \prod_{j=1}^n x_j^{b_{jk}}$.
Again, it {\em depends on $\ve$}.
If we set $\ve$
to the tropical sign $\ve(y_k)$ of the universal $y$-variables,
then the relation \eqref{eq:xmut7} reduces to the exchange relation \eqref{eq:xmut5}
of $x$-variables with tropical $y$-variables as coefficients.

We have a natural extension of Proposition \ref{prop:yhat1}.
\begin{prop}
\label{prop:yhat2}
Under the signed mutation $(B',x',y')=\mu^{(\varepsilon)}_k(B,x,y)$
with the exchange relations
\eqref{eq:ymut6} and \eqref{eq:xmut7},
 $\hat{y}$-variables
 $\hat{y}_i = y_i \prod_{j=1}^n x_j^{b_{ji}}$
satisfy the  exchange relation
\begin{align}
\label{eq:ymut7}
\hat{y}'_i&=
\begin{cases}
\hat{y}_k{}^{-1} & i= k\\
\displaystyle
\hat{y}_i \hat{y}_{k}{}^{[\varepsilon b_{ki}]_+}
(1+ \hat{y}_k{}^{\varepsilon})^{-b_{ki}}
& i\neq k,\\
\end{cases}
\end{align}
which is equivalent to \eqref{eq:ymut4}.
In particular, the mutation of $\hat{y}$-variables does not depend on the sign
$\ve$.
\end{prop}

The proof can be done in a similar (and a little easier) calculation as for Proposition \ref{prop:yhat1}.
However, the result is new in the literature.

\subsection{Periodicity in cluster algebras}
Let us introduce the notion of periodicity  in a cluster algebra.
We call a sequence $\vec{k}=(k_t)_{t=1}^N$ 
with $k_t\in \{1,\dots,n\}$
a {\em mutation sequence},
and we naturally identify it with
the sequence (composition) 
of mutations $\mu_{\vec{k}}:=\mu_{k_N}\circ \cdots \circ \mu_1$.

\begin{thm}
\label{thm:period1}
 Let  $\mathcal{A}(B^0,x^0,y^0;\bbQ_+(y^0))$ be a cluster algebra
with  universal $y$-variables,
and let $\vec{k}=(k_t)_{t=1}^N$ be a mutation sequence.
Let $(B,x,y)$ and $(B',x',y')$ be seeds of
$\mathcal{A}(B^0,x^0,y^0;\bbQ_+(y^0))$ such that
$(B',x',u')= \mu_{\vec{k}}(B,x,y)$,
and let $\nu$ be a permutation of $\{1,\dots, n\}$.
Then, the following conditions are equivalent.
\par (a). $b'_{\nu(i)\nu(j)}=b_{ij}$, $x'_{\nu(i)}=x_i$, and $y'_{\nu(i)}=y_i$ hold for any $i$
and $j$.
\par (b). $y'_{\nu(i)}=y_i$ holds for any $i$.
\par (c). $x'_{\nu(i)}=x_i$ holds for any $i$.
\par (d). $[y'_{\nu(i)}]=[y_i]$ holds for any $i$.
\par (e). $[x'_{\nu(i)}]=[x_i]$ holds for any $i$,
where $[x_i]\in \bbQ(\mathrm{Trop}(y^0))(w)$ is the one obtained from $x_i$
by the tropicalization of  $y$-variables.
\end{thm}
\begin{proof}
The implications (a) $\Rightarrow$ (b) $\Rightarrow$ (d)
and  (a) $\Rightarrow$ (c) $\Rightarrow$ (e) are obvious,
while (d) $\Rightarrow$ (a) is the result of \cite{Inoue10a,Plamondon10b}.
Let us show $(e) \Rightarrow (a)$.
It follows from the assumption (e) that the corresponding $g$-vectors
in  \cite{Fomin07}
have the same periodicity.
Then, the claim (a) follows again from the result of \cite{Plamondon10b}.
\end{proof}

\begin{defn} 
\label{defn:period1}
A mutation sequence $\vec{k}=(k_t)_{t=1}^N$ is called a {\em $\nu$-period
of $(B,x,y)$} if one of the conditions (a)--(e) in Theorem \ref{thm:period1} holds for
the seed $(B',x',y')=\mu_{\vec{k}} (B,x,y)$.
\end{defn}

Many interesting examples of  periodicities of seeds
are known \cite{Fomin07,Keller08, Inoue10a,Inoue10b,Nakanishi12c}.
Here, we give the simplest example,
which will be used as the running example throughout the paper.

\begin{ex}[Pentagon relation (1)]
\label{ex:pentagon1}
Consider the cluster algebra $\mathcal{A}(B^0,x^0,y^0;\bbQ_+(y^0))$
whose initial exchange matrix $B^0$ and
the corresponding quiver $Q^0$
are given by
\begin{align}
B^0=
\begin{pmatrix}
0 & 1 \\
-1 & 0\\
\end{pmatrix},
\quad
Q^0 =\ 
\raisebox{-10pt}{
\begin{picture}(40,5)(0,-15)
%
%
\put(0,0){\circle{6}}
\put(40,0){\circle{6}}
\put(3,0){\vector(1,0){34}}
\put(-3,-15){$1$}
\put(37,-15){$2$}
\end{picture}
}
\, .
\end{align}
This is the cluster algebra of type $A_2$ in
the classification of \cite{Fomin02}.
In particular, it is of finite type,
namely, there are only finitely many seeds.
Set $(B(1),x(1),y(1))$ to be the initial seed $(B^0,x^0,y^0)$,
and consider the  mutation sequence
$\vec{k}=(1,2,1,2,1)$, i.e.,
\begin{align}
\label{eq:seedmutseq3}
&(B(1),x(1),y(1))
\
\mathop{\rightarrow}^{\mu_{1}}
\
(B(2),x(2),y(2))
\
\mathop{\rightarrow}^{\mu_{2}}
\
\cdots
\
\mathop{\rightarrow}^{\mu_{1}}
\
(B(6),x(6),y(6))
.
\end{align}
For simplicity, we write
the initial variables $x_i=x^0_i$ and 
$y_i=y^0_i$.
According to \eqref{eq:yhat}, we set
the initial $\hat{y}$-variables as 
\begin{align}
\hat{y}_1 = y_1 x_2^{-1},
\quad \hat{y}_2=y_2 x_1.
\end{align}
Then, using the exchange relations \eqref{eq:ymut}
and \eqref{eq:xmut3},
we obtain the following explicit form of seeds:
\begin{alignat*}{3}
Q(1)\quad
&
\begin{picture}(40,15)(0,-5)
%
%
\put(0,0){\circle{12}}
\put(0,0){\circle{6}}
\put(40,0){\circle{6}}
\put(3,0){\vector(1,0){34}}
\put(-3,-15){$1$}
\put(37,-15){$2$}
\end{picture}
&\quad&
\begin{cases}
x_1(1)=x_1\\
x_2(1)=x_2,\\
\end{cases}
&&
\begin{cases}
y_1(1)=y_1\\
y_2(1)=y_2,\\
\end{cases}
\notag
\\
Q(2)\quad
&
\begin{picture}(40,15)(0,-5)
%
%
\put(40,0){\circle{12}}
\put(0,0){\circle{6}}
\put(40,0){\circle{6}}
\put(37,0){\vector(-1,0){34}}
\put(-3,-15){$1$}
\put(37,-15){$2$}
\end{picture}
&&
\begin{cases}
\displaystyle
x_1(2)=x_1^{-1}x_2\frac{1+\hat{y}_1}{1\oplus y_1}\\
x_2(2)=x_2,\\
\end{cases}
&&
\begin{cases}
y_1(2)=y_1^{-1}\\
y_2(2)=y_1y_2(1\oplus y_1)^{-1},\\
\end{cases}
\\
Q(3)\quad
&
\begin{picture}(40,15)(0,-5)
%
%
\put(0,0){\circle{12}}
\put(0,0){\circle{6}}
\put(40,0){\circle{6}}
\put(3,0){\vector(1,0){34}}
\put(-3,-15){$1$}
\put(37,-15){$2$}
\end{picture}
&&
\begin{cases}
\displaystyle
x_1(3)=x_1^{-1}x_2\frac{1+ \hat{y}_1}{1\oplus y_1}\\
\displaystyle
x_2(3)=x_1^{-1}
\frac{1 + \hat{y}_1 +\hat{y}_1\hat{y}_2}{1\oplus y_1\oplus y_1y_2},\\
\end{cases}
&&
\begin{cases}
y_1(3)=y_2(1\oplus y_1 \oplus y_1y_2)^{-1}\\
y_2(3)=y_1^{-1}y_2^{-1}(1\oplus y_1),\\
\end{cases}
\\
Q(4)\quad
&
\begin{picture}(40,15)(0,-5)
%
%
\put(40,0){\circle{12}}
\put(0,0){\circle{6}}
\put(40,0){\circle{6}}
\put(37,0){\vector(-1,0){34}}
\put(-3,-15){$1$}
\put(37,-15){$2$}
\end{picture}
&&
\begin{cases}
\displaystyle
x_1(4)=x_2^{-1}\frac{1+\hat{y}_2}{1\oplus y_2}
\\
\displaystyle
x_2(4)=
x_1^{-1}
\frac{1 + \hat{y}_1 +\hat{y}_1\hat{y}_2}{1\oplus y_1\oplus y_1y_2},\\
\end{cases}
&&
\begin{cases}
y_1(4)=y_2^{-1}(1\oplus y_1 \oplus y_1y_2)\\
y_2(4)=y_1^{-1}(1\oplus y_2)^{-1},\\
\end{cases}
\\
Q(5)\quad
&
\begin{picture}(40,15)(0,-5)
%
%
\put(0,0){\circle{12}}
\put(0,0){\circle{6}}
\put(40,0){\circle{6}}
\put(3,0){\vector(1,0){34}}
\put(-3,-15){$1$}
\put(37,-15){$2$}
\end{picture}
&&
\begin{cases}
\displaystyle
x_1(5)=x_2^{-1}\frac{1+\hat{y}_2}{1\oplus y_2}\\
x_2(5)=x_1,\\
\end{cases}
&&
\begin{cases}
y_1(5)=y_2^{-1}\\
y_2(5)=y_1(1\oplus y_2),\\
\end{cases}
\\
Q(6)\quad
&
\begin{picture}(40,15)(0,-5)
%
%
\put(40,0){\circle{12}}
\put(0,0){\circle{6}}
\put(40,0){\circle{6}}
\put(37,0){\vector(-1,0){34}}
\put(-3,-15){$1$}
\put(37,-15){$2$}
\end{picture}
&&
\begin{cases}
x_1(6)=x_2\\
x_2(6)=x_1.\\
\end{cases}
&&
\begin{cases}
y_1(6)=y_2\\
y_2(6)=y_1.\\
\end{cases}
\end{alignat*}
Here, the encircled vertices in quivers
are the {\em mutation points\/}
in the sequence \eqref{eq:seedmutseq3}.
We see that the mutation sequence $\vec{k}$ is a $\nu$-period
of  $(B(1),x(1),y(1))$,
where $\nu=(12)$ is the permutation of $1$ and $2$.
This periodicity is known as the {\em pentagon relation}.
The tropical $y$-variables at the mutation points
are
\begin{align}
\label{eq:tropyp1}
[y_1(1)]=y_1,
\quad
[y_2(2)]=y_1y_2,
\quad
[y_1(3)]=y_2,
\quad
[y_2(4)]=y_1^{-1},
\quad
[y_1(5)]=y_2^{-1},
\quad
\end{align}
and the  corresponding tropical signs $\ve_t = \ve(y_{k_t}(t))$  are
\begin{align}
\label{eq:tropsign2}
\ve_1 = +,
\quad
\ve_2 = +,
\quad
\ve_3 = +,
\quad
\ve_4 = -,
\quad
\ve_5 = -.
\end{align}
\end{ex}

\section{Surface realization of cluster algebras}

There is a class of cluster algebras which can be realized (in various sense) by triangulations
of  surfaces \cite{Gekhtman05,Fock03b,Fomin08,Fomin08b}.
This construction is often referred to as the {\em surface realization\/} of cluster algebras.
 Since careful treatment of mutations involving a self-folded triangle
 is crucial throughout the paper, we explain in detail how there are related
 to tagged triangulations and signed triangulations.
We mostly follow  \cite{Fomin08,Fomin08b}, but
we do some reformulation to work with {\em labeled\/} triangulations.
The extended seeds and their signed mutations and pops
are also defined.

\subsection{Ideal triangulations of bordered surface with marked points}

To start, we choose a compact connected oriented  surface possibly 
with boundary $\bfS$,
and a finite  set $\bfM$ of {\em marked points\/} on $\bfS$ that
is nonempty and  includes at least one marked point
on each boundary component and  possibly some interior points.
If a marked point is an interior point of $\bfS$, then it is called a puncture.
We impose the following assumption by a technical reason \cite{Fomin08}:
\begin{ass}
\label{ass:bs1}
The following cases of $(\bfS,\bfM)$ are excluded:
\begin{itemize}
\item a sphere with less than four punctures,
\item an unpunctured or once-punctured monogon,
\item an unpunctured digon,
\item an unpunctured triangle.
\end{itemize}
\end{ass}
In the above a surface homeomorphic to a disk with
$n$ marked points on the boundary is called a {\em polygon\/},
and, in particular, 
{\em monogon}, {\em digon}, {\em triangle\/}
for $n=1,2,3$, respectively.

A pair $(\bfS,\bfM)$ satisfying Assumption \ref{ass:bs1}
is called a {\em bordered surface with marked points},
or a {\em bordered surface}, for simplicity.

\begin{rem} In \cite{Fomin08,Fomin08b}
$\bfS$ is assumed to be a {\em Riemann surface}.
In our application we  need neither a complex structure  nor a metric.
\end{rem}

First we consider triangulations of $(\bfS,\bfM)$ by ordinary arcs,
where ``ordinary'' means ``not tagged'' which will be introduced later.

\begin{defn}
\label{defn:arc1}
 An {\em arc\/} $\alpha$ in a bordered surface
$(\bfS,\bfM)$ is a curve in $\bfS$ such that
\begin{itemize}
\item the endpoints of $\alpha$ are marked points,
\item $\alpha$ does not intersect itself except for the endpoints,
\item $\alpha$ is away from punctures and boundaries except for the endpoints,
\item $\alpha$ is not contractible into a marked point or onto a boundary of $\bfS$.
\end{itemize}
Furthermore, each arc $\alpha$ is considered {\em up to isotopy\/} in the class of such curves. For example, ``distinct (resp. identical) arcs'' means 
distinct (resp. identical) isotopy classes of curves unless otherwise mentioned.
\end{defn}

Two arcs are said to be {\em compatible\/} if there are representatives in their respective isotopy
classes such that they do not intersect each other in the interior of $\bfS$.

\begin{defn} An {\em ideal triangulation\/} $T=\{\alpha_i\}_{i\in I}$ of $(\bfS,\bfM)$ is
a maximal set of distinct pairwise compatible  arcs in $(\bfS,\bfM)$.
\end{defn}

See  Figure \ref{fig:triangulation-AD} for examples of ideal triangulations.
We also put labels 1, 2, \dots to the arcs   for the later use.
As in the second example, 
some {\em degenerate triangles\/},
such as {\em self-folded triangles\/} (e.g., 4-5) and  {\em triangles with identified vertices} (e.g., 2-3-4), may appear
around a puncture.
From now on we also call them triangles (of an ideal triangulation).
\begin{figure}
\begin{center}
\begin{pspicture}(-1,-2)(1,1.6)
\psset{unit=16mm}
\psset{linewidth=0.5pt}
\psset{fillstyle=solid, fillcolor=black}
\pscircle(1,0){0.05} 
\pscircle(0.7,0.7){0.05} 
\pscircle(0,1){0.05} 
\pscircle(-0.7,0.7){0.05} 
\pscircle(-1,0){0.05} 
\pscircle(-0.7,-0.7){0.05} 
\pscircle(0,-1){0.05} 
\pscircle(0.7,-0.7){0.05} 
\psline(1,0)(0.7,0.7)
\psline(0.7,0.7)(0,1)
\psline(0,1)(-0.7,0.7)
\psline(-0.7,0.7)(-1,0)
\psline(-1,0)(-0.7,-0.7)
\psline(-0.7,-0.7)(0,-1)
\psline(0,-1)(0.7,-0.7)
\psline(0.7,-0.7)(1,0)
\psset{fillstyle=none}
\psline(0,1)(-1,0) 
\psline(0,1)(0.7,-0.7) 
\psline(-1,0)(0.7,-0.7) 
\psline(-0.7,-0.7)(0.7,-0.7) 
\psline(0,1)(1,0) 
\rput[c]{0}(-0.57,0.57){\small 1}
\rput[c]{0}(0.23,0.2){\small 2}
\rput[c]{0}(-0.1,-0.25){\small 3}
\rput[c]{0}(0,-0.59){\small 4}
\rput[c]{0}(0.57,0.57){\small 5}
\rput[c]{0}(0,-1.5){(a)}
\end{pspicture}
\hskip80pt
\begin{pspicture}(-1,-2)(1,1.2)
\psset{unit=16mm}
\psset{linewidth=0.5pt}
\psset{fillstyle=solid, fillcolor=black}
\pscircle(1,0){0.05} 
\pscircle(0.7,0.7){0.05} 
\pscircle(0,1){0.05} 
\pscircle(-0.7,0.7){0.05} 
\pscircle(-1,0){0.05} 
\pscircle(-0.7,-0.7){0.05} 
\pscircle(0,-1){0.05} 
\pscircle(0.7,-0.7){0.05} 
\pscircle(0,0){0.05} 
\psline(1,0)(0.7,0.7)
\psline(0.7,0.7)(0,1)
\psline(0,1)(-0.7,0.7)
\psline(-0.7,0.7)(-1,0)
\psline(-1,0)(-0.7,-0.7)
\psline(-0.7,-0.7)(0,-1)
\psline(0,-1)(0.7,-0.7)
\psline(0.7,-0.7)(1,0)
\psset{fillstyle=none}
\psline(0,1)(-1,0) 
\psline(-0.7,-0.7)(0,0) 
\pscurve(-0.7,-0.7)(-0.2,0.2)(0.2,0.2)(0.2,-0.2)(-0.7,-0.7) 
\pscurve(0,1)(-0.42,0.4)(-0.58,0)(-0.7,-0.7) 
\pscurve(0,1)(0.4,0.2)(0.3,-0.4)(-0.7,-0.7) 
\pscurve(0,1)(0.55,0.2)(0.68,-0.2)(0.7,-0.7) 
\psline(-0.7,-0.7)(0.7,-0.7) 
\psline(0,1)(1,0) 
\rput[c]{0}(-0.57,0.57){\small 1}
\rput[c]{0}(-0.6,0.2){\small 2}
\rput[c]{0}(0.41,-0.4){\small 3}
\rput[c]{0}(-0.2,0.33){\small 4}
\rput[c]{0}(-0.2,-0.06){\small 5}
\rput[c]{0}(0,-0.8){\small 6}
\rput[c]{0}(0.7,0){\small 7}
\rput[c]{0}(0.57,0.57){\small 8}
\rput[c]{0}(0,-1.5){(b)}
\end{pspicture}
\end{center}
\caption{Examples of labeled ideal triangulations of a polygon without puncture (a)
and a polygon with one puncture (b).}
\label{fig:triangulation-AD}
\end{figure}
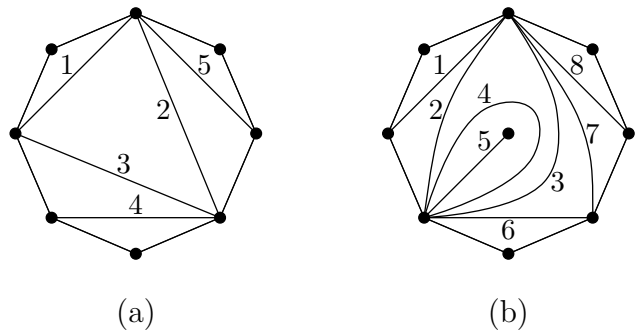
\begin{defn} 
For an ideal triangulation $T$ of $(\bfS,\bfM)$ and an arc $\alpha\in T$,
if there is another arc $\alpha'$ such that
that $T'=(T-\{\alpha\} )\cup \{\alpha'\}$ is an ideal triangulation,
then $\alpha'$ is called a {\em flip of $\alpha$},
 and
$T'$ is called a {\em flip of $T$ at $\alpha$}.
(As wee see soon, such $\alpha'$ is unique for each $\alpha$ if it exists.)
\end{defn}

Now we encounter a problem that
{\em not all arcs are flippable}.
For example, in the triangulation in Figure \ref{fig:triangulation-AD} (b),
the arc with label 5 is not flippable.
Luckily this is the only situation where an arc is not flippable.
Indeed, 
 if an arc $\alpha$  is the {\em inner side of a self-folded triangle}\/
 (an {\em inner arc}, for short),
 it is not flippable.
Otherwise,
$\alpha$ is  the common side of two  (possibly degenerate) triangles.
These triangles make
 a quadrilateral
 such that $\alpha$ is one of its diagonal.
Then, $\alpha$ is {\em uniquely\/} flipped to another diagonal $\alpha'$ of the same quadrilateral,
and {\em vice versa}.
See Figure \ref{fig:flip1}.
In particular, the uniqueness of the flip was also shown.

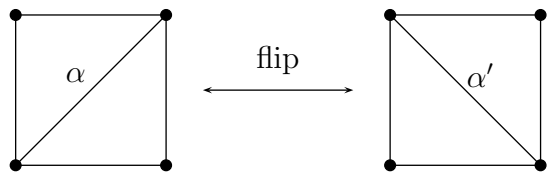
\begin{figure}
\begin{center}
\begin{pspicture}(-1,-1.2)(1,1.2)
\psset{linewidth=0.5pt}
\psset{fillstyle=solid, fillcolor=black}
\pscircle(1,1){0.08} 
\pscircle(-1,1){0.08} 
\pscircle(-1,-1){0.08} 
\pscircle(1,-1){0.08} 
\psset{fillstyle=none}
\psline(1,1)(-1,1)
\psline(-1,1)(-1,-1)
\psline(-1,-1)(1,-1)
\psline(1,-1)(1,1)
\psline(1,1)(-1,-1) 
\rput[c]{0}(-0.2,0.2){$\alpha$}
\end{pspicture}
\hskip10pt
\begin{pspicture}(-1,-1.2)(1,1.2)
%
\psset{linewidth=0.5pt}
\psline[arrows=<->](-1,0)(1,0) 
\rput[c]{0}(0,0.4){flip}
\end{pspicture}
\hskip10pt
\begin{pspicture}(-1,-1.2)(1,1.2)
\psset{linewidth=0.5pt}
\psset{fillstyle=solid, fillcolor=black}
\pscircle(1,1){0.08} 
\pscircle(-1,1){0.08} 
\pscircle(-1,-1){0.08} 
\pscircle(1,-1){0.08} 
\psset{fillstyle=none}
\psline(1,1)(-1,1)
\psline(-1,1)(-1,-1)
\psline(-1,-1)(1,-1)
\psline(1,-1)(1,1)
\psline(-1,1)(1,-1) 
\rput[c]{0}(0.2,0.2){$\alpha'$}
\end{pspicture}
\end{center}
\caption{Flip of a diagonal arc.}
\label{fig:flip1}
\end{figure}

For a given bordered surface $(\bfS,\bfM)$  it is known that all ideal triangulations are connected by
a sequence of flips \cite{Hatcher91}.
 In particular, they share the same cardinality of arcs.
For an ideal triangulation $T$ with $|T|=n$, one can label
the arcs in $T$ by the set $\{1, \dots, n\}$ as
in Figure \ref{fig:triangulation-AD}.
We call it  a {\em labeled\/} ideal triangulation, and we still write it as $T$.
In other words, a labeled ideal triangulation $T$ is not simply a set of $n$ arcs;
rather, it is an {\em $n$-tuple of arcs\/} $(\alpha_i)_{i=1}^n$,
where $\alpha_i$ is the arc with label $i$.
A flip of an (unlabeled) ideal triangulation induces a flip of a labeled ideal triangulation by preserving the labels of the unflipped arcs.
Suppose that  $T=(\alpha_i)_{i=1}^n$ and $T'=(\alpha'_i)_{i=1}^n$ are labeled ideal triangulations.
Then, if $T'$ is a flip of $T$ at $\alpha_k$, 
then $T$ is  also a flip of $T'$ at $\alpha'_k$.
and $\alpha'_i=\alpha_i$ for $i\neq k$.
So, we  can write them as $T'=\mu_k(T)$ and $T=\mu_k(T')$,
like the mutation in cluster algebras,
and call them the {\em flip at $k$}.

\begin{defn}
To each labeled ideal triangulation $T=(\alpha_i)_{i=1}^n$ of $(\bfS,\bfM)$ with $|T|=n$,
we assign a skew-symmetric matrix $B=B(T)=(b_{ij})_{i,j=1}^n$,
called the {\em (signed) adjacency matrix of $T$}, as follows.
\par
(a). {\em The case when neither $\alpha_i$ nor $\alpha_j$ are  inner arcs in $T$.}
First, for any triangle $\Delta$ in $T$ which is not self-folded,
 we define
\begin{align}
\label{eq:b1}
b^{\Delta}_{ij}
=
\begin{cases}
1& \text{$\alpha_i$ and $\alpha_j$ are different sides of $\Delta$,}\\
&\text{and the direction of the angle from $\alpha_i$ to $\alpha_j$ is  counter-clockwise;} \\
-1& \text{$\alpha_i$ and $\alpha_j$ are different sides of $\Delta$,}\\
&\text{and the direction of the angle from $\alpha_i$ to $\alpha_j$ is  clockwise;} \\
0& \text{otherwise.}\\
\end{cases}
\end{align}
Then, we define
\begin{align}
b_{ij}=\sum_{\Delta} b^{\Delta}_{ij},
\end{align}
where the sum runs over all triangles $\Delta$ in $T$ which are not self-folded.
\par
(b). {\em The rest of the case}. For an inner arc $\alpha_i$ in $T$,
let $\alpha_{\overline{i}}$ be the outer side of the self-folded triangle which
$\alpha_i$ belongs to.
Then, we define
\begin{align}
b_{ij}=
\begin{cases}
b_{\overline{i}j} & \text{$\alpha_i$ is an inner arc, and $\alpha_j$ is not;}\\
b_{i \overline{j}} & \text{$\alpha_j$ is an inner arc, and $\alpha_i$ is not;}\\
b_{\overline{i}\,\overline{j}} & \text{both $\alpha_i$ and $\alpha_j$ are inner arcs,}\\
\end{cases}
\end{align}
where the right hand side is defined in \eqref{eq:b1}.
\end{defn}

\begin{ex}
\label{ex:BQ1}
 For the ideal triangulations in Figure \ref{fig:triangulation-AD},
the corresponding skew-symmetric matrices and quivers are given as follows.
\begin{align*}
\begin{matrix}
\begin{pmatrix}
0& 1 & -1 &0 &0\\
-1&0&1&0&1\\
1&-1&0&1&0\\
0&0&-1&0&0\\
0&-1&0&0&0\\
\end{pmatrix},
&
\begin{pmatrix}
0&1&0&0&0&0&0&0\\
-1&0&1&-1&-1&0&0&0\\
0&-1&0&1&1&-1&1&0\\
0&1&-1&0&0&0&0&0\\
0&1&-1&0&0&0&0&0\\
0&0&1&0&0&0&-1&0\\
0&0&-1&0&0&1&0&1\\
0&0&0&0&0&0&-1&0\\
\end{pmatrix}.\\
\raisebox{12.5pt}{
\begin{pspicture}(0,-0.5)(1.5,1.4)
\psset{linewidth=0.5pt}
\pscircle(0,0){0.1} 
\pscircle(0.5,0.7){0.1} 
\pscircle(1,0){0.1} 
\pscircle(2,0){0.1} 
\pscircle(1.5,0.7){0.1} 
\psset{fillstyle=none}
\psline[arrows=->](0.8,0)(0.2,0)
\psline[arrows=->](1.2,0)(1.8,0)
\psline[arrows=->](0.7,0.7)(1.3,0.7)
\psline[arrows=->](0.1,0.14)(0.4,0.56)
\psline[arrows=->](0.6,0.56)(0.9,0.14)
\rput[c]{0}(0,-0.4){\small 1}
\rput[c]{0}(0.5,1.1){\small 2}
\rput[c]{0}(1,-0.4){\small 3}
\rput[c]{0}(2,-0.4){\small 4}
\rput[c]{0}(1.5,1.1){\small 5}
\end{pspicture}
}
&
\begin{pspicture}(0,-1.4)(1.5,1.4)
\psset{linewidth=0.5pt}
\pscircle(0,0){0.1} 
\pscircle(1,0){0.1} 
\pscircle(2,0){0.1} 
\pscircle(1.5,0.7){0.1} 
\pscircle(1.5,-0.7){0.1} 
\pscircle(2.5,0.7){0.1} 
\pscircle(2.5,-0.7){0.1} 
\pscircle(3.5,-0.7){0.1} 
\psset{fillstyle=none}
\psline[arrows=->](0.2,0)(0.8,0)
\psline[arrows=->](1.2,0)(1.8,0)
\psline[arrows=->](1.4,0.56)(1.1,0.14)
\psline[arrows=->](1.9,0.14)(1.6,0.56)
\psline[arrows=->](1.4,-0.56)(1.1,-0.14)
\psline[arrows=->](1.9,-0.14)(1.6,-0.56)
\psline[arrows=->](2.4,0.56)(2.1,0.14)
\psline[arrows=->](2.1,-0.14)(2.4,-0.56)
\psline[arrows=->](2.5,-0.5)(2.5,0.5)
\psline[arrows=->](2.7,-0.7)(3.3,-0.7)
\rput[c]{0}(0,-0.4){\small 1}
\rput[c]{0}(1,-0.4){\small 2}
\rput[c]{0}(2,-0.4){\small 3}
\rput[c]{0}(1.5,1.1){\small 4}
\rput[c]{0}(1.5,-1.1){\small 5}
\rput[c]{0}(2.5,1.1){\small 6}
\rput[c]{0}(2.5,-1.1){\small 7}
\rput[c]{0}(3.5,-1.1){\small 8}
\end{pspicture}
\hskip50pt
\end{matrix}
\end{align*}
See \cite[Section 4]{Fomin08} for more exotic examples.
\end{ex}

The following fact is a key to connect  triangulations and cluster algebras.

\begin{thm}[\cite{Fock05,Gekhtman05}] Let $T$ be a labeled ideal triangulation
of $(\bfS,\bfM)$,
and let $\alpha_k$ not be  an inner arc in $T$.
Then, $B(\mu_k (T))=\mu_k (B(T))$.
\end{thm}

The theorem says that  the flip of labeled ideal triangulations and the mutation of
the corresponding 
skew-symmetric matrices (equivalently, quivers) in \eqref{eq:bmut} are compatible,
{\em if the targeted arc is flippable}.
However, recall that the inner arcs are not flippable,
while skew-symmetric matrices can be mutated to {\em any\/} direction.
See Figure \ref{fig:digon1} for an illuminating example of 
a digon with a puncture in some ideal triangulation.
Fomin, Shapiro, and Thurston \cite{Fomin08}
remedied   this discrepancy
by introducing   the {\em tagged\/} triangulations.

\begin{figure}
\begin{center}
\begin{pspicture}(-1,-1.2)(3.3,1.2)
\psset{linewidth=0.5pt}
\psset{fillstyle=solid, fillcolor=black}
\pscircle(0,1){0.08} 
\pscircle(0,-1){0.08} 
\pscircle(0,0){0.08} 
\psset{fillstyle=none}
\psline(0,1)(0,0)
\psline(0,0)(0,-1)
\pscurve(0,1)(0.6,0.45)(0.6,-0.45)(0,-1)
\pscurve(0,1)(-0.6,0.45)(-0.6,-0.45)(0,-1)
\rput[c]{0}(-0.9,0){\small 1}
\rput[c]{0}(0.9,0){\small 4}
\rput[c]{0}(0.2,0.5){\small 2}
\rput[c]{0}(0.2,-0.5){\small 3}
%
\pscircle(2,0){0.1} 
\pscircle(3,0){0.1} 
\pscircle(2.5,0.7){0.1} 
\pscircle(2.5,-0.7){0.1} 
%
\psline[arrows=->](2.4,-0.56)(2.1,-0.14)
\psline[arrows=->](2.1,0.14)(2.4,0.56)
\psline[arrows=->](2.6,0.56)(2.9,0.14)
\psline[arrows=->](2.9,-0.14)(2.6,-0.56)
\rput[c]{0}(1.7,0){\small 1}
\rput[c]{0}(2.5,1.1){\small 2}
\rput[c]{0}(3.3,0){\small 4}
\rput[c]{0}(2.5,-1.1){\small 3}
\end{pspicture}
\break
\begin{pspicture}(0,0)(1,1)
\psset{linewidth=0.5pt}
\psline[arrows=<->](0,0)(1,1)
\rput[c]{0}(0.3,0.7){$\mu_2$}
\end{pspicture}
\hskip100pt
\begin{pspicture}(0,0)(1,1)
\psset{linewidth=0.5pt}
\psline[arrows=<->](0,1)(1,0)
\rput[c]{0}(0.7,0.7){$\mu_3$}
\end{pspicture}
\break
\begin{pspicture}(-1,-1.2)(3.3,1.2)
\psset{linewidth=0.5pt}
\psset{fillstyle=solid, fillcolor=black}
\pscircle(0,1){0.08} 
\pscircle(0,-1){0.08} 
\pscircle(0,0){0.08} 
\psset{fillstyle=none}
\pscurve(0,-1)(-0.4,0)(0,0.4)(0.4,0)(0,-1)
\psline(0,0)(0,-1)
\pscurve(0,1)(0.6,0.45)(0.6,-0.45)(0,-1)
\pscurve(0,1)(-0.6,0.45)(-0.6,-0.45)(0,-1)
\rput[c]{0}(-0.9,0){\small 1}
\rput[c]{0}(0.9,0){\small 4}
\rput[c]{0}(0,0.6){\small 2}
\rput[c]{0}(0.15,-0.4){\small 3}
%
\pscircle(2,0){0.1} 
\pscircle(3,0){0.1} 
\pscircle(2.5,0.7){0.1} 
\pscircle(2.5,-0.7){0.1} 
%
\psline[arrows=->](2.4,-0.56)(2.1,-0.14)
\psline[arrows=->](2.4,0.56)(2.1,0.14)
\psline[arrows=->](2.9,0.14)(2.6,0.56)
\psline[arrows=->](2.9,-0.14)(2.6,-0.56)
\psline[arrows=->](2.2,0)(2.8,0)
\rput[c]{0}(1.7,0){\small 1}
\rput[c]{0}(2.5,1.1){\small 2}
\rput[c]{0}(3.3,0){\small 4}
\rput[c]{0}(2.5,-1.1){\small 3}
\end{pspicture}
\hskip50pt
\begin{pspicture}(-1,-1.2)(3.3,1.2)
\psset{linewidth=0.5pt}
\psset{fillstyle=solid, fillcolor=black}
\pscircle(0,1){0.08} 
\pscircle(0,-1){0.08} 
\pscircle(0,0){0.08} 
\psset{fillstyle=none}
\pscurve(0,1)(-0.4,0)(0,-0.4)(0.4,0)(0,1)
\psline(0,0)(0,1)
\pscurve(0,1)(0.6,0.45)(0.6,-0.45)(0,-1)
\pscurve(0,1)(-0.6,0.45)(-0.6,-0.45)(0,-1)
\rput[c]{0}(-0.9,0){\small 1}
\rput[c]{0}(0.9,0){\small 4}
\rput[c]{0}(0,-0.6){\small 3}
\rput[c]{0}(0.15,0.4){\small 2}
%
\pscircle(2,0){0.1} 
\pscircle(3,0){0.1} 
\pscircle(2.5,0.7){0.1} 
\pscircle(2.5,-0.7){0.1} 
%
\psline[arrows=->](2.1,-0.14)(2.4,-0.56)
\psline[arrows=->](2.1,0.14)(2.4,0.56)
\psline[arrows=->](2.6,0.56)(2.9,0.14)
\psline[arrows=->](2.6,-0.56)(2.9,-0.14)
\psline[arrows=->](2.8,0)(2.2,0)
\rput[c]{0}(1.7,0){\small 1}
\rput[c]{0}(2.5,1.1){\small 2}
\rput[c]{0}(3.3,0){\small 4}
\rput[c]{0}(2.5,-1.1){\small 3}
\end{pspicture}
\break
\begin{pspicture}(0,0)(1,1)
\psset{linewidth=0.5pt}
\psline[arrows=<->](0,1)(1,0)
\rput[c]{0}(0.7,0.7){$\mu_3$}
\end{pspicture}
\hskip100pt
\begin{pspicture}(0,0)(1,1)
\psset{linewidth=0.5pt}
\psline[arrows=<->](0,0)(1,1)
\rput[c]{0}(0.3,0.7){$\mu_2$}
\end{pspicture}
\break
\begin{pspicture}(-1,-1.2)(3.3,1.2)
\psset{linewidth=0.5pt}
\rput[c]{0}(0,0){missing}
\pscircle(2,0){0.1} 
\pscircle(3,0){0.1} 
\pscircle(2.5,0.7){0.1} 
\pscircle(2.5,-0.7){0.1} 
%
\psline[arrows=->](2.1,-0.14)(2.4,-0.56)
\psline[arrows=->](2.4,0.56)(2.1,0.14)
\psline[arrows=->](2.9,0.14)(2.6,0.56)
\psline[arrows=->](2.6,-0.56)(2.9,-0.14)
\rput[c]{0}(1.7,0){\small 1}
\rput[c]{0}(2.5,1.1){\small 2}
\rput[c]{0}(3.3,0){\small 4}
\rput[c]{0}(2.5,-1.1){\small 3}
\end{pspicture}
\end{center}
\caption{Flips of arcs in a digon with a puncture  and  
mutations of the corresponding quivers.}
\label{fig:digon1}
\end{figure}
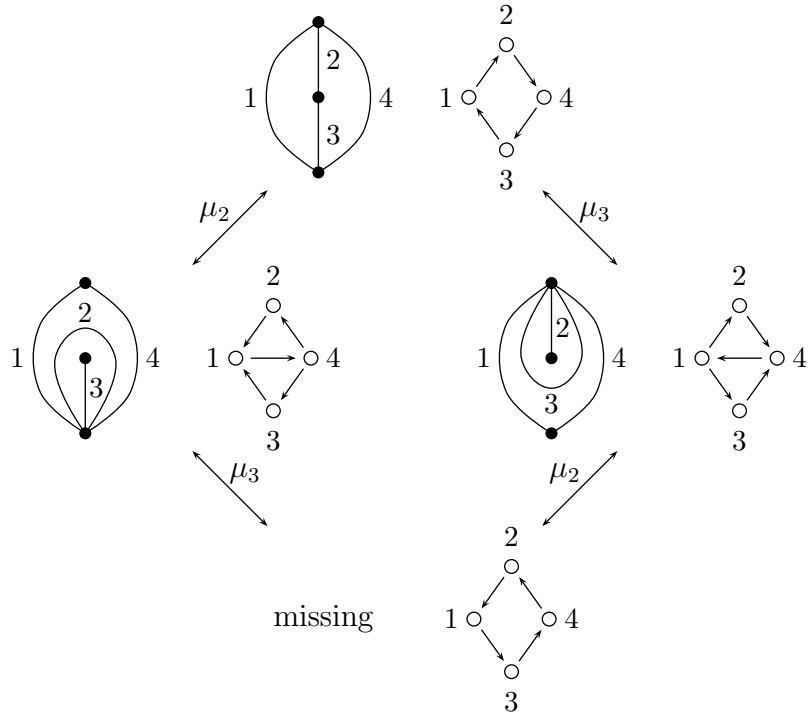

\subsection{Tagged triangulations}
\label{subsubsec:tagged}

For each arc $\alpha$ in a bordered surface $(\bfS,\bfM)$,
cut $\alpha$ into three pieces and throw out the middle part.
The remaining two parts are called the {\em ends of $\alpha$}.

\begin{defn}
An arc $\alpha$ in $(\bfS,\bfM)$ is called a {\em tagged arc\/}
if the following conditions are satisfied:
\begin{itemize}
\item[(a).] $\alpha$ is not a loop inside which there is exactly one puncture.
\item[(b).]
Each end of $\alpha$ is {\em tagged\/} in one of two ways,
{\em plain} or {\em notched} such that 
\begin{itemize}
\item any end with the endpoint on the boundary is tagged plain,
\item both ends of a loop are tagged in the same way.
\end{itemize}
\end{itemize}
In figures, the plain tags are omitted, while the notched tags are shown by
the symbol $\notch$, following \cite{Fomin08}.
\end{defn}

For example, the loop with label 4 in Figure \ref{fig:triangulation-AD} (b)
is not an tagged arc anymore due to the condition (a).
If $(\bfS,\bfM)$ does not have any puncture, then
arcs and tagged arcs are the same thing.

\begin{defn}
Two tagged arcs $\alpha$ are $\beta$ are said to be {\em compatible\/}
if the following conditions are satisfied:
\begin{itemize}
\item
the untagged versions of $\alpha$ and $\beta$ are compatible,
\item
if the untagged versions of $\alpha$ and $\beta$ are distinct,
and they share an endpoint $p$,
then the ends of $\alpha$ and $\beta$ with endpoint $p$ have the same tag,
\item
if the untagged versions of $\alpha$ and $\beta$ are identical,
then at least one end of $\alpha$ and the corresponding end of $\beta$
have the same tag.
\end{itemize}
\end{defn}

\begin{ex}
\label{ex:tag1}
Suppose that $\alpha$, $\beta$, and $\gamma$ are three {\em distinct\/}
pairwise compatible tagged arcs.
Then, {\em their untagged versions are not all identical}.
To show it, let $p$ and $q$ be their common endpoints,
and suppose that the untagged versions
of $\alpha$ and $\beta$ are identical.
Then,  $\alpha$ and $\beta$ have the same tag at one of the ends
with endpoint, say, $p$; furthermore, they have the different tags
at the end with $q$, since they are  distinct tagged arcs.
Now suppose further that  the untagged versions
of $\beta$ and $\gamma$ are identical.
If $\beta$ and $\gamma$ have the same tag at the end with  $p$,
then they have the different tags at the end with $q$.
Therefore, $\alpha$ and $\gamma$ are identical as tagged arcs,
which is a contradiction.
So, $\beta$ and $\gamma$ should have the same tag at the end with  $q$,
then they have the different tags at the end with $p$.
Then, $\alpha$ and $\gamma$ do not have the same tag at  both ends.
Therefore, $\alpha$ and $\gamma$ is not compatible,
which is again a contradiction.
\end{ex}

\begin{defn} An (unlabeled) {\em tagged triangulation\/} $T=\{\alpha_i\}_{i\in I}$ of $(\bfS,\bfM)$ is
a maximal set of distinct pairwise compatible  tagged arcs in $(\bfS,\bfM)$.
\end{defn}

A {\em labeled\/} tagged triangulation is defined in the same way as in the ordinary case.
Examples of labeled tagged triangulations are given
in Figure \ref{fig:tagged-triangulation-D}.
Observe that 
the untagged versions of tagged arcs $\alpha_4$ and $\alpha_5$ in (b) are identical.

\begin{defn} 
For an unlabeled tagged triangulation $T$ of $(\bfS,\bfM)$
 and a tagged arc $\alpha\in T$,
if there is another tagged arc $\alpha'$ such that
that $T'=(T-\{\alpha\} )\cup \{\alpha'\}$ is a tagged triangulation,
then $\alpha'$ is called a {\em flip of $\alpha$},
 and
$T'$ is called a {\em flip of $T$ at $\alpha$}.
Accordingly, a flip of a {\em labeled\/} tagged triangulation at $k$ is also defined  in the same
 way as in the ordinary case.
\end{defn}

\begin{figure}
\begin{center}
\begin{pspicture}(-1,-2)(1,1.6)
\psset{unit=16mm}
\psset{linewidth=0.5pt}
\psset{fillstyle=solid, fillcolor=black}
\pscircle(1,0){0.05} 
\pscircle(0.7,0.7){0.05} 
\pscircle(0,1){0.05} 
\pscircle(-0.7,0.7){0.05} 
\pscircle(-1,0){0.05} 
\pscircle(-0.7,-0.7){0.05} 
\pscircle(0,-1){0.05} 
\pscircle(0.7,-0.7){0.05} 
\pscircle(0,0){0.05} 
\psline(1,0)(0.7,0.7)
\psline(0.7,0.7)(0,1)
\psline(0,1)(-0.7,0.7)
\psline(-0.7,0.7)(-1,0)
\psline(-1,0)(-0.7,-0.7)
\psline(-0.7,-0.7)(0,-1)
\psline(0,-1)(0.7,-0.7)
\psline(0.7,-0.7)(1,0)
\psset{fillstyle=none}
\psline(0,1)(-1,0) 
\psline(-0.7,-0.7)(0,0) 
\psline(0,1)(0,0) 
\psline(0.7,-0.7)(0,0) 
\pscurve(0,1)(-0.42,0.4)(-0.58,0)(-0.7,-0.7) 
\pscurve(0,1)(0.55,0.2)(0.68,-0.2)(0.7,-0.7) 
\psline(-0.7,-0.7)(0.7,-0.7) 
\psline(0,1)(1,0) 
\rput[c]{0}(-0.57,0.57){\small 1}
\rput[c]{0}(-0.6,0.2){\small 2}
\rput[c]{0}(0.08,0.5){\small 5}
\rput[c]{0}(-0.3,-0.4){\small 4}
\rput[c]{0}(0.27,-0.4){\small 3}
\rput[c]{0}(0,-0.8){\small 6}
\rput[c]{0}(0.7,0){\small 7}
\rput[c]{0}(0.57,0.57){\small 8}
\rput[c]{0}(0,0.2){$\notch$}
\rput[c]{-45}(-0.2,-0.2){$\notch$}
\rput[c]{45}(0.2,-0.2){$\notch$}
\rput[c]{0}(0,-1.5){(a)}
\end{pspicture}
\hskip80pt
\begin{pspicture}(-1,-2)(1,1.2)
\psset{unit=16mm}
\psset{linewidth=0.5pt}
\psset{fillstyle=solid, fillcolor=black}
\pscircle(1,0){0.05} 
\pscircle(0.7,0.7){0.05} 
\pscircle(0,1){0.05} 
\pscircle(-0.7,0.7){0.05} 
\pscircle(-1,0){0.05} 
\pscircle(-0.7,-0.7){0.05} 
\pscircle(0,-1){0.05} 
\pscircle(0.7,-0.7){0.05} 
\pscircle(0,0){0.05} 
\psline(1,0)(0.7,0.7)
\psline(0.7,0.7)(0,1)
\psline(0,1)(-0.7,0.7)
\psline(-0.7,0.7)(-1,0)
\psline(-1,0)(-0.7,-0.7)
\psline(-0.7,-0.7)(0,-1)
\psline(0,-1)(0.7,-0.7)
\psline(0.7,-0.7)(1,0)
\psset{fillstyle=none}
\psline(0,1)(-1,0) 
\pscurve(-0.7,-0.7)(-0.5,-0.2)(0,0)
\pscurve(-0.7,-0.7)(-0.2,-0.5)(0,0)
\pscurve(0,1)(-0.42,0.4)(-0.58,0)(-0.7,-0.7) 
\pscurve(0,1)(0.4,0.2)(0.3,-0.4)(-0.7,-0.7) 
\pscurve(0,1)(0.55,0.2)(0.68,-0.2)(0.7,-0.7) 
\psline(-0.7,-0.7)(0.7,-0.7) 
\psline(0,1)(1,0) 
\rput[c]{0}(-0.57,0.57){\small 1}
\rput[c]{0}(-0.6,0.2){\small 2}
\rput[c]{0}(0.41,-0.4){\small 3}
\rput[c]{0}(-0.45,-0.27){\small 5}
\rput[c]{0}(-0.25,-0.42){\small 4}
\rput[c]{0}(0,-0.8){\small 6}
\rput[c]{0}(0.7,0){\small 7}
\rput[c]{0}(0.57,0.57){\small 8}
\rput[c]{-30}(-0.05,-0.2){$\notch$}
\rput[c]{0}(0,-1.5){(b)}
\end{pspicture}
\end{center}
\caption{Examples of labeled tagged triangulations of a polygon with one puncture. }
\label{fig:tagged-triangulation-D}
\end{figure}
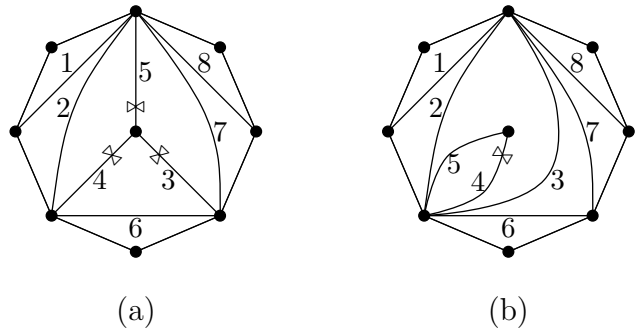

The following theorem is the first step
 to remedy the aforementioned
discrepancy.

\begin{thm}[{\cite[Theorem 7.9]{Fomin08}}]
Any tagged arc of a tagged triangulation is uniquely flippable.
\end{thm}

\begin{ex} (a).
Some examples of tagged triangulations
are given in Figure \ref{fig:flip1a}.
Note that the tag of the flipped	arc is uniquely
determined by the compatibility condition.
\par
(b).
In Figure \ref{fig:tagged-triangulation-D},
starting from the labeled tagged triangulation in (a),
flipping at $3$, then at $5$,
one obtains the labeled tagged triangulation in (b).
\end{ex}

\begin{figure}
\begin{center}
\begin{pspicture}(-1,-1)(1,1.2)
\psset{linewidth=0.5pt}
\psset{fillstyle=solid, fillcolor=black}
\pscircle(1,1){0.08} 
\pscircle(-1,1){0.08} 
\pscircle(-1,-1){0.08} 
\pscircle(1,-1){0.08} 
\psset{fillstyle=none}
\psline(1,1)(-1,1)
\psline(-1,1)(-1,-1)
\psline(-1,-1)(1,-1)
\psline(1,-1)(1,1)
\psline(1,1)(-1,-1) 
\rput[c]{0}(-0.2,0.2){$k$}
\rput[c]{0}(-1,0.6){$\notch$}
\rput[c]{90}(-0.6,1){$\notch$}
\rput[c]{0}(-1,-0.6){$\notch$}
\rput[c]{-45}(-0.7,-0.7){$\notch$}
\rput[c]{90}(-0.6,-1){$\notch$}
\end{pspicture}
\hskip0pt
\begin{pspicture}(-1,-1)(1,1.2)
%
\psset{linewidth=0.5pt}
\psline[arrows=<->](-0.5,0)(0.5,0) 
\rput[c]{0}(0,0.4){$\mu_k$}
\end{pspicture}
\hskip0pt
\begin{pspicture}(-1,-1)(1,1.2)
\psset{linewidth=0.5pt}
\psset{fillstyle=solid, fillcolor=black}
\pscircle(1,1){0.08} 
\pscircle(-1,1){0.08} 
\pscircle(-1,-1){0.08} 
\pscircle(1,-1){0.08} 
\psset{fillstyle=none}
\psline(1,1)(-1,1)
\psline(-1,1)(-1,-1)
\psline(-1,-1)(1,-1)
\psline(1,-1)(1,1)
\psline(-1,1)(1,-1) 
\rput[c]{0}(0.2,0.2){$k$}
\rput[c]{0}(-1,0.6){$\notch$}
\rput[c]{90}(-0.6,1){$\notch$}
\rput[c]{0}(-1,-0.6){$\notch$}
\rput[c]{45}(-0.7,0.7){$\notch$}
\rput[c]{90}(-0.6,-1){$\notch$}
\end{pspicture}
\hskip40pt
\begin{pspicture}(-0.6,-1)(0.6,1.2)
\psset{linewidth=0.5pt}
\psset{fillstyle=solid, fillcolor=black}
\pscircle(0,1){0.08} 
\pscircle(0,-1){0.08} 
\pscircle(0,0){0.08} 
\psset{fillstyle=none}
%
\pscurve(0,-1)(0.25,-0.5)(0,0)
\pscurve(0,-1)(-0.25,-0.5)(0,0)
\pscurve(0,1)(0.6,0.45)(0.6,-0.45)(0,-1)
\pscurve(0,1)(-0.6,0.45)(-0.6,-0.45)(0,-1)
%
\rput[c]{0}(0.4,-0.4){\small $k$}
\rput[c]{40}(0.17,-0.23){$\notch$}
%
%
\end{pspicture}
\hskip10pt
\begin{pspicture}(-0.6,-1)(0.6,1.2)
%
\psset{linewidth=0.5pt}
\psline[arrows=<->](-0.5,0)(0.5,0) 
\rput[c]{0}(0,0.4){$\mu_k$}
\end{pspicture}
\hskip10pt
\begin{pspicture}(-0.6,-1)(0.6,1.2)
\psset{linewidth=0.5pt}
\psset{fillstyle=solid, fillcolor=black}
\pscircle(0,1){0.08} 
\pscircle(0,-1){0.08} 
\pscircle(0,0){0.08} 
\psset{fillstyle=none}
%
\psline(0,-1)(0,0)
\psline(0,1)(0,0)
\pscurve(0,1)(0.6,0.45)(0.6,-0.45)(0,-1)
\pscurve(0,1)(-0.6,0.45)(-0.6,-0.45)(0,-1)
%
\rput[c]{0}(0.2,0.5){\small $k$}
%
%
\end{pspicture}
\end{center}
\caption{Examples of flips of tagged triangulations}
\label{fig:flip1a}
\end{figure}
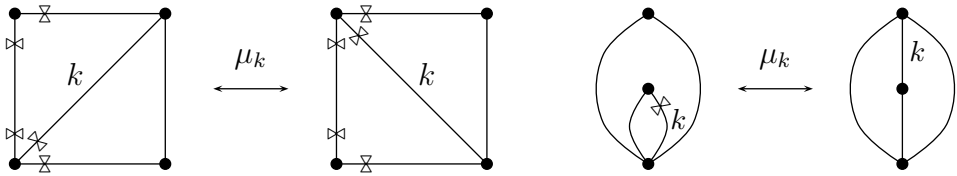

Next, we assign the adjacency matrix to
each labeled tagged triangulation.
To do that, we note that for any labeled tagged triangulation $T$,
every puncture $p$ of $T$ can be classified in  one of the following three types
{\cite{Fomin08}}.
\begin{itemize}
\item {\em Type  $1$}. 
All  tags of ends with $p$ are {\em plain}.
For example, consider the one
where all notched tags in Figure \ref{fig:tagged-triangulation-D} (a)
are replaced with plain.
\item {\em Type $-1$}. 
All  tags of ends with $p$ are {\em notched}.
See  Figure \ref{fig:tagged-triangulation-D} (a), for example.
\item {\em Type $0$}. 
There are both notched and plain tags of ends with $p$.
See  Figure \ref{fig:tagged-triangulation-D} (b), for example.
In fact, according to
Example \ref{ex:tag1},
there exists exactly a pair of tagged arcs $\alpha$ and $\beta$ which end at $p$
 such that their untagged versions are identical.
\end{itemize}

Having this classification in mind,
to each labeled tagged triangulation $T=(\alpha_i)_{i=1}^n$,
we assign a labeled ideal triangulation $T^{\circ}$ as follows:
\begin{itemize}
\item {\em Step 1}.
For each puncture $p$ of type $-1$,
replace all notched tags of ends with $p$ to plain. For example,
in Figure \ref{fig:tagged-triangulation-D} (a), replace the tags of arcs with labels 3,
4, and 5 to plain.
\item {\em Step 2}.
For each puncture $p$ of type $0$,
do the following replacement.
\begin{center}
\begin{pspicture}(-1,-1)(6,1.2)
\psset{linewidth=0.5pt}
\psset{fillstyle=solid, fillcolor=black}
\pscircle(0,1){0.08} 
\pscircle(0,-1){0.08} 
\pscircle(0,0){0.08} 
\psset{fillstyle=none}
%
\pscurve(0,-1)(0.25,-0.5)(0,0)
\pscurve(0,-1)(-0.25,-0.5)(0,0)
\pscurve(0,1)(0.6,0.45)(0.6,-0.45)(0,-1)
\pscurve(0,1)(-0.6,0.45)(-0.6,-0.45)(0,-1)
%
\rput[c]{0}(-0.4,-0.4){\small $i$}
\rput[c]{0}(0.4,-0.4){\small $j$}
\rput[c]{40}(0.17,-0.23){$\notch$}
\psline[arrows=->](2,0)(3,0)
\psset{fillstyle=solid, fillcolor=black}
\pscircle(5,1){0.08} 
\pscircle(5,-1){0.08} 
\pscircle(5,0){0.08} 
\psset{fillstyle=none}
\pscurve(5,-1)(4.6,0)(5,0.4)(5.4,0)(5,-1)
\psline(5,0)(5,-1)
\pscurve(5,1)(5.6,0.45)(5.6,-0.45)(5,-1)
\pscurve(5,1)(4.4,0.45)(4.4,-0.45)(5,-1)
\rput[c]{0}(4.9,-0.35){\small $i$}
\rput[c]{0}(5,0.6){\small $j$}

\end{pspicture}
\end{center}
For example,
if $T$ is the one in Figure \ref{fig:tagged-triangulation-D} (b), 
$T^{\circ}$ is given by the one in Figure \ref{fig:triangulation-AD} (b).

\end{itemize}

Now we extend the definition of the adjacency matrix to the tagged triangulations. 
\begin{defn}
To any labeled tagged triangulation $T$, we assign a skew-symmetric matrix
$B(T):=B(T^{\circ})$, where $T^{\circ}$ is the labeled  ideal triangulation associated with $T$
defined above,
and we call it the {\em adjacency matrix of $T$}.
\end{defn}

For example, for $T$
being the one in Figure \ref{fig:tagged-triangulation-D} (b),
$B(T)$ is given by the second matrix in Example
\ref{ex:BQ1}.

Finally we have the resolution of the discrepancy.
\begin{thm}[{\cite[Lemma 9.7]{Fomin08}}] For any
  labeled tagged triangulation $T=(\alpha_i)_{i=1}^n$
   of $(\bfS,\bfM)$ and for any $k=1,\dots,n$,
we have $B(\mu_k (T))=\mu_k (B(T))$.
\end{thm}

See Figure \ref{fig:digon2} for an example and compare it with
Figure \ref{fig:digon1}.
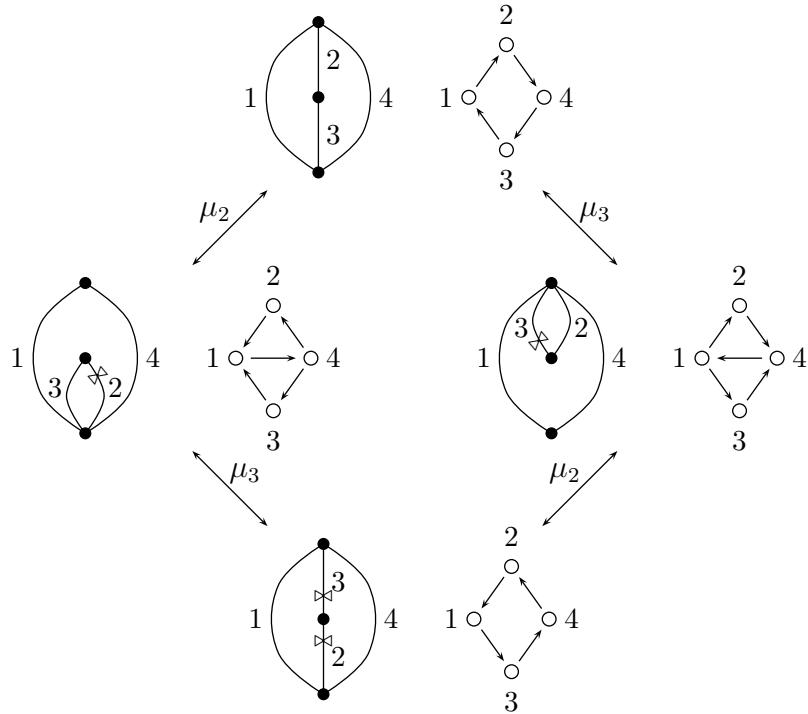
\begin{figure}
\begin{center}
\begin{pspicture}(-1,-1.2)(3.3,1.2)
\psset{linewidth=0.5pt}
\psset{fillstyle=solid, fillcolor=black}
\pscircle(0,1){0.08} 
\pscircle(0,-1){0.08} 
\pscircle(0,0){0.08} 
\psset{fillstyle=none}
\psline(0,1)(0,0)
\psline(0,0)(0,-1)
\pscurve(0,1)(0.6,0.45)(0.6,-0.45)(0,-1)
\pscurve(0,1)(-0.6,0.45)(-0.6,-0.45)(0,-1)
\rput[c]{0}(-0.9,0){\small 1}
\rput[c]{0}(0.9,0){\small 4}
\rput[c]{0}(0.2,0.5){\small 2}
\rput[c]{0}(0.2,-0.5){\small 3}
%
\pscircle(2,0){0.1} 
\pscircle(3,0){0.1} 
\pscircle(2.5,0.7){0.1} 
\pscircle(2.5,-0.7){0.1} 
%
\psline[arrows=->](2.4,-0.56)(2.1,-0.14)
\psline[arrows=->](2.1,0.14)(2.4,0.56)
\psline[arrows=->](2.6,0.56)(2.9,0.14)
\psline[arrows=->](2.9,-0.14)(2.6,-0.56)
\rput[c]{0}(1.7,0){\small 1}
\rput[c]{0}(2.5,1.1){\small 2}
\rput[c]{0}(3.3,0){\small 4}
\rput[c]{0}(2.5,-1.1){\small 3}
\end{pspicture}
\break
\begin{pspicture}(0,0)(1,1)
\psset{linewidth=0.5pt}
\psline[arrows=<->](0,0)(1,1)
\rput[c]{0}(0.3,0.7){$\mu_2$}
\end{pspicture}
\hskip100pt
\begin{pspicture}(0,0)(1,1)
\psset{linewidth=0.5pt}
\psline[arrows=<->](0,1)(1,0)
\rput[c]{0}(0.7,0.7){$\mu_3$}
\end{pspicture}
\break
\begin{pspicture}(-1,-1.2)(3.3,1.2)
\psset{linewidth=0.5pt}
\psset{fillstyle=solid, fillcolor=black}
\pscircle(0,1){0.08} 
\pscircle(0,-1){0.08} 
\pscircle(0,0){0.08} 
\psset{fillstyle=none}
%
\pscurve(0,-1)(0.25,-0.5)(0,0)
\pscurve(0,-1)(-0.25,-0.5)(0,0)
\pscurve(0,1)(0.6,0.45)(0.6,-0.45)(0,-1)
\pscurve(0,1)(-0.6,0.45)(-0.6,-0.45)(0,-1)
\rput[c]{0}(-0.9,0){\small 1}
\rput[c]{0}(0.9,0){\small 4}
\rput[c]{0}(-0.4,-0.4){\small $3$}
\rput[c]{0}(0.4,-0.4){\small $2$}
\rput[c]{40}(0.17,-0.23){$\notch$}
%
\pscircle(2,0){0.1} 
\pscircle(3,0){0.1} 
\pscircle(2.5,0.7){0.1} 
\pscircle(2.5,-0.7){0.1} 
%
\psline[arrows=->](2.4,-0.56)(2.1,-0.14)
\psline[arrows=->](2.4,0.56)(2.1,0.14)
\psline[arrows=->](2.9,0.14)(2.6,0.56)
\psline[arrows=->](2.9,-0.14)(2.6,-0.56)
\psline[arrows=->](2.2,0)(2.8,0)
\rput[c]{0}(1.7,0){\small 1}
\rput[c]{0}(2.5,1.1){\small 2}
\rput[c]{0}(3.3,0){\small 4}
\rput[c]{0}(2.5,-1.1){\small 3}
\end{pspicture}
\hskip50pt
\begin{pspicture}(-1,-1.2)(3.3,1.2)
\psset{linewidth=0.5pt}
\psset{fillstyle=solid, fillcolor=black}
\pscircle(0,1){0.08} 
\pscircle(0,-1){0.08} 
\pscircle(0,0){0.08} 
\psset{fillstyle=none}
%
\pscurve(0,1)(0.25,0.5)(0,0)
\pscurve(0,1)(-0.25,0.5)(0,0)
\pscurve(0,1)(0.6,0.45)(0.6,-0.45)(0,-1)
\pscurve(0,1)(-0.6,0.45)(-0.6,-0.45)(0,-1)
\rput[c]{0}(-0.9,0){\small 1}
\rput[c]{0}(0.9,0){\small 4}
\rput[c]{0}(-0.4,0.4){\small $3$}
\rput[c]{0}(0.4,0.4){\small $2$}
\rput[c]{40}(-0.17,0.23){$\notch$}
%
\pscircle(2,0){0.1} 
\pscircle(3,0){0.1} 
\pscircle(2.5,0.7){0.1} 
\pscircle(2.5,-0.7){0.1} 
%
\psline[arrows=->](2.1,-0.14)(2.4,-0.56)
\psline[arrows=->](2.1,0.14)(2.4,0.56)
\psline[arrows=->](2.6,0.56)(2.9,0.14)
\psline[arrows=->](2.6,-0.56)(2.9,-0.14)
\psline[arrows=->](2.8,0)(2.2,0)
\rput[c]{0}(1.7,0){\small 1}
\rput[c]{0}(2.5,1.1){\small 2}
\rput[c]{0}(3.3,0){\small 4}
\rput[c]{0}(2.5,-1.1){\small 3}
\end{pspicture}
\break
\begin{pspicture}(0,0)(1,1)
\psset{linewidth=0.5pt}
\psline[arrows=<->](0,1)(1,0)
\rput[c]{0}(0.7,0.7){$\mu_3$}
\end{pspicture}
\hskip100pt
\begin{pspicture}(0,0)(1,1)
\psset{linewidth=0.5pt}
\psline[arrows=<->](0,0)(1,1)
\rput[c]{0}(0.3,0.7){$\mu_2$}
\end{pspicture}
\break
\begin{pspicture}(-1,-1.2)(3.3,1.2)
\psset{linewidth=0.5pt}
\psset{fillstyle=solid, fillcolor=black}
\pscircle(0,1){0.08} 
\pscircle(0,-1){0.08} 
\pscircle(0,0){0.08} 
\psset{fillstyle=none}
\psline(0,1)(0,0)
\psline(0,0)(0,-1)
\pscurve(0,1)(0.6,0.45)(0.6,-0.45)(0,-1)
\pscurve(0,1)(-0.6,0.45)(-0.6,-0.45)(0,-1)
\rput[c]{0}(-0.9,0){\small 1}
\rput[c]{0}(0.9,0){\small 4}
\rput[c]{0}(0.2,0.5){\small 3}
\rput[c]{0}(0.2,-0.5){\small 2}
\rput[c]{0}(0,0.3){$\notch$}
\rput[c]{0}(0,-0.3){$\notch$}
\pscircle(2,0){0.1} 
\pscircle(3,0){0.1} 
\pscircle(2.5,0.7){0.1} 
\pscircle(2.5,-0.7){0.1} 
%
\psline[arrows=->](2.1,-0.14)(2.4,-0.56)
\psline[arrows=->](2.4,0.56)(2.1,0.14)
\psline[arrows=->](2.9,0.14)(2.6,0.56)
\psline[arrows=->](2.6,-0.56)(2.9,-0.14)
\rput[c]{0}(1.7,0){\small 1}
\rput[c]{0}(2.5,1.1){\small 2}
\rput[c]{0}(3.3,0){\small 4}
\rput[c]{0}(2.5,-1.1){\small 3}
\end{pspicture}
\end{center}
\caption{Flips of tagged arcs in  a digon with a puncture
and mutations of the corresponding quivers.}
\label{fig:digon2}
\end{figure}

\subsection{Realization of exchange graph of labeled seeds}

So far, we have concentrated on realizing the exchange matrix part of seeds.
We now turn to the realization of the {\em exchange graph\/} of the labeled seeds.
\begin{defn}
The {\em exchange graph of the labeled seeds of
a cluster algebra $\mathcal{A}=\mathcal{A}(B^0,x^0,y^0;\bbP)$}
is a graph
whose  vertices are the labeled seeds of $\mathcal{A}$ 
and a edge are drawn between two vertices if they  are related by
a mutation.
\end{defn}

The following definition is parallel to Definition \ref{defn:period1}.

\begin{defn} 
\label{defn:period2}
Let $T=(\alpha_i)_{i=1}^n$ be a labeled tagged triangulation
of a bordered surface $(\bfS,\bfM)$,
and 
let
$\nu$ be a  permutation 
of $\{1,\dots, n\}$.
A mutation sequence $\vec{k}=(k_t)_{t=1}^N$ is called a {\em $\nu$-period
of $T$} if,
for $T'=(\alpha'_i)_{i=1}^n:=\mu_{\vec{k}} (T)$,
 $\alpha'_{\nu(i)}=\alpha_i$ holds for any $i$.
\end{defn}

Let us fix the {\em initial labeled tagged triangulation} $T^0=(\alpha^0_i)_{i=1}^n$ of
  $(\bfS,\bfM)$,
  which is any labeled tagged triangulation.
 Let $B^0=B(T^0)$ be the adjacency matrix of $T^0$.
Then, we have the associated cluster algebra $\mathcal{A}(B^0,x^0,y^0;\bbQ_+(y^0))$,
 where the choice of $x^0$ is not essential.

 \begin{thm}[{cf.  \cite[Theorem 7.11]{Fomin08},
\label{thm:period2}
 \cite[Theorem 6.1]{Fomin08b}}] 
 Let $(B,x,y)$ and $T$ the ones obtained from
 $(B^0,x^0,y^0)$ and $T^0$ by the {\em same
 sequence\/} $\vec{k}$ of mutations/flips.
Then,
 a mutation sequence $\vec{k}$ is a $\nu$-period of $(B,x,y)$ if and only if 
 it is a $\nu$-period of $T$.
\end{thm}

\begin{proof} Let us set $(B',x',y')=\mu_{\vec{k}}(B,x,y)$
and $T'=\mu_{\vec{k}}(T)$.
Then,
Theorem 6.1 of
  \cite{Fomin08b}  tells
that $x'_{\nu(i)}=x_i$ if and only if
$\alpha'_{\nu(i)}=\alpha_i$.
 \end{proof}

\begin{rem}  Theorem 6.1 of  \cite{Fomin08b} is
the unlabeled version of Theorem \ref{thm:period2}
and here we reformulated (a part of)  it
for the labeled one
 with the help of
Theorem
\ref{thm:period1}.
\end{rem}

\begin{ex}[Pentagon relation (2)]
\label{ex:pentagon2}
The counterpart of the pentagon relation of the seeds in
Example \ref{ex:pentagon1}
is given by the mutation sequence of labeled triangulations
of a pentagon without a puncture
 in Figure \ref{fig:pentagon1}.
\end{ex}

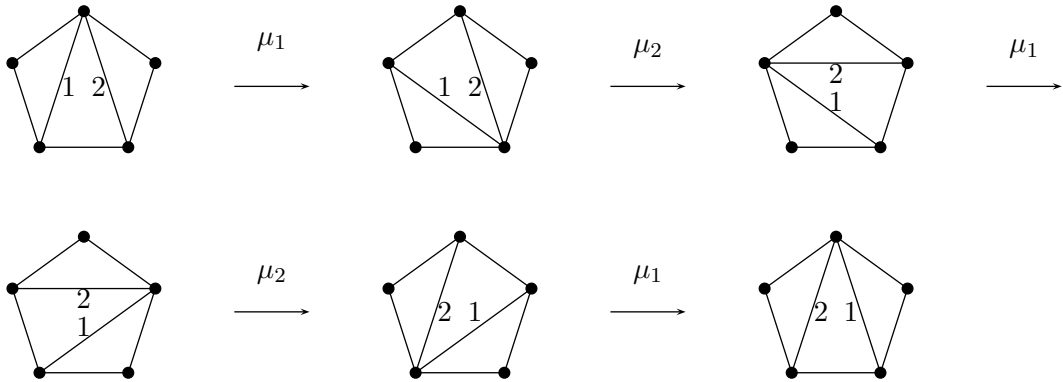
\begin{figure}
\begin{center}
\begin{pspicture}(-1,-1)(13,4.2)
%
\psset{linewidth=0.5pt}
\psset{fillstyle=solid, fillcolor=black}
\pscircle(0,4){0.08} 
\pscircle(0.95,3.31){0.08} 
\pscircle(0.59,2.19){0.08} 
\pscircle(-0.59,2.19){0.08} 
\pscircle(-0.95,3.31){0.08} 
\psset{fillstyle=none}
\psline(0,4)(0.95,3.31)
\psline(0.95,3.31)(0.59,2.19)
\psline(0.59,2.19)(-0.59,2.19)
\psline(-0.59,2.19)(-0.95,3.31)
\psline(-0.95,3.31)(0,4)
\psline(0,4)(0.59,2.19)
\psline(0,4)(-0.59,2.19)
\rput[c]{0}(2.5,3.6){\small $\mu_1$}
\psline[arrows=->](2,3)(3,3)
\rput[c]{0}(-0.2,3){\small 1}
\rput[c]{0}(0.2,3){\small 2}
\rput[c]{0}(-0.32,3){\small }
%
\psset{linewidth=0.5pt}
\psset{fillstyle=solid, fillcolor=black}
\pscircle(5,4){0.08} 
\pscircle(5.95,3.31){0.08} 
\pscircle(5.59,2.19){0.08} 
\pscircle(4.41,2.19){0.08} 
\pscircle(4.05,3.31){0.08} 
\psset{fillstyle=none}
\psline(5,4)(5.95,3.31)
\psline(5.95,3.31)(5.59,2.19)
\psline(5.59,2.19)(4.41,2.19)
\psline(4.41,2.19)(4.05,3.31)
\psline(4.05,3.31)(5,4)
\psline(5,4)(5.59,2.19)
\psline(4.05,3.31)(5.59,2.19)
\rput[c]{0}(7.5,3.5){\small $\mu_2$}
\psline[arrows=->](7,3)(8,3)
\rput[c]{0}(4.8,3){\small 1}
\rput[c]{0}(5.2,3){\small 2}
%
\psset{linewidth=0.5pt}
\psset{fillstyle=solid, fillcolor=black}
\pscircle(10,4){0.08} 
\pscircle(10.95,3.31){0.08} 
\pscircle(10.59,2.19){0.08} 
\pscircle(9.41,2.19){0.08} 
\pscircle(9.05,3.31){0.08} 
\psset{fillstyle=none}
\psline(10,4)(10.95,3.31)
\psline(10.95,3.31)(10.59,2.19)
\psline(10.59,2.19)(9.41,2.19)
\psline(9.41,2.19)(9.05,3.31)
\psline(9.05,3.31)(10,4)
\psline(9.05,3.31)(10.95,3.31)
\psline(9.05,3.31)(10.59,2.19)
\rput[c]{0}(12.5,3.5){\small $\mu_1$}
\psline[arrows=->](12,3)(13,3)
\rput[c]{0}(10,2.8){\small 1}
\rput[c]{0}(10,3.17){\small 2}
\psset{linewidth=0.5pt}
\psset{fillstyle=solid, fillcolor=black}
\pscircle(0,1){0.08} 
\pscircle(0.95,0.31){0.08} 
\pscircle(0.59,-0.81){0.08} 
\pscircle(-0.59,-0.81){0.08} 
\pscircle(-0.95,0.31){0.08} 
\psset{fillstyle=none}
\psline(0,1)(0.95,0.31)
\psline(0.95,0.31)(0.59,-0.81)
\psline(0.59,-0.81)(-0.59,-0.81)
\psline(-0.59,-0.81)(-0.95,0.31)
\psline(-0.95,0.31)(0,1)
\psline(-0.95,0.31)(0.95,0.31)
\psline(0.95,0.31)(-0.59,-0.81)
\rput[c]{0}(2.5,0.5){\small $\mu_2$}
\psline[arrows=->](2,0)(3,0)
\rput[c]{0}(0,-0.2){\small 1}
\rput[c]{0}(0,0.17){\small 2}
\psset{linewidth=0.5pt}
\psset{fillstyle=solid, fillcolor=black}
\pscircle(5,1){0.08} 
\pscircle(5.95,0.31){0.08} 
\pscircle(5.59,-0.81){0.08} 
\pscircle(4.41,-0.81){0.08} 
\pscircle(4.05,0.31){0.08} 
\psset{fillstyle=none}
\psline(5,1)(5.95,0.31)
\psline(5.95,0.31)(5.59,-0.81)
\psline(5.59,-0.81)(4.41,-0.81)
\psline(4.41,-0.81)(4.05,0.31)
\psline(4.05,0.31)(5,1)
\psline(5,1)(4.41,-0.81)
\psline(5.95,0.31)(4.41,-0.81)
\rput[c]{0}(7.5,0.5){\small $\mu_1$}
\psline[arrows=->](7,0)(8,0)
\rput[c]{0}(4.8,0){\small 2}
\rput[c]{0}(5.2,0){\small 1}
\psset{linewidth=0.5pt}
\psset{fillstyle=solid, fillcolor=black}
\pscircle(10,1){0.08} 
\pscircle(10.95,0.31){0.08} 
\pscircle(10.59,-0.81){0.08} 
\pscircle(9.41,-0.81){0.08} 
\pscircle(9.05,0.31){0.08} 
\psset{fillstyle=none}
\psline(10,1)(10.95,0.31)
\psline(10.95,0.31)(10.59,-0.81)
\psline(10.59,-0.81)(9.41,-0.81)
\psline(9.41,-0.81)(9.05,0.31)
\psline(9.05,0.31)(10,1)
\psline(10,1)(9.41,-0.81)
\psline(10,1)(10.59,-0.81)
\rput[c]{0}(9.8,0){\small 2}
\rput[c]{0}(10.2,0){\small 1}
\end{pspicture}
\end{center}
\caption{Pentagon relation of labeled triangulations.}
\label{fig:pentagon1}
\end{figure}

Let $\mathrm{LTT}(T^0)$ be the set of all labeled tagged triangulations obtained from the initial
labeled tagged triangulation $T^0$ by  sequences
of flips.

By setting $\nu=\mathrm{id}$ in Theorem \ref{thm:period2},
we have the following corollary.
\begin{cor} 
\label{cor:surface1}
There is a bijection $\Psi:
\mathrm{Seed}(B^0,x^0,y^0;
\bbQ_+(y^0))
\rightarrow \mathrm{LTT}(T^0)$
such that $\Psi(T^0)=(B^0,x^0,y^0)$
and $\Psi$ commutes with flips/mutations.
\end{cor}

In other words, the exchange graph of 
$\mathcal{A}(B^0,x^0,y^0;\bbQ_+(y^0))$
is identified with the exchange graph of
the labeled tagged triangulations 
in $\mathrm{LTT}(T^0)$  by flips.

To present a general statement on 
$\mathrm{LTT}(T^0)$, we need to exclude
some ``exceptional'' bordered surfaces.

\begin{defn}
A bordered surface $(\bfS,\bfM)$
is said to be {\em generic\/} if it is not one of the following:
\begin{itemize}
\item an once-punctured digon
\item an unpunctured annulus with  one marked point on each boundary component
\item an  once-punctured  torus
\end{itemize} 
\end{defn}

This definition is motivated by the following property
which  holds only for generic bordered surfaces.

\begin{lem}
\label{lem:generic1}
Let $(\bfS,\bfM)$ be a generic bordered surface,
let $T=(\alpha_i)_{i=1}^n$ be 
 any labeled tagged triangulation of $(\bfS,\bfM)$,
 and let $B=B(T)$.
Then, for any pair of distinct indices $i,j\in \{1,\dots,n\}$,
 there is a sequence of indices $i_0=i,i_1,\dots,i_r=j$ such that
 $|b_{i_s i_{s+1}}|=1$ for any $s=0,\dots,r-1$.
\end{lem}
\begin{proof}
This is an immediate consequence of the construction of ideal triangulations
and the associated adjacency matrices
by ``puzzle pieces" and ``blocks" in \cite[Theorem 13.3]{Fomin08}.
\end{proof}

\begin{prop}[{cf.~\cite[Proposition 7.10]{Fomin08b}}]
\label{prop:bij1}
Let $(\bfS,\bfM)$ be a generic bordered surface.
\par
(a).
 If $(\bfS,\bfM)$ is not a closed surface with exactly one puncture,
then $\mathrm{LTT}(T^0)$ consists of all
 labeled tagged triangulations of  $(\bfS,\bfM)$.
\par
(b).
If $(\bfS,\bfM)$ is  a closed surface with exactly one puncture,
then $\mathrm{LTT}(T^0)$ consists of all
 labeled tagged triangulations of  $(\bfS,\bfM)$ whose
 arcs are tagged in the same way as the arcs
of $T^0$.
\end{prop}
\begin{proof} 
The unlabeled version of the statement is true 
by \cite[Proposition 7.10]{Fomin08b}
(including the nongeneric case).
On the other hand,
thanks to Lemma \ref{lem:generic1},
for any  labeled tagged triangulation $T$ of $(\bfS,\bfM)$,
one can exchange the labels of any pair of arcs of $T$
by repeatedly applying  the sequence of flips for a pentagon in 
Figure \ref{fig:pentagon1} (and flips in 
Figure \ref{fig:digon2} if necessary).
Thus, the statement is also true for the labeled one.
\end{proof}

\subsection{Reformulation by signed triangulations}
\label{subsec:reformulation}

Let us explain the notion of {\em signed triangulations\/}
recently introduced by \cite{Labardini12,Bridgeland13}.
It is nothing but an
alternative way of expressing
tagged triangulations,
but it involves the operation called {\em pop}.

\begin{defn}  A {\em labeled signed triangulation\/} of
a bordered surface $(\bfS,\bfM)$ is a pair $T_{\sigma}=(T,\sigma)$ such that  
$T$ is a labeled ideal triangulation  of $(\bfS,\bfM)$ and $\sigma$ is a {\em sign function\/} from
the set of the punctures in $(\bfS,\bfM)$ to the sign set $\{+,-\}=\{1,-1\}$.
The sign $\sigma(p)$ at $p$ is denoted by $\sigma_p$.
\end{defn}

\begin{figure}
\begin{center}
\begin{pspicture}(4,-1)(9,1.2)
%
\psset{linewidth=0.5pt}
\psset{fillstyle=solid, fillcolor=black}
\pscircle(5,1){0.08} 
\pscircle(5,-1){0.08} 
\pscircle(5,0){0.08} 
\psset{fillstyle=none}
\pscurve(5,-1)(4.6,0)(5,0.4)(5.4,0)(5,-1)
\psline(5,0)(5,-1)
\pscurve(5,1)(5.6,0.45)(5.6,-0.45)(5,-1)
\pscurve(5,1)(4.4,0.45)(4.4,-0.45)(5,-1)
\rput[c]{0}(5.1,-0.4){\small $i$}
\rput[c]{0}(5,0.6){\small $j$}
\rput[c]{0}(5.25,0){\small $+$}
\rput[c]{0}(4.75,0){\small $p$}
\rput[c]{0}(6.5,0.3){\small $\kappa_p$}
\psline[arrows=<->](6,0)(7,0)
\psset{fillstyle=solid, fillcolor=black}
\pscircle(8,1){0.08} 
\pscircle(8,-1){0.08} 
\pscircle(8,0){0.08} 
\psset{fillstyle=none}
\pscurve(8,-1)(7.6,0)(8,0.4)(8.4,0)(8,-1)
\psline(8,0)(8,-1)
\pscurve(8,1)(8.6,0.45)(8.6,-0.45)(8,-1)
\pscurve(8,1)(7.4,0.45)(7.4,-0.45)(8,-1)
\rput[c]{0}(8.1,-0.4){\small $j$}
\rput[c]{0}(8,0.6){\small $i$}
\rput[c]{0}(8.25,0){\small $-$}
\rput[c]{0}(7.75,0){\small $p$}
\end{pspicture}
\end{center}
\caption{Pop at a puncture $p$ inside a self-folded triangle.}
\label{fig:pop1}
\end{figure}
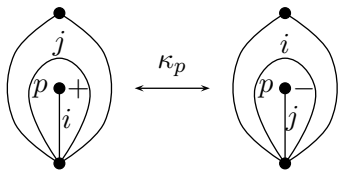

Let $T_{\sigma}$ be a labeled signed triangulation.
For a puncture $p$ inside a self-folded triangle of $T_{\sigma}$,
we define the operation $\kappa_p$ 
 as illustrated in Figure \ref{fig:pop1},
 and we call it {\em the pop at $p$}.
It is important that {\em the labels $i$ and $j$ are  interchanged\/} by a pop.
(It was first introduced by \cite{Gaiotto09} without sign.)
Let us introduce an equivalence relation, called  the {\em pop-equivalence\/},
among the labeled signed triangulations
of $(\bfS,\bfM)$ such that $T_{\sigma} \sim T'_{\sigma'}$
if they are related by a finite sequence of pops, including the empty sequence.
The equivalence class of ${T_{\sigma}}$ is denoted by $[{T_{\sigma}}] $
and called the {\em pop-equivalence class\/} of ${T_{\sigma}}$.

\begin{prop}[\cite{Labardini12,Bridgeland13}]
\label{prop:tagged1}
There is a natural one-to-one correspondence
between
the labeled tagged triangulations
 of $(\bfS,\bfM)$
 and the pop-equivalence classes
of the labeled signed triangulations
 of $(\bfS,\bfM)$.
\end{prop}

The correspondence is given as follows.
A labeled tagged triangulation 
$T$ is identified with
 the pop-equivalence class $[{T'_{\sigma}}] $,
 where its representative $T'_{\sigma}=(T',\sigma)$
 is  obtained from $T$ by
doing firstly the following operation for each puncture $p$ in $T$,
and then removing the the tags of all tagged arcs:
\begin{itemize}
\item
if  $p$ is of type 1 (as defined in Section \ref{subsubsec:tagged}), 
assign the sign $\sigma_p=+$,
\item
if $p$ is of type $-1$, 
assign the sign $\sigma_p=-$,
\item
if $p$ is of type 0,
we may do one of two ways
  (see Figure \ref{fig:pop2}):
\begin{itemize}
\item
(i)  replace the notched tagged arc
ending at $p$ with the loop surrounding $p$,
and assign the sign $\sigma_p=+$, or
\item
(ii) replace the plain tagged arc
ending at $p$ with the loop surrounding $p$,
and assign the sign $\sigma_p=-$.
\end{itemize}
Two choices are  exactly connected by the pop $\kappa_p$,
thus they define the same pop-equivalence class.
\end{itemize}

\begin{figure}
\begin{center}
\begin{pspicture}(-1,-1)(9,1.2)
\psset{linewidth=0.5pt}
\psset{fillstyle=solid, fillcolor=black}
\pscircle(0,1){0.08} 
\pscircle(0,-1){0.08} 
\pscircle(0,0){0.08} 
\psset{fillstyle=none}
%
\pscurve(0,-1)(0.25,-0.5)(0,0)
\pscurve(0,-1)(-0.25,-0.5)(0,0)
\pscurve(0,1)(0.6,0.45)(0.6,-0.45)(0,-1)
\pscurve(0,1)(-0.6,0.45)(-0.6,-0.45)(0,-1)
%
\rput[c]{0}(-0.25,0){\small $p$}
\rput[c]{0}(-0.4,-0.4){\small $i$}
\rput[c]{0}(0.4,-0.4){\small $j$}
\rput[c]{40}(0.17,-0.23){$\notch$}
\psline[arrows=<->](2,0)(3,0)
\psset{fillstyle=solid, fillcolor=black}
\pscircle(5,1){0.08} 
\pscircle(5,-1){0.08} 
\pscircle(5,0){0.08} 
\psset{fillstyle=none}
\pscurve(5,-1)(4.6,0)(5,0.4)(5.4,0)(5,-1)
\psline(5,0)(5,-1)
\pscurve(5,1)(5.6,0.45)(5.6,-0.45)(5,-1)
\pscurve(5,1)(4.4,0.45)(4.4,-0.45)(5,-1)
\rput[c]{0}(4.75,0){\small $p$}
\rput[c]{0}(5.1,-0.4){\small $i$}
\rput[c]{0}(5,0.6){\small $j$}
\rput[c]{0}(5.25,0){\small $+$}
\rput[c]{0}(6.5,0.4){\small $\kappa_p$}
\psline[arrows=<->](6.2,0)(6.8,0)
\psset{fillstyle=solid, fillcolor=black}
\pscircle(8,1){0.08} 
\pscircle(8,-1){0.08} 
\pscircle(8,0){0.08} 
\psset{fillstyle=none}
\pscurve(8,-1)(7.6,0)(8,0.4)(8.4,0)(8,-1)
\psline(8,0)(8,-1)
\pscurve(8,1)(8.6,0.45)(8.6,-0.45)(8,-1)
\pscurve(8,1)(7.4,0.45)(7.4,-0.45)(8,-1)
\rput[c]{0}(7.75,0){\small $p$}
\rput[c]{0}(8.1,-0.4){\small $j$}
\rput[c]{0}(8,0.6){\small $i$}
\rput[c]{0}(8.25,0){\small $-$}
\end{pspicture}
\end{center}
\caption{Two representatives 
 of a labeled tagged triangulation inside a digon with a puncture
 by labeled signed triangulations.} 
\label{fig:pop2}
\end{figure}
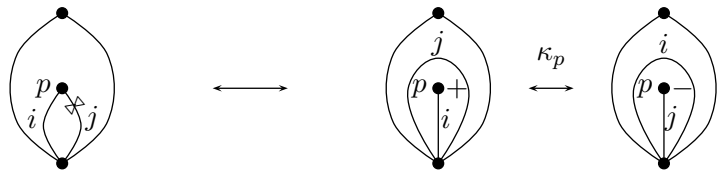

Using this new presentation, our familiar example of flips inside a digon with a puncture
looks as in Figure \ref{fig:digon3}.

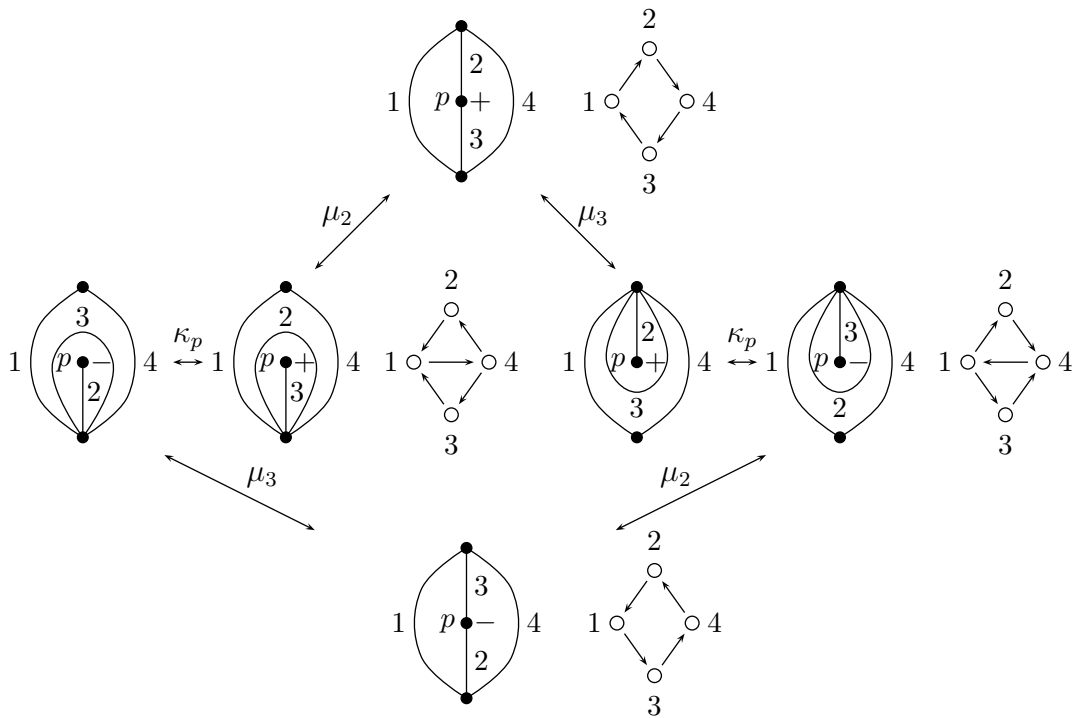
\begin{figure}
\begin{center}
\begin{pspicture}(-1,-1.2)(3.3,1.2)
\psset{linewidth=0.5pt}
\psset{fillstyle=solid, fillcolor=black}
\pscircle(0,1){0.08} 
\pscircle(0,-1){0.08} 
\pscircle(0,0){0.08} 
\psset{fillstyle=none}
\psline(0,1)(0,0)
\psline(0,0)(0,-1)
\pscurve(0,1)(0.6,0.45)(0.6,-0.45)(0,-1)
\pscurve(0,1)(-0.6,0.45)(-0.6,-0.45)(0,-1)
\rput[c]{0}(-0.9,0){\small 1}
\rput[c]{0}(0.9,0){\small 4}
\rput[c]{0}(0.2,0.5){\small 2}
\rput[c]{0}(0.2,-0.5){\small 3}
\rput[c]{0}(0.25,0){\small $+$}
\rput[c]{0}(-0.25,0){\small $p$}
%
\pscircle(2,0){0.1} 
\pscircle(3,0){0.1} 
\pscircle(2.5,0.7){0.1} 
\pscircle(2.5,-0.7){0.1} 
%
\psline[arrows=->](2.4,-0.56)(2.1,-0.14)
\psline[arrows=->](2.1,0.14)(2.4,0.56)
\psline[arrows=->](2.6,0.56)(2.9,0.14)
\psline[arrows=->](2.9,-0.14)(2.6,-0.56)
\rput[c]{0}(1.7,0){\small 1}
\rput[c]{0}(2.5,1.1){\small 2}
\rput[c]{0}(3.3,0){\small 4}
\rput[c]{0}(2.5,-1.1){\small 3}
\end{pspicture}
\break
\begin{pspicture}(0,0)(12.2,1)
\psset{linewidth=0.5pt}
\psline[arrows=<->](3,0)(4,1)
\rput[c]{0}(3.3,0.7){$\mu_2$}
\psline[arrows=<->](6,1)(7,0)
\rput[c]{0}(6.7,0.7){$\mu_3$}
\end{pspicture}
\break
\begin{pspicture}(-3.4,-1.2)(3.3,1.2)
\psset{linewidth=0.5pt}
\psset{fillstyle=solid, fillcolor=black}
\pscircle(-2.4,1){0.08} 
\pscircle(-2.4,-1){0.08} 
\pscircle(-2.4,0){0.08} 
\psset{fillstyle=none}
\pscurve(-2.4,-1)(-2.8,0)(-2.4,0.4)(-2,0)(-2.4,-1)
\psline(-2.4,0)(-2.4,-1)
\pscurve(-2.4,1)(-1.8,0.45)(-1.8,-0.45)(-2.4,-1)
\pscurve(-2.4,1)(-3,0.45)(-3,-0.45)(-2.4,-1)
\rput[c]{0}(-3.3,0){\small 1}
\rput[c]{0}(-1.5,0){\small 4}
\rput[c]{0}(-2.25,-0.4){\small $2$}
\rput[c]{0}(-2.4,0.6){\small $3$}
\rput[c]{0}(-2.15,0){\small $-$}
\rput[c]{0}(-2.65,0){\small $p$}
%
\rput[c]{0}(-1,0.3){$\kappa_p$}
\psline[arrows=<->](-1.2,0)(-0.8,0)
\psset{fillstyle=solid, fillcolor=black}
\pscircle(0.3,1){0.08} 
\pscircle(0.3,-1){0.08} 
\pscircle(0.3,0){0.08} 
\psset{fillstyle=none}
\pscurve(0.3,-1)(-0.1,0)(0.3,0.4)(0.7,0)(0.3,-1)
\psline(0.3,0)(0.3,-1)
\pscurve(0.3,1)(0.9,0.45)(0.9,-0.45)(0.3,-1)
\pscurve(0.3,1)(-0.3,0.45)(-0.3,-0.45)(0.3,-1)
\rput[c]{0}(-0.6,0){\small 1}
\rput[c]{0}(1.2,0){\small 4}
\rput[c]{0}(0.3,0.6){\small 2}
\rput[c]{0}(0.45,-0.4){\small 3}
\rput[c]{0}(0.55,0){\small $+$}
\rput[c]{0}(0.05,0){\small $p$}
%
\pscircle(2,0){0.1} 
\pscircle(3,0){0.1} 
\pscircle(2.5,0.7){0.1} 
\pscircle(2.5,-0.7){0.1} 
%
\psline[arrows=->](2.4,-0.56)(2.1,-0.14)
\psline[arrows=->](2.4,0.56)(2.1,0.14)
\psline[arrows=->](2.9,0.14)(2.6,0.56)
\psline[arrows=->](2.9,-0.14)(2.6,-0.56)
\psline[arrows=->](2.2,0)(2.8,0)
\rput[c]{0}(1.7,0){\small 1}
\rput[c]{0}(2.5,1.1){\small 2}
\rput[c]{0}(3.3,0){\small 4}
\rput[c]{0}(2.5,-1.1){\small 3}
\end{pspicture}
\hskip15pt
\begin{pspicture}(-3.4,-1.2)(3.6,1.2)
\psset{linewidth=0.5pt}
\psset{fillstyle=solid, fillcolor=black}
\pscircle(-2.4,1){0.08} 
\pscircle(-2.4,-1){0.08} 
\pscircle(-2.4,0){0.08} 
\psset{fillstyle=none}
\pscurve(-2.4,1)(-2.8,0)(-2.4,-0.4)(-2,0)(-2.4,1)
\psline(-2.4,0)(-2.4,1)
\pscurve(-2.4,1)(-1.8,0.45)(-1.8,-0.45)(-2.4,-1)
\pscurve(-2.4,1)(-3,0.45)(-3,-0.45)(-2.4,-1)
\rput[c]{0}(-3.3,0){\small 1}
\rput[c]{0}(-1.5,0){\small 4}
\rput[c]{0}(-2.4,-0.6){\small 3}
\rput[c]{0}(-2.25,0.4){\small 2}
\rput[c]{0}(-2.15,0){\small $+$}
\rput[c]{0}(-2.65,0){\small $p$}
%
\rput[c]{0}(-1,0.3){$\kappa_p$}
\psline[arrows=<->](-1.2,0)(-0.8,0)
\psset{fillstyle=solid, fillcolor=black}
\pscircle(0.3,1){0.08} 
\pscircle(0.3,-1){0.08} 
\pscircle(0.3,0){0.08} 
\psset{fillstyle=none}
\pscurve(0.3,1)(-0.1,0)(0.3,-0.4)(0.7,0)(0.3,1)
\psline(0.3,0)(0.3,1)
\pscurve(0.3,1)(0.9,0.45)(0.9,-0.45)(0.3,-1)
\pscurve(0.3,1)(-0.3,0.45)(-0.3,-0.45)(0.3,-1)
\rput[c]{0}(-0.6,0){\small 1}
\rput[c]{0}(1.2,0){\small 4}
\rput[c]{0}(0.3,-0.6){\small 2}
\rput[c]{0}(0.45,0.4){\small 3}
\rput[c]{0}(0.55,0){\small $-$}
\rput[c]{0}(0.05,0){\small $p$}
%
\pscircle(2,0){0.1} 
\pscircle(3,0){0.1} 
\pscircle(2.5,0.7){0.1} 
\pscircle(2.5,-0.7){0.1} 
%
\psline[arrows=->](2.1,-0.14)(2.4,-0.56)
\psline[arrows=->](2.1,0.14)(2.4,0.56)
\psline[arrows=->](2.6,0.56)(2.9,0.14)
\psline[arrows=->](2.6,-0.56)(2.9,-0.14)
\psline[arrows=->](2.8,0)(2.2,0)
\rput[c]{0}(1.7,0){\small 1}
\rput[c]{0}(2.5,1.1){\small 2}
\rput[c]{0}(3.3,0){\small 4}
\rput[c]{0}(2.5,-1.1){\small 3}
\end{pspicture}
\break
\begin{pspicture}(0,0)(12.2,1)
\psset{linewidth=0.5pt}
\psline[arrows=<->](1,1)(3,0)
\rput[c]{0}(2.3,0.7){$\mu_3$}
\psline[arrows=<->](9,1)(7,0)
\rput[c]{0}(7.8,0.7){$\mu_2$}
\end{pspicture}
\break
\begin{pspicture}(-1,-1.2)(3.3,1.2)
\psset{linewidth=0.5pt}
\psset{fillstyle=solid, fillcolor=black}
\pscircle(0,1){0.08} 
\pscircle(0,-1){0.08} 
\pscircle(0,0){0.08} 
\psset{fillstyle=none}
\psline(0,1)(0,0)
\psline(0,0)(0,-1)
\pscurve(0,1)(0.6,0.45)(0.6,-0.45)(0,-1)
\pscurve(0,1)(-0.6,0.45)(-0.6,-0.45)(0,-1)
\rput[c]{0}(-0.9,0){\small 1}
\rput[c]{0}(0.9,0){\small 4}
\rput[c]{0}(0.2,0.5){\small 3}
\rput[c]{0}(0.2,-0.5){\small 2}
\rput[c]{0}(0.25,0){\small $-$}
\rput[c]{0}(-0.25,0){\small $p$}
\pscircle(2,0){0.1} 
\pscircle(3,0){0.1} 
\pscircle(2.5,0.7){0.1} 
\pscircle(2.5,-0.7){0.1} 
%
\psline[arrows=->](2.1,-0.14)(2.4,-0.56)
\psline[arrows=->](2.4,0.56)(2.1,0.14)
\psline[arrows=->](2.9,0.14)(2.6,0.56)
\psline[arrows=->](2.6,-0.56)(2.9,-0.14)
\rput[c]{0}(1.7,0){\small 1}
\rput[c]{0}(2.5,1.1){\small 2}
\rput[c]{0}(3.3,0){\small 4}
\rput[c]{0}(2.5,-1.1){\small 3}
\end{pspicture}
\caption{Flips and pops of labeled signed triangulations of a digon with a puncture.}
\label{fig:digon3}
\end{center}
\end{figure}

\subsection{Local rescaling  and signed pops of extended seeds}
\label{subsec:local}

Here we point out a hidden symmetry of the exchange relation
\eqref{eq:xmut7} called the {\em local rescaling}.
This symmetry presents when the seeds admit  surface realization.
Using it, we define the {\em signed pops\/} for {\em extended seeds}.

Let  $T$ be a labeled ideal triangulation
of a bordered surface $(\bfS,\bfM)$,
and let
$B$ be the adjacency matrix of $T$.
We concentrate on a seed $(B,x,y)$ 
with coefficients in  $\mathrm{Trop}(y^0)$,
though the argument can be applicable to a more general situation.
Let $\mathcalP$ be the set of the punctures of
$(\bfS,\bfM)$.

\begin{defn}
For any puncture $p\in\mathcalP$
and 
any nonzero rational number $c$,
we call  the following operation
for each $x$-variable $x_i$ of $x$
 the {\em local rescaling
 at  $p$ by the constant $c$}:
\begin{itemize}
\item If the corresponding arc $\alpha_i$ ends at the
puncture $p$, then multiply $c$ for $x_i$.
\item If the corresponding arc $\alpha_i$ is the outer edge
of a self-folded triangle with $p$ inside it,
then multiply $c^{-1}$ for $x_i$.
\item Otherwise, leave $x_i$ as it is.
\end{itemize}
\end{defn}

\begin{lem}
\label{lem:yinv1}
For any $k$, the factor $\hat{y}_k$ in  \eqref{eq:xmut7} is invariant under the local rescaling.
\end{lem}
\begin{proof}
This can be verified by case-check of configurations involving the puncture 
$p$ and the arc $\alpha_k$.
\end{proof}

Suppose that the arc $\alpha_k$ of $T$ is flippable.
We apply the  flip $T'=\mu_k(T)$
 and 
the signed mutation $(B',x',y')=\mu^{(\ve)}_k(B,x,y)$
in Section \ref{subsec:monomial},
respectively.
(Note that $B'=B(T')$ holds for any $\ve$.)
Then, the local rescaling is defined also for $x'$ by $T'$.

\begin{prop}
\label{prop:local1}
The signed mutation $\mu^{(\ve)}_k$ and
the local rescaling at $p$ by $c$ commute with each other.
\end{prop}
\begin{proof}
By  Lemma \ref{lem:yinv1},
it is enough to show that 
$x'_k$ and $x_k^{-1}\prod_{j=1}^n
x_j^{[-\ve b_{jk}]_+}$ in \eqref{eq:xmut7}
 rescale by the same factor.
 This can be verified by case-check.
\end{proof}

Having the above property in mind,
we introduce the notion of extended seeds and their signed pops.
  Recall that our cluster algebra is a $\bbZ\bbP$-subalgebra
 of the ambient field $\bbQ\bbP(w)$ 
 for some variables $w=(w_i)$ with
$\bbP=\mathrm{Trop}(y^0)$.
 We introduce a $\mathcalP$-tuple of new algebraically independent variables
$\tilde{y}^0=(\tilde{y}^0_p)_{p\in \mathcalP}$.
Let $\tilde{\bbQ}:=\bbQ(\tilde{y}^0)$ be the rational function filed of 
 $\tilde{y}^0$ over $\bbQ$.
In particular,
\begin{align}
1-\tilde{y}^0_p,
1-(\tilde{y}^0_p)^{-1},
(1-\tilde{y}^0_p)^{-1},
(1-(\tilde{y}^0_p)^{-1})^{-1}
\in \tilde{\bbQ}.
\end{align}
We extend the ambient field  $\bbQ\bbP(w)$
to $\tilde{\bbQ}\bbP(w)$.

Let $B=B(T)$ be the adjacency matrix of a labeled ideal triangulation $T$ of
a bordered surface $(\bfS,\bfM)$.
We extend a labeled seed $(B,x,y)$
 to
an {\em labeled extended  seed\/} $(B,x,y,\tilde{y})$, where $\tilde{y}=(\tilde{y}_p)_{p\in \mathcalP}$,
  $\tilde{y}_p\in \{\tilde{y}_p^0, (\tilde{y}_p^{0})^{-1}\}$.
We call $\tilde{y}_p$ the {\em coefficient
at $p$},
or simply  a {\em $\tilde{y}$-variable}.
In particular, we extend the initial seed $(B^0,x^0,y^0)$ to
the initial extended seed $(B^0,x^0,y^0,\tilde{y}^0)$, where $\tilde{y}^0=(\tilde{y}^0_p)_{p\in \mathcalP}$
are the ones as above.
Then, we extend the signed mutation of $(B',x',y')=\mu_k^{(\ve)}(B,x,y)$
in \eqref{eq:ymut6} and \eqref{eq:xmut7} to
the {\em signed mutation\/}
 $(B',x',y',\tilde{y}')
=\mu_k^{(\ve)}(B,x,y,\tilde{y})$ (for labeled extended  seeds)
in a trivial way
by keeping \eqref{eq:ymut6} and \eqref{eq:xmut7}  and setting $\tilde{y}'=
\tilde{y}$.

Finally,
for a puncture $p$ inside a self-folded triangle in $T$,
we define the {\em signed pop  $(B',x',y',\tilde{y}')=\kappa_{p}^{(\ve)}(B,x,y,\tilde{y})$
at $p$ with sign $\ve$}
(for labeled extended  seeds)
 by setting
$B'=B$, $y'=y$, 
and
\begin{align}
\label{eq:vmut1}
\tilde{y}'_q&=
\begin{cases}
\tilde{y}_p^{-1} & q=p\\
\tilde{y}_q & q \neq p,
\end{cases}
\\
\label{eq:xpop1}
x'_i&=
\begin{cases}
(1-\tilde{y}_{p}{}^{\ve}) x_{i_p} & i=i_{p}\\
(1-\tilde{y}_{p}{}^{\ve})^{-1} x_{j_p} & i=j_{p}\\
x_i & i \neq i_p,j_p,
\end{cases}
\end{align}
where $i_{p}$ and $j_{p}$ are the labels of the inner and outer arcs
 of the self-folded triangle around $p$ in $T_{s,a}$.
The signed pop
 $\kappa^{(\ve)}_{p}$ acts on $x$ as the local rescaling
 at $p$ by the constant $1-\tilde{y}_{p}{}^{\ve}$
 (in the extended  field $\tilde{\bbQ}$).
It is easy to see that $\kappa_{p}^{(\ve)}$ 
is not an involution, but $\kappa_{p}^{(+)}$
and $\kappa_{p}^{(-)}$ are inverse to each other.

\section{Mutation of Stokes graphs}
\label{sec:mutationofStokes}
In this section 
we study the mutation of Stokes graphs,
which is purely geometric.
We  introduce Stokes triangulations, and their signed flips and pops.
They effectively control
the mutation of Stokes graphs;
moreover, they give a bridge between the exact WKB analysis and cluster algebra theory.
We also introduce the simple paths and the simple cycles of a Stokes graph,
and give their mutation formulas.

\subsection{Stokes triangulations, signed flips, and signed pops}

To work with the mutation of
Stokes graphs, it is natural to extend the notions of 
bordered surfaces and their ideal triangulations.

\begin{defn}
For a bordered surface $(\bfS,\bfM)$, let $m$ be the total number of 
triangles in any ideal triangulation $T$ of $(\bfS,\bfM)$,
where $m$ does not depend on $T$. Accordingly,
we introduce a set $\bfA$ consisting of $m$  points  of $\bfS$
such that $\bfA\cap \partial \bfS=
 \bfA \cap \bfM=\emptyset$.
We call $a\in \bfA$ a {\em midpoint (of a triangle)},
and in figures it will be shown by a cross.
For brevity, we still call $(\bfS,\bfM,\bfA)$ a bordered surface.
\end{defn}

\begin{defn} 
An {\em arc $\alpha$ in a bordered surface $(\bfS,\bfM,\bfA)$}
is a curve in $\bfS\setminus\bfA$ 
satisfying the four conditions in Definition \ref{defn:arc1}.
Each arc $\alpha$ is considered up to isotopy in the class of such curves.
\end{defn}

When we consider an arc $\alpha$ in $(\bfS,\bfM,\bfA)$,
it is sometimes convenient to regard it 
 as an  arc in $(\bfS,\bfM)$ by forgetting the midpoints.
 In that case we write the latter arc  as  $\tilde{\alpha}$
to avoid confusion.

\begin{defn}
An $n$-tuple $T=(\alpha_i)_{i=1}^n$ of arcs
in $(\bfS,\bfM,\bfA)$
is called a {\em labeled Stokes triangulation of $(\bfS,\bfM,\bfA)$\/} 
if the following conditions are satisfied:
\begin{itemize}
\item
The arcs $\alpha_1$, \dots, $\alpha_n$
in $(\bfS,\bfM,\bfA)$
 are pairwise compatible
 (in the same sense as before but considered in the isotopy classes
 for arcs in $(\bfS,\bfM,\bfA)$).
\item
The $n$-tuple $\tilde{T}=(\tilde{\alpha}_i)_{i=1}^n$
of arcs in  $(\bfS,\bfM)$ yields
a labeled ideal triangulation of $(\bfS,\bfM)$.
\item
Every triangle of $T$ contains exactly one midpoint.
\end{itemize}
\end{defn}

Some examples of labeled Stokes triangulations of a pentagon are given 
in Figure \ref{fig:pent1}.
Three triangulations therein are distinct as
labeled Stokes triangulations, but
they are identical as labeled ideal triangulations by
forgetting the midpoints.

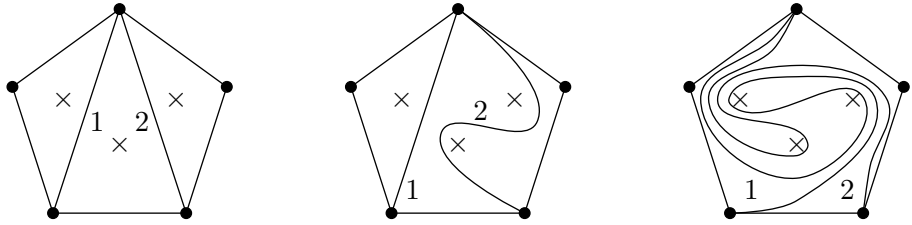
\begin{figure}
\begin{center}
\begin{pspicture}(2,3.6)(7,6.2)
\psset{unit=15mm}
%
\psset{linewidth=0.5pt}
\psset{fillstyle=solid, fillcolor=black}
\pscircle(0,4){0.053} 
\pscircle(0.95,3.31){0.053} 
\pscircle(0.59,2.19){0.053} 
\pscircle(-0.59,2.19){0.053} 
\pscircle(-0.95,3.31){0.053} 
\psset{fillstyle=none}
\psline(0,4)(0.95,3.31)
\psline(0.95,3.31)(0.59,2.19)
\psline(-0.59,2.19)(-0.95,3.31)
\psline(-0.95,3.31)(0,4)
%
\psline(0.59,2.19)(-0.59,2.19)
\psline(-0.59,2.19)(0,4)
\psline(0,4)(0.59,2.19)
%

%
\rput[c]{0}(-0.2,3){\small 1}
\rput[c]{0}(0.2,3){\small 2}
\rput[c]{0}(0,2.8){\small $\times$}
\rput[c]{0}(-0.5,3.2){\small $\times$}
\rput[c]{0}(0.5,3.2){\small $\times$}
\rput[c]{0}(3,0)
{

\psset{linewidth=0.5pt}
\psset{fillstyle=solid, fillcolor=black}
\pscircle(0,4){0.053} 
\pscircle(0.95,3.31){0.053} 
\pscircle(0.59,2.19){0.053} 
\pscircle(-0.59,2.19){0.053} 
\pscircle(-0.95,3.31){0.053} 
\psset{fillstyle=none}
\psline(0,4)(0.95,3.31)
\psline(0.95,3.31)(0.59,2.19)
\psline(0.59,2.19)(-0.59,2.19)
\psline(-0.59,2.19)(-0.95,3.31)
\psline(-0.95,3.31)(0,4)
\pscurve(0,4)(0.7,3)(-0.15,2.9)(0.59,2.19)
\psline(0,4)(-0.59,2.19)

\rput[c]{0}(-0.4,2.4){\small 1}
\rput[c]{0}(0.2,3.1){\small 2}
\rput[c]{0}(0,2.8){\small $\times$}
\rput[c]{0}(-0.5,3.2){\small $\times$}
\rput[c]{0}(0.5,3.2){\small $\times$}
}
\rput[c]{0}(6,0)
{
\psset{linewidth=0.5pt}
\psset{fillstyle=solid, fillcolor=black}
\pscircle(0,4){0.053} 
\pscircle(0.95,3.31){0.053} 
\pscircle(0.59,2.19){0.053} 
\pscircle(-0.59,2.19){0.053} 
\pscircle(-0.95,3.31){0.053} 
\psset{fillstyle=none}
\psline(0,4)(0.95,3.31)
\psline(0.95,3.31)(0.59,2.19)
\psline(0.59,2.19)(-0.59,2.19)
\psline(-0.59,2.19)(-0.95,3.31)
\psline(-0.95,3.31)(0,4)
%
\pscurve(0,4)(-0.2,3.75)(-0.85,3.2)(0,2.5)(0.6,3.2)
(-0.5,3.1)(-0.6,3.2)(0,3.4)(0.7,3.2)(0,2.3)
(-0.59,2.19)
\pscurve(0,4)(-0.2,3.65)(-0.78,3.2)(0,2.7)(0.1,2.8)
(-0.7,3.2)(0,3.5)(0.8,3.2)(0.65,2.6)
(0.59,2.19)
\rput[c]{0}(-0.4,2.4){\small 1}
\rput[c]{0}(0.45,2.4){\small 2}
\rput[c]{0}(0,2.8){\small $\times$}
\rput[c]{0}(-0.5,3.2){\small $\times$}
\rput[c]{0}(0.5,3.2){\small $\times$}
}
\end{pspicture}
\end{center}
\caption{Examples of labeled Stokes triangulations of a pentagon.}
\label{fig:pent1}
\end{figure}

For a  labeled Stokes triangulation $T=(\alpha_i)_{i=1}^n$ of $(\bfS,\bfM,\bfA)$,
an arc $\alpha_k$ is said to be {\em flippable\/} if it is not
an inner arc of $T$.
However,
unlike the case of labeled ideal triangulations of $(\bfS,\bfM)$,
the ``flip''  of a flippable arc $\alpha_k$ is not unique.
In fact, there are infinitely many choices of doing ``flip'' of a flippable arc $\alpha_k$
to obtain a new Stokes triangulation.
It is not difficult
 to see that they are generated by the following two elementary flips.

\begin{defn}
For a labeled Stokes triangulation 
 $T=(\alpha_i)_{i=1}^n$ 
of $(\bfS,\bfM,\bfA)$,
a flippable arc $\alpha_k$ of $T$, and a sign $\ve\in \{+,-\}$,
the {\em signed flip 
$T'=\mu_k^{(\ve)}(T)$ at $k$ with sign $\ve$} is a labeled Stokes triangulation
of  $(\bfS,\bfM,\bfA)$ obtained from $T$ by replacing the arc $\alpha_k$ with
the one in Figure \ref{fig:flip2}.
Namely, the new arc $\alpha'_k$ is obtained from
 $\alpha_k$ by sliding
each end point of $\alpha_k$ along an edge
of the surrounding quadrilateral of  $\alpha_k$
clockwise for $\ve=+$ and   anticlockwise for  $\ve=-$.
\end{defn}

Clearly, $\mu_{k}^{(\ve)}$ is not an involution any more,
but $\mu_{k}^{(+)}$ and $\mu_{k}^{(-)}$ are inverse to each other.
Figure \ref{fig:flip22} demonstrates how a sequence of signed flips act
on a quadrilateral.

We also introduce the {\em signed\/} pops
for labeled Stokes triangulations.

\begin{defn}
For a labeled Stokes triangulation $T$ of  $(\bfS,\bfM,\bfA)$,
a puncture $p$ inside a self-folded triangle in $T$, and a sign $\ve\in \{+,-\}$,
the {\em signed pop 
$T'=\kappa_p^{(\ve)}(T)$ at $p$ with sign $\ve$}
is a labeled Stokes triangulation of  $(\bfS,\bfM,\bfA)$
obtained from $T$ by replacing the inner  arc
of the self-folded triangle around $p$
with the one in Figure \ref{fig:pop3},
then exchanging the labels of the inner and outer arcs thereof.
Namely, if $\alpha_i$ and $\alpha_j$ is the inner and outer arcs of $T$,
then $\alpha'_i=\alpha_j$, and $\alpha'_j$ is obtained from
 $\alpha_i$
 by sliding the outside end point of $\alpha_i$ along
 the surrounding monogon
clockwise for $\ve=+$ and   anticlockwise for  $\ve=-$.
\end{defn}

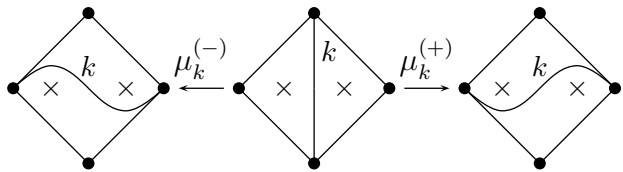
\begin{figure}
\begin{center}
\begin{pspicture}(-4,-1)(4,1.2)
\psset{linewidth=0.5pt}
\psset{fillstyle=solid, fillcolor=black}
\pscircle(0,1){0.08} 
\pscircle(0,-1){0.08} 
\pscircle(-1,0){0.08} 
\pscircle(1,0){0.08} 
\psset{fillstyle=none}
\psline(0,1)(0,0)
\psline(0,0)(0,-1)
\psline(0,1)(1,0)
\psline(0,1)(-1,0)
\psline(0,-1)(1,0)
\psline(0,-1)(-1,0)
%
\rput[c]{0}(0.2,0.5){\small $k$}
\rput[c]{0}(-0.4,0){\small $\times$}
\rput[c]{0}(0.4,0){\small $\times$}
%
\psline[arrows=->](1.2,0)(1.8,0)
\psline[arrows=<-](-1.8,0)(-1.2,-0)
\rput[c]{0}(1.5,0.4){$\mu_{k}^{(+)}$}
\rput[c]{0}(-1.5,0.4){$\mu_{k}^{(-)}$}
%
\psset{linewidth=0.5pt}
\psset{fillstyle=solid, fillcolor=black}
\pscircle(3,1){0.08} 
\pscircle(3,-1){0.08} 
\pscircle(2,0){0.08} 
\pscircle(4,0){0.08} 
\psset{fillstyle=none}
\psline(3,1)(2,0)
\psline(3,1)(4,0)
\psline(3,-1)(4,0)
\psline(3,-1)(2,0)
\pscurve(2,0)(2.5,-0.3)(3.5,0.3)(4,0)
%
\rput[c]{0}(3,0.3){\small $k$}
\rput[c]{0}(2.5,0){\small $\times$}
\rput[c]{0}(3.5,-0){\small $\times$}

\psset{linewidth=0.5pt}
\psset{fillstyle=solid, fillcolor=black}
\pscircle(-3,1){0.08} 
\pscircle(-3,-1){0.08} 
\pscircle(-2,0){0.08} 
\pscircle(-4,0){0.08} 
\psset{fillstyle=none}
\psline(-3,1)(-2,0)
\psline(-3,1)(-4,0)
\psline(-3,-1)(-4,0)
\psline(-3,-1)(-2,0)
\pscurve(-2,0)(-2.5,-0.3)(-3.5,0.3)(-4,0)
%
\rput[c]{0}(-3,0.3){\small $k$}
\rput[c]{0}(-2.5,0){\small $\times$}
\rput[c]{0}(-3.5,-0){\small $\times$}
%
%
\end{pspicture}
\caption{Signed flips of labeled Stokes triangulations.}
\label{fig:flip2}
\end{center}
\end{figure}

Again, $\kappa_{p}^{(\ve)}$ is not an involution any more,
but $\kappa_{p}^{(+)}$ and $\kappa_{p}^{(-)}$ are inverse to each other.
Figure \ref{fig:pop33} demonstrates how a sequence of signed pops act.

The signed flips and the signed pops of Stokes triangulations
are supposed to be the counterparts
of the signed mutations and the signed pops of
extended seeds
in Section \ref{subsec:local}.
To be more precise,
we have the following
 conjecture,
  which naturally extends
  Corollary \ref{cor:surface1} and
Proposition
\ref{prop:bij1}.
(We thank Yuuki Hirako for the discussion.)

\begin{conj}
\label{conj:bij1}
Let $T^0$ be any labeled Stokes triangulation of $(\bfS,\bfM,\bfA)$,
and let $B^0=B(T^0)$. 
\par
(i)
Let $\mathrm{LST}(T^0)$ be the set of all labeled Stokes triangulations obtained from the initial
labeled Stokes triangulation $T^0$ of $(\bfS,\bfM,\bfA)$ by  sequences
of signed flips and signed pops,
and let $\mathrm{Seed}(B^0,x^0,y^0,\tilde{y}^0)$ be
the set of all labeled extended seeds which
are mutation-equivalent to the initial one 
$(B^0,x^0,y^0,\tilde{y}^0)$ by
signed mutations and signed pops.
Then, there is a bijection $\Psi:\mathrm{LST}(T^0)
\rightarrow \mathrm{Seed}(B^0,x^0,y^0,\tilde{y}^0)$ such that
$\Psi(T^0)=(B^0,x^0,y^0,\tilde{y}^0)$ and 
$\Psi$ commutes with 
signed flips/mutations and signed pops. 
\par
(ii)
Assume that $(\bfS,\bfM)$ is generic and not a closed surface with exactly one puncture.
Then, $\mathrm{LST}(T^0)$
consists of
all labeled Stokes triangulations of $(\bfS,\bfM,\bfA)$.
\end{conj}

In the rest of paper we do not rely on this conjecture,
but having it in mind will be useful.

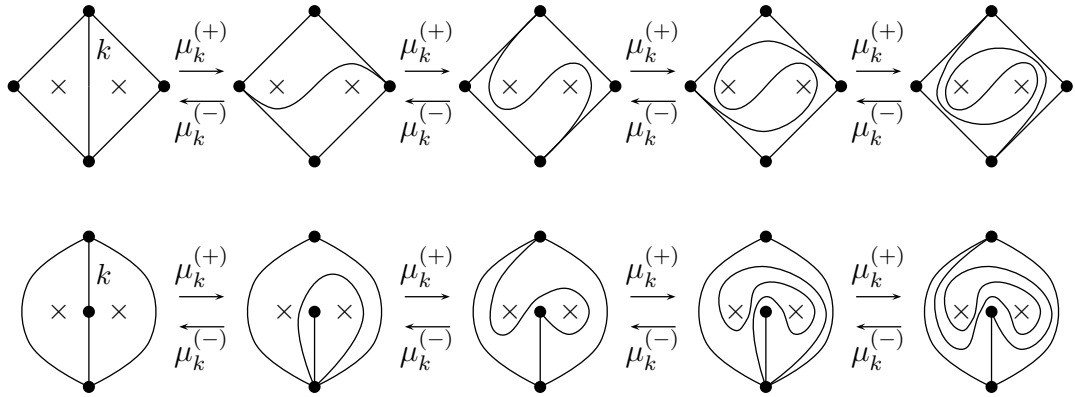
\begin{figure}
\begin{center}
\begin{pspicture}(-1,-4)(13,1.2)
\psset{linewidth=0.5pt}
\psset{fillstyle=solid, fillcolor=black}
\pscircle(0,1){0.08} 
\pscircle(0,-1){0.08} 
\pscircle(-1,0){0.08} 
\pscircle(1,0){0.08} 
\psset{fillstyle=none}
\psline(0,1)(0,0)
\psline(0,0)(0,-1)
\psline(0,1)(1,0)
\psline(0,1)(-1,0)
\psline(0,-1)(1,0)
\psline(0,-1)(-1,0)
\rput[c]{0}(0.2,0.5){\small $k$}
\rput[c]{0}(-0.4,0){\small $\times$}
\rput[c]{0}(0.4,0){\small $\times$}
\psline[arrows=->](1.2,0.2)(1.8,0.2)
\psline[arrows=<-](1.2,-0.2)(1.8,-0.2)
\rput[c]{0}(1.5,0.6){$\mu_{k}^{(+)}$}
\rput[c]{0}(1.5,-0.5){$\mu_{k}^{(-)}$}
\psset{linewidth=0.5pt}
\psset{fillstyle=solid, fillcolor=black}
\pscircle(3,1){0.08} 
\pscircle(3,-1){0.08} 
\pscircle(2,0){0.08} 
\pscircle(4,0){0.08} 
\psset{fillstyle=none}
\psline(3,1)(2,0)
\psline(3,1)(4,0)
\psline(3,-1)(4,0)
\psline(3,-1)(2,0)
\pscurve(2,0)(2.5,-0.3)(3.5,0.3)(4,0)
\rput[c]{0}(2.5,0){\small $\times$}
\rput[c]{0}(3.5,-0){\small $\times$}
\psline[arrows=->](4.2,0.2)(4.8,0.2)
\psline[arrows=<-](4.2,-0.2)(4.8,-0.2)
\rput[c]{0}(4.5,0.6){$\mu_{k}^{(+)}$}
\rput[c]{0}(4.5,-0.5){$\mu_{k}^{(-)}$}
\rput[c]{0}(6,0)
{
\psset{linewidth=0.5pt}
\psset{fillstyle=solid, fillcolor=black}
\pscircle(0,1){0.08} 
\pscircle(0,-1){0.08} 
\pscircle(-1,0){0.08} 
\pscircle(1,0){0.08} 
\psset{fillstyle=none}
%
\psline(0,1)(1,0)
\psline(0,1)(-1,0)
\psline(0,-1)(1,0)
\psline(0,-1)(-1,0)
\pscurve(0,1)(-0.68,0)(-0.5,-0.3)(0.5,0.3)(0.68,0)(0,-1)
%
\rput[c]{0}(-0.4,0){\small $\times$}
\rput[c]{0}(0.4,0){\small $\times$}
\psline[arrows=->](1.2,0.2)(1.8,0.2)
\psline[arrows=<-](1.2,-0.2)(1.8,-0.2)
\rput[c]{0}(1.5,0.6){$\mu_{k}^{(+)}$}
\rput[c]{0}(1.5,-0.5){$\mu_{k}^{(-)}$}
\psset{linewidth=0.5pt}
\psset{fillstyle=solid, fillcolor=black}
\pscircle(3,1){0.08} 
\pscircle(3,-1){0.08} 
\pscircle(2,0){0.08} 
\pscircle(4,0){0.08} 
\psset{fillstyle=none}
\psline(3,1)(2,0)
\psline(3,1)(4,0)
\psline(3,-1)(4,0)
\psline(3,-1)(2,0)
\pscurve(4,0)(3,0.6)(2.32,0)(2.5,-0.3)(3.5,0.3)(3.68,0)(3,-0.6)(2,0)
\rput[c]{0}(2.5,0){\small $\times$}
\rput[c]{0}(3.5,-0){\small $\times$}
\psline[arrows=->](4.2,0.2)(4.8,0.2)
\psline[arrows=<-](4.2,-0.2)(4.8,-0.2)
\rput[c]{0}(4.5,0.6){$\mu_{k}^{(+)}$}
\rput[c]{0}(4.5,-0.5){$\mu_{k}^{(-)}$}
}
\rput[c]{0}(12,0)
{
\psset{linewidth=0.5pt}
\psset{fillstyle=solid, fillcolor=black}
\pscircle(0,1){0.08} 
\pscircle(0,-1){0.08} 
\pscircle(-1,0){0.08} 
\pscircle(1,0){0.08} 
\psset{fillstyle=none}
%
\psline(0,1)(1,0)
\psline(0,1)(-1,0)
\psline(0,-1)(1,0)
\psline(0,-1)(-1,0)
\pscurve(0,-1)(0.5,-0.43)(0.6,-0.28)(0.7,0.2)(0,0.5)(-0.6,0)
(-0.5,-0.3)(0.5,0.3)(0.6,0)(-0,-0.5)(-0.7,-0.2)(-0.6,0.28)(-0.5,0.43)(0,1)
\rput[c]{0}(-0.4,0){\small $\times$}
\rput[c]{0}(0.4,0){\small $\times$}
}
%
\rput[c]{0}(0,-3)
{
%
\psset{linewidth=0.5pt}
\psset{fillstyle=solid, fillcolor=black}
\pscircle(0,1){0.08} 
\pscircle(0,-1){0.08} 
\pscircle(0,0){0.08} 
\psset{fillstyle=none}
\psline(0,1)(0,0)
\psline(0,0)(0,-1)
\pscurve(0,1)(0.8,0.4)(0.8,-0.4)(0,-1)
\pscurve(0,1)(-0.8,0.4)(-0.8,-0.4)(0,-1)
%
\rput[c]{0}(0.2,0.5){\small $k$}
\rput[c]{0}(-0.4,0){\small $\times$}
\rput[c]{0}(0.4,0){\small $\times$}
%
%
\psline[arrows=->](1.2,0.2)(1.8,0.2)
\psline[arrows=<-](1.2,-0.2)(1.8,-0.2)
\rput[c]{0}(1.5,0.6){$\mu_{k}^{(+)}$}
\rput[c]{0}(1.5,-0.5){$\mu_{k}^{(-)}$}
\psset{linewidth=0.5pt}
\psset{fillstyle=solid, fillcolor=black}
\pscircle(3,1){0.08} 
\pscircle(3,-1){0.08} 
\pscircle(3,0){0.08} 
\psset{fillstyle=none}
\psline(3,0)(3,-1)
\pscurve(3,1)(3.8,0.4)(3.8,-0.4)(3,-1)
\pscurve(3,1)(2.2,0.4)(2.2,-0.4)(3,-1)
\pscurve(3,-1)(2.8,0)(3.25,0.5)(3.65,0)(3,-1)
%
\rput[c]{0}(2.6,0){\small $\times$}
\rput[c]{0}(3.4,0){\small $\times$}
\psline[arrows=->](4.2,0.2)(4.8,0.2)
\psline[arrows=<-](4.2,-0.2)(4.8,-0.2)
\rput[c]{0}(4.5,0.6){$\mu_{k}^{(+)}$}
\rput[c]{0}(4.5,-0.5){$\mu_{k}^{(-)}$}
\rput[c]{0}(6,0)
{
\psset{linewidth=0.5pt}
\psset{fillstyle=solid, fillcolor=black}
\pscircle(0,1){0.08} 
\pscircle(0,-1){0.08} 
\pscircle(0,0){0.08} 
\psset{fillstyle=none}
%
\psline(0,0)(0,-1)
\pscurve(0,1)(0.8,0.4)(0.8,-0.4)(0,-1)
\pscurve(0,1)(-0.8,0.4)(-0.8,-0.4)(0,-1)
\pscurve(0,1)(-0.6,0.3)(-0.55,-0.3)(0,0.2)(0.4,0.3)(0.65,0)(0.4,-0.3)(0,0)
%
\rput[c]{0}(-0.4,0){\small $\times$}
\rput[c]{0}(0.4,0){\small $\times$}
%
%
\psline[arrows=->](1.2,0.2)(1.8,0.2)
\psline[arrows=<-](1.2,-0.2)(1.8,-0.2)
\rput[c]{0}(1.5,0.6){$\mu_{k}^{(+)}$}
\rput[c]{0}(1.5,-0.5){$\mu_{k}^{(-)}$}
\psset{linewidth=0.5pt}
\psset{fillstyle=solid, fillcolor=black}
\pscircle(3,1){0.08} 
\pscircle(3,-1){0.08} 
\pscircle(3,0){0.08} 
\psset{fillstyle=none}
\psline(3,0)(3,-1)
\pscurve(3,1)(3.8,0.4)(3.8,-0.4)(3,-1)
\pscurve(3,1)(2.2,0.4)(2.2,-0.4)(3,-1)
\pscurve(3,-1)(3.8,0)(3,0.6)(2.4,0.3)(2.6,-0.25)(2.75,-0.1)(2.9,0.35)
(3.4,0.3)(3.65,0)(3.4,-0.3)(3.2,0.05)(3,0.2)(2.85,0)(3,-1)
%
\rput[c]{0}(2.6,0){\small $\times$}
\rput[c]{0}(3.4,0){\small $\times$}
\psline[arrows=->](4.2,0.2)(4.8,0.2)
\psline[arrows=<-](4.2,-0.2)(4.8,-0.2)
\rput[c]{0}(4.5,0.6){$\mu_{k}^{(+)}$}
\rput[c]{0}(4.5,-0.5){$\mu_{k}^{(-)}$}
}
\rput[c]{0}(12,0)
{
\psset{linewidth=0.5pt}
\psset{fillstyle=solid, fillcolor=black}
\pscircle(0,1){0.08} 
\pscircle(0,-1){0.08} 
\pscircle(0,0){0.08} 
\psset{fillstyle=none}
%
\psline(0,0)(0,-1)
\pscurve(0,1)(0.8,0.4)(0.8,-0.4)(0,-1)
\pscurve(0,1)(-0.8,0.4)(-0.8,-0.4)(0,-1)
\pscurve(0,0)(0.4,-0.5)(0.8,0)(0,0.6)(-0.6,0.3)(-0.4,-0.25)(-0.25,-0.1)(-0.1,0.35)
(0.4,0.3)(0.65,0)(0.4,-0.3)(0.2,0.05)(0,0.2)(-0.15,0)(-0.4,-0.5)(-0.75,0.2)(0,1)
%
\rput[c]{0}(-0.4,0){\small $\times$}
\rput[c]{0}(0.4,0){\small $\times$}
}
}
\end{pspicture}
\caption{Examples of sequences of signed flips.}
\label{fig:flip22}
\end{center}
\end{figure}

\begin{figure}
\begin{center}
\begin{pspicture}(-2,-1)(2,1.2)
%
\psset{linewidth=0.5pt}
\psset{fillstyle=solid, fillcolor=black}
\pscircle(0,-1){0.08} 
\pscircle(0,0){0.08} 
\psset{fillstyle=none}
\pscurve(0,-1)(-0.7,0.4)(0,1)(0.7,0.4)(0,-1)
\psline(0,0)(0,-1)
\rput[c]{0}(-0.25,-0.1){\small $p$}
\rput[c]{0}(0,0.4){\small $\times$}
\rput[c]{0}(0.8,0.9){\small $j$}
\rput[c]{0}(0.12,-0.4){\small $i$}
\psline[arrows=->](1.2,0)(1.8,0)
\psline[arrows=->](-1.2,-0)(-1.8,-0)
\rput[c]{0}(1.5,0.4){$\kappa_{p}^{(+)}$}
\rput[c]{0}(-1.5,0.4){$\kappa_{p}^{(-)}$}
%
\rput[c]{0}(3,0)
{
\psset{linewidth=0.5pt}
\psset{fillstyle=solid, fillcolor=black}
\pscircle(0,-1){0.08} 
\pscircle(0,0){0.08} 
\psset{fillstyle=none}
\pscurve(0,-1)(-0.7,0.4)(0,1)(0.7,0.4)(0,-1)
\pscurve(0,-1)(0.3,0.35)(0,0.7)(-0.3,0.35)(0,0)
\rput[c]{0}(-0.25,-0.1){\small $p$}
\rput[c]{0}(0,0.4){\small $\times$}
\rput[c]{0}(0.8,0.9){\small $i$}
\rput[c]{0}(0,-0.4){\small $j$}
}
\rput[c]{0}(-3,0)
{
\psset{linewidth=0.5pt}
\psset{fillstyle=solid, fillcolor=black}
\pscircle(0,-1){0.08} 
\pscircle(0,0){0.08} 
\psset{fillstyle=none}
\pscurve(0,-1)(-0.7,0.4)(0,1)(0.7,0.4)(0,-1)
\pscurve(0,-1)(-0.3,0.35)(0,0.7)(0.3,0.35)(0,0)
\rput[c]{0}(-0.25,-0.1){\small $p$}
\rput[c]{0}(0,0.4){\small $\times$}
\rput[c]{0}(0.8,0.9){\small $i$}
\rput[c]{0}(0,-0.4){\small $j$}
}
\end{pspicture}
\end{center}
\caption{Signed pops of labeled Stokes triangulations.}
\label{fig:pop3}
\end{figure}
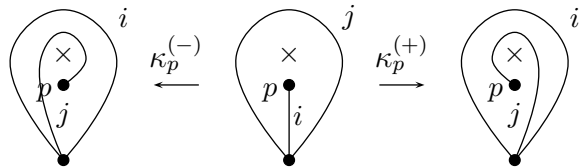

\begin{figure}
\begin{center}
\begin{pspicture}(0,-1)(10,1.2)
\psset{linewidth=0.5pt}
\psset{fillstyle=solid, fillcolor=black}
\pscircle(0,-1){0.08} 
\pscircle(0,0){0.08} 
\psset{fillstyle=none}
\pscurve(0,-1)(-0.7,0.4)(0,1)(0.7,0.4)(0,-1)
\psline(0,0)(0,-1)
\rput[c]{0}(-0.25,-0.1){\small $p$}
\rput[c]{0}(0,0.4){\small $\times$}
\rput[c]{0}(0.8,0.9){\small $j$}
\rput[c]{0}(0.12,-0.4){\small $i$}
\psline[arrows=->](1.2,0.2)(1.8,0.2)
\psline[arrows=<-](1.2,-0.2)(1.8,-0.2)
\rput[c]{0}(1.5,0.6){$\kappa_{p}^{(+)}$}
\rput[c]{0}(1.5,-0.5){$\kappa_{p}^{(-)}$}
%
%
\rput[c]{0}(3,0)
{
\psset{linewidth=0.5pt}
\psset{fillstyle=solid, fillcolor=black}
\pscircle(0,-1){0.08} 
\pscircle(0,0){0.08} 
\psset{fillstyle=none}
\pscurve(0,-1)(-0.7,0.4)(0,1)(0.7,0.4)(0,-1)
\pscurve(0,-1)(0.3,0.35)(0,0.7)(-0.3,0.35)(0,0)
%
\rput[c]{0}(0,0.4){\small $\times$}
\rput[c]{0}(0.8,0.9){\small $i$}
\rput[c]{0}(0,-0.4){\small $j$}
\psline[arrows=->](1.2,0.2)(1.8,0.2)
\psline[arrows=<-](1.2,-0.2)(1.8,-0.2)
\rput[c]{0}(1.5,0.6){$\kappa_{p}^{(+)}$}
\rput[c]{0}(1.5,-0.5){$\kappa_{p}^{(-)}$}
}
%
%
\rput[c]{0}(6,0)
{
\psset{linewidth=0.5pt}
\psset{fillstyle=solid, fillcolor=black}
\pscircle(0,-1){0.08} 
\pscircle(0,0){0.08} 
\psset{fillstyle=none}
\pscurve(0,-1)(-0.7,0.4)(0,1)(0.7,0.4)(0,-1)
\pscurve(0,-1)(0.4,0.35)(0,0.8)(-0.4,0.35)(0,-0.2)(0.2,0.2)(0,0.6)(-0.2,0.2)(0,0)
%
\rput[c]{0}(0,0.4){\small $\times$}
\rput[c]{0}(0.8,0.9){\small $j$}
\rput[c]{0}(0,-0.4){\small $i$}
\psline[arrows=->](1.2,0.2)(1.8,0.2)
\psline[arrows=<-](1.2,-0.2)(1.8,-0.2)
\rput[c]{0}(1.5,0.6){$\kappa_{p}^{(+)}$}
\rput[c]{0}(1.5,-0.5){$\kappa_{p}^{(-)}$}
}
%
%
\rput[c]{0}(9,0)
{
\psset{linewidth=0.5pt}
\psset{fillstyle=solid, fillcolor=black}
\pscircle(0,-1){0.08} 
\pscircle(0,0){0.08} 
\psset{fillstyle=none}
\pscurve(0,-1)(-0.7,0.4)(0,1)(0.7,0.4)(0,-1)
\pscurve(0,-1)(0.5,0.35)(0,0.85)(-0.5,0.35)(0,-0.3)
(0.35,0.2)(0,0.7)(-0.35,0.2)(0,-0.15)(0.2,0.2)(0,0.55)(-0.2,0.2)(0,0)
%
\rput[c]{0}(0,0.4){\small $\times$}
\rput[c]{0}(0.8,0.9){\small $i$}
\rput[c]{0}(0,-0.5){\small $j$}
}
\end{pspicture}
\end{center}
\caption{Examples of sequences of signed pops.}
\label{fig:pop33}
\end{figure}
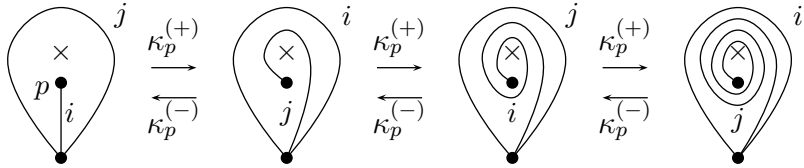

\begin{rem}
The notions of Stokes triangulations and signed flips (for surfaces without
punctures) also appear recently in \cite{Qiu14}
in the study of the spherical twists on 3-Calabi-Yau categories.
\end{rem}

\subsection{Construction of Stokes triangulation from Stokes graph}

Let $G=G(\phi)$ be the
 Stokes graph of a quadratic differential $\phi$ on a compact
 Riemann surface $\Sigma$.
 The classification of the Stokes regions of $G$ was 
 given  in Section \ref{section:Stokes-graphs}
 under Assumptions \ref{assumption:zeros and poles} 
and \ref{assumption:at-most-one-saddle}.
To be more specific,
 if $G$ is {\em saddle-free},
 a degenerate ring domain does not appear.
Thus, any Stokes region of $\phi$ falls into  one of the following three patterns,
which are depicted in
Figure \ref{fig:Stokes1} \cite[Section 3.4]{Bridgeland13},
where we split horizontal strips in two patterns:
  \begin{itemize}
   \item[(a).] {\em Regular horizontal strip.} This is a generic case.
   The Stokes region is inside a quadrilateral with two simple zeros $q_1$, $q_2$ and
   two poles $p_1$, $p_2$ of orders $m_1$, $m_2\geq 2$.
   The poles $p_1$, $p_2$ may coincide.
   \item[(b).] {\em Degenerate horizontal strip.} This may be regarded as
   the folding of the two edges $q_1p_2$ and $q_2p_2$ in the case (a).
   The orders of poles $p_1$ and  $p_2$ are $m_1\geq 2$ and $m_2=2$,
   respectively.
    \item[(c).] {\em Half plane.} This occurs only for a pole $p_1$ with order $m_1\geq 3$.
  \end{itemize}
Note that the pictures 
in
Figure \ref{fig:Stokes1}
are schematic ones,
and actual  trajectories entering in a pole should obey the local property in Section
\ref{section:Stokes-graphs},
   depending on the order of the pole.
  The dashed arc is a representative of the isotopy class of trajectories 
   inside the Stokes region.

\begin{figure}
\begin{center}
\begin{pspicture}(-2.2,-2.2)(12.2,1.5)
\psset{linewidth=0.5pt}
\psset{fillstyle=solid, fillcolor=black}
\pscircle(-1.5,0){0.08} 
\pscircle(1.5,0){0.08} 
\psset{fillstyle=none}
\psline(0,1.2)(0,0.6)
\psline(0,-1.2)(0,-0.6)
\psline[linestyle=dashed](1.5,0)(-1.5,0)
%
\pscurve(1.5,0)(1.4,0)(0.6,0.25)(0,0.6)
\pscurve(-1.5,0)(-1.4,0)(-0.6,0.25)(0,0.6)
\pscurve(1.5,0)(1.4,0)(0.6,-0.25)(0,-0.6)
\pscurve(-1.5,0)(-1.4,0)(-0.6,-0.25)(0,-0.6)
\rput[c]{0}(0,-0.6){\small $\times$}
\rput[c]{0}(0,0.6){\small $\times$}
\rput[c]{0}(-1.5,-0.4){\small $p_1$}
\rput[c]{0}(1.5,-0.4){\small $p_2$}
\rput[c]{0}(0.4,0.7){\small $q_1$}
\rput[c]{0}(0.4,-0.7){\small $q_2$}
\rput[c]{0}(0,-2){\footnotesize (a) regular horizontal strip}
\psset{fillstyle=solid, fillcolor=black}
\pscircle(5,-1){0.08} 
\pscircle(5,0.6){0.08} 
\psset{fillstyle=none}
\pscurve(5,-1)(5.3,-0.6)(5.5,-0.2)(5.6,0.8)(5,1.2)(4.4,0.8)(4.4,0.2)(4.6,-0.2)
\pscurve(5,-1)(4.8,-0.8)(4.68,-0.6)(4.6,-0.2)
\pscurve(4.6,-0.2)(5,0)(5.25,0.6)(5,0.78)(4.85,0.6)(5,0.48)(5.1,0.6)(5,0.67)
\pscurve[linestyle=dashed](5,-1)(5.4,0)(5.4,0.6)(5,0.85)(4.78,0.6)(5,0.42)%
(5.15,0.6)(5,0.73)(4.9,0.6)(5,0.52)
\rput[c]{0}(4.6,-0.2){\small $\times$}
\rput[c]{0}(5,-1.4){\small $p_1$}
\rput[c]{0}(4.6,0.6){\small $p_2$}
\rput[c]{0}(4.3,-0.2){\small $q_1$}
\rput[c]{0}(5,-2){\footnotesize (b) degenerate horizontal strip}
\psset{fillstyle=solid, fillcolor=black}
\pscircle(10,-1){0.08} 
\psset{fillstyle=none}
\psline(10,1.2)(10,0.6)
\pscurve(10,-1)(10.4,-0.6)(10.6,0.2)(10,0.6)
\pscurve(10,-1)(9.8,-0.8)(9.4,0.2)(10,0.6)
\pscurve[linestyle=dashed](10,-1)(10.2,-0.78)(10.3,-0.25)(10,-0.05)(9.7,-0.25)(9.8,-0.78)(10,-1)
\rput[c]{0}(10,0.6){\small $\times$}
\rput[c]{0}(10,-1.4){\small $p_1$}
\rput[c]{0}(10.4,0.7){\small $q_1$}
\rput[c]{0}(10,-2){\footnotesize {(c) half plane}}
\end{pspicture}
\end{center}
\caption{Patterns of Stokes regions for saddle-free Stokes graph.
The dashed arc is a representative of the isotopy class of trajectories
inside the Stokes region.}
\label{fig:Stokes1}
\end{figure}
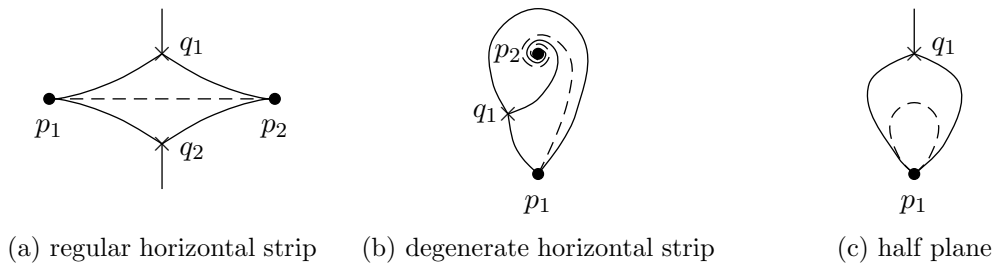

Let us introduce labeling of a saddle-free Stokes graph.
\begin{defn} 
Let $G=G(\phi)$ be the  Stokes graph of a
saddle-free quadratic differential $\phi$ on $\Sigma$.
We put labels, say, $D_1$, \dots, $D_n$ for the Stokes regions of $G$
which are (regular or degenerate) horizontal strips.
It is called a {\em labeled Stokes graph},
and denoted by the same symbol $G$.
(We do not put labels for the Stokes regions which are {\em half planes}.)
\par
\end{defn}

To each labeled  Stokes graph $G=G(\phi)$,
we will assign a  labeled Stokes  triangulation $T$ 
of a certain bordered  surface $(\bfS,\bfM,\bfA)$.
We  follow the construction of
an ideal triangulation by \cite{Kawai05,Gaiotto09,Bridgeland13}, 
which is automatically upgraded to a labeled  Stokes triangulation.

{\em Step 1. Construction of  the bordered surface $(\bfS,\bfM,\bfA)$.}
Let $p_1,\dots,p_r$ be the double poles of $\phi$,
and $p'_1,\dots,p'_s$ be the poles of orders $m_1,\dots, m_s \geq 3$  of $\phi$.
Then, the bordered surface $\bfS=\bfS(\phi)$ is obtained from $\Sigma$ by cutting out a small hole
around each $p'_i$ such that
no other poles and zeros of $\phi$ are removed.
Let $B_i$ denote the resulting boundary component for $p'_i$.
To each $B_i$ we put $m_i-2$ marked points.
Then, the set of the marked points $\bfM=\bfM(\phi)$ consists of
these marked points at boundaries {\em and\/} the double poles
$p_1$,\dots, $p_r$. (Thus, $p_1$,\dots, $p_r$ are  the punctures of $(\bfS,\bfM)$.)
Also,  the set  of  the midpoints $\bfA=\bfA(\phi)$ 
 are given by the  zeros of $\phi$.
 
{\em Step 2. Construction of  the labeled Stokes triangulation $T$ of $(\bfS,\bfM,\bfA)$.}
For each Stokes region $D$
which is a (regular or degenerate) {\em horizontal strip},
choose any representative $\beta$ 
of trajectories in $D$ up to isotopy.
We identify  $\beta$ with an {\em arc\/ $\alpha$ 
of $(\bfS,\bfM,\bfA)$}
 in the following way.
If the poles
$p_1$ and $p_2$ in Figure \ref{fig:Stokes1} (a) or (b) are double poles,
then the arc $\alpha$ is the one connecting $p_1$ and $p_2$ therein.
If $p_i$ is a pole of order $m\geq 3$, we do the following modification:
We identify the $m-2$ marked points at the boundary component for $p_i$
with the $m-2$ tangent directions of trajectories around $p_i$,
 keeping the clockwise order.
Then the arc $\alpha$ ends at
the marked point at the boundary component for $p_i$ corresponding to the
tangent direction of $\beta'$ as in Figure \ref{fig:boundary1}.
\par

Let us collect the resulting arcs $\alpha_1,\dots,\alpha_n$ corresponding to the
Stokes regions $D_1$, \dots, $D_n$ which are horizontal strips.
Let us show that $T=T(\phi)=(\alpha_i)_{i=1}^n$
is a labeled Stokes triangulation of $(\bfS,\bfM,\bfA)$.

\begin{prop}[{\cite[Lemma 10.1]{Bridgeland13}}]
\label{prop:triang1}
The $n$-tuple $\tilde{T}=(\tilde{\alpha}_i)_{i=1}^n$
of arcs in  $(\bfS,\bfM)$ is a labeled ideal triangulation of $(\bfS,\bfM)$.
\end{prop}
Note that the arcs corresponding to the {\em degenerate\/} horizontal strips
are the {\em inner arcs\/} in $T$.
See Figure \ref{fig:degenerate1}.

\begin{rem}
For each Stokes region $D$ which is a {\em half plane},
we can naturally identify a  representative $\beta'$ 
of trajectories in $D$ with the {\em edge\/} $\delta$ connecting
the two marked points at the boundary component for $p_i$ corresponding to the
tangent directions of the both ends of $\beta'$ as in Figure \ref{fig:boundary1}.
\end{rem}


\begin{figure}
\begin{center}
\begin{pspicture}(-1.5,-1.5)(7,1.5)
\psset{linewidth=0.5pt}
\psset{fillstyle=solid, fillcolor=black}
\pscircle(0,0){0.08} 
\psset{fillstyle=none}
\psline[linestyle=dashed](1.5,0)(-1.5,0)
\psline[linestyle=dashed](0,1.5)(0,-1.5)
%
\pscurve(0,0)(0.08,0.64)(0.4,0.86)(0.76,0.76)(0.86,0.4)(0.64,0.08)(0,0)
\pscurve(0,0)(-0.3,-0.01)(-1,-0.23)(-1.5,-0.7)
%
%
%
%
\rput[c](-0.6,-0.4){\small $\beta$}
\rput[c](1.2,0.4){\small $\beta'$}
\psset{linewidth=0.5pt}
\psset{fillstyle=solid, fillcolor=lightgray,linestyle=solid}
\pscircle(5,0){0.4} 
\psset{fillstyle=none,linestyle=solid}
\psset{linewidth=0.5pt}
\psset{fillstyle=solid, fillcolor=black}
\pscircle(5.4,0){0.08} 
\pscircle(5,0.4){0.08} 
\pscircle(5,-0.4){0.08} 
\pscircle(4.6,0){0.08} 
\psset{fillstyle=none}
\psset{linewidth=1.5pt}
\psarc(5,0){0.4}{0}{90} 
\psset{linewidth=0.5pt}
\pscurve(4.6,0)(4,-0.27)(3.5,-0.7)
\rput[c](4.4,-0.4){\small $\alpha$}
\rput[c](5.6,0.4){\small $\delta$}
\end{pspicture}
\end{center}
\caption{Identification of trajectories
  around a pole of order $m\geq 3$ with arcs and edges. The case $m=6$ is shown.}
\label{fig:boundary1}
\end{figure}
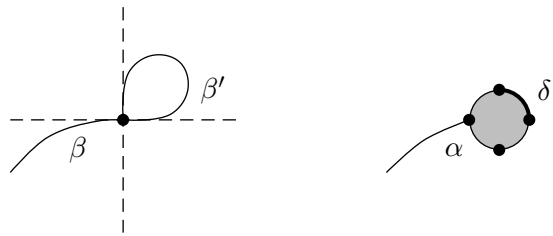

\begin{figure}
\begin{center}
\begin{pspicture}(-1,-1)(7,1.2)
%
\psset{linewidth=0.5pt}
\psset{fillstyle=solid, fillcolor=black}
\pscircle(6,1){0.08} 
\pscircle(6,-1){0.08} 
\pscircle(6,0){0.08} 
\psset{fillstyle=none}
\pscurve(6,-1)(5.6,0)(6,0.4)(6.4,0)(6,-1)
\pscurve(6,0)(5.8,-0.5)(6,-1)
\pscurve(6,1)(6.6,0.45)(6.6,-0.45)(6,-1)
\pscurve(6,1)(5.4,0.45)(5.4,-0.45)(6,-1)
\psset{fillstyle=solid, fillcolor=black}
\pscircle(0,1){0.08} 
\pscircle(0,-1){0.08} 
\pscircle(0,0){0.08} 
\psset{fillstyle=none}
\psline(0,1)(0,0.6)
\pscurve(0,-1)(0.5,-0.4)(0.45,0.35)(0,0.6)
\pscurve(0,-1)(-0.5,-0.4)(-0.45,0.35)(0,0.6)
\pscurve(0,-1)(0,-0.4)
\psline(0,-1)(0,-0.4)
\pscurve(0,-0.4)(-0.2,-0.2)(-0.2,0.1)(0,0.2)(0.15,0)(0,-0.1)(-0.05,-0.05)(0,0)
\pscurve(0,-0.4)(0.32,-0.2)(0.35,0.2)(0,0.4)(-0.35,0.2)(-0.35,-0.3)(0,-1)
\rput[c]{0}(0,0.6){\small $\times$}
\rput[c]{0}(0,-0.4){\small $\times$}
%
%
\rput[c]{0}(6,0.6){\small $\times$}
\rput[c]{0}(6,-0.4){\small $\times$}
\end{pspicture}
\end{center}
\caption{Degenerate horizontal strip (left) and inner arc (right)} 
\label{fig:degenerate1}
\end{figure}
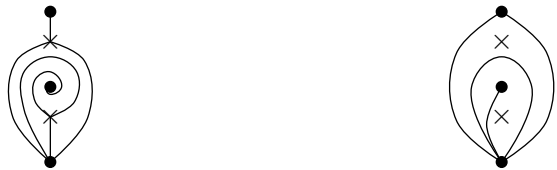

\begin{lem}
Every triangle of $T$ contains exactly one midpoint.
\end{lem}
\begin{proof}
The claim is divided into the following two claims.
\begin{itemize}
\item
Every triangle in $T$ contains at most one zero.
\item
Every triangle in $T$ contains at least one zero.
\end{itemize}
Both claims follow from the classification of Stokes regions
in Figure \ref{fig:Stokes1} and the construction of the triangulation $T$.
\end{proof}

In summary, we have the desired extension of Proposition \ref{prop:triang1}.

\begin{prop}
The $n$-tuple $T$
 is a labeled Stokes triangulation of $(\bfS,\bfM,\bfA)$.
\end{prop}

We call   $T=T(\phi)$ a {\em labeled Stokes triangulation
of $(\bfS(\phi),\bfM(\phi),\bfA(\phi))$ associated
with $G=G(\phi)$}.

\begin{rem}
 Assumption  \ref{assumption:zeros and poles} for $\phi$ automatically guarantees Assumption \ref{ass:bs1} for the corresponding
 bordered surface
  $(\bfS(\phi), \bfM(\phi))$,
except for the following cases:
\begin{itemize}
\item a once punctured monogon,
\item an unpunctured triangle.
\end{itemize}
These exceptional cases are trivial from the cluster algebraic point of view,
and we do not mind this discrepancy seriously.
(However, they are basic and important examples in the exact WKB analysis.)
\end{rem}

\begin{ex}[Pentagon relation (3)]
\label{ex:pentagon3}
Let us take $\Sigma=\bbP^1$ and  the quadratic differential
$\phi=Q(z)dz$ with $Q_0(z)=z(z+1)(z+i)$ 
in Figure \ref{fig:examples-of-Stokes-graphs} (b).
The quadratic differential $\phi$ has zeros at $-i$, $0$, and
$-1$.
It has also a pole $p_1=\infty$ with order $7$.
Thus, there are five tangent directions at $p$.
The labeled Stokes graph of $\phi$ in $\bbC$ is drawn schematically 
(i.e., up to isotopy and rotation)  in Figure \ref{fig:pentagon2} (a).
Then, we have the associated
labeled Stokes triangulation as 
in Figure \ref{fig:pentagon2} (b),
where
 the boundary of the pentagon is
 identified with
the boundary of the hole cut out around the pole $p=\infty$. 

\begin{figure}
\begin{center}
\begin{pspicture}(-0.4,-0.9)(7,2.6)
\psset{linewidth=0.5pt}
\psset{fillstyle=solid, fillcolor=black}
\psset{fillstyle=none}
\psline(1,2.4)(1,0.5)
\psline(1.6,1.4)(2.4,1.4)
\psline(0.4,1.4)(-0.4,1.4)
\pscurve(1.8,-0.2)(1.4,0.2)(1,0.5)
\pscurve(0.2,-0.2)(0.6,0.2)(1,0.5)
\pscurve(1.1,2.4)(1.3,1.8)(1.6,1.4)
\pscurve(0.9,2.4)(0.7,1.8)(0.4,1.4)
\pscurve(1.6,1.4)(1.7,0.5)(1.9,-0.2)
\pscurve(0.4,1.4)(0.3,0.5)(0.1,-0.2)
\pscurve[linestyle=dashed](1.05,2.4)(1.3,1)(1.85,-0.2)
\pscurve[linestyle=dashed](0.95,2.4)(0.7,1)(0.15,-0.2)
\pscurve[linestyle=dashed](2.5,1.1)(2.1,0.6)(2,0)
\pscurve[linestyle=dashed](-0.5,1.1)(-0.1,0.6)(0,0)
\pscurve[linestyle=dashed](0.45,-0.4)(1,-0.25)(1.55,-0.4)
\pscurve[linestyle=dashed](1.3,2.4)(1.75,1.95)(2.4,1.7)
\pscurve[linestyle=dashed](0.7,2.4)(0.25,1.95)(-0.4,1.7)
%
\rput[c]{0}(0.4,1.4){\small $\times$}
\rput[c]{0}(1,0.5){\small $\times$}
\rput[c]{0}(1.6,1.4){\small $\times$}
%
%
\rput[c]{0}(0.75,0.95){\small $D_1$}
\rput[c]{0}(1.28,0.95){\small $D_2$}
\rput[c]{0}(1,-0.8){\small (a)}
%
\psset{linewidth=0.5pt}
\psset{fillstyle=solid, fillcolor=black}
\pscircle(6,2){0.08} 
\pscircle(6.95,1.31){0.08} 
\pscircle(6.59,0.19){0.08} 
\pscircle(5.41,0.19){0.08} 
\pscircle(5.05,1.31){0.08} 
\psset{fillstyle=none}
\psline(6,2)(6.95,1.31)
\psline(6.95,1.31)(6.59,0.19)
\psline(6.59,0.19)(5.41,0.19)
\psline(5.41,0.19)(5.05,1.31)
\psline(5.05,1.31)(6,2)
\psline(6,2)(6.59,0.19)
\psline(6,2)(5.41,0.19)
\rput[c]{0}(5.4,1.2){\small $\times$}
\rput[c]{0}(6,0.8){\small $\times$}
\rput[c]{0}(6.6,1.2){\small $\times$}
\rput[c]{0}(5.8,1){\small $1$}
\rput[c]{0}(6.2,1){\small $2$}
\rput[c]{0}(6,-0.8){\small (b)}
\end{pspicture}
\end{center}
\caption{Example of labeled Stokes graph (drawn schematically) and associated
labeled Stokes triangulation.} 
\label{fig:pentagon2}
\end{figure}
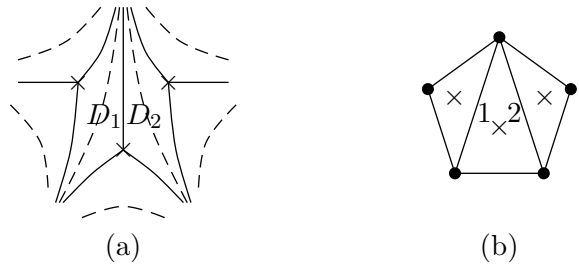

\end{ex}

\subsection{Signed flips and signed pops of Stokes graphs}

In Section   \ref{section:Voros-coefficients-and-DDP-formula}
we  already treated  the flips and the pops of Stokes graphs.
Here, we refine them as the {\em signed flips\/} and the {\em signed pops\/} to
incorporate them with the cluster algebra formulation.

To start, let us formulate the mutation of Stokes graphs in a more general situation
than before.
Suppose that there is 
 a  continuous 1-parameter family of quadratic differentials
 $\{ \phi_t \mid 0\leq t\leq 1\}$ on $\Sigma$
satisfying
the following conditions:
\begin{cond}
\label{cond:regular1}
\begin{itemize}
\item[(i).] The positions of zeros and poles of $\phi_t$ may change, but their orders
remain the same through the deformation.
\item[(ii).] The Stokes graph $G_t=G(\phi_t)$  is saddle-free for any $0\leq t \leq1$.
\end{itemize}
\end{cond}

By condition (i),
there is  a homeomorphism $\tau=\tau(\phi_0,\phi_1)$ of $\Sigma$,
by which
one can naturally identify the zeros and the poles of $G_1$
with those of $G_0$.
Furthermore, by condition (ii)
and \cite[Proposition 4.9]{Bridgeland13},
  $G_1$ is isotopic to $G_0$,
namely,
$G_t$  deforms continuously
without changing its topology form $t=0$ to $1$.
Thus, one can  identify the Stokes regions of $G_1$
with those of $G_0$ by $\tau$.
In particular, a given label of $G_0$ induces a label of $G_1$.
We call this labeled Stokes graph $G_1$  a {\em regular deformation\/} of
a labeled Stokes graph $G_0$.

\begin{rem}
\label{rem:pullback1}
The homeomorphism $\tau=\tau(\phi_0,\phi_1)$ also induces the pull-back $\tau^*(T)$ of
any labeled Stokes triangulation $T$ of 
$(\bfS(\phi_1),\bfM(\phi_1),\bfA(\phi_1))$
to
a labeled Stokes triangulation of 
$(\bfS(\phi_0),\bfM(\phi_0),\bfA(\phi_0))$.
Let $T_0=T(\phi_0)$ and $T_1=T(\phi_1)$
be  the labeled Stokes triangulations
associated with $G_0=G(\phi_0)$ and $G_1=G(\phi_1)$,
respectively.
Then, by construction, $\tau^*(T_1)$ coincides with $T_0$.
\end{rem}
 
 On the other hand,
when the saddle-free condition (ii) is violated at some $t=t_0$,
the Stokes graph $G_t$  changes its topology at $t_0$.
We call this phenomenon the {\em mutation\/} of Stokes graphs.
In this paper, we concentrate on the simplest situation
where such $G_{t_0}$ has a unique saddle trajectory
under Assumption \ref{assumption:at-most-one-saddle}.
Then, by
\cite[Proposition 5.5]{Bridgeland13},
the mutation of Stokes graphs locally reduces to the saddle reduction of
the saddle trajectory of $G_{t_0}$ studied in Section \ref{section:saddle-reduction}.
Thus, it is described by a flip and a pop, depending on whether
the saddle trajectory is regular or degenerate.

\begin{rem} When the Stokes graph $G_{t_0}$ have two saddle trajectories,
besides the combinations of flips and pops,
another type of mutation called a {\em juggle} may occur \cite{Gaiotto09}.
This is also an important subject, but we do not treat it in this paper.
\end{rem}

As mentioned, we refine the flips and the pops as the 
{\em signed flips\/} and the {\em signed pops\/}
in parallel with Stokes triangulations.

First, let us consider the signed flips.
Suppose that a  Stokes graph $G_0=G(\phi)$ of a quadratic differential $\phi$ 
has the unique saddle trajectory, which is
{\em regular}.
For a sufficiently small $\delta>0$,
let $G_{\pm \delta}=G(e^{\pm2i\delta}\phi)$ be the
saddle reductions of $G_0$ in Section \ref{section:saddle-reduction}.
See Figure \ref{fig:flip3}.
We assign the same label to the pair of the Stokes regions of $G_{+\delta}$ and
 $G_{-\delta}$ which
 degenerate into the saddle
 trajectory of $G_0$.
 We also assign the same label
 to each pair of the Stokes regions  of $G_{+\delta}$ and
 $G_{-\delta}$  which are naturally identified by an isotopy.
 Let $k$ be the label of the Stokes regions
 of $G_{+\delta}$ and
 $G_{-\delta}$ which degenerate to saddle 
 trajectory.
 Then, we write $G_{-\delta}=\mu^{(+)}_k (G_{+\delta})$
 and $G_{+\delta}=\mu^{(-)}_k (G_{-\delta})$
 as labeled Stokes graphs.

\begin{defn}
\label{defn:flipStokes1}
 For a  pair of labeled Stokes graphs $G=G(\phi)$ and $G'=G(\phi')$,
 suppose that there is a pair of labeled Stokes graphs $G_{+\delta}$ and
 $G_{-\delta}$ which are the saddle reductions
  of a Stokes graph $G_0$ with a unique regular saddle trajectory
 such that 
 \begin{itemize}
 \item
 $G_{+\delta}$ and $G_{-\delta}$ are regular
 deformations of $G$ and $G'$, respectively,
\item
$G_{-\delta}=\mu^{(+)}_k (G_{+\delta})$ in the above sense.
\end{itemize}
Then, we write $G'=\mu^{(+)}_k (G)$
and  $G=\mu^{(-)}_k (G')$,
and call $G'$ a {\em signed flip of $G$ at $k$ with sign $+$ \/} and {\em vice versa}.
\end{defn}

\begin{rem}
It follows from \cite[Proposition 4.9]{Bridgeland13} that
if $G'=\mu^{(\ve)}_k (G)$ and $G''=\mu^{(\ve)}_k (G)$,
then $G''$ is a regular deformation of $G''$.
Namely, a signed flip $\mu^{(\ve)}_k(G)$ is unique up to a regular deformation.
\end{rem}

As expected, the signed flips of labeled Stokes graphs and labeled Stokes triangulations
are compatible.

\begin{prop} 
Suppose that $G'=\mu^{(\ve)}_k (G)$.
Let $T$ and $T'$ be  labeled Stokes triangulations associated with $G$
and $G'$, respectively,
and let $\tau=\tau(\phi,\phi')$ be a  homeomorphism of $\Sigma$
naturally identifying the zeros and poles of $G$ and $G'$.
Then, $\mu^{(\ve)}_k(T)=\tau^*(T')$ holds,
where $\tau^*(T')$ is the pullback of $T'$ induced by $\tau$.
\end{prop}

\begin{proof}
Thanks to  Remark \ref{rem:pullback1},
we can assume that $G=G_{\delta}$ and $G'=G_{-\delta}$.
Since the zeros and poles do not move for $G_{\theta}$ ($-\delta \leq
\theta \leq \delta$),
$\tau$ is taken to be the identity map.
Then, this is clear from  Figure \ref{fig:flip3}.
\end{proof}

Next, let us consider the signed pops.
Let 
$G$ be
 a saddle-free labeled Stokes graph,
 $T$ be  a labeled Stokes triangulation
associated with $G$,
and $p$ be a puncture inside a self-folded triangle in $T$.
Then, a {\em signed pop $G'=\kappa^{(\ve)}_p(G)$
at $p$\/}  ($\ve=\pm$)  is defined in a  parallel way
by replacing a {\em regular\/} saddle trajectory of $G_0$ in the above
with a {\em degenerate\/} one surrounding the double pole of $G_0$ corresponding to $p$.
See Figure \ref{fig:pop4}.
The only speciality is that
the labels of the inner and outer Stokes regions surrounding 
the double pole should be interchanged
by the signed pops.
The rest is the same as the signed flip and we do not repeat it.
Again, we have the following compatibility.

\begin{prop} 
Suppose that $G'=\kappa^{(\ve)}_p (G)$.
Let $T$ and $T'$ be  labeled Stokes triangulations associated with $G$
and $G'$, respectively,
and let $\tau=\tau(\phi,\phi')$ be a  homeomorphism of $\Sigma$
naturally identifying the zeros and poles of $G$ and $G'$.
Then, $\kappa^{(\ve)}_p(T)=\tau^*(T')$ holds,
where $\tau^*(T')$ is the pullback of $T'$ induced by $\tau$.
\end{prop}

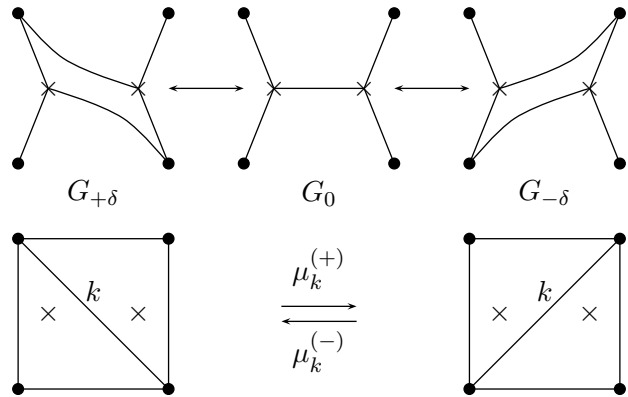
\begin{figure}
\begin{center}
\begin{pspicture}(-1,-1)(7,4.2)
\psset{linewidth=0.5pt}
\psset{fillstyle=solid, fillcolor=black}
\pscircle(1,1){0.08} 
\pscircle(1,-1){0.08} 
\pscircle(-1,1){0.08} 
\pscircle(-1,-1){0.08} 
\psset{fillstyle=none}
\psline(1,1)(1,-1)
\psline(1,-1)(-1,-1)
\psline(-1,-1)(-1,1)
\psline(-1,1)(1,1)
\psline(-1,1)(1,-1)
%
\rput[c]{0}(-0.6,0){\small $\times$}
\rput[c]{0}(0.6,0){\small $\times$}
%
%
\psline[arrows=->](2.5,0.1)(3.5,0.1)
\psline[arrows=<-](2.5,-0.1)(3.5,-0.1)
\psline[arrows=<->](1,3)(2,3)
\psline[arrows=<->](4,3)(5,3)
\rput[c]{0}(0,0.3){\small $k$}
\rput[c]{0}(3,0.6){\small $\mu_{k}^{(+)}$}
\rput[c]{0}(3,-0.5){\small $\mu_{k}^{(-)}$}
\psset{fillstyle=solid, fillcolor=black}
\pscircle(7,1){0.08} 
\pscircle(7,-1){0.08} 
\pscircle(5,1){0.08} 
\pscircle(5,-1){0.08} 
\psset{fillstyle=none}
\psline(7,1)(7,-1)
\psline(7,-1)(5,-1)
\psline(5,-1)(5,1)
\psline(5,1)(7,1)
\psline(7,1)(5,-1)

\rput[c]{0}(5.4,0){\small $\times$}
\rput[c]{0}(6.6,-0){\small $\times$}
\rput[c]{0}(6,0.3){\small $k$}
%
%
\psset{fillstyle=solid, fillcolor=black}
\pscircle(1,4){0.08} 
\pscircle(1,2){0.08} 
\pscircle(-1,4){0.08} 
\pscircle(-1,2){0.08} 
\psset{fillstyle=none}
\psline(-1,2)(-0.6,3)
\psline(-1,4)(-0.6,3)
\psline(1,2)(0.6,3)
\psline(1,4)(0.6,3)
\pscurve(0.6,3)(-0.4,3.4)(-1,4)
\pscurve(-0.6,3)(0.4,2.6)(1,2)
\rput[c]{0}(-0.6,3){\small $\times$}
\rput[c]{0}(0.6,3){\small $\times$}
\rput[c]{0}(0,1.6){\small $G_{+\delta}$}
\rput[c]{0}(3,1.6){\small $G_0$}
\rput[c]{0}(6,1.6){\small $G_{-\delta}$}
\psset{fillstyle=solid, fillcolor=black}
\pscircle(4,4){0.08} 
\pscircle(4,2){0.08} 
\pscircle(2,4){0.08} 
\pscircle(2,2){0.08} 
\psset{fillstyle=none}
\psline(2,4)(2.4,3)
\psline(2,2)(2.4,3)
\psline(4,4)(3.6,3)
\psline(4,2)(3.6,3)
\psline(3.6,3)(2.4,3)
\rput[c]{0}(2.4,3){\small $\times$}
\rput[c]{0}(3.6,3){\small $\times$}
%
%
\psset{fillstyle=solid, fillcolor=black}
\pscircle(7,4){0.08} 
\pscircle(7,2){0.08} 
\pscircle(5,4){0.08} 
\pscircle(5,2){0.08} 
\psset{fillstyle=none}
\psline(5,4)(5.4,3)
\psline(5,2)(5.4,3)
\psline(7,4)(6.6,3)
\psline(7,2)(6.6,3)
\pscurve(6.6,3)(5.6,2.6)(5,2)
\pscurve(5.4,3)(6.4,3.4)(7,4)
\rput[c]{0}(5.4,3){\small $\times$}
\rput[c]{0}(6.6,3){\small $\times$}
\end{pspicture}
\end{center}
\caption{Signed flips of Stokes graphs (upper row) and Stokes triangulations
(lower row).} 
\label{fig:flip3}
\end{figure}


\begin{figure}
\begin{center}
\begin{pspicture}(-1,-1)(7,4.2)
\psset{linewidth=0.5pt}
\psset{fillstyle=solid, fillcolor=black}
\pscircle(0,1){0.08} 
\pscircle(0,-1){0.08} 
\pscircle(0,0){0.08} 
\psset{fillstyle=none}
%
\pscurve(0,-1)(-0.4,0)(0,0.4)(0.4,0)(0,-1)
\pscurve(0,0)(-0.2,-0.5)(0,-1)
\pscurve(0,1)(0.6,0.45)(0.6,-0.45)(0,-1)
\pscurve(0,1)(-0.6,0.45)(-0.6,-0.45)(0,-1)
\rput[c]{0}(0,0.6){\small $\times$}
\rput[c]{0}(0,-0.4){\small $\times$}
%
%
\rput[c]{0}(0.32,0.5){\small $j$}
\rput[c]{0}(-0.2,-0){\small $i$}
%
%
\psline[arrows=->](2.5,0.1)(3.5,0.1)
\psline[arrows=<-](2.5,-0.1)(3.5,-0.1)
\psline[arrows=<->](1,3)(2,3)
\psline[arrows=<->](4,3)(5,3)
\rput[c]{0}(3,0.6){\small $\kappa_{p}^{(+)}$}
\rput[c]{0}(3,-0.5){\small $\kappa_{p}^{(-)}$}
%
\psset{fillstyle=solid, fillcolor=black}
\pscircle(6,1){0.08} 
\pscircle(6,-1){0.08} 
\pscircle(6,0){0.08} 
\psset{fillstyle=none}
\pscurve(6,-1)(5.6,0)(6,0.4)(6.4,0)(6,-1)
\pscurve(6,0)(6.2,-0.5)(6,-1)
\pscurve(6,1)(6.6,0.45)(6.6,-0.45)(6,-1)
\pscurve(6,1)(5.4,0.45)(5.4,-0.45)(6,-1)
\psset{fillstyle=solid, fillcolor=black}
\pscircle(0,4){0.08} 
\pscircle(0,2){0.08} 
\pscircle(0,3){0.08} 
\psset{fillstyle=none}
\psline(0,4)(0,3.6)
\pscurve(0,2)(0.5,2.6)(0.45,3.35)(0,3.6)
\pscurve(0,2)(-0.5,2.6)(-0.45,3.35)(0,3.6)
\psline(0,2)(0,2.6)
\pscurve(0,2.6)(-0.2,2.8)(-0.2,3.1)(0,3.2)(0.15,3)(0,2.9)(-0.05,2.95)(0,3)
\pscurve(0,2.6)(0.32,2.8)(0.35,3.2)(0,3.4)(-0.35,3.2)(-0.35,2.7)(0,2)
\rput[c]{0}(0,3.6){\small $\times$}
\rput[c]{0}(0,2.6){\small $\times$}
%
\psset{fillstyle=solid, fillcolor=black}
\pscircle(3,3){0.08} 
\pscircle(3,2){0.08} 
\pscircle(3,4){0.08} 
\psset{fillstyle=none}
\pscurve(3,2)(3.5,2.6)(3.45,3.35)(3,3.6)
\pscurve(3,2)(2.5,2.6)(2.55,3.35)(3,3.6)
\psline(3,2)(3,2.6)
\psline(3,4)(3,3.6)
\pscircle(3,3){0.41}
\rput[c]{0}(3,3.6){\small $\times$}
\rput[c]{0}(3,2.6){\small $\times$}
%
%
\psset{fillstyle=solid, fillcolor=black}
\pscircle(6,4){0.08} 
\pscircle(6,2){0.08} 
\pscircle(6,3){0.08} 
\psset{fillstyle=none}
\psline(6,4)(6,3.6)
\pscurve(6,2)(6.5,2.6)(6.45,3.35)(6,3.6)
\pscurve(6,2)(5.5,2.6)(5.55,3.35)(6,3.6)
\psline(6,2)(6,2.6)
\pscurve(6,2.6)(6.2,2.8)(6.2,3.1)(6,3.2)(5.85,3)(6,2.9)(6.05,2.95)(6,3)
\pscurve(6,2.6)(5.68,2.8)(5.65,3.2)(6,3.4)(6.35,3.2)(6.35,2.7)(6,2)
\rput[c]{0}(6,3.6){\small $\times$}
\rput[c]{0}(6,2.6){\small $\times$}
\rput[c]{0}(0,1.6){\small $G_{+\delta}$}
\rput[c]{0}(3,1.6){\small $G_0$}
\rput[c]{0}(6,1.6){\small $G_{-\delta}$}
%
\rput[c]{0}(6,0.6){\small $\times$}
\rput[c]{0}(6,-0.4){\small $\times$}
\rput[c]{0}(6.32,0.5){\small $i$}
\rput[c]{0}(6.2,0){\small $j$}
\end{pspicture}
\end{center}
\caption{Signed pops of Stokes graphs (upper row) and  Stokes triangulations
(lower row).} 
\label{fig:pop4}
\end{figure}
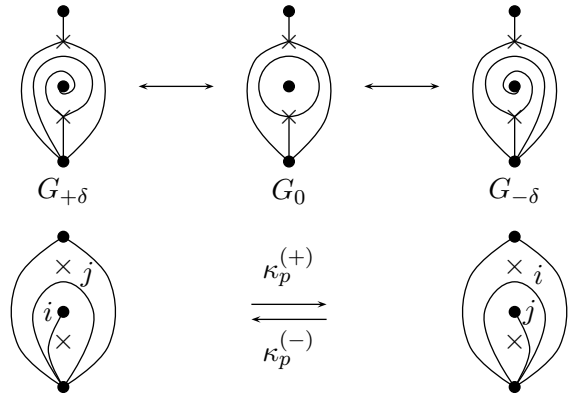

\subsection{Simple paths and simple cycles}
\label{subsec:simple1}

Let $\phi$ be a quadratic differential on $\Sigma$.
Recall that
we call elements of $H_1(\hat{\Sigma}\setminus \hat{P}_0,\hat{P}_{\infty})$
and
$H_1(\hat{\Sigma}\setminus \hat{P})$
  {\em paths\/} and {\em cycles\/},
respectively, in Section \ref{section:Voros-coefficients-and-DDP-formula}.

Let $\beta^*$ (resp., $\gamma^*$)
 be the image of a path $\beta$ (resp.,  a cycle  $\gamma$) by the  covering involution $\tau$ while keeping the direction.
Let
\begin{align}
\label{eq:sym1}
\mathrm{Sym}(H_1(\hat{\Sigma}\setminus \hat{P}_0,\hat{P}_{\infty}))
&=
\{ \beta \in H_1(\hat{\Sigma}\setminus \hat{P}_0,\hat{P}_{\infty})
\mid
\beta^*=\beta
\},\\
\mathrm{Sym}(H_1(\hat{\Sigma}\setminus \hat{P}))
&=
\{ \gamma \in H_1(\hat{\Sigma}\setminus \hat{P})
\mid
\gamma^*=\gamma
\}.
\end{align}
We introduce the {\em $*$-equivalence} of paths $\beta\equiv \beta'$
(resp., cycles $\gamma\equiv\gamma'$)
by the condition $\beta-\beta'\in 
\mathrm{Sym}(H_1(\hat{\Sigma}\setminus \hat{P}_0,\hat{P}_{\infty}))$
(resp., 
$\gamma-\gamma'\in 
\mathrm{Sym}(H_1(\hat{\Sigma}\setminus \hat{P}))$.
For example, $\beta\equiv -\beta^*$ and $\gamma\equiv -\gamma^*$
hold.

\begin{ex} Typical deformations of paths and cycles
modulo the $*$-equivalence are given in
 Figures  \ref{fig:path1}  and  \ref{fig:cycle1}.
Observe that these deformations are not quite obvious (like puzzle rings!).
\end{ex}

We introduce the quotients by the $*$-equivalence,
namely,
\begin{align}
\tilde{\Gamma}^{\vee}&=
H_1(\hat{\Sigma}\setminus \hat{P}_0,\hat{P}_{\infty})/
\mathrm{Sym}(H_1(\hat{\Sigma}\setminus \hat{P}_0,\hat{P}_{\infty})),\\
\tilde{\Gamma}&=
H_1(\hat{\Sigma}\setminus \hat{P})/
\mathrm{Sym}(H_1(\hat{\Sigma}\setminus \hat{P})).
\end{align}


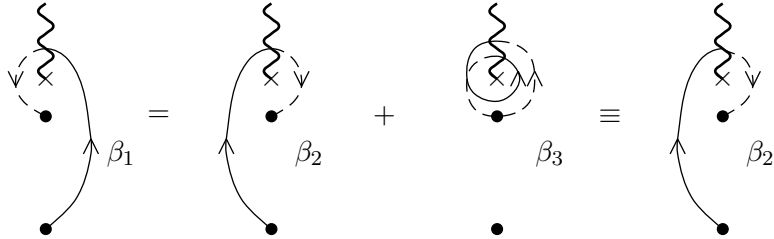
\begin{figure}
\begin{center}
\begin{pspicture}(-0.6,0)(9.7,3.2)
\psset{fillstyle=solid, fillcolor=black}
\pscircle(0,0){0.08} 
\pscircle(0,1.5){0.08} 
\psset{fillstyle=none}
\psset{linewidth=1pt}
\pscurve(0,2)(-0.1,2.1)(0,2.2)(0.1,2.3)(0,2.4)%
(-0.1,2.5)(0,2.6)(0.1,2.7)(0,2.8)(-0.1,2.9)(0,3)%
(0,3)
\psset{linewidth=0.5pt}
\psecurve(0,0)(0,0)(0.4,0.4)(0.6,1)(0.6,1.2)(0.3,2.25)(0,2.4)(-0.3,2.25)
\psecurve[linestyle=dashed](0,1.5)(0,1.5)(-0.3,1.7)(-0.4,1.9)(-0.3,2.25)(0,2.4)(0.3,2.25)
\psline(0.6,1.2)(0.5,1)
\psline(0.6,1.2)(0.7,1)
\psline(-0.4,1.9)(-0.3,2.1)
\psline(-0.4,1.9)(-0.5,2.1)
\rput[c]{0}(0,2){\small $\times$}
\rput[c]{0}(1,1){\small $\beta_1$}
\rput[c]{0}(1.5,1.5){\small $=$}
\psset{fillstyle=solid, fillcolor=black}
\pscircle(3,0){0.08} 
\pscircle(3,1.5){0.08} 
\psset{fillstyle=none}
\psset{linewidth=1pt}
\pscurve(3,2)(2.9,2.1)(3,2.2)(3.1,2.3)(3,2.4)%
(2.9,2.5)(3,2.6)(3.1,2.7)(3,2.8)(2.9,2.9)(3,3)%
(3,3)
\psset{linewidth=0.5pt}
\psecurve(3,0)(3,0)(2.6,0.4)(2.4,1)(2.4,1.2)(2.7,2.25)(3,2.4)(3.3,2.25)
\psecurve[linestyle=dashed](3,1.5)(3,1.5)(3.3,1.7)(3.4,1.9)(3.3,2.25)(3,2.4)(2.7,2.25)
\psline(2.4,1.2)(2.5,1)
\psline(2.4,1.2)(2.3,1)
\psline(3.4,1.9)(3.3,2.1)
\psline(3.4,1.9)(3.5,2.1)
\rput[c]{0}(3,2){\small $\times$}
\rput[c]{0}(3.5,1){\small $\beta_2$}
\rput[c]{0}(4.5,1.5){\small $+$}
\psset{fillstyle=solid, fillcolor=black}
\pscircle(6,0){0.08} 
\pscircle(6,1.5){0.08} 
\psset{fillstyle=none}
\psset{linewidth=1pt}
\pscurve(6,2)(5.9,2.1)(6,2.2)(6.1,2.3)(6,2.4)%
(5.9,2.5)(6,2.6)(6.1,2.7)(6,2.8)(5.9,2.9)(6,3)%
(6,3)
\psset{linewidth=0.5pt,linestyle=dashed}
\psecurve(6.2,2.2)(6,2.3)(5.7,2.2)(5.6,2)(5.7,1.6)(6,1.5)%
(6.4,1.65)(6.5,2)(6.4,2.35)(6,2.5)(5.7,2.4)
\psecurve[linestyle=solid](6.4,2.35)(6,2.5)(5.7,2.4)(5.6,2)%
(5.7,1.8)(6,1.7)(6.2,1.8)(6.3,2)(6.2,2.2)(6,2.3)(5.7,2.2)
\psset{linestyle=solid}
\psline(6.5,2.1)(6.4,1.9)
\psline(6.5,2.1)(6.6,1.9)
\psline(6.28,2.1)(6.38,1.9)
\psline(6.28,2.1)(6.18,1.9)
\rput[c]{0}(6,2){\small $\times$}
\rput[c]{0}(6.7,1){\small $\beta_3$}
\rput[c]{0}(7.5,1.5){\small $\equiv$}
\psset{fillstyle=solid, fillcolor=black}
\pscircle(9,0){0.08} 
\pscircle(9,1.5){0.08} 
\psset{fillstyle=none}
\psset{linewidth=1pt}
\pscurve(9,2)(8.9,2.1)(9,2.2)(9.1,2.3)(9,2.4)%
(8.9,2.5)(9,2.6)(9.1,2.7)(9,2.8)(8.9,2.9)(9,3)%
(9,3)
\psset{linewidth=0.5pt}
\psecurve(9,0)(9,0)(8.6,0.4)(8.4,1)(8.4,1.2)(8.7,2.25)(9,2.4)(9.3,2.25)
\psecurve[linestyle=dashed](9,1.5)(9,1.5)(9.3,1.7)(9.4,1.9)(9.3,2.25)(9,2.4)(2.7,9.25)
\psline(8.4,1.2)(8.5,1)
\psline(8.4,1.2)(8.3,1)
\psline(9.4,1.9)(9.3,2.1)
\psline(9.4,1.9)(9.5,2.1)
\rput[c]{0}(9,2){\small $\times$}
\rput[c]{0}(9.5,1){\small $\beta_2$}
%
%
%
\end{pspicture}
\end{center}
\caption{Typical deformations of paths
modulo the $*$-equivalence.}
\label{fig:path1}
\end{figure}

\begin{figure}
\begin{center}
\begin{pspicture}(-0.6,-1)(10.3,4.6)
\psset{linewidth=1pt}
\pscurve(0,2)(-0.1,2.1)(0,2.2)(0.1,2.3)(0,2.4)%
(-0.1,2.5)(0,2.6)(0.1,2.7)(0,2.8)(-0.1,2.9)(0,3)%
(0,3)(0.1,3.1)(0,3.2)(-0.1,3.3)(0,3.4)%
(0.1,3.5)(0,3.6)(-0.1,3.7)(0,3.8)(0.1,3.9)(0,4)%
\pscurve(0,-1)(0.1,-0.9)(0,-0.8)(-0.1,-0.7)(0,-0.6)%
(0.1,-0.5)(0,-0.4)(-0.1,-0.3)(0,-0.2)(0.1,-0.1)(0,0)%
\psset{linewidth=0.5pt}
\psecurve(-0.3,1.75)(0,1.6)(0.3,1.75)(0.6,3)(0.3,4.25)(0,4.4)(-0.3,4.25)
\psecurve(0.3,1.75)(0,1.6)(-0.3,1.75)(-0.6,3)(-0.3,4.25)(0,4.4)(0.3,4.25)
\psecurve(-0.3,-0.25)(0,-0.4)(0.3,-0.25)(0.6,1)(0.3,2.25)(0,2.4)(-0.3,2.25)
\psecurve[linestyle=dashed](0.3,-0.25)(0,-0.4)(-0.3,-0.25)(-0.6,1)(-0.3,2.25)(0,2.4)(0.3,2.25)
\psline(0.6,3.1)(0.5,2.9)
\psline(0.6,3.1)(0.7,2.9)
\psline(-0.6,2.9)(-0.5,3.1)
\psline(-0.6,2.9)(-0.7,3.1)
\psline(0.6,1.1)(0.5,0.9)
\psline(0.6,1.1)(0.7,0.9)
\psline(-0.6,0.9)(-0.5,1.1)
\psline(-0.6,0.9)(-0.7,1.1)
\rput[c]{0}(0,0){\small $\times$}
\rput[c]{0}(0,2){\small $\times$}
\rput[c]{0}(0,4){\small $\times$}
\rput[c]{0}(1,3){\small $\gamma_1$}
\rput[c]{0}(1,1){\small $\gamma_2$}
\rput[c]{0}(1.5,2){\small $=$}
\psset{linewidth=1pt}
\pscurve(3,2)(2.9,2.1)(3,2.2)(3.1,2.3)(3,2.4)%
(2.9,2.5)(3,2.6)(3.1,2.7)(3,2.8)(2.9,2.9)(3,3)%
(3,3)(3.1,3.1)(3,3.2)(2.9,3.3)(3,3.4)%
(3.1,3.5)(3,3.6)(2.9,3.7)(3,3.8)(3.1,3.9)(3,4)%
\pscurve(3,-1)(3.1,-0.9)(3,-0.8)(2.9,-0.7)(3,-0.6)%
(3.1,-0.5)(3,-0.4)(2.9,-0.3)(3,-0.2)(3.1,-0.1)(3,0)%
\psset{linewidth=0.5pt}
\psecurve(2.7,-0.25)(3,-0.4)(3.3,-0.25)(3.85,2)(3.3,4.25)(3,4.4)(2.7,4.25)%
(2.6,4)(2.75,3.7)(3,3.6)(3.25,3.7)
\psecurve[linestyle=dashed](3.3,-0.25)(3,-0.4)(2.7,-0.25)(2.6,0)(2.75,0.3)(2.95,0.39)(3,0.4)(3.2,0.55)(3.3,0.75)%
(3.45,2)(3.3,3.25)(3.2,3.45)(3,3.6)(2.7,3.35)
\psline(3.85,2.1)(3.75,1.9)
\psline(3.85,2.1)(3.95,1.9)
\psline(3.45,1.9)(3.35,2.1)
\psline(3.45,1.9)(3.55,2.1)
\rput[c]{0}(3,0){\small $\times$}
\rput[c]{0}(3,2){\small $\times$}
\rput[c]{0}(3,4){\small $\times$}
\rput[c]{0}(4.2,2.5){\small $\gamma_3$}
\rput[c]{0}(4.5,2){\small $+$}
\psset{linewidth=1pt}
\pscurve(6,2)(5.9,2.1)(6,2.2)(6.1,2.3)(6,2.4)%
(5.9,2.5)(6,2.6)(6.1,2.7)(6,2.8)(5.9,2.9)(6,3)%
(6,3)(6.1,3.1)(6,3.2)(5.9,3.3)(6,3.4)%
(6.1,3.5)(6,3.6)(5.9,3.7)(6,3.8)(6.1,3.9)(6,4)%
\pscurve(6,-1)(6.1,-0.9)(6,-0.8)(5.9,-0.7)(6,-0.6)%
(6.1,-0.5)(6,-0.4)(5.9,-0.3)(6,-0.2)(6.1,-0.1)(6,0)%
\psset{linewidth=0.5pt}
\psecurve(6.2,2.2)(6,2.3)(5.7,2.2)(5.6,2)(5.7,1.6)(6,1.5)%
(6.4,1.65)(6.5,2)(6.4,2.35)(6,2.5)(5.7,2.4)
\psecurve[linestyle=dashed](6.4,2.35)(6,2.5)(5.7,2.4)(5.6,2)%
(5.7,1.8)(6,1.7)(6.2,1.8)(6.3,2)(6.2,2.2)(6,2.3)(5.7,2.2)
\psline(6.5,2.1)(6.4,1.9)
\psline(6.5,2.1)(6.6,1.9)
\psline(6.28,2.1)(6.38,1.9)
\psline(6.28,2.1)(6.18,1.9)
\rput[c]{0}(6,0){\small $\times$}
\rput[c]{0}(6,2){\small $\times$}
\rput[c]{0}(6,4){\small $\times$}
\rput[c]{0}(6.7,2.5){\small $\gamma_4$}
\rput[c]{0}(7.5,2){\small $\equiv$}
\psset{linewidth=1pt}
\pscurve(9,2)(8.9,2.1)(9,2.2)(9.1,2.3)(9,2.4)%
(8.9,2.5)(9,2.6)(9.1,2.7)(9,2.8)(8.9,2.9)(9,3)%
(9,3)(9.1,3.1)(9,3.2)(8.9,3.3)(9,3.4)%
(9.1,3.5)(9,3.6)(8.9,3.7)(9,3.8)(9.1,3.9)(9,4)%
\pscurve(9,-1)(9.1,-0.9)(9,-0.8)(8.9,-0.7)(9,-0.6)%
(9.1,-0.5)(9,-0.4)(8.9,-0.3)(9,-0.2)(9.1,-0.1)(9,0)%
\psset{linewidth=0.5pt}
\psecurve(8.7,-0.25)(9,-0.4)(9.3,-0.25)(9.85,2)(9.3,4.25)(9,4.4)(8.7,4.25)%
(8.6,4)(8.75,3.7)(9,3.6)(9.25,3.7)
\psecurve[linestyle=dashed](9.3,-0.25)(9,-0.4)(8.7,-0.25)(8.6,0)(8.75,0.3)(8.95,0.39)(9,0.4)(9.2,0.55)(9.3,0.75)%
(9.45,2)(9.3,3.25)(9.2,3.45)(9,3.6)(8.7,3.35)
\psline(9.85,2.1)(9.75,1.9)
\psline(9.85,2.1)(9.95,1.9)
\psline(9.45,1.9)(9.35,2.1)
\psline(9.45,1.9)(9.55,2.1)
\rput[c]{0}(9,0){\small $\times$}
\rput[c]{0}(9,2){\small $\times$}
\rput[c]{0}(9,4){\small $\times$}
\rput[c]{0}(10.2,2.5){\small $\gamma_3$}
%
%
%
\end{pspicture}
\end{center}
\vskip0.5cm
\begin{center}
\begin{pspicture}(-1,-0.2)(9.6,4.6)
\psset{linewidth=1pt}
\psset{fillstyle=solid, fillcolor=black}
\pscircle(0,1){0.08} 
\psset{fillstyle=none}
\psset{linewidth=1pt}
\pscurve(0,2)(-0.1,2.1)(0,2.2)(0.1,2.3)(0,2.4)%
(-0.1,2.5)(0,2.6)(0.1,2.7)(0,2.8)(-0.1,2.9)(0,3)%
(0,3)(0.1,3.1)(0,3.2)(-0.1,3.3)(0,3.4)%
(0.1,3.5)(0,3.6)(-0.1,3.7)(0,3.8)(0.1,3.9)(0,4)%
\psset{linewidth=0.5pt}
\psecurve (-0.2,2)(0,2.2)(0.6,2)(1,1)(0.6,0)(0,-0.2)%
(-0.6,0)(-1,2)(-0.3,4.25)(0,4.4)(0.3,4.25)%
(0.4,4)(0.25,3.7)(0,3.6)(-0.25,3.7)
\psecurve[linestyle=dashed](0.6,2)(0,2.2)(-0.2,2)(0,1.8)%
(0.4,1.6)(0.6,1)(0.4,0.4)(0,0.2)(-0.4,0.4)(-0.6,1)%
(-0.6,2)(-0.3,3.25)(-0.2,3.45)(0,3.6)(0.3,3.35)
\psset{linewidth=0.5pt}
%
%
\psline(-1,0.9)(-1.1,1.1)
\psline(-1,0.9)(-0.9,1.1)
\psline(-0.6,1.1)(-0.5,0.9)
\psline(-0.6,1.1)(-0.7,0.9)
\psline(1,1.1)(0.9,0.9)
\psline(1,1.1)(1.1,0.9)
\psline(0.6,0.9)(0.5,1.1)
\psline(0.6,0.9)(0.7,1.1)
\rput[c]{0}(0,2){\small $\times$}
\rput[c]{0}(0,4){\small $\times$}
\rput[c]{0}(1.4,1){\small $\gamma_1$}
\rput[c]{0}(2,2){\small $\equiv$}
\psset{fillstyle=solid, fillcolor=black}
\pscircle(4,1){0.08} 
\psset{fillstyle=none}
\psset{linewidth=1pt}
\pscurve(4,2)(3.9,2.1)(4,2.2)(4.1,2.3)(4,2.4)%
(3.9,2.5)(4,2.6)(4.1,2.7)(4,2.8)(3.9,2.9)(4,3)%
(4,3)(4.1,3.1)(4,3.2)(3.9,3.3)(4,3.4)%
(4.1,3.5)(4,3.6)(3.9,3.7)(4,3.8)(4.1,3.9)(4,4)%
\psset{linewidth=0.5pt}
\psecurve(4.2,2)(4,2.2)(3.4,2)(3,1)(3.4,0)(4,-0.2)%
(4.6,0)(5,1)(4.6,2)(4,2.2)(3.8,2)
\psecurve[linestyle=dashed](4.6,2)(4,2.2)(3.8,2)(4,1.8)%
(4.4,1.6)(4.6,1)(4.4,0.4)(4,0.2)(3.6,0.4)(3.4,1)(3.6,1.6)(4,1.8)(4.2,2)(4,2.2)(3.4,2)
\psecurve(3.7,1.75)(4,1.6)(4.3,1.75)(4.6,3)(4.3,4.25)(4,4.4)(3.7,4.25)
\psecurve(4.3,1.75)(4,1.6)(3.7,1.75)(3.4,3)(3.7,4.25)(4,4.4)(4.3,4.25)
\psline(4.6,3.1)(4.5,2.9)
\psline(4.6,3.1)(4.7,2.9)
\psline(3.4,2.9)(3.3,3.1)
\psline(3.4,2.9)(3.5,3.1)
\psline(5,1.1)(4.9,0.9)
\psline(5,1.1)(5.1,0.9)
\psline(4.6,0.9)(4.5,1.1)
\psline(4.6,0.9)(4.7,1.1)
\psline(3.4,1.1)(3.3,0.9)
\psline(3.4,1.1)(3.5,0.9)
\psline(3,0.9)(2.9,1.1)
\psline(3,0.9)(3.1,1.1)
\rput[c]{0}(4,2){\small $\times$}
\rput[c]{0}(4,4){\small $\times$}
\rput[c]{0}(5,3){\small $\gamma_2$}
\rput[c]{0}(5.4,1){\small $\gamma_3$}
\rput[c]{0}(6,2){\small $\equiv$}
\psset{fillstyle=solid, fillcolor=black}
\pscircle(8,1){0.08} 
\psset{fillstyle=none}
\psset{linewidth=1pt}
\pscurve(8,2)(7.9,2.1)(8,2.2)(8.1,2.3)(8,2.4)%
(7.9,2.5)(8,2.6)(8.1,2.7)(8,2.8)(7.9,2.9)(8,3)%
(8,3)(8.1,3.1)(8,3.2)(7.9,3.3)(8,3.4)%
(8.1,3.5)(8,3.6)(7.9,3.7)(8,3.8)(8.1,3.9)(8,4)%
\psset{linewidth=0.5pt}
\psecurve (8.2,2)(8,2.2)(7.4,2)(7,1)(7.4,0)(8,-0.2)%
(8.6,0)(9,2)(8.3,4.25)(8,4.4)(7.7,4.25)%
(7.6,4)(7.75,3.7)(8,3.6)(8.25,3.7)
\psecurve[linestyle=dashed](7.4,2)(8,2.2)(8.2,2)(8,1.8)%
(7.6,1.6)(7.4,1)(7.6,0.4)(8,0.2)(8.4,0.4)(8.6,1)%
(8.6,2)(8.3,3.25)(8.2,3.45)(8,3.6)(7.7,3.35)
%
%
\psline(7,0.9)(6.9,1.1)
\psline(7,0.9)(7.1,1.1)
\psline(7.4,1.1)(7.5,0.9)
\psline(7.4,1.1)(7.3,0.9)
\psline(9,1.1)(8.9,0.9)
\psline(9,1.1)(9.1,0.9)
\psline(8.6,0.9)(8.5,1.1)
\psline(8.6,0.9)(8.7,1.1)
\rput[c]{0}(8,2){\small $\times$}
\rput[c]{0}(8,4){\small $\times$}
\rput[c]{0}(9.4,2,2){\small $\gamma_4$}
%
%
%
\end{pspicture}
\end{center}
\caption{Typical deformations of cycles
modulo the $*$-equivalence.} 
\label{fig:cycle1}
\end{figure}
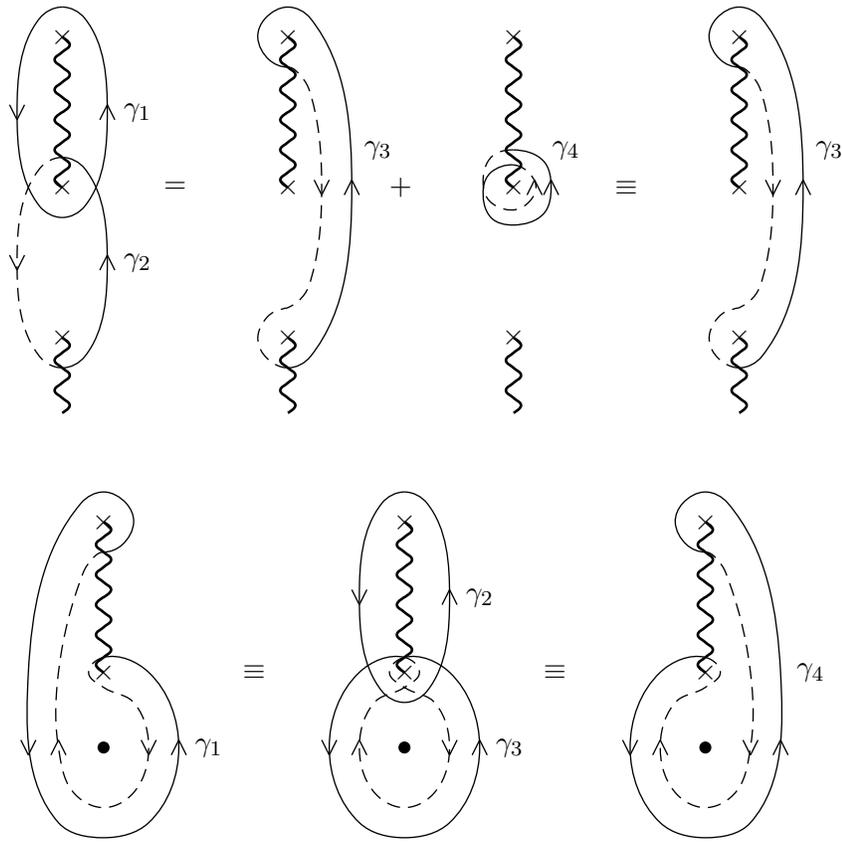

From now on, we  identify
paths $\beta$ and cycles $\gamma$ {\em modulo the $*$-equivalence}.
In other words,
we consider classes
$[\beta] \in \tilde{\Gamma}^{\vee}$
and
$[\gamma] \in \tilde{\Gamma}$
and denote them
by their representatives $\beta$ and $\gamma$,
for the notational simplicity.
{ Any path $\beta \in \tilde{\Gamma}^{\vee}$
and any cycle $\gamma \in \tilde{\Gamma}$
are specified by presenting their representatives
in a particular choice of branch.
In Figure \ref{fig:path-cycle1}
some examples of representatives of a path
$\beta \in \tilde{\Gamma}^{\vee}$ and a cycle
$\gamma \in \tilde{\Gamma}$
in different branches are given.
}

\begin{figure}
\begin{center}
\begin{pspicture}(-5,0)(-1,4.4)
\psset{unit=10mm}
%
\psset{fillstyle=solid, fillcolor=black}

\rput[c]{0}(-3,0){\small $\oplus$}
\rput[c]{0}(-3,4){\small $\oplus$}
\rput[c]{0}(-5,2){\small $\oplus$}
\rput[c]{0}(-1,2){\small $\ominus$}
\psset{fillstyle=none}
\psset{linewidth=1pt}
\pscurve(-3,1)(-3.1,0.9)(-3.2,1)(-3.3,1.1)(-3.4,1)(-3.5,0.9)
(-3.6,1)(-3.7,1.1)(-3.8,1)(-3.9,0.9)(-4.0,1)
\pscurve(-3,3)(-3.1,3.1)(-3.2,3)(-3.3,2.9)(-3.4,3)(-3.5,3.1)
(-3.6,3)(-3.7,2.9)(-3.8,3)(-3.9,3.1)(-4.0,3)
\psset{linewidth=0.5pt}

\psline(-3,3.9)(-3,3)
\psline(-3,0.1)(-3,1)
\pscurve(-3,3)(-2.2,2.5)(-1.1,2.05)
\pscurve(-3,1)(-2.2,1.5)(-1.1,1.95)
\pscurve(-3,3)(-3.8,2.5)(-4.9,2.05)
\pscurve(-3,1)(-3.8,1.5)(-4.9,1.95)
\psset{linewidth=1pt}
\psline(-4.9,2)(-1.1,2)
\psecurve(-3.6,2)(-3.45,1)(-3.3,0.75)(-3,0.6)(-2.7,0.75)(-2.4,2)(-2.7,3.25)(-3,3.4)
(-3.45,3)(-3.6,2)
\psecurve[linestyle=dashed]
(-3.3,0.75)(-3.45,1)(-3.6,2)(-3.45,3)(-3.3,3.25)
\psline(-2.4,2.3)(-2.5,2.1)
\psline(-2.4,2.3)(-2.3,2.1)
\psline(-3.6,2.1)(-3.7,2.3)
\psline(-3.6,2.1)(-3.5,2.3)
\psline(-2.1,2)(-1.9,2.1)
\psline(-2.1,2)(-1.9,1.9)
%
%
%
\rput[c]{0}(-3,1){\small $\times$}
\rput[c]{0}(-3,3){\small $\times$}
\rput[c]{0}(-2.3,3){\small $\gamma$}
\rput[c]{0}(-3,2.3){\small $\beta$}
\end{pspicture}
\hskip30pt
\begin{pspicture}(-5,0)(-1,4.4)
\psset{unit=10mm}
%
\psset{fillstyle=solid, fillcolor=black}

\rput[c]{0}(-3,0){\small $\ominus$}
\rput[c]{0}(-3,4){\small $\ominus$}
\rput[c]{0}(-5,2){\small $\ominus$}
\rput[c]{0}(-1,2){\small $\oplus$}
\psset{fillstyle=none}
\psset{linewidth=1pt}
\pscurve(-3,1)(-3.1,0.9)(-3.2,1)(-3.3,1.1)(-3.4,1)(-3.5,0.9)
(-3.6,1)(-3.7,1.1)(-3.8,1)(-3.9,0.9)(-4.0,1)
\pscurve(-3,3)(-3.1,3.1)(-3.2,3)(-3.3,2.9)(-3.4,3)(-3.5,3.1)
(-3.6,3)(-3.7,2.9)(-3.8,3)(-3.9,3.1)(-4.0,3)
\psset{linewidth=0.5pt}

\psline(-3,3.9)(-3,3)
\psline(-3,0.1)(-3,1)
\pscurve(-3,3)(-2.2,2.5)(-1.1,2.05)
\pscurve(-3,1)(-2.2,1.5)(-1.1,1.95)
\pscurve(-3,3)(-3.8,2.5)(-4.9,2.05)
\pscurve(-3,1)(-3.8,1.5)(-4.9,1.95)
\psset{linewidth=1pt}
\psline(-4.9,2)(-1.1,2)
\psecurve(-3.6,2)(-3.45,1)(-3.3,0.75)(-3,0.6)(-2.7,0.75)(-2.4,2)(-2.7,3.25)(-3,3.4)
(-3.45,3)(-3.6,2)
\psecurve[linestyle=dashed]
(-3.3,0.75)(-3.45,1)(-3.6,2)(-3.45,3)(-3.3,3.25)
%
\psline(-2.4,2.1)(-2.5,2.3)
\psline(-2.4,2.1)(-2.3,2.3)
\psline(-3.6,2.3)(-3.7,2.1)
\psline(-3.6,2.3)(-3.5,2.1)
\psline(-2.1,2.1)(-1.9,2)
\psline(-2.1,1.9)(-1.9,2)
%
%
%
\rput[c]{0}(-3,1){\small $\times$}
\rput[c]{0}(-3,3){\small $\times$}
\rput[c]{0}(-2.3,3){\small $\gamma$}
\rput[c]{0}(-3,2.3){\small $\beta$}
\end{pspicture}
\hskip30pt
\begin{pspicture}(-5,0)(-1,4.4)
\psset{unit=10mm}
%
\psset{fillstyle=solid, fillcolor=black}

\rput[c]{0}(-3,0){\small $\oplus$}
\rput[c]{0}(-3,4){\small $\oplus$}
\rput[c]{0}(-5,2){\small $\ominus$}
\rput[c]{0}(-1,2){\small $\ominus$}
\psset{fillstyle=none}
\psset{linewidth=1pt}
\pscurve(-3,1)(-3.1,1.1)(-3,1.2)(-2.9,1.3)(-3,1.4)%
(-3.1,1.5)(-3,1.6)(-2.9,1.7)(-3,1.8)(-3.1,1.9)(-3,2)%
(-2.9,2.1)(-3,2.2)(-3.1,2.3)(-3,2.4)%
(-2.9,2.5)(-3,2.6)(-3.1,2.7)(-3,2.8)(-2.9,2.9)(-3,3)%
\psset{linewidth=0.5pt}

\psline(-3,3.9)(-3,3)
\psline(-3,0.1)(-3,1)
\pscurve(-3,3)(-2.2,2.5)(-1.1,2.05)
\pscurve(-3,1)(-2.2,1.5)(-1.1,1.95)
\pscurve(-3,3)(-3.8,2.5)(-4.9,2.05)
\pscurve(-3,1)(-3.8,1.5)(-4.9,1.95)
\psset{linewidth=1pt}
\psline[linestyle=dashed](-4.9,2)(-3,2)
\psline(-3,2)(-1.1,2)
\psecurve(-3.3,0.75)(-3,0.6)(-2.7,0.75)(-2.4,2)(-2.7,3.25)(-3,3.4)(-4.3,3.25)
\psecurve(-2.7,0.75)(-3,0.6)(-3.3,0.75)(-3.6,2)(-3.3,3.25)(-3,3.4)(-2.7,3.25)
\psline(-2.4,2.3)(-2.5,2.1)
\psline(-2.4,2.3)(-2.3,2.1)
\psline(-3.6,2.1)(-3.7,2.3)
\psline(-3.6,2.1)(-3.5,2.3)
\psline(-2.1,2)(-1.9,2.1)
\psline(-2.1,2)(-1.9,1.9)
\psline(-4.1,2)(-3.9,2.1)
\psline(-4.1,2)(-3.9,1.9)
\rput[c]{0}(-3,1){\small $\times$}
\rput[c]{0}(-3,3){\small $\times$}
\rput[c]{0}(-2.3,3){\small $\gamma$}
\rput[c]{0}(-2.7,2.3){\small $\beta$}
\end{pspicture}
\end{center}
\caption{Examples of representatives of a path $\beta\in \tilde{\Gamma}^{\vee}$ 
and a cycle $\gamma\in \tilde{\Gamma}$
in different branches.}
\label{fig:path-cycle1}
\end{figure}
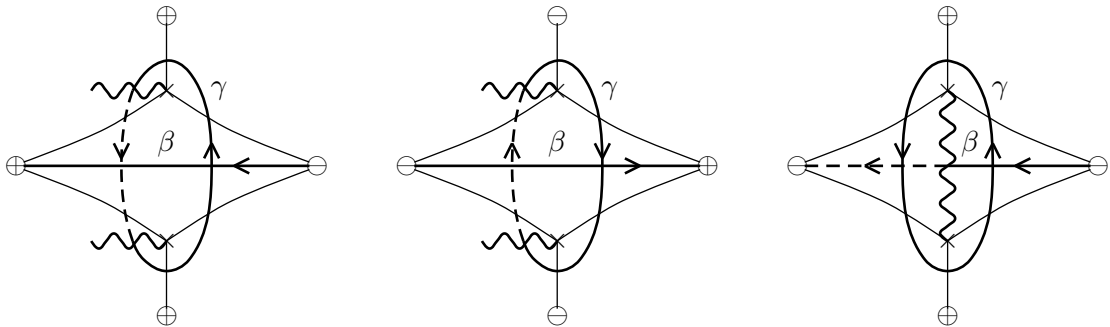

Now we introduce one of key objects in our work.
\begin{defn}
Let $\phi$ be a saddle-free quadratic differential on $\Sigma$.
Let $G=G(\phi)$ be a  labeled Stokes graph of $\phi$,
and let $D_1$, \dots, $D_n$ be the Stokes regions of $G$
which are regular or degenerate horizontal strips.
To each  region $D_i$, we assign
a path ${\beta}_i \in \tilde{\Gamma}^{\vee}$
and a cycle ${\gamma}_i \in \tilde{\Gamma}^{\vee}$
as follows:
\begin{itemize}
\item
If $D_i$ is  a regular horizontal strip not surrounding a degenerate horizontal strip,
${\beta}_i$ and $\gamma_i$ are given  in
Figure \ref{fig:bg1}.
(The presented therein are their representatives in a particular choice of branch.)

\item
If $D_i$ is a degenerate horizontal strip
and $D_j$ is  the regular horizontal strip surrounding $D_i$,
then ${\beta}_i, {\beta}_j, \gamma_i,\gamma_j$ are given  in
 Figure \ref{fig:bg2}.
(Both figures in   Figure \ref{fig:bg2} yield the same $*$-equivalence classes
as demonstrated in Figures 
\ref{fig:path1} and \ref{fig:cycle1}.)
\end{itemize}
We call ${\beta}^{}_1, \dots, \beta^{}_n\in \Gamma^{\vee}$
and $\gamma^{}_1,\dots, \gamma^{}_n\in \Gamma$ the  {\em simple paths\/} and
the {\em simple cycles\/} of $G$,
respectively.

\end{defn}

\begin{figure}
\begin{center}
\begin{pspicture}(-6,0)(-1,5)
\psset{unit=12mm}
%
\psset{fillstyle=solid, fillcolor=black}

\rput[c]{0}(-3,0){\small $\oplus$}
\rput[c]{0}(-3,4){\small $\oplus$}
\rput[c]{0}(-5,2){\small $\ominus$}
\rput[c]{0}(-1,2){\small $\ominus$}
\psset{fillstyle=none}
\psset{linewidth=1pt}
\pscurve(-3,1)(-3.1,1.1)(-3,1.2)(-2.9,1.3)(-3,1.4)%
(-3.1,1.5)(-3,1.6)(-2.9,1.7)(-3,1.8)(-3.1,1.9)(-3,2)%
(-2.9,2.1)(-3,2.2)(-3.1,2.3)(-3,2.4)%
(-2.9,2.5)(-3,2.6)(-3.1,2.7)(-3,2.8)(-2.9,2.9)(-3,3)%
\psset{linewidth=0.5pt}

\psline(-3,3.9)(-3,3)
\psline(-3,0.1)(-3,1)
\pscurve(-3,3)(-2.2,2.5)(-1.1,2.05)
\pscurve(-3,1)(-2.2,1.5)(-1.1,1.95)
\pscurve(-3,3)(-3.8,2.5)(-4.9,2.05)
\pscurve(-3,1)(-3.8,1.5)(-4.9,1.95)
\psset{linewidth=1pt}
\psline[linestyle=dashed](-4.9,2)(-3,2)
\psline(-3,2)(-1.1,2)
\psecurve(-3.3,0.75)(-3,0.6)(-2.7,0.75)(-2.4,2)(-2.7,3.25)(-3,3.4)(-4.3,3.25)
\psecurve(-2.7,0.75)(-3,0.6)(-3.3,0.75)(-3.6,2)(-3.3,3.25)(-3,3.4)(-2.7,3.25)
\psline(-2.4,2.3)(-2.5,2.1)
\psline(-2.4,2.3)(-2.3,2.1)
\psline(-3.6,2.1)(-3.7,2.3)
\psline(-3.6,2.1)(-3.5,2.3)
\psline(-2.1,2)(-1.9,2.1)
\psline(-2.1,2)(-1.9,1.9)
\psline(-4.1,2)(-3.9,2.1)
\psline(-4.1,2)(-3.9,1.9)
\rput[c]{0}(-3,1){\small $\times$}
\rput[c]{0}(-3,3){\small $\times$}
\rput[c]{0}(-2.3,3){\small $\gamma_i$}
\rput[c]{0}(-2.7,2.3){\small $\beta_i$}
\end{pspicture}
\end{center}
\caption{Simple path and simple cycle
for a regular horizontal strip $D_i$ not surrounding a degenerate horizontal strip.
}
\label{fig:bg1}
\end{figure}
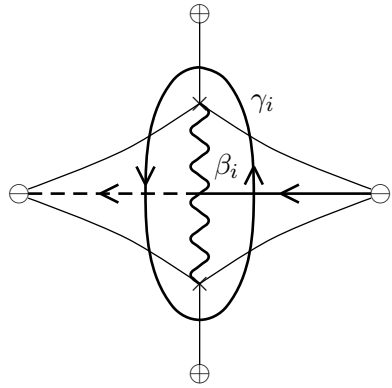
%
%
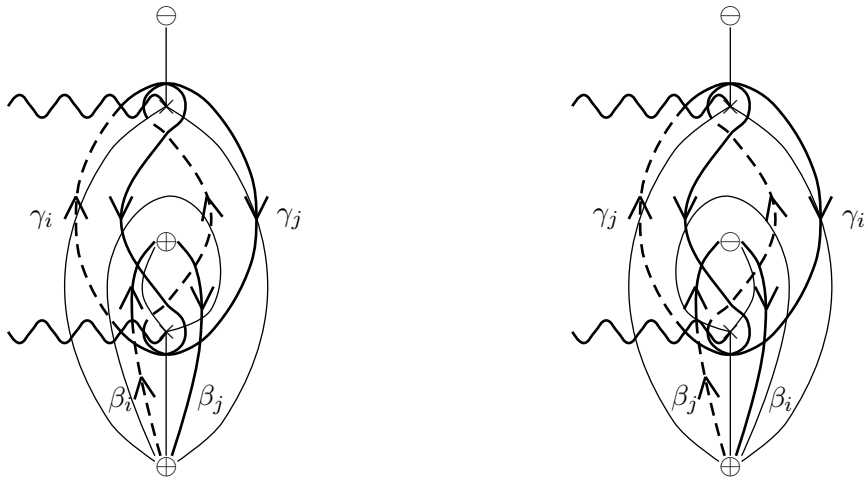
\begin{figure}
\begin{center}
\begin{pspicture}(-1.4,0)(9,6.2)
\psset{unit=15mm}
\psset{linewidth=0.5pt}
\rput[c]{0}(0,0){\small $\oplus$}
\rput[c]{0}(0,2){\small $\oplus$}
\rput[c]{0}(0,4){\small $\ominus$}
\rput[c]{0}(0,3.2){\small $\times$}
\rput[c]{0}(0,1.2){\small $\times$}
\rput[c]{0}(-1.1,2.2){\small $\gamma_i{}$}
\rput[c]{0}(1.1,2.2){\small $\gamma_j{}$}
\rput[c]{0}(-0.4,0.6){\small $\beta_i{}$}
\rput[c]{0}(0.4,0.6){\small $\beta_j{}$}
\psset{linewidth=0.5pt,linestyle=solid}
\pscurve(0.1,0.05)(0.5,0.4)(0.9,1.6)(0.5,2.8)(0,3.2)
\pscurve(-0.1,0.05)(-0.5,0.4)(-0.9,1.6)(-0.5,2.8)(-0,3.2)
\pscurve(0,1.2)(0.4,1.4)(0.35,2.2)(0,2.4)(-0.5,1.8)(-0.1,0.1)

\pscurve(0,1.2)(-0.2,1.5)(-0.2,1.7)(-0.1,1.95)
\psline(0,0.1)(0,1.2)
\psline(0,3.2)(0,3.9)
\psset{linewidth=1pt,linestyle=solid}
\pscurve(0,1.2)(-0.1,1.1)(-0.2,1.2)(-0.3,1.3)(-0.4,1.2)
(-0.5,1.1)(-0.6,1.2)(-0.7,1.3)(-0.8,1.2)(-0.9,1.1)(-1.0,1.2)
(-1.1,1.3)(-1.2,1.2)(-1.3,1.1)(-1.4,1.2)
\pscurve(0,3.2)(-0.1,3.3)(-0.2,3.2)(-0.3,3.1)(-0.4,3.2)
(-0.5,3.3)(-0.6,3.2)(-0.7,3.1)(-0.8,3.2)(-0.9,3.3)(-1,3.2)
(-1.1,3.1)(-1.2,3.2)(-1.3,3.3)(-1.4,3.2)
\psecurve[linewidth=1pt,linestyle=solid]
(-0.75,2.6)(-0.4,3.2)(-0.2,3.35)(0,3.4)(0.15,3.3)(0.15,3.1)(0,2.95)(-0.4,2.2)(0,1.45)
(0.15,1.3)(0.15,1.1)(0,1)(-0.2,1.05)(-0.4,1.2)(-0.75,1.8)
\psecurve[linewidth=1pt,linestyle=dashed]
(-0.2,1.05)(-0.4,1.2)(-0.75,1.8)(-0.8,2.2)(-0.75,2.6)(-0.4,3.2)(-0.2,3.35)
\psecurve[linewidth=1pt,linestyle=solid]
(-0.17,1.3)(-0.2,1.2)(-0.17,1.1)(0,1)(0.2,1.05)(0.4,1.2)(0.75,1.8)(0.8,2.2)(0.75,2.6)(0.4,3.2)(0.2,3.35)(0,3.4)(-0.17,3.3)(-0.2,3.2)(-0.17,3.1)
\psecurve[linewidth=1pt,linestyle=dashed]
(-0.17,3.3)(-0.2,3.2)(-0.17,3.1)(0,2.95)(0.4,2.2)(0,1.45)
(-0.17,1.3)(-0.2,1.2)(-0.17,1.1)
\psset{linewidth=1pt,linestyle=solid}
\psline(-0.4,2.2)(-0.3,2.4)
\psline(-0.4,2.2)(-0.5,2.4)
\psline(0.8,2.2)(0.7,2.4)
\psline(0.8,2.2)(0.9,2.4)
\psline(-0.8,2.4)(-0.7,2.2)
\psline(-0.8,2.4)(-0.9,2.2)
\psline(0.36,2.4)(0.5,2.22)
\psline(0.36,2.4)(0.32,2.18)

\pscurve[linewidth=1pt,linestyle=solid]
(0.1,2)(0.25,1.8)(0.3,1.1)(0.05,0.1)
\psecurve[linewidth=1pt,linestyle=solid]
(-0.1,2)(-0.1,2)(-0.25,1.8)(-0.3,1.3)(-0.05,0.1)
\psecurve[linewidth=1pt,linestyle=dashed]
(-0.25,1.8)(-0.3,1.3)(-0.05,0.1)(-0.05,0.1)
\psset{linewidth=1pt,linestyle=solid}
\psline(-0.31,1.6)(-0.21,1.4)
\psline(-0.31,1.6)(-0.41,1.4)
\psline(0.32,1.4)(0.42,1.6)
\psline(0.32,1.4)(0.22,1.6)
\psline(-0.22,0.8)(-0.28,0.6)
\psline(-0.22,0.8)(-0.1,0.62)

%
%
\psset{linewidth=0.5pt}
\rput[c]{0}(5,0){\small $\oplus$}
\rput[c]{0}(5,2){\small $\ominus$}
\rput[c]{0}(5,4){\small $\ominus$}
\rput[c]{0}(5,3.2){\small $\times$}
\rput[c]{0}(5,1.2){\small $\times$}
\rput[c]{0}(3.9,2.2){\small $\gamma_j$}
\rput[c]{0}(6.1,2.2){\small $\gamma_i$}
\rput[c]{0}(4.6,0.6){\small $\beta_j{}$}
\rput[c]{0}(5.45,0.6){\small $\beta_i{}$}

\psset{linewidth=0.5pt,linestyle=solid}
\pscurve(5.1,0.05)(5.5,0.4)(5.9,1.6)(5.5,2.8)(5,3.2)
\pscurve(4.9,0.05)(4.5,0.4)(4.1,1.6)(4.5,2.8)(5,3.2)
\pscurve(5,1.2)(4.6,1.4)(4.65,2.2)(5,2.4)(5.5,1.8)(5.1,0.1)
\pscurve(5,1.2)(5.2,1.5)(5.2,1.7)(5.1,1.95)
\psline(5,0.1)(5,1.2)
\psline(5,3.2)(5,3.9)
\psset{linewidth=1pt,linestyle=solid}
\pscurve(5,1.2)(4.9,1.1)(4.8,1.2)(4.7,1.3)(4.6,1.2)
(4.5,1.1)(4.4,1.2)(4.3,1.3)(4.2,1.2)(4.1,1.1)(4,1.2)
(3.9,1.3)(3.8,1.2)(3.7,1.1)(3.6,1.2)
\pscurve(5,3.2)(4.9,3.3)(4.8,3.2)(4.7,3.1)(4.6,3.2)
(4.5,3.3)(4.4,3.2)(4.3,3.1)(4.2,3.2)(4.1,3.3)(4,3.2)
(3.9,3.1)(3.8,3.2)(3.7,3.3)(3.6,3.2)
\psecurve[linewidth=1pt,linestyle=solid]
(4.25,2.6)(4.6,3.2)(4.8,3.35)(5,3.4)(5.15,3.3)(5.15,3.1)(5,2.95)(4.6,2.2)(5,1.45)
(5.15,1.3)(5.15,1.1)(5,1)(4.8,1.05)(4.6,1.2)(4.25,1.8)
\psecurve[linewidth=1pt,linestyle=dashed]
(4.8,1.05)(4.6,1.2)(4.25,1.8)(4.2,2.2)(4.25,2.6)(4.6,3.2)(4.8,3.35)
\psecurve[linewidth=1pt,linestyle=solid]
(4.83,1.3)(4.8,1.2)(4.83,1.1)(5,1)(5.2,1.05)(5.4,1.2)(5.75,1.8)(5.8,2.2)(5.75,2.6)(5.4,3.2)(5.2,3.35)(5,3.4)(4.83,3.3)(4.8,3.2)(4.83,3.1)
\psecurve[linewidth=1pt,linestyle=dashed]
(4.83,3.3)(4.8,3.2)(4.83,3.1)(5,2.95)(5.4,2.2)(5,1.45)
(4.83,1.3)(4.8,1.2)(4.83,1.1)

\psset{linewidth=1pt,linestyle=solid}
\psline(4.6,2.2)(4.7,2.4)
\psline(4.6,2.2)(4.5,2.4)
\psline(5.8,2.2)(5.7,2.4)
\psline(5.8,2.2)(5.9,2.4)
\psline(4.2,2.4)(4.3,2.2)
\psline(4.2,2.4)(4.1,2.2)
\psline(5.36,2.4)(5.5,2.22)
\psline(5.36,2.4)(5.32,2.18)
\pscurve[linewidth=1pt,linestyle=solid]
(5.1,2)(5.25,1.8)(5.3,1.1)(5.05,0.1)
\psecurve[linewidth=1pt,linestyle=solid]
(4.9,2)(4.9,2)(4.75,1.8)(4.7,1.3)(4.95,0.1)
\psecurve[linewidth=1pt,linestyle=dashed]
(4.75,1.8)(4.7,1.3)(4.95,0.1)(4.95,0.1)
\psset{linewidth=1pt,linestyle=solid}
\psline(4.69,1.6)(4.79,1.4)
\psline(4.69,1.6)(4.59,1.4)
\psline(5.32,1.4)(5.42,1.6)
\psline(5.32,1.4)(5.22,1.6)
\psline(4.78,0.8)(4.72,0.6)
\psline(4.78,0.8)(4.9,0.62)
\end{pspicture}
\end{center}
\caption{Simple paths and simple cycles
for a degenerate horizontal strip $D_i$
and for the regular horizontal strip $D_j$ surrounding $D_i$.}
\label{fig:bg2}
\end{figure}

\begin{rem} 
The simple cycles  correspond to
the  {\em modified basis\/} in
\cite{Bridgeland13} in their convention of the homology group.
\end{rem}

We define $\Gamma^{\vee}$ to be
the subgroup of 
$\tilde{\Gamma}^{\vee}
$ generated by $
\beta_1, \dots, 
\beta_n$.
Similarly,
we define $\Gamma$ to be
the subgroup of 
$\tilde{\Gamma}
$ generated by $
\gamma_1, \dots, 
\gamma_n$.
It can be easily checked that the intersection forms in \eqref{eq:braket1} and \eqref{eq:intersection1}
induce the intersection forms $\langle ~,~\rangle: \Gamma \times \Gamma^{\vee}
\rightarrow \bbZ$
and 
$(~,~): \Gamma \times \Gamma
\rightarrow \bbZ$.

\begin{prop}
\label{prop:dual1}
We have
\begin{align}
\label{eq:gb3}
\langle \gamma_i, \beta_j  \rangle =\delta_{ij},
\end{align}
so that the simple paths $
\beta_1, \dots, 
\beta_n \in \Gamma^{\vee}$
and the simple cycles $
\gamma_1, \dots, 
\gamma_n \in \Gamma$
are the dual bases to each other.
\end{prop}
\begin{proof}
This can be checked by  inspecting
Figures \ref{fig:bg1} and \ref{fig:bg2}.
\end{proof}

Let us observe that the simple paths and the simple cycles are naturally integrated into
 cluster algebra theory.
Let $T$ be the labeled Stokes triangulation
associated with $G$.
The following formula is the first justification of the definition
of the simple cycles.

\begin{prop}
\label{prop:bmat1}
 Let $B=(b_{ij})_{i,j=1}^n$ be the adjacency matrix of $T$.
Then,
we have
\begin{align}
\label{eq:gamma1}
( \gamma_i, \gamma_j ) &= b_{ij}.
\end{align}
\end{prop}
\begin{proof}
We prove it by case-check.
There are essentially two cases to consider.

Case 1. Suppose that the regions  $D_i$ and $D_j$ are regular
horizontal strips.
Then, $b_{ij}=1$ or $2$.
The case $b_{ij}=1$ is depicted in Figure \ref{fig:gamma2},
and we see that 
$( \gamma_i, \gamma_j ) = 1= b_{ij}$.
In the case $b_{ij}=2$,
identify the paths $\beta_j$ and $-\beta_k$
in Figure \ref{fig:gamma2}.
Then, we have
$( \gamma_i, \gamma_j ) = 2= b_{ij}$.

Case 2. Suppose that the region $D_i$ is a degenerate horizontal strip
and the region $D_j$ is the regular horizontal strip
surrounding $D_i$.
Then, 
by Figure \ref{fig:bg2} we have
$( \gamma_i, \gamma_{j} ) =0
 = b_{ij}$,
 while $( \gamma_i, \gamma_{k} ) 
 =( \gamma_{j}, \gamma_{k} )
 =b_{ik}=b_{jk}$ for
 any $k\neq i,j$. 
\end{proof}

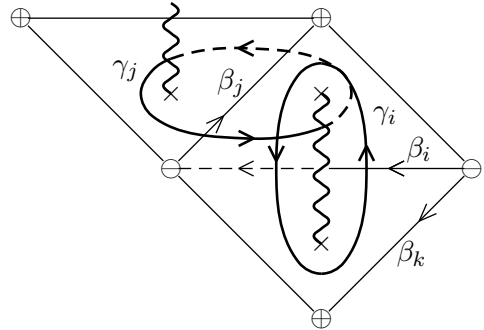
\begin{figure}
\begin{center}
\begin{pspicture}(-7,0)(-1,4.4)
%
\psset{fillstyle=solid, fillcolor=black}
%
\rput[c]{0}(-3,0){\small $\oplus$}
\rput[c]{0}(-3,4){\small $\oplus$}
\rput[c]{0}(-5,2){\small $\ominus$}
\rput[c]{0}(-1,2){\small $\ominus$}
\rput[c]{0}(-7,4){\small $\oplus$}

\psset{fillstyle=none}
\psset{linewidth=1pt}
\pscurve(-3,1)(-3.1,1.1)(-3,1.2)(-2.9,1.3)(-3,1.4)%
(-3.1,1.5)(-3,1.6)(-2.9,1.7)(-3,1.8)(-3.1,1.9)(-3,2)%
(-2.9,2.1)(-3,2.2)(-3.1,2.3)(-3,2.4)%
(-2.9,2.5)(-3,2.6)(-3.1,2.7)(-3,2.8)(-2.9,2.9)(-3,3)%
\pscurve(-5,3)(-5.1,3.1)(-5,3.2)(-4.9,3.3)(-5,3.4)%
(-5.1,3.5)(-5,3.6)(-4.9,3.7)(-5,3.8)(-5.1,3.9)(-5,4)%
(-4.9,4.1)(-5,4.2)%
%
\psset{linewidth=0.5pt}
\psline(-3.1,3.9)(-4.9,2.1)
\psline(-4.9,1.9)(-3.1,0.1)
\psline(-2.9,0.1)(-1.1,1.9)
\psline(-1.1,2.1)(-2.9,3.9)
\psline[linestyle=dashed](-4.9,2)(-3.1,2)
\psline(-2.9,2)(-1.1,2)
\psline(-5.1,2.1)(-6.9,3.9)
\psline(-6.9,4)(-3.1,4)
\psset{linewidth=1pt}
\psecurve(-3.3,0.75)(-3,0.6)(-2.7,0.75)(-2.4,2)(-2.7,3.25)(-3,3.4)(-4.3,3.25)
\psecurve(-2.7,0.75)(-3,0.6)(-3.3,0.75)(-3.6,2)(-3.3,3.25)(-3,3.4)(-2.7,3.25)
\psecurve(-2.75,2.7)(-3,2.54)(-4,2.4)(-5.25,2.7)(-5.4,3)(-5.25,3.3)(-5,3.46)(-4,3.6)%
\psecurve[linestyle=dashed](-5.25,3.3)(-5,3.46)(-4,3.6)(-2.75,3.3)(-2.6,3)(-2.75,2.7)(-3,2.54)(-4,2.4)
\psset{linewidth=0.5pt}
\psline(-2.1,2)(-1.9,2.1)
\psline(-2.1,2)(-1.9,1.9)

\psline(-2.1,2)(-1.9,2.1)
\psline(-2.1,2)(-1.9,1.9)
\psline(-4.1,2)(-3.9,2.1)
\psline(-4.1,2)(-3.9,1.9)
\psline(-4.3,2.7)(-4.53,2.6)
\psline(-4.3,2.7)(-4.4,2.47)
\psline(-1.7,1.3)(-1.47,1.4)
\psline(-1.7,1.3)(-1.6,1.53)
\psset{linewidth=1pt}
\psline(-3.6,2.1)(-3.7,2.3)
\psline(-3.6,2.1)(-3.5,2.3)
\psline(-2.4,2.3)(-2.5,2.1)
\psline(-2.4,2.3)(-2.3,2.1)
\psline(-3.9,2.4)(-4.1,2.3)
\psline(-3.9,2.4)(-4.1,2.5)
\psline(-4.1,3.6)(-3.9,3.7)
\psline(-4.1,3.6)(-3.9,3.5)
\rput[c]{0}(-5,3){\small $\times$}
\rput[c]{0}(-3,1){\small $\times$}
\rput[c]{0}(-3,3){\small $\times$}
\rput[c]{0}(-2.15,2.7){\small $\gamma_i$}
\rput[c]{0}(-5.6,3.3){\small $\gamma_j$}
\rput[c]{0}(-1.7,2.25){\small $\beta_i$}
\rput[c]{0}(-4.2,3.14){\small $\beta_j$}
\rput[c]{0}(-1.8,0.86){\small $\beta_k$}
%
\end{pspicture}
\end{center}
\caption{Calculation of the intersection number
 $( \gamma_i, \gamma_j)=1$.
} 
\label{fig:gamma2}
\end{figure}

\begin{prop}
\label{prop:bg3}
As an element of $\Gamma^{\vee}$,
$\gamma_i$ decomposes as follows:
\begin{align}
\label{eq:bg5}
\gamma_i
=
\sum_{j=1}^n
b_{ji} 
\beta_j.
\end{align}
\end{prop}
\begin{proof}
This can be verified by case-check
using Figures \ref{fig:bg2} and \ref{fig:gamma2}.
\end{proof}

\subsection{Mutation of simple paths and simple cycles}

Let us examine how the simple paths and the simple cycles
transform (= mutate) under the mutation of
Stokes graphs.
Suppose that there are two labeled  Stokes graphs
$G=G(\phi)$ and $G'=G(\phi')$
which are related
by a signed flip or a signed pop.
Let us distinguish the corresponding homology groups $\Gamma^{\vee}$
 and $\Gamma$
for $\phi$ and $\phi'$
as
 $\Gamma^{\vee}_G$, $\Gamma_G$
and
 $\Gamma^{\vee}_{G'}$, $\Gamma_{G'}$,
 respectively.
By assumption,
the zeros and poles
of $\phi$ continuously move to those of $\phi'$.
This  induces the canonical isomorphisms 
of the homology groups
\begin{align}
\label{eq:iso1}
 \tau_{G,G'}^{\vee}:
  \Gamma^{\vee}_{G'}\overset{\sim}{\rightarrow}\Gamma^{\vee}_G,
  \quad
\tau_{G,G'}:\Gamma_{G'}\overset{\sim}{\rightarrow}
 \Gamma_{G}.
\end{align}
Let $\beta_1,\dots,\beta_n\in\Gamma^{\vee}_G $ and 
$\gamma_1,\dots,\gamma_n\in \Gamma_G$
be the simple paths and cycles of $G$,
and 
$\beta_1',\dots,\beta_n'\in \Gamma^{\vee}_{G'}$ and 
$\gamma_1',\dots,\gamma_n'\in \Gamma_{G'}$
be the ones of $G'$.
Then, the isomorphisms $ \tau_{G,G'}^{\vee}$ and 
$ \tau_{G,G'}$ are explicitly given as follows.

\begin{prop} (a).
({cf. \cite[Lemma 9.11]{Bridgeland13}})
\label{prop:cycle1}
If $G$ and $G'$ are related by
the signed flip $G'=\mu_k^{(\ve)}(G)$, 
then the isomorphisms  $ \tau_{G,G'}^{\vee}$ and 
$ \tau_{G,G'}$ are given by
\begin{align}
\label{eq:bemut1}
\tau_{G,G'}^{\vee}(\beta_i')&=
\begin{cases}
-\beta_k
+\sum_{j=1}^n [-\varepsilon b_{jk}]_+ \beta_j
 & i=k\\
\beta_i 
& i \neq k,
\end{cases}\\
\label{eq:gamut1}
\tau_{G,G'}(\gamma_i')&=
\begin{cases}
-\gamma_k & i=k\\
\gamma_i + [\varepsilon b_{ki}]_+ \gamma_k
& i \neq k.
\end{cases}
\end{align}
\par
(b).
If $G$ and $G'$ are related by
the signed pop $G'=\kappa_p^{(\ve)}(G)$, 
then the isomorphisms  $ \tau_{G,G'}^{\vee}$ and 
$ \tau_{G,G'}$ are trivial,
i.e.,
\begin{align}
\label{eq:bgmut1}
\tau_{G,G'}^{\vee}(\beta_i')= \beta_i,
\quad
\tau_{G,G'}( \gamma_i')=\gamma_i.
\end{align}
\end{prop}

\begin{proof}
(a).
Two transformations \eqref{eq:bemut1}
and \eqref{eq:gamut1} preserve the duality
\eqref{eq:gb3} (see \cite[Section 3.3]{Nakanishi11c}).
Therefore, it is enough to prove \eqref{eq:bemut1}.
This can be done by case-check with respect to the
configuration of $\gamma_i$ and $\gamma_k$.
We only present the most typical case in Figure \ref{fig:hom1},
where $b_{ki}=1$.
Then, for $\ve=+$, $\gamma_i'=\gamma_i+\gamma_k$
by the deformation of cycles
in Figure \ref{fig:cycle1},
for $\ve=-$, $\gamma_i'=\gamma_i$,
and in either case, $\gamma_k'=-\gamma_k$.
This agrees with \eqref{eq:gamut1}.

\par
(b).
Again, it is enough to prove it for
the  simple paths.
The essential case is given in Figure \ref{fig:hom2},
where the label of the arcs is the same one
in Figure \ref{fig:pop4}.
\end{proof}

\begin{figure}
\begin{center}
\begin{pspicture}(-1.6,-0.5)(13,3.4)
\psset{unit=11.5mm}
\psset{linewidth=0.5pt}
\rput[c]{0}(0,3){\small $\oplus$}
\rput[c]{0}(1.6,1.8){\small $\oplus$}
\rput[c]{0}(1,0){\small $\oplus$}
\rput[c]{0}(-1,0){\small $\ominus$}
\rput[c]{0}(-1.6,1.8){\small $\ominus$}
\psset{fillstyle=none}
\psline[arrows=<-](2,1.6)(3,1.6)
\psline[arrows=->](2,1)(3,1)
\rput[c]{0}(2.5,1.9){\small $\mu_k^{(-)}$}
\rput[c]{0}(2.5,0.7){\small $\mu_k^{(+)}$}
\psline(-0.9,0)(0.9,0)
\psline(1.05,0.1)(1.55,1.7)
\psline(-1.05,0.1)(-1.55,1.7)
\psline(0.1,2.9)(1.5,1.85)
\psline(-0.1,2.9)(-1.5,1.85)
\psline(-1.5,1.8)(1.5,1.8)
\psline(-0.9,0.05)(1.5,1.75)
%
%
\rput[c]{0}(0,2.2){\small $\times$}
\rput[c]{0}(-0.5,1.1){\small $\times$}
\rput[c]{0}(0.6,0.6){\small $\times$}
\rput[c]{0}(-0.2,2){\small $a$}
\rput[c]{0}(-0.5,1.5){\small $b$}
\rput[c]{0}(0.4,0.8){\small $c$}
\rput[c]{0}(-0.8,2.1){\small $\gamma_i'$}
\rput[c]{0}(-1,1){\small $\gamma_k'$}
\psset{linewidth=1pt,linestyle=solid}
\pscurve(-0.5,1.1)(-0.4,1.2)(-0.3,1.1)(-0.2,1)(-0.1,1.1)(0.0,1.2)(0.1,1.1)
(0.2,1)(0.3,1.1)(0.4,1.2)(0.5,1.1)(0.6,1)(0.7,1.1)(0.8,1.2)(0.9,1.1)
(1.0,1)(1.1,1.1)(1.2,1.2)(1.3,1.1)(1.4,1)(1.5,1.1)(1.6,1.2)(1.7,1.1)
\pscurve(0.6,0.6)(0.5,0.5)(0.4,0.6)(0.3,0.7)(0.2,0.6)
(0.1,0.5)(0,0.6)(-0.1,0.7)(-0.2,0.6)(-0.3,0.5)(-0.4,0.6)
(-0.5,0.7)(-0.6,0.6)(-0.7,0.5)(-0.8,0.6)(-0.9,0.7)(-1.0,0.6)
(-1.1,0.5)(-1.2,0.6)(-1.3,0.7)(-1.4,0.6)(-1.5,0.5)(-1.6,0.6)
\pscurve(0,2.2)(0.1,2.3)(0.2,2.2)(0.3,2.1)(0.4,2.2)
(0.5,2.3)(0.6,2.2)(0.7,2.1)(0.8,2.2)(0.9,2.3)(1,2.2)
(1.1,2.1)(1.2,2.2)(1.3,2.3)(1.4,2.2)
\psset{linewidth=1pt,linestyle=solid}
\pscurve(0.2,2.2)(0.05,1.5)(-0.25,1.05)
\psset{linewidth=1pt,linestyle=dashed}
\psecurve(0.05,1.5)(0.2,2.2)(0,2.4)(-0.2,2.4)(-0.7,1.1)(-0.45,0.95)(-0.25,1.05)(0.05,1.5)
\psset{linewidth=1pt,linestyle=solid}
\psecurve(0,1.25)(0.32,1.12)(0.6,0.9)(0.8,0.65)(0.75,0.4)(-0.23,0.4)(-0.6,0.7)(-0.7,0.9)
\psset{linewidth=1pt,linestyle=dashed}
\psecurve(-0.23,0.4)(-0.6,0.7)(-0.7,0.9)(-0.55,1.3)(0,1.27)(0.32,1.12)(0.6,0.9)%
\psline(0.1,1.6)(0.07,1.8)
\psline(0.1,1.6)(0.21,1.78)
\psline(0.2,0.3)(0,0.23)
\psline(0.2,0.3)(0.03,0.42)
\psset{linewidth=0.5pt}
\psset{fillstyle=none}
\rput[c]{0}(5,3){\small $\oplus$}
\rput[c]{0}(6.6,1.8){\small $\oplus$}
\rput[c]{0}(6,0){\small $\oplus$}
\rput[c]{0}(4,0){\small $\ominus$}
\rput[c]{0}(3.4,1.8){\small $\ominus$}
\psset{linewidth=0.5pt,linestyle=solid}
\psline(4.1,0)(5.9,0)
\psline(6.05,0.1)(6.55,1.7)
\psline(3.95,0.1)(3.45,1.7)
\psline(5.1,2.9)(6.5,1.85)
\psline(4.9,2.9)(3.5,1.85)
\psline(3.5,1.8)(6.5,1.8)
\psline(5.9,0.05)(3.5,1.75)
%
%
\rput[c]{0}(7.5,1.9){\small $\mu_k^{(+)}$}
\psline[arrows=->](7,1.6)(8,1.6)
\rput[c]{0}(7.5,0.7){\small $\mu_k^{(-)}$}
\psline[arrows=<-](7,1)(8,1)
%
%
\rput[c]{0}(5,2.2){\small $\times$}
\rput[c]{0}(4.5,0.6){\small $\times$}
\rput[c]{0}(5.6,1.1){\small $\times$}
\rput[c]{0}(5.2,2){\small $a$}
\rput[c]{0}(5.5,1.5){\small $b$}
\rput[c]{0}(4.6,0.8){\small $c$}
\rput[c]{0}(4.2,2.1){\small $\gamma_i$}
\rput[c]{0}(4,1){\small $\gamma_k$}
\rput[c]{0}(0,-0.5){\small $G'$}
\rput[c]{0}(5,-0.5){\small $G$}
\rput[c]{0}(10,-0.5){\small $G'$}
\psset{linewidth=1pt,linestyle=solid}
\pscurve(4.5,0.6)(4.4,0.5)(4.3,0.6)(4.2,0.7)(4.1,0.6)
(4.0,0.5)(3.9,0.6)(3.8,0.7)(3.7,0.6)(3.6,0.5)(3.5,0.6)(3.4,0.7)(3.3,0.6)
\pscurve(5.6,1.1)(5.7,1.2)(5.8,1.1)(5.9,1.0)(6,1.1)
(6.1,1.2)(6.2,1.1)(6.3,1.0)(6.4,1.1)(6.5,1.2)(6.6,1.1)
\pscurve(5,2.2)(5.1,2.3)(5.2,2.2)(5.3,2.1)(5.4,2.2)
(5.5,2.3)(5.6,2.2)(5.7,2.1)(5.8,2.2)(5.9,2.3)(6,2.2)
(6.1,2.1)(6.2,2.2)(6.3,2.3)(6.4,2.2)
\psset{linewidth=1pt,linestyle=solid}
\psecurve(5,2.4)(5.4,2.1)(5.6,1.8)(5.8,1.1)(5.7,0.9)
\psset{linewidth=1pt,linestyle=dashed}
\psecurve(5.6,1.8)(5.4,2.1)(5,2.4)(4.75,2.35)(4.95,1.4)(5.35,1)(5.7,0.9)(5.8,1.1)(5.6,1.8)
\psset{linewidth=1pt,linestyle=solid}
\psecurve(5.6,1.3)(5.8,1.1)(5.65,0.75)(4.8,0.3)(4.3,0.35)(4.2,0.7)(4.8,1.2)
\psset{linewidth=1pt,linestyle=dashed}
\psecurve(4.3,0.35)(4.2,0.7)(4.8,1.2)(5.6,1.3)(5.8,1.1)(5.65,0.75)
\psset{linewidth=1pt,linestyle=solid}
\psline(5.7,1.6)(5.7,1.82)
\psline(5.7,1.6)(5.53,1.75)
\psline(5,0.35)(5.15,0.52)
\psline(5,0.35)(5.2,0.35)
\psset{linewidth=0.5pt,linestyle=solid}
\rput[c]{0}(10,3){\small $\oplus$}
\rput[c]{0}(11.6,1.8){\small $\oplus$}
\rput[c]{0}(11,0){\small $\oplus$}
\rput[c]{0}(9,0){\small $\ominus$}
\rput[c]{0}(8.4,1.8){\small $\ominus$}
\psset{fillstyle=none}
\psline(9.1,0)(10.9,0)
\psline(11.05,0.1)(11.55,1.7)
\psline(8.95,0.1)(8.45,1.7)
\psline(10.1,2.9)(11.5,1.85)
\psline(9.9,2.9)(8.5,1.85)
\psline(8.5,1.8)(11.5,1.8)
\psline(9.1,0.05)(11.5,1.75)
%
%
\rput[c]{0}(10,2.2){\small $\times$}
\rput[c]{0}(9.5,1.1){\small $\times$}
\rput[c]{0}(10.6,0.6){\small $\times$}
\rput[c]{0}(9.8,2){\small $a$}
\rput[c]{0}(9.5,1.5){\small $c$}
\rput[c]{0}(10.4,0.8){\small $b$}
\rput[c]{0}(9.2,2.1){\small $\gamma_i'$}
\rput[c]{0}(9,0.8){\small $\gamma_k'$}
\psset{linewidth=1pt,linestyle=solid}
\pscurve(9.5,1.1)(9.4,1)(9.3,1.1)(9.2,1.2)(9.1,1.1)
(9.0,1)(8.9,1.1)(8.8,1.2)(8.7,1.1)(8.6,1)(8.5,1.1)(8.4,1.2)(8.3,1.1)
\pscurve(10.6,0.6)(10.7,0.7)(10.8,0.6)(10.9,0.5)(11,0.6)
(11.1,0.7)(11.2,0.6)(11.3,0.5)(11.4,0.6)(11.5,0.7)(11.6,0.6)
\pscurve(10,2.2)(10.1,2.3)(10.2,2.2)(10.3,2.1)(10.4,2.2)
(10.5,2.3)(10.6,2.2)(10.7,2.1)(10.8,2.2)(10.9,2.3)(11,2.2)
(11.1,2.1)(11.2,2.2)(11.3,2.3)(11.4,2.2)
\psset{linewidth=1pt,linestyle=solid}
\pscurve(10.2,2.2)(10.05,1.5)(9.75,1.05)(9.55,0.95)(9.3,1.1)
\psset{linewidth=1pt,linestyle=dashed}
\psecurve(10.05,1.5)(10.2,2.2)(10,2.4)(9.8,2.4)(9.3,1.1)(9.55,0.95)
\psset{linewidth=1pt,linestyle=solid}
\psecurve(9.4,0.8)(9.3,1.1)(9.5,1.3)(10,1.25)(10.6,0.9)(10.8,0.6)(10.7,0.4)
\psset{linewidth=1pt,linestyle=dashed}
\psecurve(10.6,0.9)(10.8,0.6)(10.7,0.4)(9.4,0.8)(9.3,1.1)(9.5,1.3)
\psset{linewidth=1pt,linestyle=solid}
\psline(10.1,1.6)(10.07,1.8)
\psline(10.1,1.6)(10.21,1.78)
\psline(10.2,1.18)(10,1.15)
\psline(10.2,1.18)(10.03,1.33)
\end{pspicture}
\end{center}
\caption{Mutation of simple cycles for signed flips.}
\label{fig:hom1}
\end{figure}
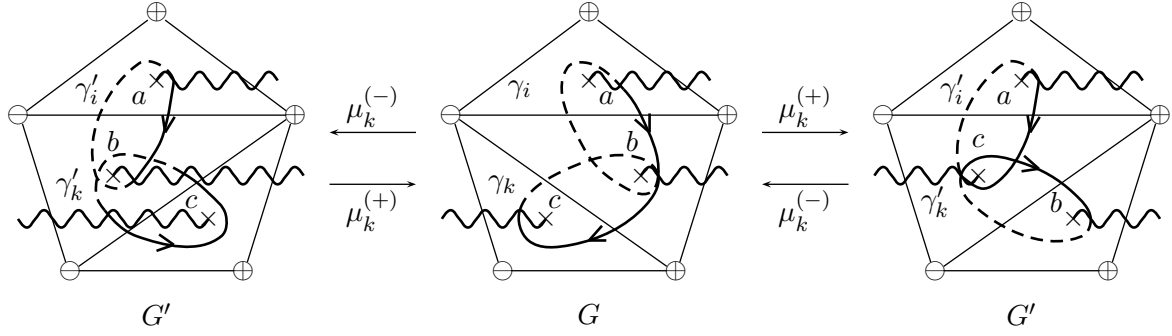

\begin{figure}
\begin{center}
\begin{pspicture}(-1.4,-0.5)(7,4.7)
\psset{unit=11.5mm}
\psset{linewidth=0.5pt}
\rput[c]{0}(0,0){\small $\oplus$}
\rput[c]{0}(0,2){\small $\oplus$}
\rput[c]{0}(0,4){\small $\ominus$}
\rput[c]{0}(0,3.2){\small $\times$}
\rput[c]{0}(0,1.2){\small $\times$}
\rput[c]{0}(-1.2,2.2){\small $\gamma_i{}$}
\rput[c]{0}(1.2,2.2){\small $\gamma_j{}$}

\psset{linewidth=0.5pt,linestyle=solid}
\pscurve(0.1,0.05)(0.5,0.4)(0.9,1.6)(0.5,2.8)(0,3.2)
\pscurve(-0.1,0.05)(-0.5,0.4)(-0.9,1.6)(-0.5,2.8)(-0,3.2)

\pscurve(0,1.2)(0.4,1.4)(0.35,2.2)(0,2.4)(-0.5,1.8)(-0.1,0.1)

\pscurve(0,1.2)(-0.2,1.5)(-0.2,1.7)(-0.1,1.95)
\psline(0,0.1)(0,1.2)
\psline(0,3.2)(0,3.9)
\psset{linewidth=1pt,linestyle=solid}
\pscurve(0,1.2)(-0.1,1.1)(-0.2,1.2)(-0.3,1.3)(-0.4,1.2)
(-0.5,1.1)(-0.6,1.2)(-0.7,1.3)(-0.8,1.2)(-0.9,1.1)(-1.0,1.2)
(-1.1,1.3)(-1.2,1.2)(-1.3,1.1)(-1.4,1.2)
\pscurve(0,3.2)(-0.1,3.3)(-0.2,3.2)(-0.3,3.1)(-0.4,3.2)
(-0.5,3.3)(-0.6,3.2)(-0.7,3.1)(-0.8,3.2)(-0.9,3.3)(-1,3.2)
(-1.1,3.1)(-1.2,3.2)(-1.3,3.3)(-1.4,3.2)
\psecurve[linewidth=1pt,linestyle=solid]
(-0.75,2.6)(-0.4,3.2)(-0.2,3.35)(0,3.4)(0.15,3.3)(0.15,3.1)(0,2.95)(-0.4,2.2)(0,1.45)
(0.15,1.3)(0.15,1.1)(0,1)(-0.2,1.05)(-0.4,1.2)(-0.75,1.8)
\psecurve[linewidth=1pt,linestyle=dashed]
(-0.2,1.05)(-0.4,1.2)(-0.75,1.8)(-0.8,2.2)(-0.75,2.6)(-0.4,3.2)(-0.2,3.35)
\psecurve[linewidth=1pt,linestyle=solid]
(-0.17,1.3)(-0.2,1.2)(-0.17,1.1)(0,1)(0.2,1.05)(0.4,1.2)(0.75,1.8)(0.8,2.2)(0.75,2.6)(0.4,3.2)(0.2,3.35)(0,3.4)(-0.17,3.3)(-0.2,3.2)(-0.17,3.1)
\psecurve[linewidth=1pt,linestyle=dashed]
(-0.17,3.3)(-0.2,3.2)(-0.17,3.1)(0,2.95)(0.4,2.2)(0,1.45)
(-0.17,1.3)(-0.2,1.2)(-0.17,1.1)

\psset{linewidth=1pt,linestyle=solid}
\psline(-0.4,2.2)(-0.3,2.4)
\psline(-0.4,2.2)(-0.5,2.4)
\psline(0.8,2.2)(0.7,2.4)
\psline(0.8,2.2)(0.9,2.4)
\rput[c]{0}(0,-0.5){\small $G$}
%
\psset{fillstyle=none}
\psline[arrows=->](2,2.3)(3,2.3)
\psline[arrows=<-](2,1.7)(3,1.7)
\rput[c]{0}(2.5,2.6){\small $\kappa_p^{(+)}$}
\rput[c]{0}(2.5,1.4){\small $\kappa_p^{(-)}$}
%
\psset{linewidth=0.5pt}
\rput[c]{0}(5,0){\small $\oplus$}
\rput[c]{0}(5,2){\small $\ominus$}
\rput[c]{0}(5,4){\small $\ominus$}
\rput[c]{0}(5,3.2){\small $\times$}
\rput[c]{0}(5,1.2){\small $\times$}
\rput[c]{0}(3.8,2.2){\small $\gamma_i'$}
\rput[c]{0}(6.2,2.2){\small $\gamma_j'$}

\psset{linewidth=0.5pt,linestyle=solid}
\pscurve(5.1,0.05)(5.5,0.4)(5.9,1.6)(5.5,2.8)(5,3.2)
\pscurve(4.9,0.05)(4.5,0.4)(4.1,1.6)(4.5,2.8)(5,3.2)
\pscurve(5,1.2)(4.6,1.4)(4.65,2.2)(5,2.4)(5.5,1.8)(5.1,0.1)
\pscurve(5,1.2)(5.2,1.5)(5.2,1.7)(5.1,1.95)
\psline(5,0.1)(5,1.2)
\psline(5,3.2)(5,3.9)
\psset{linewidth=1pt,linestyle=solid}
\pscurve(5,1.2)(4.9,1.1)(4.8,1.2)(4.7,1.3)(4.6,1.2)
(4.5,1.1)(4.4,1.2)(4.3,1.3)(4.2,1.2)(4.1,1.1)(4,1.2)
(3.9,1.3)(3.8,1.2)(3.7,1.1)(3.6,1.2)
\pscurve(5,3.2)(4.9,3.3)(4.8,3.2)(4.7,3.1)(4.6,3.2)
(4.5,3.3)(4.4,3.2)(4.3,3.1)(4.2,3.2)(4.1,3.3)(4,3.2)
(3.9,3.1)(3.8,3.2)(3.7,3.3)(3.6,3.2)
\psecurve[linewidth=1pt,linestyle=solid]
(4.25,2.6)(4.6,3.2)(4.8,3.35)(5,3.4)(5.15,3.3)(5.15,3.1)(5,2.95)(4.6,2.2)(5,1.45)
(5.15,1.3)(5.15,1.1)(5,1)(4.8,1.05)(4.6,1.2)(4.25,1.8)
\psecurve[linewidth=1pt,linestyle=dashed]
(4.8,1.05)(4.6,1.2)(4.25,1.8)(4.2,2.2)(4.25,2.6)(4.6,3.2)(4.8,3.35)
\psecurve[linewidth=1pt,linestyle=solid]
(4.83,1.3)(4.8,1.2)(4.83,1.1)(5,1)(5.2,1.05)(5.4,1.2)(5.75,1.8)(5.8,2.2)(5.75,2.6)(5.4,3.2)(5.2,3.35)(5,3.4)(4.83,3.3)(4.8,3.2)(4.83,3.1)
\psecurve[linewidth=1pt,linestyle=dashed]
(4.83,3.3)(4.8,3.2)(4.83,3.1)(5,2.95)(5.4,2.2)(5,1.45)
(4.83,1.3)(4.8,1.2)(4.83,1.1)

\psset{linewidth=1pt,linestyle=solid}
\psline(4.6,2.2)(4.7,2.4)
\psline(4.6,2.2)(4.5,2.4)
\psline(5.8,2.2)(5.7,2.4)
\psline(5.8,2.2)(5.9,2.4)
\rput[c]{0}(5,-0.5){\small $G'$}
\end{pspicture}
\end{center}
\caption{Mutation of simple cycles for signed pops.}
\label{fig:hom2}
\end{figure}
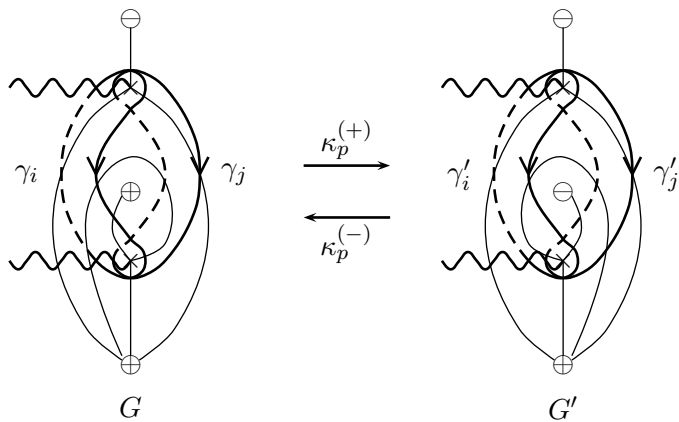

Proposition \ref{prop:cycle1}  tells that
the simple paths and the simple cycles mutate
as the (logarithm of) 
monomial $x$-variables and monomial $y$-variables
in Section \ref{subsec:monomial},
respectively.
The formula for the simple cycles already appeared 
in \cite[Lemma 9.11]{Bridgeland13} as the mutation of the modified basis therein.

\subsection{Periodicity of signed mutations and signed flips}


Let $G^0=G(\phi)$ be a labeled  Stokes graph taken as the ``initial" one,
let $T^0$ be the associated labeled Stokes triangulation
of $(\bfS(\phi),
\bfM(\phi),\bfA(\phi))$,
and let $B^0=B(T^0)$ be the adjacency matrix of $T^0$.
For 
the cluster algebra $\mathcal{A}(B^0,x^0,y^0;\bbQ_+(y^0))$,
suppose that  $\vec{k}=(k_1,\dots,k_N)$
is a $\nu$-period of  the initial seed $(B^0,x^0,y^0)$;
namely,
we have a  mutation sequence
\begin{align}
\label{eq:seedseq0}
\begin{split}
(B(1),x(1),y(1))=(B^0,x^0,y^0) \buildrel{\mu_{k_1}} \over{\rightarrow}
&(B(2),x(2),y(2)) \buildrel{\mu_{k_2}} \over{\rightarrow}
\cdots\\
&\cdots
\buildrel{\mu_{k_N}} \over{\rightarrow}
(B(N+1),x(N+1),y(N+1))
\end{split}
\end{align}
such that $y_{\nu(i)}(N+1)=y^0_i$.
Let $\varepsilon_t = \varepsilon (y_{k_t}(t))$ ($t=1,\dots,N$)
be the tropical sign of $y_{k_t}(t)$ with respect to the initial
$y$-variables $y^0=y(1)$.

To make use of such a sequence in our application,
 we need to reexpress it by
signed mutations and signed pops.
For that purpose,
we  introduce 
a mutation sequence 
of    labeled Stokes triangulations of $(\bfS(\phi),
\bfM(\phi),\bfA(\phi))$,
\begin{align}
\label{eq:Stokesseq2}
\begin{split}
&T(1)=T^0
\buildrel{\tilde{\mu}^{(\ve_1)}_{k_1}} \over{\rightarrow}
 T(2)\buildrel{\tilde{\mu}^{(\ve_2)}_{k_2}} \over{\rightarrow}
\cdots
\buildrel{\tilde{\mu}^{(\ve_N)}_{k_N}} \over{\rightarrow}
T(N+1)
 \buildrel{\tilde{\kappa}} \over{\rightarrow}
T(N+2).
\end{split}
\end{align}
Here, for $T(t)=(\alpha_i(t) )_{i=1}^n$,
we set
\begin{align}
\label{eq:mu1}
 \tilde{\mu}^{(\ve_t)}_{k_t}=
 \begin{cases}
 \mu^{(\ve_t)}_{k_t} &
 \text{if 
$\alpha_{k_t}(t)$ in $T(t)$ is  flippable}
 \\
\mu^{(\ve_t)}_{k_t}\circ
\kappa_{p_t}^{(s_t)}
& \text{otherwise},
\end{cases}
\end{align}
where $p_t$ is the puncture inside the self-folded triangle in $T(t)$
which $\alpha_{k_t}(t)$ belongs to,
and the signs $s_t\in \{+,-\}$ are determined by the following rule:
\begin{itemize}
\item  For $t=1,\dots,N$, let
\begin{align}
n(t):=|\{ t'   \mid 1\leq t'<t, \ p_{t'}=p_t\}|,
\end{align}
and we set $s_t=(-1)^{n(t)}$.
Namely, if a puncture $p=p_t$ appears at the first time
in the sequence
\eqref{eq:Stokesseq2} at $t$,
then $s_t=+$, and the sign for $p$ alternates after that.
\end{itemize}
Also, we set
\begin{align}
\label{eq:kappa1}
\tilde{\kappa}=\prod_{p}{}'
\kappa_p^{(-)},
\end{align}
where the product runs over the punctures $p$ such that
$|\{ 1\leq t\leq N | p_t=p \}|$, i.e., the number of the appearance of $p$
in the sequence \eqref{eq:Stokesseq2}, is odd.

\begin{rem}
If we project the sequence  \eqref{eq:Stokesseq2} to the
sequence of labeled signed triangulations in Section
\ref{subsec:reformulation}
by forgetting the signs of flips and pops,
the periodicity of
the sequence \eqref{eq:seedseq0}
and Theorem \ref{thm:period2}
imply that the labeled signed triangulation at $t=N+1$, say, $T'(N+1)$  is
{\em pop-equivalent\/} to the initial one $T'(1)$.
Meanwhile, if $|\{ 1\leq t\leq N | p_t=p \}|$ is odd for a puncture $p$,
then the sign $\sigma_p$ at $p$
of $T'(N+1)$ and $T'(1)$ is opposite.
These together mean that $p$ is inside the self-folded triangle
in $T'(N+1)$ so that it should be popped;
 therefore, it is also so in $T(N+1)$.
This guarantees that the signed pops of $\tilde{\kappa}$ in \eqref{eq:kappa1}
is well-defined.
\end{rem}

First, let us consider a mutations sequence of labeled extended  seeds  which is parallel to 
\eqref{eq:Stokesseq2},
\begin{align}
\label{eq:seedseq1}
\begin{split}
(B(1),&x(1),y(1),\tilde{y}(1))=(B^0,x^0,y^0,\tilde{y}^0)
\buildrel{\tilde{\mu}^{(\ve_1)}_{k_1}} \over{\rightarrow}
 (B(2),x(2),y(2),\tilde{y}(2)) \buildrel{\tilde{\mu}^{(\ve_2)}_{k_2}} \over{\rightarrow}
\cdots\\
&\cdots
\buildrel{\tilde{\mu}^{(\ve_N)}_{k_N}} \over{\rightarrow}
 (B(N+1),\dots,\tilde{y}(N+1))
 \buildrel{\tilde{\kappa}} \over{\rightarrow}
 (B(N+2),\dots,\tilde{y}(N+2)),
\end{split}
\end{align}
where  $\tilde{\mu}^{(\ve_t)}_{k_t}$ and $\tilde{\kappa}$
are the ones in \eqref{eq:mu1} and \eqref{eq:kappa1}
but acting on extended seeds.

\begin{prop}
\label{prop:period3}
The mutation sequence \eqref{eq:seedseq1} of extended seeds
is $\nu$-periodic in the following sense:
\begin{align}
\begin{split}
b_{\nu(i)\nu(j)}(N+2)&=b_{ij}(1),\quad
x_{\nu(i)}(N+2)=x_i(1),\\
y_{\nu(i)}(N+2)&=y_i(1),
\quad
\tilde{y}_{i}(N+2)=\tilde{y}_i(1).
\end{split}
\end{align}
\end{prop}
\begin{proof}
It is enough to show the periodicity of $\tilde{y}$- and $x$-variables.
First, let us  show the periodicity of $\tilde{y}$-variables.
By the definition of $\tilde{\kappa}$ in \eqref{eq:kappa1},
for each puncture $p$,
the numbers of $\kappa^{(+)}_{p}$ and 
 $\kappa^{(-)}_{p}$ appearing in
 the sequence \eqref{eq:seedseq1} are equal.
Then, the periodicity of $x$-variables
follows from Proposition \ref{prop:local1}
and the property $\kappa_{p}^{\pm}\kappa_{p}^{\mp}=1$.
It also implies $\tilde{y}_{p}(N+2)=\tilde{y}_p(1)$ by 
 \eqref{eq:vmut1}.
\end{proof}

Next, consider a   mutation sequence
of  labeled  Stokes graphs,
\begin{align}
\label{eq:Stokesseq0}
\begin{split}
 G(1)=G
\buildrel{\tilde{\mu}^{(\ve_1)}_{k_1}} \over{\rightarrow}
G(2)\buildrel{\tilde{\mu}^{(\ve_2)}_{k_2}} \over{\rightarrow}
\cdots
\buildrel{\tilde{\mu}^{(\ve_N)}_{k_N}} \over{\rightarrow}
G(N+1)
 \buildrel{\tilde{\kappa}} \over{\rightarrow}
G(N+2),
\end{split}
\end{align}
where  $\tilde{\mu}^{(\ve_t)}_{k_t}$ and $\tilde{\kappa}$
are the same ones in \eqref{eq:mu1} and \eqref{eq:kappa1}
but for labeled Stokes graphs.

Let $\beta_i(t)$ and $\gamma_i(t)$ be the simple paths
and the simple cycles of $G(t)$.
Let
\begin{align}
\tau_{G(t),G(t+1)}^{\vee}:\Gamma^{\vee}_{G(t+1)}
\overset{\sim}{\rightarrow} \Gamma^{\vee}_{G(t)},
\quad
\tau_{G(t),G(t+1)}:\Gamma_{G(t+1)}
\overset{\sim}{\rightarrow} \Gamma_{G(t)}
\end{align}
be the isomorphisms
given by Proposition
\ref{prop:cycle1}
for the sequence \eqref{eq:Stokesseq0},
and let
 \begin{align}
\tau_{G(1),G(N+2)}^{\vee}:\Gamma^{\vee}_{G(N+2)}
\overset{\sim}{\rightarrow} \Gamma^{\vee}_{G(1)},
\quad
\tau_{G(1),G(N+2)}:\Gamma_{G(N+2)}
\overset{\sim}{\rightarrow} \Gamma_{G(1)}
\end{align}
be their compositions.

Then, we have the following $\nu$-periodicity of the simple paths and the simple cycles
for the sequence \eqref{eq:Stokesseq0}.
\begin{prop}
\label{prop:bg1}
\begin{align}
\label{eq:beta1}
\tau_{G(1),G(N+2)}^{\vee}(\beta_{\nu(i)}(N+2))&= \beta_i(1),\\
\label{eq:gamma3}
\tau_{G(1),G(N+2)}( \gamma_{\nu(i)}(N+2))&= \gamma_i(1).
\end{align}
\end{prop}
\begin{proof}
Let us first show \eqref{eq:gamma3}.
By Proposition \ref{prop:cycle1}
and the choice of the signs $\ve_t$,
the simple cycles exactly transform as
the logarithm of the tropical $y$-variables (see \eqref{eq:ymut5})
for the sequence \eqref{eq:seedseq0}.
Therefore, we have the periodicity  \eqref{eq:gamma3}.
The periodicity \eqref{eq:beta1} follows from
\eqref{eq:gamma3} and the duality in Proposition \ref{prop:dual1}.
\end{proof}

To conclude the section, let us remark that Proposition \ref{prop:period3} and
Conjecture
\ref{conj:bij1} together imply
 the following $\nu$-periodicity of the labeled Stokes triangulations
for the sequence \eqref{eq:Stokesseq2}.
\begin{conj}
\label{conj:period1}
 As arcs in $(\bfS(\phi),
\bfM(\phi),\bfA(\phi))$,
\label{prop:Stokes1}
 \begin{align}
 \label{eq:per2}
\alpha_{\nu(i)}(N+2)&=\alpha_i(1).
\end{align}
\end{conj}

We have a partial result on the conjecture.

\begin{prop}
\label{prop:alpha5}
 The equality \eqref{eq:per2} is true for any $i$
such that both endpoints of $\alpha_i(1)$ are on the boundary
of $\bfS(\phi)$.
In particular, Conjecture \ref{conj:period1} is true if
$(\bfS(\phi),
\bfM(\phi))$ have no punctures.
\end{prop}

\begin{proof}
We have the following facts.

Fact 1. Let us temporarily forget the signs
of all signed flips and pops in the sequence
\eqref{eq:Stokesseq2}, and regard it as a sequence of
signed ideal triangulations in Section
\ref{subsec:reformulation}.
Then, 
by Theorem \ref{thm:period2},
we deduce that the equality \eqref{eq:per2} holds {\em as arcs in 
$(\bfS(\phi),
\bfM(\phi))$}.

Fact 2. The group
$\mathrm{Sym}(H_1(\hat{\Sigma}\setminus \hat{P}_0,\hat{P}_{\infty}))$
in \eqref{eq:sym1} is generated by two types of elements of
$H_1(\hat{\Sigma}\setminus \hat{P}_0,\hat{P}_{\infty})$;
one is $\gamma + \gamma*$, where $\gamma$ is any path
connecting two points in $\hat{P}_{\infty}$,
and the other is any cycle (= closed path) in $\hat{\Sigma}
\setminus \hat{P}_{0}$  around a point in $\hat{P}_0$.

Fact 3. By \eqref{eq:beta1}, the periodicity
$\tau_{G(1),G(N+2)}^{\vee}(\beta_{\nu(i)}(N+2))= \beta_i(1)$ holds.

Fact 4. By the definition of arcs, 
neither $\alpha_{\nu(i)}(N+2)$
nor  $\alpha_i(1)$ intersects itself.

These facts, together with the assumption that
both endpoints of $\alpha_i(1)$ are on the boundary
of $\bfS(\phi)$,
imply the equality \eqref{eq:per2}
{\em as arcs in $(\bfS(\phi),
\bfM(\phi),\bfA(\phi))$}.
\end{proof}

\begin{rem} On the other hand,
if an arc $\alpha_i$ has an endpoint at a puncture,
one cannot  deduce the equality \eqref{eq:per2} as above
due to the deformation of paths in Figure \ref{fig:path1}.
In the above proof, such a deformation does not occur
because of the presence of the boundary at each endpoint.
\end{rem}

\begin{ex} [Pentagon relation (4)]
\label{ex:pentagon4}
Let observe how the periodicity 
\eqref{eq:per2} occurs in our running 
Examples \ref{ex:pentagon1},
\ref{ex:pentagon2}, and \ref{ex:pentagon3}.
Let us take the labeled Stokes graph in
Figure \ref{fig:pentagon2} (a) as the initial labeled Stokes graph.
We apply the mutation sequence $\vec{k}=(1,2,1,2,1)$ in 
Examples \ref{ex:pentagon1} and 
\ref{ex:pentagon2}.
According to \eqref{eq:tropsign2},
we take the sequence of the tropical signs $\vec{\ve}=
(+,+,+,-,-)$.
Then, the pentagon relations of
the corresponding labeled Stokes triangulations
and labeled Stokes graphs are presented
in Figures \ref{fig:pentagon4} and 
\ref{fig:pentagon5},
respectively.
Note that the boundary trajectories 
in Figure \ref{fig:pentagon5} vanish 
in $\Gamma$ and $\Gamma^{\vee}$.
\end{ex}

\begin{figure}
\begin{center}
\begin{pspicture}(-1,-1)(13,4.2)
%
\psset{linewidth=0.5pt}
\psset{fillstyle=solid, fillcolor=black}
\pscircle(0,4){0.08} 
\pscircle(0.95,3.31){0.08} 
\pscircle(0.59,2.19){0.08} 
\pscircle(-0.59,2.19){0.08} 
\pscircle(-0.95,3.31){0.08} 
\psset{fillstyle=none}
\psline(0,4)(0.95,3.31)
\psline(0.95,3.31)(0.59,2.19)
\psline(0.59,2.19)(-0.59,2.19)
\psline(-0.59,2.19)(-0.95,3.31)
\psline(-0.95,3.31)(0,4)
\psline(0,4)(0.59,2.19)
\psline(0,4)(-0.59,2.19)
\rput[c]{0}(2.5,3.6){\small $\mu_1^{(+)}$}
\psline[arrows=->](2,3)(3,3)
\rput[c]{0}(-0.2,3){\small 1}
\rput[c]{0}(0.2,3){\small 2}
\rput[c]{0}(0,2.8){\small $\times$}
\rput[c]{0}(-0.6,3.2){\small $\times$}
\rput[c]{0}(0.6,3.2){\small $\times$}
%
\psset{linewidth=0.5pt}
\psset{fillstyle=solid, fillcolor=black}
\pscircle(5,4){0.08} 
\pscircle(5.95,3.31){0.08} 
\pscircle(5.59,2.19){0.08} 
\pscircle(4.41,2.19){0.08} 
\pscircle(4.05,3.31){0.08} 
\psset{fillstyle=none}
\psline(5,4)(5.95,3.31)
\psline(5.95,3.31)(5.59,2.19)
\psline(5.59,2.19)(4.41,2.19)
\psline(4.41,2.19)(4.05,3.31)
\psline(4.05,3.31)(5,4)
\psline(5,4)(5.59,2.19)
\pscurve(4.05,3.31)(4.3,3)(5,3)(5.59,2.19)

\rput[c]{0}(7.5,3.5){\small $\mu_2^{(+)}$}
\psline[arrows=->](7,3)(8,3)
\rput[c]{0}(4.8,2.8){\small 1}
\rput[c]{0}(5.5,3){\small 2}
 \rput[c]{0}(5,2.8){\small $\times$}
\rput[c]{0}(4.4,3.2){\small $\times$}
\rput[c]{0}(5.6,3.2){\small $\times$}
%
\psset{linewidth=0.5pt}
\psset{fillstyle=solid, fillcolor=black}
\pscircle(10,4){0.08} 
\pscircle(10.95,3.31){0.08} 
\pscircle(10.59,2.19){0.08} 
\pscircle(9.41,2.19){0.08} 
\pscircle(9.05,3.31){0.08} 
\psset{fillstyle=none}
\psline(10,4)(10.95,3.31)
\psline(10.95,3.31)(10.59,2.19)
\psline(10.59,2.19)(9.41,2.19)
\psline(9.41,2.19)(9.05,3.31)
\psline(9.05,3.31)(10,4)
%
\pscurve(9.05,3.31)(9.3,3.05)(10.5,3.45)(10.95,3.31)
\pscurve(9.05,3.31)(9.3,3)(10,3)(10.59,2.19)

\rput[c]{0}(12.5,3.5){\small $\mu_1^{(+)}$}
\psline[arrows=->](12,3)(13,3)
\rput[c]{0}(9.8,2.8){\small 1}
\rput[c]{0}(10,3.5){\small 2}
\rput[c]{0}(10,2.8){\small $\times$}
\rput[c]{0}(9.4,3.2){\small $\times$}
\rput[c]{0}(10.6,3.2){\small $\times$}
\psset{linewidth=0.5pt}
\psset{fillstyle=solid, fillcolor=black}
\pscircle(0,1){0.08} 
\pscircle(0.95,0.31){0.08} 
\pscircle(0.59,-0.81){0.08} 
\pscircle(-0.59,-0.81){0.08} 
\pscircle(-0.95,0.31){0.08} 
\psset{fillstyle=none}
\psline(0,1)(0.95,0.31)
\psline(0.95,0.31)(0.59,-0.81)
\psline(0.59,-0.81)(-0.59,-0.81)
\psline(-0.59,-0.81)(-0.95,0.31)
\psline(-0.95,0.31)(0,1)
\pscurve(-0.95,0.31)(-0.7,0.05)(0.5,0.45)(0.95,0.31)
\pscurve(-0.59,-0.81)(-0,-0.5)(0.5,0.35)(0.95,0.31)
%
\rput[c]{0}(2.5,0.5){\small $\mu_2^{(-)}$}
\psline[arrows=->](2,0)(3,0)
\rput[c]{0}(0.3,-0.4){\small 1}
\rput[c]{0}(0,0.5){\small 2}
\rput[c]{0}(0,-0.2){\small $\times$}
\rput[c]{0}(-0.6,0.2){\small $\times$}
\rput[c]{0}(0.6,0.2){\small $\times$}
\psset{linewidth=0.5pt}
\psset{fillstyle=solid, fillcolor=black}
\pscircle(5,1){0.08} 
\pscircle(5.95,0.31){0.08} 
\pscircle(5.59,-0.81){0.08} 
\pscircle(4.41,-0.81){0.08} 
\pscircle(4.05,0.31){0.08} 
\psset{fillstyle=none}
\psline(5,1)(5.95,0.31)
\psline(5.95,0.31)(5.59,-0.81)
\psline(5.59,-0.81)(4.41,-0.81)
\psline(4.41,-0.81)(4.05,0.31)
\psline(4.05,0.31)(5,1)
\pscurve(4.41,-0.81)(5,-0.5)(5.5,0.35)(5.95,0.31)
\psline(5,1)(4.41,-0.81)
%
\rput[c]{0}(7.5,0.5){\small $\mu_1^{(-)}$}
\psline[arrows=->](7,0)(8,0)
\rput[c]{0}(4.8,0){\small 2}
\rput[c]{0}(5.3,-0.4){\small 1}
\rput[c]{0}(5,-0.2){\small $\times$}
\rput[c]{0}(4.4,0.2){\small $\times$}
\rput[c]{0}(5.6,0.2){\small $\times$}
%

\psset{linewidth=0.5pt}
\psset{fillstyle=solid, fillcolor=black}
\pscircle(10,1){0.08} 
\pscircle(10.95,0.31){0.08} 
\pscircle(10.59,-0.81){0.08} 
\pscircle(9.41,-0.81){0.08} 
\pscircle(9.05,0.31){0.08} 
\psset{fillstyle=none}
\psline(10,1)(10.95,0.31)
\psline(10.95,0.31)(10.59,-0.81)
\psline(10.59,-0.81)(9.41,-0.81)
\psline(9.41,-0.81)(9.05,0.31)
\psline(9.05,0.31)(10,1)
\psline(10,1)(9.41,-0.81)
\psline(10,1)(10.59,-0.81)
\rput[c]{0}(9.8,0){\small 2}
\rput[c]{0}(10.2,0){\small 1}
\rput[c]{0}(10,-0.2){\small $\times$}
\rput[c]{0}(9.4,0.2){\small $\times$}
\rput[c]{0}(10.6,0.2){\small $\times$}
\end{pspicture}
\end{center}
\caption{Pentagon relation of labeled Stokes triangulations.}
\label{fig:pentagon4}
\end{figure}
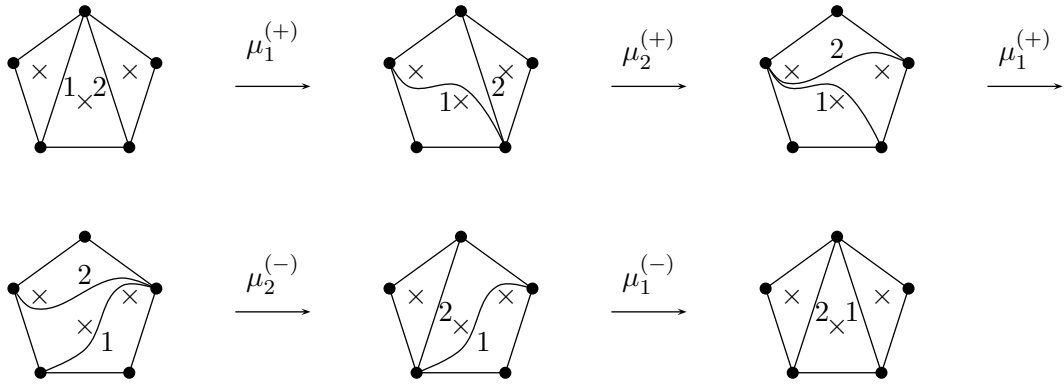


\begin{figure}
\begin{center}
\begin{pspicture}(0,-0.4)(14,6.6)
\psset{linewidth=0.5pt}
%
\psset{fillstyle=none}
\psline(1,6.4)(1,4.5)
\psline(1.6,5.4)(2.4,5.4)
\psline(0.4,5.4)(-0.4,5.4)
\pscurve(1.8,3.8)(1.4,4.2)(1,4.5)
\pscurve(0.2,3.8)(0.6,4.2)(1,4.5)
\pscurve(1.1,6.4)(1.3,5.8)(1.6,5.4)
\pscurve(0.9,6.4)(0.7,5.8)(0.4,5.4)
\pscurve(1.6,5.4)(1.7,4.5)(1.9,3.8)
\pscurve(0.4,5.4)(0.3,4.5)(0.1,3.8)
\pscurve[linestyle=dashed](1.05,6.4)(1.3,5)(1.85,3.8)
\pscurve[linestyle=dashed](0.95,6.4)(0.7,5)(0.15,3.8)
%
%
\rput[c]{0}(0.4,5.4){\small $\times$}
\rput[c]{0}(1,4.5){\small $\times$}
\rput[c]{0}(1.6,5.4){\small $\times$}
\rput[c]{0}(0.1,5.1){\small $a$}
\rput[c]{0}(1,4.1){\small $b$}
\rput[c]{0}(1.9,5.1){\small $c$}
\rput[c]{0}(0.77,4.7){\small $1$}
\rput[c]{0}(1.22,4.7){\small $2$}
\pscurve[linestyle=dashed](2.5,5.1)(2.1,4.6)(2,4)
\pscurve[linestyle=dashed](-0.5,5.1)(-0.1,4.6)(0,4)
\pscurve[linestyle=dashed](0.45,3.6)(1,3.75)(1.55,3.6)
\pscurve[linestyle=dashed](1.3,6.4)(1.75,5.95)(2.4,5.7)
\pscurve[linestyle=dashed](0.7,6.4)(0.25,5.95)(-0.4,5.7)
\rput[c]{0}(3.5,5.6){\small $\mu_1^{(+)}$}
\psline[arrows=->](3,5)(4,5)
\psset{fillstyle=none}
\psline(5.7,5.4)(6.8,3.8)
\psline(5.9,4.3)(5.3,3.75)
\psline(6.6,4.8)(7.4,5.2)
\pscurve(6,6.4)(5.9,5.75)(5.7,5.4)
\pscurve(4.6,5.4)(5.3,5.35)(5.7,5.4)
\pscurve(5.9,4.3)(5.2,4.95)(4.6,5.3)
\pscurve(6.6,4.8)(6.25,5.6)(6.1,6.4)
\pscurve(6.9,3.85)(6.7,4.3)(6.6,4.8)
\pscurve(6.7,3.75)(6.4,4)(5.9,4.3)
\pscurve[linestyle=dashed](6.85,3.8)(6.2,5.05)(6.05,6.4)
\pscurve[linestyle=dashed](6.75,3.75)(5.8,4.8)(4.6,5.35)
%
%
\rput[c]{0}(6.6,4.8){\small $\times$}
\rput[c]{0}(5.9,4.3){\small $\times$}
\rput[c]{0}(5.7,5.4){\small $\times$}
\rput[c]{0}(6.9,4.7){\small $c$}
\rput[c]{0}(5.5,5.7){\small $a$}
\rput[c]{0}(5.9,4){\small $b$}
\rput[c]{0}(6,5.4){\small $2$}
\rput[c]{0}(5.6,5.1){\small $1$}
\pscurve[linestyle=dashed](7.5,5.1)(7.1,4.6)(7,4)
\pscurve[linestyle=dashed](4.5,5.1)(4.9,4.6)(5,4)
\pscurve[linestyle=dashed](5.45,3.6)(6,3.75)(6.55,3.6)
\pscurve[linestyle=dashed](6.3,6.4)(6.75,5.95)(7.4,5.7)
\pscurve[linestyle=dashed](5.7,6.4)(5.25,5.95)(4.6,5.7)
\rput[c]{0}(8.5,5.6){\small $\mu_2^{(+)}$}
\psline[arrows=->](8,5)(9,5)
\psset{fillstyle=none}
\psline(11.4,4.8)(9.6,5.3)
\psline(10.4,4.5)(10.17,3.8)
\psline(10.7,5.6)(10.95,6.4)
\pscurve(11.4,4.8)(11.5,4.3)(11.8,3.75)
\pscurve(11.4,4.8)(11.8,5.1)(12.4,5.35)
\pscurve(10.4,4.5)(11.2,4.1)(11.7,3.65)
\pscurve(10.7,5.6)(11.6,5.45)(12.4,5.45)
\pscurve(10.4,4.5)(9.9,5)(9.55,5.2)
\pscurve(10.7,5.6)(10.05,5.42)(9.65,5.4)
\pscurve[linestyle=dashed](11.75,3.7)(10.9,4.65)(9.6,5.25)
\pscurve[linestyle=dashed](12.4,5.4)(11,5.2)(9.6,5.35)
%
%
\rput[c]{0}(11.4,4.8){\small $\times$}
\rput[c]{0}(10.4,4.5){\small $\times$}
\rput[c]{0}(10.7,5.6){\small $\times$}
\rput[c]{0}(11.2,5.9){\small $a$}
\rput[c]{0}(11.8,4.8){\small $c$}
\rput[c]{0}(10.5,4.2){\small $b$}
\rput[c]{0}(11.3,5.05){\small $2$}
\rput[c]{0}(11.2,4.6){\small $1$}
\pscurve[linestyle=dashed](12.5,5.1)(12.1,4.6)(12,4)
\pscurve[linestyle=dashed](9.5,5.1)(9.9,4.6)(10,4)
\pscurve[linestyle=dashed](10.45,3.6)(11,3.75)(11.55,3.6)
\pscurve[linestyle=dashed](11.3,6.4)(11.75,5.95)(12.4,5.7)
\pscurve[linestyle=dashed](10.7,6.4)(10.25,5.95)(9.6,5.7)
\rput[c]{0}(13.5,5.6){\small $\mu_1^{(+)}$}
\psline[arrows=->](13,5)(14,5)
\psset{fillstyle=none}
\psline(0.6,0.8)(2.4,1.3)
\psline(1.6,0.5)(1.83,-0.2)
\psline(1.3,1.6)(1.05,2.4)
\pscurve(0.6,0.8)(0.5,0.3)(0.2,-0.25)
\pscurve(0.6,0.8)(0.2,1.1)(-0.4,1.35)
\pscurve(1.6,0.5)(0.8,0.1)(0.3,-0.35)
\pscurve(1.3,1.6)(0.4,1.45)(-0.4,1.45)
\pscurve(1.6,0.5)(2.1,1)(2.45,1.2)
\pscurve(1.3,1.6)(1.95,1.42)(2.35,1.4)
\pscurve[linestyle=dashed](0.25,-0.3)(1.1,0.65)(2.4,1.25)
\pscurve[linestyle=dashed](-0.4,1.4)(1,1.2)(2.4,1.35)
%
%
\rput[c]{0}(0.6,0.8){\small $\times$}
\rput[c]{0}(1.6,0.5){\small $\times$}
\rput[c]{0}(1.3,1.6){\small $\times$}
\rput[c]{0}(1.5,1.9){\small $a$}
\rput[c]{0}(0.2,0.8){\small $b$}
\rput[c]{0}(1.5,0.2){\small $c$}
\rput[c]{0}(0.7,1.05){\small $2$}
\rput[c]{0}(0.8,0.6){\small $1$}
\pscurve[linestyle=dashed](2.5,1.1)(2.1,0.6)(2,0)
\pscurve[linestyle=dashed](-0.5,1.1)(-0.1,0.6)(0,0)
\pscurve[linestyle=dashed](0.45,-0.4)(1,-0.25)(1.55,-0.4)
\pscurve[linestyle=dashed](1.3,2.4)(1.75,1.95)(2.4,1.7)
\pscurve[linestyle=dashed](0.7,2.4)(0.25,1.95)(-0.4,1.7)
\rput[c]{0}(3.5,1.6){\small $\mu_2^{(-)}$}
\psline[arrows=->](3,1)(4,1)
\psset{fillstyle=none}
\psline(6.3,1.4)(5.2,-0.2)
\psline(6.1,0.3)(6.7,-0.25)
\psline(5.4,0.8)(4.6,1.2)
\pscurve(6,2.4)(6.1,1.75)(6.3,1.4)
\pscurve(7.4,1.4)(6.7,1.35)(6.3,1.4)
\pscurve(6.1,0.3)(6.8,0.95)(7.4,1.3)
\pscurve(5.4,0.8)(5.75,1.6)(5.9,2.4)
\pscurve(5.1,-0.15)(5.3,0.3)(5.4,0.8)
\pscurve(5.3,-0.25)(5.6,0)(6.1,0.3)
\pscurve[linestyle=dashed](5.15,-0.2)(5.8,1.05)(5.95,2.4)
\pscurve[linestyle=dashed](5.25,-0.25)(6.2,0.8)(7.4,1.35)
%
%
\rput[c]{0}(5.4,0.8){\small $\times$}
\rput[c]{0}(6.1,0.3){\small $\times$}
\rput[c]{0}(6.3,1.4){\small $\times$}
\rput[c]{0}(5.1,0.7){\small $a$}
\rput[c]{0}(6.5,1.7){\small $b$}
\rput[c]{0}(6.1,0){\small $c$}
\rput[c]{0}(6,1.4){\small $2$}
\rput[c]{0}(6.4,1.1){\small $1$}
\pscurve[linestyle=dashed](7.5,1.1)(7.1,0.6)(7,0)
\pscurve[linestyle=dashed](4.5,1.1)(4.9,0.6)(5,0)
\pscurve[linestyle=dashed](5.45,-0.4)(6,-0.25)(6.55,-0.4)
\pscurve[linestyle=dashed](6.3,2.4)(6.75,1.95)(7.4,1.7)
\pscurve[linestyle=dashed](5.7,2.4)(5.25,1.95)(4.6,1.7)
\rput[c]{0}(8.5,1.6){\small $\mu_1^{(-)}$}
\psline[arrows=->](8,1)(9,1)
\psset{fillstyle=none}
\psline(11,2.4)(11,0.5)
\psline(11.6,1.4)(12.4,1.4)
\psline(10.4,1.4)(9.6,1.4)
\pscurve(11.8,-0.2)(11.4,0.2)(11,0.5)
\pscurve(10.2,-0.2)(10.6,0.2)(11,0.5)
\pscurve(11.1,2.4)(11.3,1.8)(11.6,1.4)
\pscurve(10.9,2.4)(10.7,1.8)(10.4,1.4)
\pscurve(11.6,1.4)(11.7,0.5)(11.9,-0.2)
\pscurve(10.4,1.4)(10.3,0.5)(10.1,-0.2)
\pscurve[linestyle=dashed](11.05,2.4)(11.3,1)(11.85,-0.2)
\pscurve[linestyle=dashed](10.95,2.4)(10.7,1)(10.15,-0.2)
%
%
\rput[c]{0}(10.4,1.4){\small $\times$}
\rput[c]{0}(11,0.5){\small $\times$}
\rput[c]{0}(11.6,1.4){\small $\times$}
\rput[c]{0}(10.1,1.1){\small $a$}
\rput[c]{0}(11,0.1){\small $b$}
\rput[c]{0}(11.9,1.1){\small $c$}
\rput[c]{0}(10.77,0.7){\small $2$}
\rput[c]{0}(11.22,0.7){\small $1$}
\pscurve[linestyle=dashed](12.5,1.1)(12.1,0.6)(12,0)
\pscurve[linestyle=dashed](9.5,1.1)(9.9,0.6)(10,0)
\pscurve[linestyle=dashed](10.45,-0.4)(11,-0.25)(11.55,-0.4)
\pscurve[linestyle=dashed](11.3,2.4)(11.75,1.95)(12.4,1.7)
\pscurve[linestyle=dashed](10.7,2.4)(10.25,1.95)(9.6,1.7)
%
\end{pspicture}
\end{center}
\caption{Pentagon relation of labeled Stokes graphs (drawn schematically).} 
\label{fig:pentagon5}
\end{figure}
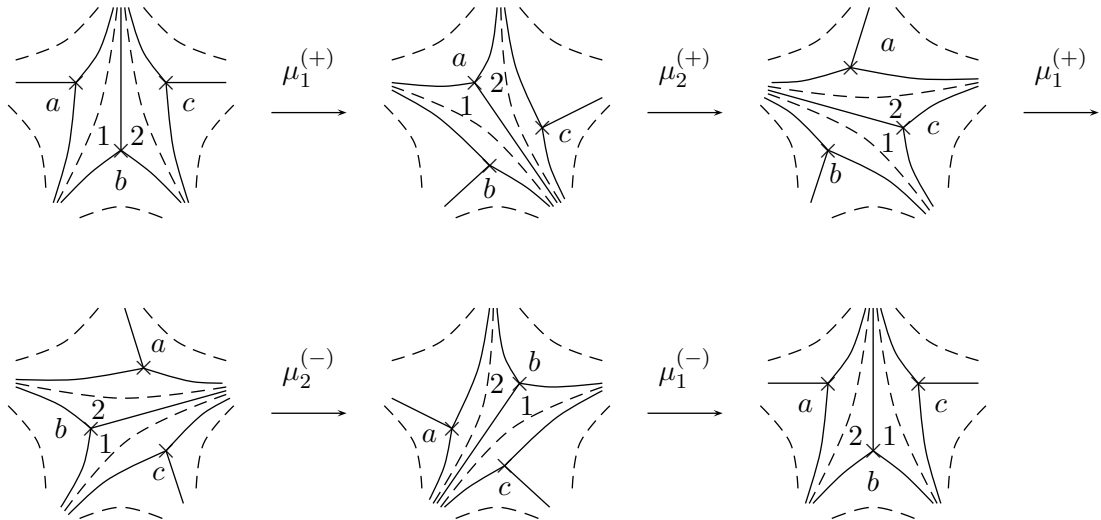

\begin{ex}
Let us illustrate the mutation sequence \eqref{eq:Stokesseq2} 
involving
signed pops.
We consider the mutation sequence 
of labeled seeds 
with period 4
represented by the labeled tagged triangulations
in Figure \ref{fig:digon2}. See also Figure \ref{fig:digon3}.
Then, 
the corresponding mutation sequence 
of the labeled Stokes triangulations
is given  by Figure \ref{fig:pop5},
where
\begin{align}
\tilde{\mu}^{(\ve_1)}_{k_1}=\mu^{(+)}_i,
\quad
\tilde{\mu}^{(\ve_2)}_{k_2}=\mu^{(+)}_j,
\quad
\tilde{\mu}^{(\ve_3)}_{k_3}=\mu^{(-)}_i\circ\kappa_p^{(+)},
\quad
\tilde{\mu}^{(\ve_4)}_{k_4}=\mu^{(-)}_j,
\quad
\kappa=\kappa_p^{(-)}.
\end{align} 
Thus, the desired periodicity \eqref{eq:per2} holds
in this example.
\end{ex}

\begin{figure}
\begin{center}
\begin{pspicture}(4,-1)(19.5,4.2)
%
\psset{linewidth=0.5pt}
%
\psset{fillstyle=solid, fillcolor=black}
\pscircle(5,4){0.08} 
\pscircle(5,2){0.08} 
\pscircle(5,3){0.08} 
\psset{fillstyle=none}
\pscurve(5,4)(4.6,3)(5,2.6)(5.4,3)(5,4)
\psline(5,3)(5,4)
\pscurve(5,4)(5.6,3.45)(5.6,2.55)(5,2)
\pscurve(5,4)(4.4,3.45)(4.4,2.55)(5,2)
%
\rput[c]{0}(5,2.4){\small $\times$}
\rput[c]{0}(5,2.8){\small $\times$}
\rput[c]{0}(5.32,2.5){\small $i$}
\rput[c]{0}(5.12,3.4){\small $j$}
\psline[arrows=->](6.7,3)(7.3,3)
\rput[c]{0}(7,3.6){$\mu_{i}^{(+)}$}
%
\psset{fillstyle=solid, fillcolor=black}
\pscircle(9,4){0.08} 
\pscircle(9,2){0.08} 
\pscircle(9,3){0.08} 
\psset{fillstyle=none}
%
\psline(9,4)(9,3)
\pscurve(9,3)(8.8,2.8)(9.2,2.4)(9,2)
\pscurve(9,4)(9.6,3.45)(9.6,2.55)(9,2)
\pscurve(9,4)(8.4,3.45)(8.4,2.55)(9,2)
%
\rput[c]{0}(9,2.4){\small $\times$}
\rput[c]{0}(9,2.8){\small $\times$}
\rput[c]{0}(9.2,2.7){\small $i$}
\rput[c]{0}(9.12,3.4){\small $j$}
\psline[arrows=->](10.7,3)(11.3,3)
\rput[c]{0}(11,3.6){$\mu_{j}^{(+)}$}
\psset{fillstyle=solid, fillcolor=black}
\pscircle(13,4){0.08} 
\pscircle(13,2){0.08} 
\pscircle(13,3){0.08} 
\psset{fillstyle=none}
\pscurve(13,2)(13.1,2.4)(12.6,2.85)(13,3.4)(13.5,3)(13,2)
\pscurve(13,3)(12.8,2.8)(13.2,2.4)(13,2)
\pscurve(13,4)(13.6,3.45)(13.6,2.55)(13,2)
\pscurve(13,4)(12.4,3.45)(12.4,2.55)(13,2)
%
\rput[c]{0}(13,2.4){\small $\times$}
\rput[c]{0}(13,2.8){\small $\times$}
%
\rput[c]{0}(13.32,3.5){\small $j$}
\rput[c]{0}(13.22,2.7){\small $i$}
\psline[arrows=->](14.7,3)(15.3,3)
\rput[c]{0}(15,3.6){$\kappa_{p}^{(+)}$}
%
\psset{fillstyle=solid, fillcolor=black}
\pscircle(17,4){0.08} 
\pscircle(17,2){0.08} 
\pscircle(17,3){0.08} 
\psset{fillstyle=none}
\pscurve(17,2)(17.1,2.4)(16.6,2.85)(17,3.4)(17.5,3)(17,2)
\pscurve(17,3)(17.2,2.5)(17,2)
\pscurve(17,4)(17.6,3.45)(17.6,2.55)(17,2)
\pscurve(17,4)(16.4,3.45)(16.4,2.55)(17,2)
%
\rput[c]{0}(17,2.4){\small $\times$}
\rput[c]{0}(17,2.8){\small $\times$}
\rput[c]{0}(17.32,3.5){\small $i$}
\rput[c]{0}(17.3,2.7){\small $j$}
\psline[arrows=->](18.7,3)(19.3,3)
\rput[c]{0}(19,3.6){$\mu_{i}^{(-)}$}
%

\psset{fillstyle=solid, fillcolor=black}
\pscircle(5,1){0.08} 
\pscircle(5,-1){0.08} 
\pscircle(5,0){0.08} 
\psset{fillstyle=none}
%
\pscurve(5,1)(4.8,-0.2)(5,-0.3)(5.1,-0.2)(5,0)
\pscurve(5,0)(5.3,-0.5)(5,-1)
\pscurve(5,1)(5.6,0.45)(5.6,-0.45)(5,-1)
\pscurve(5,1)(4.4,0.45)(4.4,-0.45)(5,-1)
%
\rput[c]{0}(5,-0.2){\small $\times$}
\rput[c]{0}(5,-0.6){\small $\times$}
\rput[c]{0}(5,0.5){\small $i$}
\rput[c]{0}(5.4,-0.3){\small $j$}
\psline[arrows=->](6.3,0)(7.3,0)
\rput[c]{0}(7,0.6){$\mu_{j}^{(-)}$}
%
\psset{fillstyle=solid, fillcolor=black}
\pscircle(9,1){0.08} 
\pscircle(9,-1){0.08} 
\pscircle(9,0){0.08} 
\psset{fillstyle=none}
\pscurve(9,1)(8.6,0)(9,-0.4)(9.4,0)(9,1)
\pscurve(9,1)(8.8,-0.2)(9,-0.3)(9.1,-0.2)(9,0)
\pscurve(9,1)(9.6,0.45)(9.6,-0.45)(9,-1)
\pscurve(9,1)(8.4,0.45)(8.4,-0.45)(9,-1)
%
\rput[c]{0}(9,-0.6){\small $\times$}
\rput[c]{0}(9,-0.2){\small $\times$}
\rput[c]{0}(9,0.5){\small $i$}
\rput[c]{0}(9.32,-0.4){\small $j$}
\psline[arrows=->](10.3,0)(11.3,0)
\rput[c]{0}(11,0.6){$\kappa_{p}^{(-)}$}
\psset{fillstyle=solid, fillcolor=black}
\pscircle(13,1){0.08} 
\pscircle(13,-1){0.08} 
\pscircle(13,0){0.08} 
\psset{fillstyle=none}
\pscurve(13,1)(12.6,0)(13,-0.4)(13.4,0)(13,1)
\psline(13,0)(13,1)
\pscurve(13,1)(13.6,0.45)(13.6,-0.45)(13,-1)
\pscurve(13,1)(12.4,0.45)(12.4,-0.45)(13,-1)
%
\rput[c]{0}(13,-0.6){\small $\times$}
\rput[c]{0}(13,-0.2){\small $\times$}
\rput[c]{0}(13.32,-0.5){\small $i$}
\rput[c]{0}(13.12,0.4){\small $j$}
%
\end{pspicture}
\end{center}
\caption{Example of periodicity of labeled Stokes triangulations involving pops.}
\label{fig:pop5}
\end{figure}
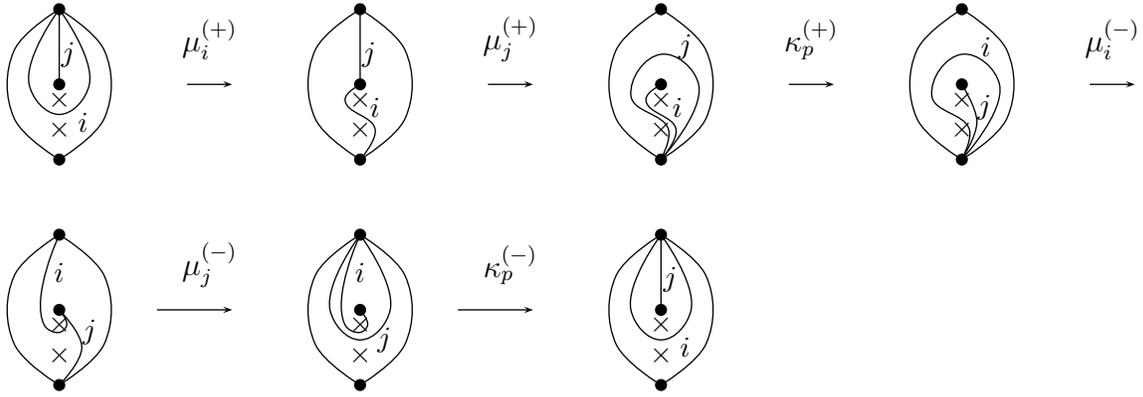

\section{Mutation of Voros symbols}
\label{sec:mut_voros}

Here we combine the analytic and geometric results
 in Sections 3 and 6
and show that the Voros symbols for the simple paths and the simple cycles
mutate exactly as $x$-variables and $\hat{y}$-variables in our extended seeds.

\subsection{Mutation formula of Voros symbols for signed flips}
\label{subsec:mutformula1}

Let us return to the situation in Section \ref{subsec:S1action}.
Let $Q(z,\eta)$ be the potential of a Schr\"odinger equation
\eqref{eq:Sch}, 
and let
$\phi=Q_0(z)dz^{\otimes 2}$ be the associated quadratic differential.
Assume that
the Stokes graph $G_0=G(\phi)$ 
has a unique {\em regular\/} saddle trajectory $\ell_0$.
Let $Q^{(\theta)}(z,\eta)$ be the $S^1$-family for  $Q(z,\eta)$
in \eqref{eq:Qtheta}.
We choose  a sufficiently small $\delta>0$ such that
$G_{+\delta}=G(\phi_{+\delta})$ and
$G_{-\delta}=G(\phi_{-\delta})$ are the saddle reductions of $G_0$.
Let us fix a sign $\ve=\pm $.
Then, we set $G=G_{\ve\delta}$
and $G'=G_{-\ve\delta}$.
We assume that $G$ and $G'$
are labeled so that they are
related by the signed flip 
$G'= \mu_k^{(\varepsilon)}(G)$.
See Figure \ref{fig:flip3}.

Let 
$\beta_1, \dots \beta_n \in \Gamma^{\vee}_G$ and 
$\gamma_1,\dots, \gamma_n \in \Gamma_G$ be 
the simple paths and the simple cycles of $G$ introduced in
Section \ref{subsec:simple1}.
We denote the Voros symbols for the potential 
$Q^{(\ve\delta)}(z,\eta)$ with respect to them
 by
\begin{align}
\label{eq:Voros1}
e^{{W}_i} = e^{{W}^{(\ve \delta)}_{\beta_i}},
\quad 
e^{{V}_i} = e^{{V}^{(\ve \delta)}_{\gamma_i}},
\end{align}
where we use the notation in \eqref{eq:WVtheta1}.
We also introduce
\begin{align}
\label{eq:vdef1}
e^{ v_i} =e^{\eta {v}^{(\ve \delta)}_{\gamma_i}},
\quad
{v}^{(\theta)}_{\gamma} 
=
{\int_{\gamma} \sqrt{Q_0^{(\theta)}(z)}\, dz}
=
e^{i\theta} \int_{\gamma} \sqrt{Q_0(z)}\, dz.
\end{align}
Thus, $e^{ v_i}$ is
the exponential factor of $e^{{V}_i}$.

Let $T$ be a labeled Stokes triangulation associated with
 $G$,
and let $B$ be the adjacency matrices of $T$.

\begin{lem} \label{lem:relation-between-x-and-y-Voros}
The following relation holds.
\begin{equation} \label{eq:relation-between-x-and-y-Voros}
e^{{V}_{i}} = e^{ v_i} \prod_{j=1}^{n} (e^{{W}_{j}})^{b_{ji}}.
\end{equation}
\end{lem}
\begin{proof}
We have 
\[
V_i = \oint_{\gamma_i}S_{\rm odd}^{(\ve \delta)}(z,\eta)dz = 
\oint_{\gamma_i}\left(\eta \sqrt{Q^{(\ve \delta)}_0(z)}
+S_{\rm odd}^{{\rm reg} (\ve \delta)}(z,\eta) \right)dz
=  v_i + \sum_{j=1}^{n} b_{ji} W_j,
\]
where
the last equality is due to
Proposition \ref{prop:bg3}.
\end{proof}
\par
Note that the  relation \eqref{eq:relation-between-x-and-y-Voros}
is
parallel to the one for  the 
$\hat{y}$-variables in \eqref{eq:yhat}.

Let $\gamma_0 \in \Gamma_G$ be the saddle class associated 
with $\ell_0$ defined in Section \ref{subsec:saddleclass}.

\begin{lem} \label{lemma:sign-of-simple-cycles}
The saddle class $\gamma_0$ coincides with $\varepsilon \gamma_k$.
\end{lem}
\begin{proof}
In the case
 $\ve=+$,
where $G=G_{+\delta}$,
  $\gamma_0=
\gamma_k$ holds as in Figure \ref{fig:cycle2}.
Similarly,
in the case  $\ve=-$,
where $G=G_{-\delta}$,
$\gamma_k= -\gamma_0$ holds as in Figure \ref{fig:cycle2}.
\end{proof}

\begin{figure}[t]
\begin{center}
\begin{pspicture}(0.5,-2.5)(16,2.5)
%
%
\psset{fillstyle=solid, fillcolor=black}
\psset{linewidth=0.5pt}
\rput[c]{0}(5.1,2.2){\small $\oplus$}
\rput[c]{0}(5.1,-2.2){\small $\oplus$}
\rput[c]{0}(0.9,2.2){\small $\ominus$}
\rput[c]{0}(0.9,-2.2){\small $\ominus$}
\rput[c]{0}(10.1,2.2){\small $\oplus$}
\rput[c]{0}(10.1,-2.2){\small $\oplus$}
\rput[c]{0}(5.9,2.2){\small $\ominus$}
\rput[c]{0}(5.9,-2.2){\small $\ominus$}
\rput[c]{0}(15.1,2.2){\small $\oplus$}
\rput[c]{0}(15.1,-2.2){\small $\oplus$}
\rput[c]{0}(10.9,2.2){\small $\ominus$}
\rput[c]{0}(10.9,-2.2){\small $\ominus$}
\psset{fillstyle=none}
\rput[c]{0}(2,0){\small $\times$}
\rput[c]{0}(4,0){\small $\times$}
\rput[c]{0}(3.2,-2.5){$G_{+\delta}$} 
\rput[c]{0}(7,0){\small $\times$}
\rput[c]{0}(9,0){\small $\times$}
\rput[c]{0}(8.1,-2.5){$G_0$} 
\rput[c]{0}(8,-0.25){$\ell_{0}$} 

\rput[c]{0}(12,0){\small $\times$}
\rput[c]{0}(14,0){\small $\times$}
\rput[c]{0}(13.2,-2.5){$G_{-\delta}$} 

\psset{linewidth=0.5pt}
\pscurve(2,0)(4,-1)(5,-2)
\pscurve(2,0)(1.5,1)(1,2)
\pscurve(2,0)(1.5,-1)(1,-2)
\pscurve(4,0)(2,1)(1,2)
\pscurve(4,0)(4.5,1)(5,2)
\pscurve(4,0)(4.5,-1)(5,-2)
\pscurve(7,0)(8,0)(9,0)
\pscurve(7,0)(6.5,1)(6,2)
\pscurve(7,0)(6.5,-1)(6,-2)
\pscurve(9,0)(9.5,1)(10,2)
\pscurve(9,0)(9.5,-1)(10,-2)
\pscurve(12,0)(14,1)(15,2)
\pscurve(12,0)(11.5,1)(11,2)
\pscurve(12,0)(11.5,-1)(11,-2)
\pscurve(14,0)(12,-1)(11,-2)
\pscurve(14,0)(14.5,1)(15,2)
\pscurve(14,0)(14.5,-1)(15,-2)
\psset{linewidth=1pt}
\pscurve(6.5,0)(6.7,0.5)(8,0.8)
\pscurve(9.5,0)(9.3,0.5)(8,0.8)
\pscurve(1.5,0)(1.7,0.5)(3,0.8)
\pscurve(4.5,0)(4.3,0.5)(3,0.8)
\pscurve(11.5,0)(11.7,0.5)(13,0.8)
\pscurve(14.5,0)(14.3,0.5)(13,0.8)
\psset{linewidth=1pt, linestyle=dashed}
\pscurve(6.5,0)(6.7,-0.5)(7.95,-0.8)
\pscurve(9.5,0)(9.3,-0.5)(8.05,-0.8)
\pscurve(1.5,0)(1.7,-0.5)(2.95,-0.8)
\pscurve(4.5,0)(4.3,-0.5)(3.05,-0.8)
\pscurve(11.5,0)(11.7,-0.5)(12.95,-0.8)
\pscurve(14.5,0)(14.3,-0.5)(13.05,-0.8)
\rput[c]{0}(8,1.2){\small $\gamma_{0}$} 
\rput[c]{0}(3,1.2){\small $\gamma_{k}$} 
\rput[c]{0}(13,1.2){\small $\gamma_{k}$} 
\rput[c]{0}(2.5,0){\small $\beta_{k}$} 
\rput[c]{0}(12.5,0){\small $\beta_{k}$} 
\psset{linewidth=0.5pt, linestyle=solid}
\pscurve(1,2)(2,0.6)(3,0)(4,-0.6)(5,-2)
\pscurve(11,-2)(12,-0.6)(13,0)(14,0.6)(15,2)
\psset{linewidth=1pt, linestyle=solid}
\psline(8,0.8)(8.2,0.9)
 \psline(8,0.8)(8.2,0.7)
\psline(8,-0.8)(7.8,-0.9) 
\psline(8,-0.8)(7.8,-0.7)
\psline(3,0.8)(3.2,0.9)
 \psline(3,0.8)(3.2,0.7)
\psline(3,-0.8)(2.8,-0.9) 
\psline(3,-0.8)(2.8,-0.7)
\psline(13,0.8)(12.8,0.9)
 \psline(13,0.8)(12.8,0.7)
\psline(13,-0.8)(13.2,-0.9) 
\psline(13,-0.8)(13.2,-0.7)
\psset{linewidth=0.5pt, linestyle=solid}
\psline(3,0)(2.8,0.2)
\psline(3,0)(2.75,0.05)
\psline(13,0)(12.8,-0.2)
\psline(13,0)(12.75,-0.05)
\psset{linewidth=1pt,linestyle=solid}
\pscurve(2,0)(1.9,-0.1)(1.8,0)(1.7,0.1)(1.6,0)
(1.5,-0.1)(1.4,0)(1.3,0.1)(1.2,0)(1.1,-0.1)(1,0)
\pscurve(4,0)(4.1,0.1)(4.2,0)(4.3,-0.1)(4.4,0)
(4.5,0.1)(4.6,0)(4.7,-0.1)(4.8,0)(4.9,0.1)(5,0)
\pscurve(7,0)(6.9,-0.1)(6.8,0)(6.7,0.1)(6.6,0)
(6.5,-0.1)(6.4,0)(6.3,0.1)(6.2,0)(6.1,-0.1)(6,0)
\pscurve(9,0)(9.1,0.1)(9.2,0)(9.3,-0.1)(9.4,0)
(9.5,0.1)(9.6,0)(9.7,-0.1)(9.8,0)(9.9,0.1)(10,0)
\pscurve(12,0)(11.9,-0.1)(11.8,0)(11.7,0.1)(11.6,0)
(11.5,-0.1)(11.4,0)(11.3,0.1)(11.2,0)(11.1,-0.1)(11,0)
\pscurve(14,0)(14.1,0.1)(14.2,0)(14.3,-0.1)(14.4,0)
(14.5,0.1)(14.6,0)(14.7,-0.1)(14.8,0)(14.9,0.1)(15,0)
\end{pspicture}
\end{center}
\caption{Cycles $\gamma_0$ and $\gamma_k$.} 
\label{fig:cycle2}
\end{figure}
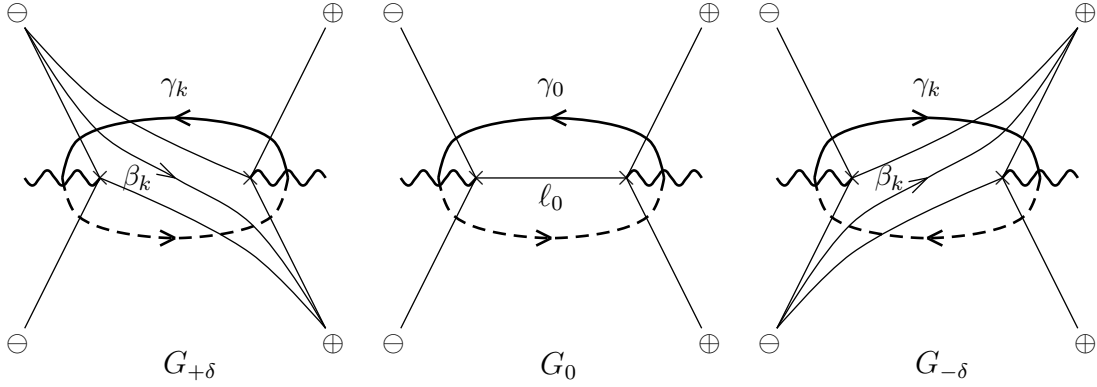

Using $\gamma_k$ instead of $\gamma_0$, the jump formula in
Theorem \ref{thm:jump1} (a)
 is restated  as follows.
\begin{prop}
 \label{prop:signed-DDP-formula}
For 
any path $\beta \in \Gamma^{\vee}_G$ and
any  cycle $\gamma \in \Gamma_G$, we have 
\begin{align} \label{eq:epsilon-expression-of-DDP-analytic}
\lim_{\delta\rightarrow +0}
{\mathcal S}[e^{{W}^{(-\ve\delta)}_{\beta}}] &=
\lim_{\delta\rightarrow +0}
 {\mathcal S}
[e^{{W}^{(\ve\delta)}_{\beta}}
\left( 1 + \Bigl(e^{{V}^{(\ve\delta)}_{\gamma_k}}\Bigr)^{\varepsilon} 
\right)^{-\langle\gamma_k,\beta\rangle}
],\\
\label{eq:epsilon-expression-of-DDP-analytic2}
\lim_{\delta\rightarrow +0}
{\mathcal S}[e^{{V}^{(-\ve\delta)}_{\gamma}}] &=
\lim_{\delta\rightarrow +0}
 {\mathcal S}[e^{{V}^{(\ve\delta)}_{\gamma}}
\left( 1 + \Bigl(e^{{V}^{(\ve\delta)}_{\gamma_k}}\Bigr)^{\varepsilon} 
\right)^{-(\gamma_k, \gamma)}].
\end{align}
\end{prop}
\begin{proof} 
Let us show \eqref{eq:epsilon-expression-of-DDP-analytic}.
When $\varepsilon = +$, $G = G_{+\delta}$, $G' = G_{-\delta}$
and
$\gamma_0 = \gamma_k$ by 
Lemma \ref{lemma:sign-of-simple-cycles}. 
Hence,
we have the equality \eqref{eq:epsilon-expression-of-DDP-analytic}
immediately from
\eqref{eq:jump1}.
When $\varepsilon = -$, $G' = G_{+\delta}$, $G = G_{-\delta}$
and
$\gamma_0 = -\gamma_k$ by 
Lemma \ref{lemma:sign-of-simple-cycles}. 
Note that we have the equality
\begin{align}
\lim_{\delta\rightarrow +0}{\mathcal S}[e^{{V}^{(-\delta)}_{\gamma_0}}] = 
\lim_{\delta\rightarrow +0}{\mathcal S}[e^{{V}^{(+\delta)}_{\gamma_0}}]
\end{align}
 by
\eqref{eq:jump1} and $(\gamma_0,\gamma_0)=0$.
Then, it follows from \eqref{eq:jump1}  that
\begin{align}
\begin{split}
\lim_{\delta\rightarrow +0}
{\mathcal S}[e^{{W}^{(+\delta)}_{\beta}}] 
&= 
\lim_{\delta\rightarrow +0}
{\mathcal S}[e^{{W}^{(-\delta)}_{\beta}}
\left( 1 + e^{{V}^{(+\delta)}_{\gamma_0}}
\right)^{(\gamma_0,\beta)}] \\
&= 
\lim_{\delta\rightarrow +0}
{\mathcal S}[e^{{W}^{(-\delta)}_{\beta}}
\left( 1 +e^{-{V}^{(-\delta)}_{\gamma_k}}
\right)^{-(\gamma_k,\beta)}],
\end{split}
\end{align}
which is the desired equality for $\ve=-$.
The equality \eqref{eq:epsilon-expression-of-DDP-analytic2}
is proved in the same manner. 
\end{proof}

We emphasize the following ``nonjump'' property
of the integral in \eqref{eq:vdef1}.
\begin{lem}
\label{lem:nojump1}
For 
any  cycle $\gamma \in \Gamma_G$, 
we have
\begin{align}
\label{eq:epsilon-expression-of-DDP-analytic3}
\lim_{\delta\rightarrow +0}
{\mathcal S}[e^{\eta{v}^{(-\ve\delta)}_{\gamma}}] &=
\lim_{\delta\rightarrow +0}
 {\mathcal S}[e^{{\eta v}^{(\ve\delta)}_{\gamma}}]
\end{align}
\end{lem}
\begin{proof}
By the definition of
the Borel sum of an exponential factor in Definition \ref{def:Borel-summability},
we have ${\mathcal S}[e^{\eta{v}^{(\theta)}_{\gamma}}]=e^{\eta{v}^{(\theta)}_{\gamma}}$.
Thus,  both hand sides of \eqref{eq:epsilon-expression-of-DDP-analytic3}
equal to $e^{\eta {v}^{(0)}_{\gamma}}$.
\end{proof}

Let 
$\beta_1', \dots \beta_n' \in \Gamma^{\vee}_{G'}$ and 
$\gamma_1',\dots, \gamma_n' \in \Gamma_{G'}$ be 
the simple paths and the simple cycles of $G'$.
Since the zeros and poles of $\phi_{\theta}$ are stable
by the $S^1$-action,
we have $\Gamma^{\vee}_G= \Gamma^{\vee}_{G'}$
and $\Gamma_G= \Gamma_{G'}$ with
$ \tau_{G,G'}^{\vee}=\mathrm{id}$ and $ \tau_{G,G'}=\mathrm{id}$.
Thus, \eqref{eq:bemut1} and \eqref{eq:gamut1} reduce
to the direct equalities
\begin{align}
\label{eq:bemut2}
\beta_i'&=
\begin{cases}
-\beta_k
+\sum_{j=1}^n [-\varepsilon b_{jk}]_+ \beta_j
 & i=k\\
\beta_i 
& i \neq k,
\end{cases}\\
\label{eq:gamut2}
\gamma_i'&=
\begin{cases}
-\gamma_k & i=k\\
\gamma_i + [\varepsilon b_{ki}]_+ \gamma_k
& i \neq k.
\end{cases}
\end{align}
In parallel to $Q^{(\ve\delta)}(z,\eta)$, we introduce
\begin{align}
\label{eq:vdef2}
e^{{W}'_i} = e^{{W}^{(-\ve \delta)}_{\beta_i'}},
\quad 
e^{{V}'_i} = e^{{V}^{(-\ve \delta)}_{\gamma_i'}},
\quad
e^{ v'_i} =e^{\eta{v}^{(-\ve \delta)}_{\gamma_i'}}.
\end{align}

Now we present the mutation formula of the Voros symbols for the signed flips.

\begin{thm}[Mutation formula of the Voros symbols for the signed flip $\mu^{(\ve)}_k$]
\label{thm:mutationV1}
For $i=1,\dots,n$, we have
\begin{align}
\label{eq:v1}
\lim_{\delta\rightarrow +0}
{\mathcal S}[e^{{v}'_{i}}] &=
\begin{cases}
 \displaystyle
\lim_{\delta\rightarrow +0}
 {\mathcal S}[(e^{{v}_{k}})^{-1}] &i=k\\
 \displaystyle
\lim_{\delta\rightarrow +0}
 {\mathcal S}[e^{{v}_{i}}
 (e^{{v}_{k}})^{[\ve b_{ki}]_+}]
& i\neq k,
\end{cases}
\\
\label{eq:W1}
\lim_{\delta\rightarrow +0}
{\mathcal S}[e^{{W}'_{i}}] &=
\begin{cases}
\displaystyle
\lim_{\delta\rightarrow +0}
 {\mathcal S}[(e^{{W}_{k}})^{-1}
 \left(
 \prod_{j=1}^n (e^{{W}_{j}})^{[- \ve b_{jk}]_+}
 \right)
\left( 1 + (e^{{V}_{k}})^{\varepsilon} 
\right)] & i=k\\
\displaystyle
\lim_{\delta\rightarrow +0}
 {\mathcal S}[e^{{W}_{i}}]
& i \neq k,
\end{cases}
\\
\label{eq:V1}
\lim_{\delta\rightarrow +0}
{\mathcal S}[e^{{V}'_{i}}] &=
\begin{cases}
\displaystyle
\lim_{\delta\rightarrow +0}
 {\mathcal S}[(e^{{V}_{k}})^{-1}] & i=k
 \\
\displaystyle
\lim_{\delta\rightarrow +0}
 {\mathcal S}[e^{{V}_{i}}
 (e^{{V}_{k}})^{[\ve b_{ki}]_+}
\left( 1 + (e^{{V}_{k}})^{\varepsilon} 
\right)^{-b_{ki}}] & i\neq k.
\end{cases}
\end{align}
\end{thm}
\begin{proof}
Let us show \eqref{eq:v1}.
By \eqref{eq:gamut2},
\begin{align}
\label{eq:ev2}
e^{{v}'_{i}}
=
e^{\eta{v}^{(-\ve \delta)}_{\gamma_i'}}
  =
\begin{cases}
\bigl((e^{\eta{v}^{(-\ve \delta)}_{\gamma_k}}\bigr)^{-1} &i=k\\
e^{\eta{v}^{(-\ve \delta)}_{\gamma_i}}
 \bigl(e^{\eta{v}^{(-\ve \delta)}_{\gamma_k}}
 \bigr)^{[\ve b_{jk}]_+}
& i\neq k,
\end{cases}
\end{align}
Applying the Borel resummation operator $\mathcal{S}$ to \eqref{eq:ev2},
taking the limit $\delta\rightarrow +0$,
then using Lemma \ref{lem:nojump1},
we obtain \eqref{eq:v1}.
The equalities \eqref{eq:W1} and \eqref{eq:V1}
are obtained in a similar way from
\eqref{eq:bemut2}, \eqref{eq:gamut2},
Proposition \ref{prop:signed-DDP-formula},
 and the facts $\langle \gamma_k,\beta_i\rangle
= \delta_{ki}$
and 
$( \gamma_k,\gamma_i)
= b_{ki}$.
\end{proof}

The formulas \eqref{eq:v1}--\eqref{eq:V1}
coincide with the exchange relation of seeds in 
\eqref{eq:ymut6}, \eqref{eq:xmut6}, and 
\eqref{eq:ymut7} under the identification
\begin{align}
\begin{array}{lcl}
\displaystyle \lim_{\delta\rightarrow +0} e^{v_i} & \leftrightarrow & y_i,\\
\displaystyle \lim_{\delta\rightarrow +0} e^{W_i} & \leftrightarrow & x_i,\\
\displaystyle \lim_{\delta\rightarrow +0} e^{V_i} & \leftrightarrow & \hat{y}_i.
\end{array}
\end{align}
We phrase this result as
{\em ``by the signed flips the Voros symbols
$x_i=e^{W_i}$, $\hat{y}_i=e^{V_i}$,
together with $y_i=e^{v_i}$,
mutate as the variables in seeds
(in the sense of Section \ref{subsec:monomial})}''.
 Observe that the monomial parts in the right hand sides
of \eqref{eq:v1}--\eqref{eq:V1} have the geometric origin,
while the non-monomial parts have the analytic origin, i.e.,
the Stokes phenomenon.

Next, we reformulate the above result in terms of the Stokes automorphisms
as in Section 
\ref{section:Voros-coefficients-and-DDP-formula}.

Let $\bbV=\bbV(Q^{(\ve\delta)}(z,\eta))$ be the
the Voros field for the potential $Q^{(\ve\delta)}(z,\eta)$,
i.e.,
the rational function field  of
the Voros symbols
$e^{{W}_1}$, \dots, $e^{{W}_n}$,
$e^{V_1}$, \dots, $e^{V_n}$
over $\bbQ$.
By Lemma \ref{lem:relation-between-x-and-y-Voros},
$\bbV$ is also generated by
$e^{{W}_1}$, \dots, $e^{{W}_n}$,
$e^{v_1}$, \dots, $e^{v_n}$.
Similarly,
let $\bbV'=\bbV(Q^{(-\ve\delta)}(z,\eta))$ be the 
Voros field for the potential $Q^{(-\ve\delta)}(z,\eta)$,
which is the rational function field  of
 the Voros symbols
$e^{{W}'_1}$, \dots, $e^{{W}'_n}$,
$e^{V'_1}$, \dots, $e^{V'_n}$.

The isomorphisms of the homology groups
 $\tau_{G,G'}$ and  $\tau^{\vee}_{G,G'}$
 in Proposition \ref{prop:cycle1} (a)
 induce the following field isomorphism  $\tau^*_{\bbV,\bbV'}:\bbV'\rightarrow
\bbV$:
\begin{align} 
\label{eq:monomial-Voros-y}
\tau^*_{\bbV,\bbV'}(e^{{v}'_i}) & = 
\begin{cases}
(e^{{v}_k})^{-1} & i = k\\
e^{{v}_i} (e^{{v}_k})^{[\varepsilon b_{ki}]_+} & i\neq k,
\end{cases}\\
\label{eq:monomial-Voros-x}
\tau^*_{\bbV,\bbV'}(e^{{W}'_i}) & = 
\begin{cases}
\displaystyle
(e^{{W}_k})^{-1}\prod_{j=1}^{n}(e^{{W}_j})
^{[-\varepsilon b_{j k}]_+}
& i = k\\
e^{{W}_i} & i\neq k.
\end{cases}
\end{align}
Compare them with \eqref{eq:ymut6} and \eqref{eq:xmut6}.
By \eqref{eq:bmut},
\eqref{eq:monomial-Voros-y},
\eqref{eq:monomial-Voros-x},
and Lemma \ref{lem:relation-between-x-and-y-Voros},
we have
\begin{align}
\label{eq:monomial-Voros-Y}
\tau^*_{\bbV,\bbV'}(e^{{V}'_i}) & = 
\begin{cases}
(e^{{V}_k})^{-1} & i = k\\
e^{{V}_i} (e^{{V}_k})^{[\varepsilon b_{ki}]_+} & i\neq k.
\end{cases}
\end{align}

Also,
in view of Proposition \ref{prop:signed-DDP-formula}
and Lemma \ref{lem:nojump1},
we  introduce the field automorphism
 ${\mathfrak S}_{\bbV,k}^{(\varepsilon)}:\bbV \rightarrow
\bbV$  as follows.
\begin{align}
 \label{eq:v4}
{\mathfrak S}_{\bbV,k}^{(\varepsilon)}(\displaystyle e^{{v}_{i}})  &=  
e^{{v}_{i}} ,
\\
\label{eq:W4}
{\mathfrak S}_{\bbV,k}^{(\varepsilon)}( e^{{W}_{i}} ) &=  
e^{{W}_{i}} 
\left( 1 + (e^{{V}_k})^{\varepsilon} \right)
^{-\delta_{ki}}.
\end{align}
By 
\eqref{eq:v4},
\eqref{eq:W4},
and Lemma \ref{lem:relation-between-x-and-y-Voros},
we have
\begin{align}
 \label{eq:V4}
{\mathfrak S}_{\bbV,k}^{(\varepsilon)}( e^{{V}_{i}} )
&= e^{{V}_{i}}
\left( 1 + (e^{{V}_k})^{\varepsilon} \right)
^{-b_{ki}}.
\end{align}
We call ${\mathfrak S}_{\bbV,k}^{(\varepsilon)}$
the {\em Stokes automorphism
associated with the  signed flip  $\mu^{(\ve)}_k$}.

For simplicity, let us denote
\begin{alignat}{3}
y_i&= e^{v_i},\quad 
x_i&= e^{{W}_{i}} ,\quad
\hat{y}_i&=e^{V_i},\\
y'_i&= e^{v'_i},\quad 
x'_i&= e^{{W}'_i},~~ 
\hat{y}'_i&= e^{V'_i}.
\end{alignat}
Then, it is easy to check that the following formulas hold.
\begin{align}
\label{eq:st1}
({\mathfrak S}_{\bbV,k}^{(\varepsilon)}\circ
\tau^*_{\bbV,\bbV'})
(y'_i)
&=
\begin{cases}
y_k^{-1} & i= k\\
\displaystyle
y_i  y_{k}{}^{[\varepsilon b_{ki}]_+}
& i\neq k,\\
\end{cases}
\\
\label{eq:st2}
({\mathfrak S}_{\bbV,k}^{(\varepsilon)}\circ
\tau^*_{\bbV,\bbV'})
(x'_i)
&=
\begin{cases}
\displaystyle
x_k{}^{-1} 
\left(
\prod_{j=1}^n x_{j}{}^{[-\varepsilon b_{jk}]_+}
\right)
(1+ \hat{y}_k{}^{\varepsilon})
& i= k\\
\displaystyle
x_i 
& i\neq k,\\
\end{cases}
\\
\label{eq:st3}
({\mathfrak S}_{\bbV,k}^{(\varepsilon)}\circ
\tau^*_{\bbV,\bbV'})
(\hat{y}'_i)
&=
\begin{cases}
\hat{y}_k{}^{-1} & i= k\\
\displaystyle
\hat{y}_i 
\hat{y}_k{}^{[\ve b_{ki}]_+}
(1+ \hat{y}_k{}^{\ve})^{-b_{ki}}
& i\neq k.
\end{cases}
\end{align}

Theorem \ref{thm:mutationV1} is compactly expressed in the following way,
where $\circ$ is the composition of  operations.
\begin{thm}
\label{thm:Stokes1}
The following equality of operations holds:
\begin{align}
\lim_{\delta\rightarrow +0} \circ\, \mathcal{S}
=
\lim_{\delta\rightarrow +0} \circ \,\mathcal{S} \circ
{\mathfrak S}_{\bbV,k}^{(\varepsilon)}\circ
\tau^*_{\bbV,\bbV'}.
\end{align}
\end{thm}
The formulation by the Stokes automorphisms 
enables us to
treat the Stokes phenomenon more algebraically,
and it will be useful
when we study its global property
in Section \ref{sec:identities}.

\subsection{Mutation formula of Voros symbols for signed pops}

Next we consider the case where
the quadratic differential $\phi$ in Section 
\ref{subsec:mutformula1}
has a unique {\em degenerate\/} saddle trajectory $\ell_0$.
Again,
we choose  a sufficiently small $\delta>0$ such that
$G_{+\delta}=G(\phi_{+\delta})$ and
$G_{-\delta}=G(\phi_{-\delta})$ are the saddle reductions of $G_0=G(\phi)$.
Let us fix a sign $\ve=\pm $,
and we set $G=G_{\ve\delta}$
and $G'=G_{-\ve\delta}$.
We assume that $G$ and $G'$
are labeled so that they are related by the signed pop
$G'=\kappa^{(\ve)}_p(G)$ at the double pole $p$ surrounded by $\ell_0$.
See Figure \ref{fig:pop4}.
The story is quite parallel to the case of the signed flips,
so we concentrate on the points which are special for this case.

Let $\gamma_0 \in \Gamma_G$
 be the saddle class associated with $\ell_0$
 defined in Section \ref{subsec:saddleclass}.
Let  $\tilde{\gamma}_p\in \Gamma_G$ be the cycle
given in  Figure \ref{fig:non-jump-formula-for-pop}.
Namely, $\tilde{\gamma}_p$ is $\gamma_0$ or
$-\gamma_0$, and its orientation is determined by
the condition $\langle \tilde{\gamma}_p, \beta \rangle=1$,
where $\beta$ is any trajectory in the degenerate horizontal strip
surrounding $p$ whose orientation is given as in Section 
\ref{section:orientation}.

\begin{figure}
\begin{center}
\begin{pspicture}(0,-9.1)(16.7,-3.8)
%
\psset{linewidth=0.5pt}
\psset{fillstyle=solid, fillcolor=black}
\pscircle(8,-5.5){0.08}
\psset{linewidth=0.5pt}
\psset{fillstyle=none}
\rput[c]{0}(3,-6.98){\small $\times$}
\rput[c]{0}(3,-5.5){\small $\ominus$}
\rput[c]{0}(3,-8.5){\small $\ominus$}
\rput[c]{0}(3.1,-9.1){\small $G_{+\delta}$}
\rput[c]{0}(3,-4.9){\small $\tilde{\gamma}_p$}
\rput[c]{0}(2.6,-5.4){\small $\beta$}
\rput[c]{0}(8,-6.98){\small $\times$}
\rput[c]{0}(8.1,-9.1){\small $G_{0}$} 
\rput[c]{0}(8,-4.3){\small $\ell_{0}$}
\rput[c]{0}(8,-8.5){\small $\ominus$}
\rput[c]{0}(8,-5.8){\small $p$} 
\rput[c]{0}(8,-4.9){\small $\gamma_0$}
\rput[c]{0}(13,-6.98){\small $\times$}
\rput[c]{0}(13,-5.5){\small $\oplus$}
\rput[c]{0}(13,-8.5){\small $\ominus$}
\rput[c]{0}(13.1,-9.1){\small $G_{-\delta}$} 
\rput[c]{0}(13,-4.9){\small $\tilde{\gamma}_p$}
\rput[c]{0}(13.4,-5.4){\small $\beta$}
%
\psline(3,-7.0)(3,-8.34)
\pscurve(3,-7)(2.7,-6.5)(2.7,-6)(2.85,-5.6)
\psset{linewidth=0.5pt,linestyle=solid}
\pscurve(3,-7.0)(4,-6.3)(4.2,-5.5)(4,-4.8)(3,-4.4)(2.0,-4.9)(1.8,-6.4)(2.0,-7.0)
\psset{linewidth=0.5pt,linestyle=solid}
\psecurve(1.8,-6.4)(2.0,-7.0)(2.9,-8.35)(2.9,-8.35)
%
\psset{linewidth=1pt,linestyle=solid}
\pscurve(3,-7)(2.9,-7.1)(2.8,-7)
(2.7,-6.9)(2.6,-7)(2.5,-7.1)(2.4,-7)
(2.3,-6.9)(2.2,-7)(2.1,-7.1)(2,-7)
(1.9,-6.9)(1.8,-7)(1.7,-7.1)(1.6,-7)
(1.5,-6.9)(1.4,-7)
\psset{linewidth=1pt}
\pscurve(8,-7)(7.9,-7.1)(7.8,-7)
(7.7,-6.9)(7.6,-7)(7.5,-7.1)(7.4,-7)
(7.3,-6.9)(7.2,-7)(7.1,-7.1)(7,-7)
(6.9,-6.9)(6.8,-7)(6.7,-7.1)(6.6,-7)
(6.5,-6.9)(6.4,-7)
\psset{linewidth=0.5pt}
\psline(8,-7.0)(8,-8.34)
\pscurve(8,-7.0)(9,-6.3)(9.2,-5.5)
(9,-4.8)(8,-4.3)(7,-4.8)
(6.8,-5.5)(7,-6.3)
(8,-7.0)
\psset{linewidth=1pt}
\pscurve(13,-7)(12.9,-7.1)(12.8,-7)
(12.7,-6.9)(12.6,-7)(12.5,-7.1)(12.4,-7)
(12.3,-6.9)(12.2,-7)(12.1,-7.1)(12,-7)
(11.9,-6.9)(11.8,-7)(11.7,-7.1)(11.6,-7)
(11.5,-6.9)(11.4,-7)
\psset{linewidth=0.5pt}
\psline(13,-7.0)(13,-8.34)
\pscurve(13,-7)(13.3,-6.5)(13.3,-6)(13.15,-5.6)

\pscurve
(13,-7.0)(12,-6.3)(11.8,-5.5)
(12,-4.8)(13,-4.4)
(14.0,-4.9)(14.2,-6.4)
(14.0,-7.0)(13.1,-8.35)
\psset{linewidth=0.5pt}
%
%
\psset{linewidth=1pt, linestyle=dashed}
\pscurve(12.75,-6.87)(13.65,-6.2)(13.9,-5.5)(13.7,-4.9)(13,-4.65)
(12.3,-4.9)(12.1,-5.5)
(12.3,-6)(13,-6.5)(13.3,-6.9)(13.2,-7.4)(12.7,-7.4)(12.2,-6.9)
\pscurve(7.75,-6.87)(8.65,-6.2)(8.9,-5.5)(8.7,-4.9)(8,-4.65)
(7.3,-4.9)(7.1,-5.5)
(7.3,-6)(8,-6.5)(8.3,-6.9)(8.2,-7.4)(7.7,-7.4)(7.2,-6.9)
\pscurve(2.75,-6.87)(3.65,-6.2)(3.9,-5.5)(3.7,-4.9)(3,-4.65)
(2.3,-4.9)(2.1,-5.5)
(2.3,-6)(3,-6.5)(3.3,-6.9)(3.2,-7.4)(2.7,-7.4)(2.2,-6.9)
\psset{linewidth=1pt,linestyle=solid}
\pscurve(12.75,-6.87)(12.7,-7.1)(13,-7.2)(14.15,-6.35)(14.4,-5.6)
(14,-4.42)(13.4,-4.04)(13,-4)
(12.6,-4.04)(12,-4.42)(11.6,-5.60)(11.8,-6.4)(12.2,-6.9)
\pscurve(7.75,-6.87)(7.7,-7.1)(8,-7.2)(9.15,-6.35)(9.4,-5.6)
(9,-4.42)(8.4,-4.04)(8,-4)
(7.6,-4.04)(7,-4.42)(6.6,-5.60)(6.8,-6.4)(7.2,-6.9)
\pscurve(2.75,-6.87)(2.7,-7.1)(3,-7.2)(4.15,-6.35)(4.4,-5.6)
(4,-4.42)(3.4,-4.04)(3,-4)
(2.6,-4.04)(2,-4.42)(1.6,-5.60)(1.8,-6.4)(2.2,-6.9)
%
\psline(11.6,-5.9)(11.42,-5.6)
\psline(11.6,-5.9)(11.78,-5.6)
\psline(12.15,-5.7)(12.12,-6)
\psline(12.15,-5.7)(12.4,-5.9)
\psline(6.6,-5.6)(6.45,-5.9)
\psline(6.6,-5.6)(6.8,-5.85)
\psline(7.0,-5.7)(7.23,-5.95)
\psline(7.23,-5.9)(7.30,-5.62)
\psline(1.6,-5.6)(1.45,-5.9)
\psline(1.6,-5.6)(1.8,-5.85)
\psline(2.0,-5.7)(2.23,-5.9)
\psline(2.23,-5.9)(2.30,-5.62)
\psset{linewidth=0.5pt}
\pscurve(13.15,-5.5)(13.6,-6)(13.6,-7)(13.05,-8.35)
\psecurve(2.4,-6)(2.4,-7)(2.95,-8.35)(2.95,-8.35)
\psset{linewidth=0.5pt,linestyle=dashed}
\pscurve(2.85,-5.5)(2.4,-6)(2.4,-7)(2.95,-8.35)
\psset{linewidth=0.5pt}
\psline(2.35,-6.4)(2.45,-6.6)
\psline(2.35,-6.4)(2.25,-6.6)
\psline(13.65,-6.4)(13.55,-6.6)
\psline(13.65,-6.4)(13.75,-6.6)
\end{pspicture}
\end{center}
\caption{Cycles $\gamma_0$ and $\tilde{\gamma}_p$.} 
\label{fig:non-jump-formula-for-pop}
\end{figure}
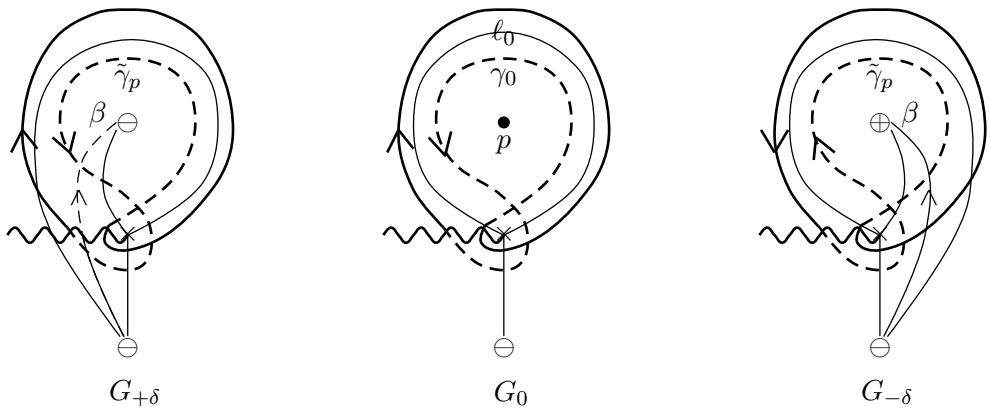


We have the counterpart of
Lemma \ref{lemma:sign-of-simple-cycles}.

\begin{lem} \label{lemma:sign-of-simple-cycles2}
The saddle class $\gamma_0$ coincides with $\varepsilon \tilde{\gamma}_p$.
\end{lem}
\begin{proof}
This is clear from 
Figure \ref{fig:non-jump-formula-for-pop}.
\end{proof}

Let us use the notation \eqref{eq:Voros1} for the Voros symbols
for $Q^{(\ve\delta)}(z,\eta)$.

The counterpart of
Proposition \ref{prop:signed-DDP-formula}
is as follows.

\begin{prop} \label{prop:jump-at-loop}
For 
any path $\beta \in \Gamma^{\vee}_G$ and
any  cycle $\gamma \in \Gamma_G$, we have 
\begin{align}
\label{eq:W2}
\lim_{\delta\rightarrow +0}
{\mathcal S}[e^{{W}^{(-\ve\delta)}_{\beta}}] &= 
\lim_{\delta\rightarrow +0}
{\mathcal S}
[e^{{W}^{(\ve\delta)}_{\beta}}
\left( 1 - \Bigl(
e^{{V}^{(\ve\delta)}_{\tilde{\gamma}_p}}
\Bigr)^{\varepsilon} 
\right)^{\langle\tilde{\gamma}_p,\beta\rangle}
],\\
\label{eq:V2}
\lim_{\delta\rightarrow +0}
{\mathcal S}[e^{{V}^{(-\ve\delta)}_{\gamma}}] &=
\lim_{\delta\rightarrow +0}
 {\mathcal S}[e^{{V}^{(\ve\delta)}_{\gamma}}
].
\end{align}
\end{prop}

\begin{proof}
The first formula follows from  Theorem \ref{thm:jump1} (b)
and Lemma \ref{lemma:sign-of-simple-cycles2}
by the same argument for Proposition \ref{prop:signed-DDP-formula}.
The second formula is  the same one in
Theorem \ref{thm:jump1} (b).
\end{proof}

Let us examine the integral ${V}^{(\ve\delta)}_{\tilde{\gamma}_p}$
appearing in \eqref{eq:W2}.
As shown in \eqref{eq:Vint1},
 we have 
\begin{align}
\label{eq:Vint2}
V^{(\ve\delta)}_{\tilde{\gamma}_p} &= \oint_{\tilde{\gamma}_p}
\eta\sqrt{Q^{(\ve\delta)}_0(z)}\, dz.
\end{align}
Furthermore, we see in Figure
\ref{fig:non-jump-formula-for-pop} that
the cycle $\tilde{\gamma}_{p}$ winds around $p$ anticlockwise 
and twice
(modulo the $*$-equivalence) on the sheet
 where $p$ has the sign $\oplus$ with respect to the integral of
 $\sqrt{Q^{(\ve\delta)}_0(z)}$.
Thus, the right hand side of \eqref{eq:Vint2} equals to
\begin{align}
4\pi i \eta\operatornamewithlimits{Res}_{z=p_{\oplus}} \sqrt{Q^{(\ve\delta)}_0(z)}\,dz,
\end{align}
where $z=p_{\oplus}$ implies taking the residue at $p$ on the above mentioned sheet.

More generally, we define,
 for any double pole $q$ of $Q^{(\ve\delta)}_0(z)$,
 which is also a double pole of $Q^{(-\ve\delta)}_0(z)$, 
\begin{align}
\tilde{v}_{q} &= 
4\pi i \eta\operatornamewithlimits{Res}_{z=q_{\oplus}} \sqrt{Q^{(\ve\delta)}_0(z)}\,dz,\\
\tilde{v}'_{q} &= 
4\pi i \eta\operatornamewithlimits{Res}_{z=q_{\oplus}} \sqrt{Q^{(-\ve\delta)}_0(z)}\,dz.
\end{align}
The definition makes sense,
because $G$ and $G'$ are saddle-free, so that
$q$ has the definite sign on each sheet
with respect to the corresponding integral.

\begin{rem} 
The integral $\tilde{v}_q$ coincides with the {\em residue at $q$\/}
in \cite[Section 2.4]{Bridgeland13} up to the sign.
\end{rem}

Now we present the mutation formula of the Voros symbols for the signed pops,
where we use the notations \eqref{eq:Voros1}, \eqref{eq:vdef1},
and \eqref{eq:vdef2}.

 \begin{thm}[Mutation formula of the Voros symbols for the signed pop $\kappa^{(\ve)}_p$]
\label{thm:mutationV2}
 For $i=1,\dots,n$ and any double pole $q$ of $G$,
 we have
\begin{align}
\label{eq:v32}
\lim_{\delta\rightarrow +0}
{\mathcal S}[e^{{v}'_{i}}] &=
\lim_{\delta\rightarrow +0}
 {\mathcal S}[e^{{v}_{i}}],
\\
\label{eq:v3}
\lim_{\delta\rightarrow +0}
{\mathcal S}[e^{{\tilde{v}}'_{q}}] &=
\lim_{\delta\rightarrow +0}
 {\mathcal S}[(e^{{\tilde{v}}_{q}})^{1-2\delta_{qp}}],
\\
\label{eq:W3}
\lim_{\delta\rightarrow +0}
{\mathcal S}[e^{{W}'_{i}}] &= 
\lim_{\delta\rightarrow +0}
{\mathcal S}[e^{{W}_{i}}
\left( 1 - (e^{{\tilde{v}}_{p}})^{\varepsilon} 
\right)^{\delta_{ii_p} - \delta_{ij_p}}
],
\\
\label{eq:V3}
\lim_{\delta\rightarrow +0}
{\mathcal S}[e^{{V}'_{i}}] &=
\lim_{\delta\rightarrow +0}
 {\mathcal S}[e^{{V}_{i}}
],
\end{align}
where $i_p$ and $j_p$ are the labels in
\eqref{eq:xpop1}.
\end{thm}

\begin{proof}
The formulas
\eqref{eq:v32}, \eqref{eq:W3}, and \eqref{eq:V3}
 are obtained 
 from
 Propositions \ref{prop:cycle1} (b) and
  \ref{prop:jump-at-loop},
and the facts $\langle \tilde{\gamma}_p,\beta_{i_p} \rangle=1$,
 $\langle \tilde{\gamma}_p,\beta_{j_p} \rangle=-1$
in the same way  as the proof of Theorem  \ref{thm:mutationV1}.
Let us prove \eqref{eq:v3}.
First, we consider the  nontrivial case $q=p$.
By Lemma \ref{lemma:sign-of-simple-cycles2},
the cycle $\tilde{\gamma}'_p$ for $G'$ is related to
$\tilde{\gamma}_p$ for $G$ as $\tilde{\gamma}'_p=-\tilde{\gamma}_p$.
Thus,
\begin{align}
e^{\tilde{v}'_p}
=e^{\eta V^{(-\ve\delta)}_{\tilde{\gamma}'_p}}
=e^{-\eta V^{(-\ve\delta)}_{\tilde{\gamma}_p}}.
\end{align}
By  \eqref{eq:Vint2}, there is no jump between $V^{(-\ve\delta)}_{\tilde{\gamma}_p}$
and $V^{(\ve\delta)}_{\tilde{\gamma}_p}$ for $\delta\rightarrow +0$.
Thus, we have
\begin{align}
\lim_{\delta\rightarrow +0}
{\mathcal S}
[e^{-\eta V^{(-\ve\delta)}_{\tilde{\gamma}_p}}]
=
\lim_{\delta\rightarrow +0}
{\mathcal S}
[e^{-\eta V^{(\ve\delta)}_{\tilde{\gamma}_p}}]
=
\lim_{\delta\rightarrow +0}
{\mathcal S}
[(e^{\tilde{v}_p})^{-1}].
\end{align}
Let us consider the case $q\neq p$.
Since we assume that $G_0$ has no saddle trajectory other than
$\ell_0$, the sign $\oplus/\ominus$ of $q$ does not
 change under the signed pop $\kappa^{(\ve)}_p$.
Thus, we have $\lim_{\delta \rightarrow +0}\tilde{v}'_q=
\lim_{\delta \rightarrow +0}\tilde{v}_q$,
and the equality \eqref{eq:v3} follows.
\end{proof}

We note that the mutation of $\tilde{y}_q= e^{\tilde{v}_q}$ by
the signed flips in Section \ref{subsec:mutformula1}
is trivial.
Thus, summarizing Theorems \ref{thm:mutationV1} and \ref{thm:mutationV2},
we obtain  our first main result.

\begin{thm}\label{thm:localmutation}
By the signed flips and the signed pops,
the Voros symbols 
$
x_i=e^{{W}_{i}}$
and
$\hat{y}_i= e^{V_i}$,
together with
${y}_i= e^{v_i}$ and 
$\tilde{y}_q= e^{\tilde{v}_q}$,
mutate as the variables of extended seeds
(in the sense of Section \ref{subsec:local}).
\end{thm}

In the same spirit of Section \ref{subsec:mutformula1},
we reformulate Theorem
\ref{thm:mutationV2} in terms of the Stokes automorphisms for the signed pops.

Let us $\tilde{\bbV}=\tilde{\bbV}(Q^{(\ve\delta)}(z,\eta))$ denote
 the extension of the Voros field $\bbV
 =\bbV(Q^{(\ve\delta)}(z,\eta))$ in Section \ref{subsec:mutformula1}
 by $e^{\tilde{v}_q}$'s.
 We call $\tilde{\bbV}$ the {\em extended Voros field of
 $Q^{(\ve\delta)}(z,\eta)$}.
  We define $\tilde{\bbV}'=\tilde{\bbV}(Q^{(-\ve\delta)}(z,\eta))$
 in the same way.

Again, the isomorphisms of the homology groups
 $\tau_{G,G'}$ and  $\tau^{\vee}_{G,G'}$
 in Proposition \ref{prop:cycle1} (b)
 induce the following field isomorphism
 $\tau^*_{\tilde{\bbV},\tilde{\bbV}'}:\tilde{\bbV}'\rightarrow
\tilde{\bbV}$:
\begin{align}
\label{eq:iso3}
\tau^*_{\tilde{\bbV},\tilde{\bbV}'}(e^{{v}'_i})  = e^{v_i},
\quad
\tau^*_{\tilde{\bbV},\tilde{\bbV}'}(e^{\tilde{v}'_q})  = (e^{\tilde{v}_q})^{1-2\delta_{pq}},
\quad
\tau^*_{\tilde{\bbV},\tilde{\bbV}'}(e^{{W}'_i})  = e^{{W}_i}.
\end{align}
Here we use the same symbol for the isomorphism 
in \eqref{eq:monomial-Voros-y}
and \eqref{eq:monomial-Voros-y},
since the both are 
By Lemma \ref{lem:relation-between-x-and-y-Voros}, we have
\begin{align}
\tau^*_{\tilde{\bbV},\tilde{\bbV}'}(e^{{V}'_i})  = e^{{V}_i}.
\end{align}

We also introduce the field automorphism
${\mathfrak K}_{\tilde{\bbV},p}^{(\varepsilon)} : 
 {\tilde{\bbV}} \rightarrow {\tilde{\bbV}}
 $ as follows.
\begin{align} \label{eq:epsilon-expression-of-DDP-formal2}
\begin{split}
{\mathfrak K}_{\tilde{\bbV},p}^{(\varepsilon)} 
( e^{{v}_{i}})  &=  
e^{{v}_{i}} ,
\\
{\mathfrak K}_{\tilde{\bbV},p}^{(\varepsilon)} 
( e^{{\tilde{v}}_{q}}  )&=  
e^{{\tilde{v}}_{q}},
\\
{\mathfrak K}_{\tilde{\bbV},p}^{(\varepsilon)} 
( e^{{W}_{i}}) &  =
e^{{W}_{i}} 
\left( 1 - (e^{{\tilde{v}}_p})^{\varepsilon} \right)^{\delta_{i i_p}-\delta_{ij_p}}.
 \end{split}
\end{align}
By Lemma \ref{lem:relation-between-x-and-y-Voros}, we have
\begin{align}
{\mathfrak K}_{\tilde{\bbV},p}^{(\varepsilon)} 
( e^{{V}_{i}}) &=
 e^{{V}_{i}}.
\end{align}
We call
${\mathfrak K}_{\tilde{\bbV},p}^{(\varepsilon)}$
the {\em  Stokes automorphism associated with the signed
pop $\kappa^{(\ve)}_p$}.

We denote
\begin{alignat}{3}
\label{eq:yV1}
y_i&= e^{v_i},\quad 
\tilde{y}_q&= e^{\tilde{v}_q},\quad 
x_i&=e^{{W}_{i}} ,\quad
\hat{y}_i&= e^{V_i},
\\
y'_i&=e^{v'_i},\quad 
\tilde{y}'_q&=e^{\tilde{v}'_q},\quad 
x'_i&=e^{{W}'_i} ,~~ 
\hat{y}'_i&= e^{V'_i}.
\end{alignat}

Then, it is easy to check that the following formulas hold.
\begin{align}
\label{eq:st4}
({\mathfrak K}_{\tilde{\bbV},p}^{(\varepsilon)}\circ
\tau^*_{\tilde{\bbV},\tilde{\bbV}'})
(y'_i)
&=
y_i,
\\
\label{eq:st5}
({\mathfrak K}_{\tilde{\bbV},p}^{(\varepsilon)}\circ
\tau^*_{\tilde{\bbV},\tilde{\bbV}'})
(\tilde{y}'_q)
&=
\tilde{y}_q{}^{1-2\delta_{qp}},
\\
\label{eq:st6}
({\mathfrak K}_{\tilde{\bbV},p}^{(\varepsilon)}\circ
\tau^*_{\tilde{\bbV},\tilde{\bbV}'})
(x'_i)
&=
x_i(1-(\tilde{y}_p)^{\ve})^{\delta_{ii_p}-\delta_{ij_p}},
\\
\label{eq:st7}
({\mathfrak K}_{\tilde{\bbV},p}^{(\varepsilon)}\circ
\tau^*_{\tilde{\bbV},\tilde{\bbV}'})
(\hat{y}'_i)
&=
\hat{y}_i.
\end{align}

Theorem
\ref{thm:mutationV2}
is compactly expressed in the following way.

\begin{thm}
\label{thm:Stokes2}
The following equality of operations holds:
\begin{align}
\lim_{\delta\rightarrow +0} \circ\, \mathcal{S}
=
\lim_{\delta\rightarrow +0} \circ \,\mathcal{S} \circ
{\mathfrak K}_{\tilde{\bbV},p}^{(\varepsilon)}\circ
\tau^*_{\tilde{\bbV},\tilde{\bbV}'}.
\end{align}
\end{thm}

For the completeness,
we also extend the isomorphisms
$\tau^*_{\bbV,\bbV'}: \bbV'\rightarrow \bbV$
and ${\mathfrak S}_{\bbV,k}^{(\varepsilon)}:
\bbV \rightarrow \bbV$  for the signed flips in Section
\ref{subsec:mutformula1}
to the ones
$\tau^*_{\tilde{\bbV},\tilde{\bbV}'}: \tilde{\bbV}'\rightarrow \tilde{\bbV}$
and ${\mathfrak S}_{\tilde{\bbV} ,k}^{(\varepsilon)}:
\tilde{\bbV} \rightarrow \tilde{\bbV}$
in a trivial way:
\begin{align}
\tau^*_{\tilde{\bbV},\tilde{\bbV}'}(\tilde{y}'_q)=\tilde{y}_q,
\quad
{\mathfrak S}_{\tilde{\bbV} ,k}^{(\varepsilon)}(\tilde{y}_q)
=\tilde{y}_q.
\end{align}
Then, we have
\begin{align}
({\mathfrak S}_{\tilde{\bbV},k}^{(\varepsilon)}\circ
\tau^*_{\tilde{\bbV},\tilde{\bbV}'})
(\tilde{y}'_q)
=
\tilde{y}_q,
\end{align}
and
Theorem
\ref{thm:Stokes1} still holds.

\section{Application: Identities of  Stokes automorphisms}
\label{sec:identities}

By combining all results in the previous sections
we derive the identities of  Stokes automorphisms associated with
periods of seeds in cluster algebras.

\subsection{Regular deformation and mutation of potentials}

In Section \ref{sec:mutationofStokes} we introduced  
regular deformations and signed mutations 
for Stokes graphs of Schr{\"o}dinger equations. 
Here we extend them to potential functions.

\begin{defn}
We say that the potential $Q(z,\eta)$ of 
a Schr{\"o}dinger equation \eqref{eq:Sch} 
is {\em saddle-free\/}
if its Stokes graph is saddle-free.
For a pair of two saddle-free potentials
$Q(z,\eta)$ and $Q'(z,\eta)$, we say that 
they are related by a {\em regular deformation of potentials} 
if there exists a  family of potentials
\begin{equation} \label{eq:regular-family-potential}
Q(z,\eta;t) = \sum_{n=0}^{\infty} \eta^{-n}Q_n(z;t)
\quad (0 \le t \le 1) 
\end{equation}
satisfying the following conditions:
\begin{itemize}
\item %
$Q(z,\eta;t)$ is a polynomial in $\eta^{-1}$ 
(i.e., $Q_n(z;t) = 0$ for $n \gg 1$).  
Each coefficient $Q_n(z;t)$ is analytic in $t$ 
and satisfies $Q(z,\eta) = Q(z,\eta;0)$ and 
$Q'(z,\eta) = Q(z,\eta;1)$. 
\item %
For any $t$, $Q(z,\eta;t)$ satisfies 
Assumption 2.3 and 2.4. Moreover, for any $n \ge 0$, 
the pole orders of $Q_n(z;t)$ are independent of $t$. 
\item %
The family 
$\{ \phi_t = Q_0(z;t)dz^{\otimes 2} \mid 0 \leq t \leq 1\}$ 
satisfies Condition \ref{cond:regular1}.
\end{itemize} 
\end{defn}

Let $S_{\rm odd}(z,\eta;t)$ be the formal series 
\eqref{eq:Sodd-and-Seven} defined from the potential 
$Q(z,\eta;t)$ satisfying the above
conditions. 
Since the coefficients of $S_{\rm odd}(z,\eta;t)$ are 
determined by the recursion relation \eqref{eq:recursion}, 
the coefficients of $S_{\rm odd}(z,\eta;t)$ are analytic 
not only in $z$ but also in $t$ as long as 
$S_{-1}(z;t) = \sqrt{Q_0(z;t)}$ never vanishes. 
Namely, they are analytic in $t$ as long as zeros 
and poles of $\phi_t$ (which may depend on $t$) 
do not coincides with $z$. 
Moreover, each coefficient of the Voros symbols are also 
analytic in $t$ since Condition \ref{cond:regular1} 
guarantees that zeros and poles of $\phi_t$ 
never confluence under a regular deformation. 
Note that, if a pair of two zeros (or a pair of 
a zero and a pole) of $\phi_t$ merges 
and some path which defines 
a Voros symbol is pinched by the merging pair at $t = t_0$ 
with some $t_0 \in [0,1]$, then the coefficients 
of the Voros symbols may not be analytic in 
$t$ at the point $t=t_0$ since the integrand 
(i.e., coefficients of $S_{\rm odd}(z,\eta;t)$) 
may have singularities at zeros and poles. 
However, Condition \ref{cond:regular1} guarantees that 
such a point $t_0$ never appears in a regular deformation.

\begin{rem}
Since the Stokes graphs of the Schr{\"o}dinger equations 
whose potentials are given 
by \eqref{eq:regular-family-potential} 
are saddle-free for any $t \in [0,1]$, 
the Voros symbols are Borel summable 
for any $t \in [0,1]$. Therefore, we expect that  
no Stokes phenomenon occurs to the Voros symbols 
under the regular deformation of potentials. 
In other words, for any $0\leq t_0 \leq 1$, 
we conjecture that the following equalities hold:
\begin{align} \label{eq:no-jump-conjecture}
\lim_{t\rightarrow t_0} \mathcal{S}
[e^{W_{\beta}(t)}]
= \mathcal{S}
[e^{W_{\beta} (t_0)}],
\quad
\lim_{t\rightarrow t_0} \mathcal{S}
[e^{V_{\gamma}(t)}]
= \mathcal{S}
[e^{V_{\gamma}(t_0)}].
\end{align}
Here $e^{W_{\beta}(t)}$ and 
$e^{V_{\gamma}(t)}$ denote 
the Voros symbols for the potential $Q(z,\eta;t)$, 
and we have identified the paths and cycles 
by the isomorphism \eqref{eq:iso1} 
induced by the regular deformation of 
the quadratic differentials $\phi_t$. 
 Under the conjecture,
we can identify the Voros fields 
${\mathbb V} = {\mathbb V}(Q(z,\eta))$
and ${\mathbb V}' = {\mathbb V}(Q'(z,\eta))$
by the natural isomorphism 
$e^{W_{i}}\mapsto e^{W'_{i}}
$,
$e^{V_{i}}\mapsto e^{V'_{i}}
$,
if they are related by a regular deformation 
of potentials. 
See also Remark \ref{rem:identity1} later. 
\end{rem}

Next we introduce the signed mutations  
{\em of potentials}. 
\begin{defn}
Let $Q(z,\eta)$ and $Q'(z,\eta)$ be a pair of
saddle-free potentials, and
let $G$ and $G'$ be their labeled Stokes graphs. 
Fix a sign $\varepsilon = \pm$.
Suppose that  there exists a potential 
$Q^{(0)}(z,\eta)$ satisfying the following conditions:
\begin{itemize}
\item %
The Stokes graph $G_0 = G(\phi_0)$
 has a unique 
{\em regular} saddle trajectory,
where $\phi_0$ is the 
quadratic differential associated with the 
potential $Q^{(0)}(z,\eta)$.
\item %
Let $Q^{(\theta)}(z,\eta)$ be the $S^1$-family for 
the potential $Q^{(0)}(z,\eta)$,
and let $\phi_{\theta}$ be the quadratic differential
associated with $Q^{(\theta)}(z,\eta)$.
Choose a sufficiently small $\delta>0$ such that
$G_{\ve\delta}=G(\phi_{\ve\delta})$ and $G_{-\ve\delta}=G(\phi_{-\ve\delta})$
are the saddle reductions of $G_0$.
Then,
 the potentials
$Q(z,\eta)$ and $Q^{(+\varepsilon\delta)}(z,\eta)$
(resp., $Q'(z,\eta)$ and 
$Q^{(-\varepsilon\delta)}(z,\eta)$) are related by 
a regular deformation of potentials. 
\item %
The labeled Stokes graphs $G$ and $G'$ are related by
the signed flip $G'=\mu^{(\ve)}_k(G)$
in the sense of Definition \ref{defn:flipStokes1}.
\end{itemize}
Then, we write $Q'(z,\eta) = \mu_k^{(\varepsilon)}(Q(z,\eta))$,
assuming the labels of $G$ and $G'$.
This defines the
{\em  signed flip $\mu_k^{(\varepsilon)}$ 
for potentials}.  
The {\em  signed pop $\kappa_p^{(\varepsilon)}$ 
for potentials} is defined by the same manner
by considering ``degenerate saddle trajectory"
instead of ``regular saddle trajectory" in the above .
\end{defn}

\subsection{Stokes automorphism for general cycle}

Let $Q(z,\eta)$ be a saddle-free potential,
and let $G$ be its labeled Stokes graph.
Let ${\bbV}={\bbV}(Q(z,\eta))$ be the
 Voros field of $Q(z,\eta)$.
For any cycle $\gamma\in \Gamma_G$,
$k=1,\dots,n$, and sign $\ve$,
we define the field automorphism
${\mathfrak S}_{\bbV,\gamma}^{(\varepsilon)}:{\bbV}
 \rightarrow {\bbV}$
as follows.
\begin{align}
 \label{eq:epsilon-expression-of-DDP-formal3}
\begin{split}
{\mathfrak S}_{\bbV,\gamma}^{(\varepsilon)}(\displaystyle e^{{v}_{i}})  &=  
e^{{v}_{i}} ,
\\
{\mathfrak S}_{\bbV,\gamma}^{(\varepsilon)}( e^{{W}_{i}} ) &=  
e^{{W}_{i}} 
\left( 1 + (e^{{V}_{\gamma}})^{\varepsilon} \right)
^{-\langle \gamma, \beta_i\rangle}.
\end{split}
\end{align}
Thanks to
Proposition \ref{prop:bg3}
and
 Lemma
\ref{lem:relation-between-x-and-y-Voros},
the following formula  holds.
\begin{align}
{\mathfrak S}_{\bbV,\gamma}^{(\varepsilon)}( e^{{V}_{i}} )
&= e^{{V}_{i}}
\left( 1 + (e^{{V}_{\gamma}})^{\varepsilon} \right)
^{-(\gamma,\gamma_i)}.
\end{align}
If we set $\gamma=\gamma_k$, these formulas reduce to those
 of the ``original'' Stokes automorphism ${\mathfrak S}_{\bbV,k}^{(\varepsilon)}$
 for the signed flip $\mu^{(\ve)}_k$
in Section \ref{subsec:mutformula1}.
We call ${\mathfrak S}_{\bbV,\gamma}^{(\varepsilon)}$
the {\em Stokes automorphism for a cycle $\gamma$ with sign $\ve$}.
We note the equality
\begin{align}
\label{eq:stokeseq1}
\mathfrak{S}_{\bbV,\gamma}^{(-)}
=
(\mathfrak{S}_{\bbV,-\gamma}^{(+)})^{-1}.
\end{align}

Let $Q(z,\eta)$ and $Q'(z,\eta)$ be a pair of
saddle-free potentials,
and let $G$ and $G'$ be their labeled Stokes graphs.
Suppose that $Q(z,\eta)$ and $Q'(z,\eta)$ 
are related by
\begin{align}
Q'(z,\eta) =
\begin{cases}
\mu^{(\ve')}_k(Q(z,\eta)) & \mbox{if $\alpha_k$ is flippable in $T$}\\
(\mu^{(\ve')}_k\circ \kappa^{(\ve'')}_{p})(Q(z,\eta))
 & \mbox{otherwise},
\end{cases}
\end{align}
where $T$ is the labeled Stokes triangulation for $G$, $p$ is the puncture
inside the self-folded triangle in $T$ which $\alpha_k$ belongs to,
and $\ve'$ and $\ve''$ are any signs.
Under this assumption,
let $\tau_{G,G'}:\Gamma_{G'} \rightarrow \Gamma_G$ be the one in \eqref{eq:gamut1},
where $\ve$ therein is replaced with $\ve'$.
Also, let $ \tau^*_{\bbV,\bbV'}:\bbV'\rightarrow \bbV$ be
the field automorphism  from
$\bbV=\bbV(Q(z,\eta))$ to $\bbV'=\bbV(Q'(z,\eta))$ defined by
\eqref{eq:monomial-Voros-y} and \eqref{eq:monomial-Voros-x},
where $\ve$ therein is replace with $\ve'$.
In particular, they are independent of the sign $\ve''$.

\begin{prop}
\label{prop:tsp}
 For any $\gamma'\in \Gamma_{G'}$,
 we have the following equality of isomorphisms from
 $\bbV'$ to $\bbV$:
 \begin{align}
\label{eq:ts1}
 \tau^*_{\bbV,\bbV'}\circ
 {\mathfrak S}_{\bbV',\gamma'}^{(\varepsilon)}
 =
 {\mathfrak S}_{\bbV,\tau_{G,G'}(\gamma')}^{(\varepsilon)}\circ \tau^*_{\bbV,\bbV'}.
 \end{align}
 Here,  the sign $\ve$ for $ {\mathfrak S}_{\bbV',\gamma'}^{(\varepsilon)}$
 and the sign $\ve'$ for $ \tau^*_{\bbV,\bbV'}$
 are taken independently.
\end{prop}
\begin{proof}
It is enough to show that the actions of  both hand sides
of \eqref{eq:ts1}  on $e^{W'_i}$ coincide.
By explicit calculation, this is equivalent to the condition
\begin{align}
\langle \tau_{G,G'}(\gamma'), \tau^{\vee}_{G,G'}(\beta_i')\rangle
=
\langle \gamma', \beta_i' \rangle.
\end{align}
This equality is  known (e.g., \cite[Section 3.3]{Nakanishi11c}),
and it is easily verified.
\end{proof}

\subsection{Identities of  Stokes automorphisms}
As the initial data
we choose a saddle-free potential  $Q^0(z,\eta)$ 
on a compact  Riemann surface $\Sigma$.
Let $G^0$ be its labeled Stokes graph,
$T^0$ be the associated Stokes triangulation of the bordered surface $(\bfS,\bfM,\bfA)$,
and $B^0$ be the adjacency matrix of $T^0$.

We consider the cluster algebra with the initial seed $(B^0,x^0,y^0)$.
Let $\vec{k}=(k_1,\dots,k_N)$ be a $\nu$-period of $(B^0,x^0,y^0)$.
Recall that from the sequence of labeled seeds \eqref{eq:seedseq0}, we obtain the sequence 
of labeled Stokes triangulations \eqref{eq:Stokesseq2},
which further induces 
the sequence of labeled extended  seeds
\eqref{eq:seedseq1}
and
the sequence of labeled Stokes graphs
\eqref{eq:Stokesseq0}.
We have the periodicity properties for them in Propositions
\ref{prop:period3} and \ref{prop:bg1}.

Suppose that there is a sequence of  deformations
of  potentials starting from $Q^0(z,\eta)$,
\begin{align}
\label{eq:deformation1}
Q(z,\eta)(1) =Q^0(z,\eta)
\buildrel{\tilde{\mu}^{(\ve_1)}_{k_1}} \over{\rightarrow}
 Q(z,\eta)(2) 
 \buildrel{\tilde{\mu}^{(\ve_2)}_{k_2}} \over{\rightarrow}
  \cdots
 \buildrel{\tilde{\mu}^{(\ve_N)}_{k_N}} \over{\rightarrow}
 Q(z,\eta)(N+1) 
 \buildrel{\tilde{\kappa}} \over{\rightarrow}
 Q(z,\eta)(N+2),
\end{align}
where
 the sign $\ve_t$
is the tropical sign of $y_{k_t}(t)$ for the sequence
\eqref{eq:seedseq0},
and
 $\tilde{\mu}^{(\ve_t)}_{k_t}$ and $\tilde{\kappa}$ are
the ones
in \eqref{eq:mu1} and \eqref{eq:kappa1}
but for potentials.

By Theorem \ref{thm:localmutation},
the periodicity of the labeled extended  seeds in 
\eqref{eq:seedseq1} is realized by Stokes automorphisms
and isomorphisms $\tau^*$ on  the extended Voros fields
$\tilde{\bbV}(t)$ for $Q(z,\eta)(t)$.
Among them, the Stokes automorphisms for the signed pops
induce the local rescaling.
By Proposition \ref{prop:local1} and the definition
of $\tilde{\kappa}$,
${\mathfrak K}_{\tilde{\bbV},p}^{(\varepsilon)}$ 
and
${\mathfrak K}_{\tilde{\bbV},p}^{(-\varepsilon)}$ 
  pairwise cancel and  they are safely removed.
Thus, we obtain  the following identity
of 
Stokes automorphisms
and isomorphisms on  the Voros fields
${\bbV}(t)$ for $Q(z,\eta)(t)$.
\begin{align}
\label{eq:stid2}
 \mathfrak{S}_{\bbV(1),\gamma_{k_1}(1)}^{(\varepsilon_1)}
\tau^*_{\bbV(1),\bbV(2)}
 \mathfrak{S}_{\bbV(2),\gamma_{k_2}(2)}^{(\varepsilon_2)}
\tau^*_{\bbV(2),\bbV(3)}
\cdots
 \mathfrak{S}_{\bbV(N),\gamma_{k_N}(N)}^{(\varepsilon_N)}
\tau^*_{\bbV(N),\bbV(N+1)}
=\nu^*_{\bbV(1),\bbV(N+1)}.
\end{align}
Here, the composition symbol $\circ$ is omitted,
and $\nu^*_{\bbV(1),\bbV(N+1)}:\bbV(N+1) \rightarrow \bbV(1)$ is
the isomorphism defined by $e^{v_{\nu(i)}(N+1)}\mapsto e^{v_i(1)}$,
$e^{W_{\nu(i)}(N+1)}\mapsto e^{W_i(1)}$.

\begin{rem}
\label{rem:identity1}
Assuming that the conjecture 
\eqref{eq:no-jump-conjecture} holds.
Then, the left hand side of the identity 
\eqref{eq:stid2} faithfully expresses
the formula describing the effect of {\em all} 
Stokes phenomena associated with the deformation 
sequence \eqref{eq:deformation1} of potentials 
(where the Stokes phenomena relevant to pops 
are canceled out). That is, the equality 
\eqref{eq:stid2} has an {\em analytic} meaning. 
On the other hand, the equality \eqref{eq:stid2} itself holds 
regardless of the validity of the conjecture 
\eqref{eq:no-jump-conjecture} or even without the existence 
of the deformation sequence \eqref{eq:deformation1},
since it expresses the periodicity of 
the labeled seeds in \eqref{eq:seedseq0}.
\end{rem}

{}From  the identity
\eqref{eq:stid2},
one can derive the  identity among  Stokes automorphisms
acting on the initial Voros field $\bbV(1)$
by pushing forward the Stokes automorphisms 
acting on the Voros fields $\bbV(t)$ for $t>1$.
 This is our second main result.

\begin{thm}
\label{thm:sid1}
The following identity holds.
\begin{align}
\label{eq:stid1}
 \mathfrak{S}_{\bbV(1),\gamma_{k_1}(1)}^{(\varepsilon_1)}
 \mathfrak{S}_{\bbV(1),\tau_{G(1),G(2)}(\gamma_{k_2}(2))}^{(\varepsilon_2)}
 \cdots
 \mathfrak{S}_{\bbV(1),\tau_{G(1),G(N)}(\gamma_{k_N}(N))}^{(\varepsilon_N)}
 =\mathrm{id},
 \end{align}
 where $\tau_{G(1),G(t)}=\tau_{G(1),G(2)}\circ\tau_{G(2),G(3)}
 \cdots \circ\tau_{G(t-1),G(t)}$.
 Furthermore,  let   $c(t)=(c_i(t))_{i=1}^n$
 be the $c$-vector
 of 
 $ y_{k_t}(t)$ for the sequence \eqref{eq:seedseq0}
 with respect to the initial $y$-variables $y_i(1)$,
 i.e.,
 \begin{align}
 \label{eq:cvec2}
 [y_{k_t}(t)]= \sum_{i=1}^n y_ i(1)^{c_i(t)}.
  \end{align}
(See Section \ref{subsec:monomial}.)
Then, the cycle $\tau_{G(1),G(t)}(\gamma_{k_t}(t))$ therein is given by
  \begin{align}
  \label{eq:tau1}
 \tau_{G(1),G(t)}(\gamma_{k_t}(t))=
 \sum_{i=1}^n c_i(t) \gamma_i(1).
 \end{align}
 \end{thm}
 
\begin{proof}
We rewrite the left hand side of the identity \eqref{eq:stid2} by repeated application of
Proposition \ref{prop:tsp} in the following manner.
\begin{align}
\begin{split}
&\quad \ 
 \mathfrak{S}_{\bbV(1),\gamma_{k_1}(1)}^{(\varepsilon_1)}
\tau^*_{\bbV(1),\bbV(2)}
 \mathfrak{S}_{\bbV(2),\gamma_{k_2}(2)}^{(\varepsilon_2)}
\tau^*_{\bbV(2),\bbV(3)}
 \mathfrak{S}_{\bbV(3),\gamma_{k_3}(3)}^{(\varepsilon_3)}
\tau^*_{\bbV(3),\bbV(4)}
\cdots\\
&=
 \mathfrak{S}_{\bbV(1),\gamma_{k_1}(1)}^{(\varepsilon_1)}
 \mathfrak{S}_{\bbV(1),\tau_{G(1),G(2)}(\gamma_{k_2}(2))}^{(\varepsilon_2)}
\tau^*_{\bbV(1),\bbV(3)}
 \mathfrak{S}_{\bbV(3),\gamma_{k_3}(3)}^{(\varepsilon_3)}
\tau^*_{\bbV(3),\bbV(4)}
\cdots\\
&=
 \mathfrak{S}_{\bbV(1),\gamma_{k_1}(1)}^{(\varepsilon_1)}
 \mathfrak{S}_{\bbV(1),\tau_{G(1),G(2)}(\gamma_{k_2}(2))}^{(\varepsilon_2)}
 \mathfrak{S}_{\bbV(1),\tau_{G(1),G(3)}(\gamma_{k_3}(3))}^{(\varepsilon_3)}
\tau^*_{\bbV(1),\bbV(4)}
\cdots\\
&=
 \mathfrak{S}_{\bbV(1),\gamma_{k_1}(1)}^{(\varepsilon_1)}
 \mathfrak{S}_{\bbV(1),\tau_{G(1),G(2)}(\gamma_{k_2}(2))}^{(\varepsilon_2)}
\cdots
 \mathfrak{S}_{\bbV(1),\tau_{G(1),G(N)}(\gamma_{k_N}(N))}^{(\varepsilon_N)}
\tau^*_{\bbV(1),\bbV(N+1)}.
\end{split}
\end{align}
Thanks to the choice of the sign $\ve_t=\ve(y_{k_t}(t))$,
$\tau_{G(t),G(t+1)}$ acts as the mutation of the (logarithm of) tropical $y$-variables
for the sequence \eqref{eq:seedseq0}.
See the remark after Proposition \ref{prop:cycle1}.
Thus, the claim  \eqref{eq:tau1} follows from \eqref{eq:cvec2}.
It also implies 
$\tau_{G(1),G(N+1)}(\gamma_{\nu(i)}) = \gamma_i$
due to the $\nu$-periodicity of 
the sequence \eqref{eq:seedseq0}.
Therefore, we  have
$\tau^*_{\bbV(1),\bbV(N+1)}=\nu^*_{\bbV(1),\bbV(N+1)}$,
which cancels
the right hand side of
the identity \eqref{eq:stid2}.
Thus, we obtain  the identity \eqref{eq:stid1}.
\end{proof}

\begin{rem}
The derivation of the identity \eqref{eq:stid1} is parallel to that
of the quantum dilogarithm identities in \cite{Keller11,Kashaev11}.
\end{rem}

\begin{ex}[Pentagon relation (5)]
Let $G=G(1)$ be the initial labeled Stokes graph in
Figure \ref{fig:pentagon4}.
Let $\gamma_1$ and $\gamma_2$
be the simple cycles of $\Gamma_G$.
Let $\bbV=\bbV(1)$ be the initial Voros field.
With the data in \eqref{eq:tropyp1} and \eqref{eq:tropsign2},
the identity \eqref{eq:stid1} reads
\begin{align}
 \mathfrak{S}_{\bbV,\gamma_{1}}^{(+)}
 \mathfrak{S}_{\bbV,\gamma_{1}+\gamma_{2}}^{(+)}
 \mathfrak{S}_{\bbV,\gamma_{2}}^{(+)}
  \mathfrak{S}_{\bbV,-\gamma_{1}}^{(-)}
   \mathfrak{S}_{\bbV,-\gamma_{2}}^{(-)}
   =\mathrm{id}.
 \end{align}
 Using the simplified notation  $\mathfrak{S}_{\bbV,\gamma}^{(+)}
 =  \mathfrak{S}_{\gamma}$
 and the equality  \eqref{eq:stokeseq1},
 the identity is written as
 \begin{align}
 \mathfrak{S}_{\gamma_{1}}
 \mathfrak{S}_{\gamma_{1}+\gamma_{2}}
 \mathfrak{S}_{\gamma_{2}}
  (\mathfrak{S}_{\gamma_{1}})^{-1}
   (\mathfrak{S}_{\gamma_{2}})^{-1}
   =\mathrm{id},
 \end{align}
 or equivalently,
  \begin{align}
   \mathfrak{S}_{\gamma_{2}}
  \mathfrak{S}_{\gamma_{1}}
  =
 \mathfrak{S}_{\gamma_{1}}
 \mathfrak{S}_{\gamma_{1}+\gamma_{2}}
 \mathfrak{S}_{\gamma_{2}}.
 \end{align}
 This is the  identity \eqref{eq:pentagon1} by \cite{Delabaere93}.
\end{ex}

\appendix
\bigskip
\noindent
{\bf Appendix}

\section{Proof of Theorem \ref{thm:DDP-analytic}}
\label{section:proof-of-DDP}
Here we give a proof of Theorem \ref{thm:DDP-analytic}.
Let us recall the situation. 
We consider the case that the Stokes graph 
$G_0 = G(\phi)$ has a {\em unique} saddle trajectory $\ell_0$, 
and it is a {\em regular} saddle trajectory. 
Note that, since there are no saddle trajectory 
other than $\ell_0$, other Stokes curves must flow 
into a point in $P_{\infty}$ at one end. 
To specify the situation, in addition to 
Figure \ref{fig:Stokes-auto},
we take branch cuts and assign $\oplus$ and $\ominus$ 
as in Figure \ref{fig:gamma0-and-alpha}. 
(Note that we can show Theorem \ref{thm:DDP-analytic} 
in the same manner as presented here 
if the signs are assigned differently.)
Then, the saddle class $\gamma_0$ associated with 
$\ell_0$ has the orientation 
shown in Figure \ref{fig:gamma0-and-alpha}.

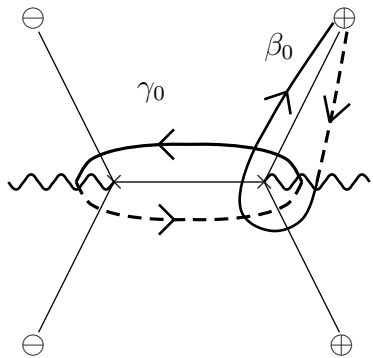
\begin{figure}[h]
\begin{center}
\begin{pspicture}(5,-2.3)(11,2.3)
%
%
\psset{fillstyle=solid, fillcolor=black}
\psset{fillstyle=none}
\rput[c]{0}(5.93,2.18){$\ominus$} 
\rput[c]{0}(5.93,-2.17){$\ominus$} 
\rput[c]{0}(10.07,2.18){$\oplus$} 
\rput[c]{0}(10.03,-2.17){$\oplus$} 
\rput[c]{0}(7,0){\small $\times$}
\rput[c]{0}(9,0){\small $\times$}
\rput[c]{0}(7.5,1.2){$\gamma_0$} 
\psset{linewidth=1.2pt}
\pscurve(6.5,0)(6.7,0.3)(8,0.5)
\pscurve(9.5,0)(9.3,0.3)(8,0.5)
\psset{linewidth=1.2pt, linestyle=dashed}
\pscurve(6.5,0)(6.7,-0.3)(7.95,-0.5)
\pscurve(9.5,0)(9.3,-0.3)(8.05,-0.5)
\psset{linewidth=1pt, linestyle=solid}
\psline(7.6,0.5)(7.8,0.7)
\psline(7.6,0.5)(7.8,0.3)
\psline(7.8,-0.5)(7.6,-0.7)
\psline(7.8,-0.5)(7.6,-0.3)
\rput[c]{0}(9.2,1.8){$\beta_0$} 
\psset{linewidth=1pt}
\pscurve(9.92,2.12)(9.3,1.2)(8.9,0.5)
\pscurve(8.9,0.5)(8.7,-0.4)(9.1,-0.65)(9.5,-0.5)(9.7,0)
\psset{linewidth=1.2pt,linestyle=dashed}
\pscurve(9.7,0)(9.92,1.0)(10.1,2.0)
\psset{linewidth=1pt, linestyle=solid}
\psline(9.3,1.2)(9.0,1.1)
\psline(9.3,1.2)(9.35,0.9)
\psline(9.87,0.8)(10.14,1.04)
\psline(9.87,0.8)(9.8,1.15)
\psset{linewidth=1pt,linestyle=solid}
\pscurve(7,0)(6.9,-0.1)(6.8,0)(6.7,0.1)
(6.6,0)(6.5,-0.1)(6.4,0)
(6.3,0.1)(6.2,0)(6.1,-0.1)(6,0)
(5.9,0.1)(5.8,0)(5.7,-0.1)(5.6,0)
\pscurve(9,0)(9.1,0.1)(9.2,0)(9.3,-0.1)(9.4,0)
(9.5,0.1)(9.6,0)(9.7,-0.1)(9.8,0)(9.9,0.1)(10,0)
(10.1,-0.1)(10.2,0)(10.3,0.1)(10.4,0)
\psset{linewidth=0.5pt}
%
\pscurve(7,0)(8,0)(9,0)
\pscurve(7,0)(6.5,1)(6,2)
\pscurve(7,0)(6.5,-1)(6,-2)
\pscurve(9,0)(9.5,1)(10,2)
\pscurve(9,0)(9.5,-1)(10,-2)
\end{pspicture}
\end{center}
\caption{The saddle class 
$\gamma_0\in H_1(\hat{\Sigma}\setminus\hat{P})$ and 
a path $\beta_0 \in H_1(\hat{\Sigma}\setminus\hat{P}_0,
\hat{P}_{\infty})$. The picture depicts a part of the Stokes 
graph $G_0$ near the saddle trajectory $\ell_0$.} 
\label{fig:gamma0-and-alpha}
\end{figure}

Let $G_{\pm\delta} = G(\phi_{\pm\delta})$ 
be the saddle reductions of $G_0$ (where $\delta>0$ is a 
sufficiently small number), and 
let ${\mathcal S}_{\pm}[e^{V_{\gamma}}]$ 
(resp., ${\mathcal S}_{\pm}[e^{W_{\beta}}]$) 
be the Borel sum of the Voros symbol 
$e^{V_{\gamma}}$ (resp., $e^{W_{\beta}}$)
in the direction $\pm\delta$ 
(see Section \ref{section:saddle-reduction}). 
Here $\gamma \in H_1(\hat{\Sigma}\setminus\hat{P})$ 
is any cycle and $\beta \in H_1(\hat{\Sigma}\setminus
\hat{P}_0,\hat{P}_{\infty})$ is any path. 
Now we will show the equality 
\begin{equation} \label{eq:DDP-analytic-appendix}
{\mathcal S}_{-}[e^{{W}_{\beta}}] = 
{\mathcal S}_{+}[e^{{W}_{\beta}}](1 + {\mathcal S}_{+}
[e^{{V}_{\gamma_0}}])^{-\langle\gamma_0,\beta\rangle},~~
{\mathcal S}_{-}[e^{{V}_{\gamma}}] = {\mathcal S}_{+}
[e^{{V}_{\gamma}}](1 + {\mathcal S}_{+}
[e^{{V}_{\gamma_0}}])^{-(\gamma_0,\gamma)}.
\end{equation}
(i.e., the equality \eqref{eq:DDP-analytic}) 
on a domain containing 
$\{\eta \in {\mathbb R}~|~ \eta \gg 1 \}$. 

Firstly, we show an important result for 
the Borel sums of the Voros symbols. 
\begin{lem} \label{lem:no-jump-formula}
If $\beta$ (resp., $\gamma$) does not intersect with 
the saddle trajectory $\ell_0$, 
the Borel sums ${\mathcal S}_{\pm}[e^{W_{\beta}}]$ 
(resp., ${\mathcal S}_{\pm}[e^{V_{\gamma}}]$) 
does not jump. That is, the equalities 
\begin{equation} \label{eq:no-jump-formula}
{\mathcal S}_{-}[e^{{W}_{\beta}}] = 
{\mathcal S}_{+}[ e^{{W}_{\beta}}]
~~\text{if $\langle\gamma_0,\beta\rangle=0$}, \quad
{\mathcal S}_{-}[e^{{V}_{\gamma}}] = 
{\mathcal S}_{+}[ e^{{V}_{\gamma}}]
~~\text{if $(\gamma_0,\gamma)=0$}
\end{equation}
hold as analytic functions of $\eta$ on 
$\{\eta \in {\mathbb R}~|~\eta \gg 1  \}$.
Especially, the Borel sum of $e^{{V}_{\gamma_0}}$ 
does not jump.
\end{lem}
\begin{proof}
It follows from the assumption and Corollary 
\ref{cor:summability} that the Voros symbols 
$e^{W_{\beta}}$ and $e^{V_{\gamma}}$ are 
Borel summable (in the direction $0$).
Since the Stokes graph $G_0$  contains
no other saddle trajectory than $\ell_0$, 
we can prove the statement in the same manner 
as in the proof of Lemma 
\ref{lemma:saddle-reduction-of-Voros}.
\end{proof}

Consequently, for a path $\beta$ which never intersects 
with $\ell_0$, both  of the Borel sums 
${\mathcal S}_{+}[e^{{W}_{\beta}}]$ 
and ${\mathcal S}_{-}[e^{{W}_{\beta}}]$
coincide with ${\mathcal S}[e^{{W}_{\beta}}]$ 
($={\mathcal S}_{0}[e^{{W}_{\beta}}]$).  
Similarly, we have 
${\mathcal S}_{\pm}[e^{{V}_{\gamma}}] = 
{\mathcal S}[e^{{V}_{\gamma}}]$ for a cycle $\gamma$ if 
it never intersects with $\ell_0$. 
Below we write 
${\mathcal S}_{\pm}[e^{{W}_{\beta}}] = 
{\mathcal S}[e^{{W}_{\beta}}]$ 
etc. when a path or a cycle 
does not intersect with $\ell_0$. 

The formula \eqref{eq:no-jump-formula} is a part 
of the desired formula \eqref{eq:DDP-analytic-appendix}. 
In what follows we try to show 
\eqref{eq:DDP-analytic-appendix} for the paths and 
the cycles which intersect with $\ell_0$. 
Note that, since any path 
$\beta \in H_1(\hat{\Sigma}\setminus
\hat{P}_0,\hat{P}_{\infty})$ and any cycle 
$\gamma \in H_1(\hat{\Sigma}\setminus\hat{P})$ 
can be written by a finite number of paths 
whose end-points are contained in $\hat{P}_{\infty}$
in the relative homology group 
$H_1(\hat{\Sigma}\setminus\hat{P}_0,\hat{P}_{\infty})$ 
(see the proof of Lemma 
\ref{lemma:saddle-reduction-of-Voros}), 
it suffices to show  \eqref{eq:DDP-analytic-appendix} 
for any such a path $\beta \in H_1(\hat{\Sigma}\setminus
\hat{P}_0,\hat{P}_{\infty})$. 

\begin{lem} 
Let $\beta_0 \in H_1(\hat{\Sigma}\setminus\hat{P}_0,
\hat{P}_{\infty})$ be the path depicted in 
Figure \ref{fig:gamma0-and-alpha}. Then, the 
following equality holds as analytic functions of $\eta$ on 
$\{\eta \in {\mathbb R}~|~\eta \gg 1  \}$:
\begin{equation} \label{eq:pre-DDP-formula}
{\mathcal S}_{-}[e^{{W}_{\beta_0}}] = 
{\mathcal S}_{+}[e^{{W}_{\beta_0}}]
(1 + {\mathcal S}[e^{{V}_{\gamma_0}}]).
\end{equation}
\end{lem}

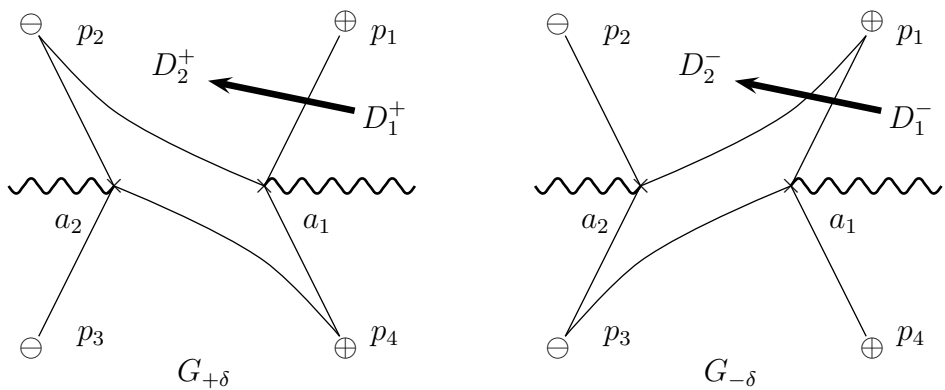
\begin{figure}[h]
\begin{center}
\begin{pspicture}(0.5,-2.5)(13,2.5)
%
%
\psset{fillstyle=solid, fillcolor=black}
\psset{fillstyle=none}
\rput[c]{0}(5.08,2.2){$\oplus$} 
\rput[c]{0}(5.08,-2.15){$\oplus$}
\rput[c]{0}(0.9,2.15){$\ominus$} 
\rput[c]{0}(0.9,-2.15){$\ominus$}
\rput[c]{0}(2,0){\small $\times$}
\rput[c]{0}(4,0){\small $\times$}
\rput[c]{0}(3.2,-2.5){$G_{+\delta}$} 
\rput[c]{0}(5.6,0.9){$D^{+}_1$} 
\rput[c]{0}(2.8,1.6){$D^{+}_2$} 
\rput[c]{0}(5.6,2.0){$p_1$} 
\rput[c]{0}(1.7,2.0){$p_2$}
\rput[c]{0}(1.7,-2.0){$p_3$}  
\rput[c]{0}(5.6,-2.0){$p_4$} 
\rput[c]{0}(4.7,-0.5){$a_1$} 
\rput[c]{0}(1.4,-0.5){$a_2$} 
\rput[c]{0}(12.08,2.2){$\oplus$} 
\rput[c]{0}(12.08,-2.15){$\oplus$}
\rput[c]{0}(7.9,2.15){$\ominus$} 
\rput[c]{0}(7.9,-2.15){$\ominus$}
\rput[c]{0}(9,0){\small $\times$}
\rput[c]{0}(11,0){\small $\times$}
\rput[c]{0}(10.2,-2.5){$G_{-\delta}$} 
\rput[c]{0}(12.6,0.9){$D^{-}_1$} 
\rput[c]{0}(9.8,1.6){$D^{-}_2$} 
\rput[c]{0}(12.6,2.0){$p_1$} 
\rput[c]{0}(8.7,2.0){$p_2$}
\rput[c]{0}(8.7,-2.0){$p_3$}  
\rput[c]{0}(12.6,-2.0){$p_4$} 
\rput[c]{0}(11.7,-0.5){$a_1$} 
\rput[c]{0}(8.4,-0.5){$a_2$} 
\psset{linewidth=1pt,linestyle=solid}
\pscurve(2,0)(1.9,-0.1)(1.8,0)(1.7,0.1)(1.6,0)(1.5,-0.1)(1.4,0)
(1.3,0.1)(1.2,0)(1.1,-0.1)(1,0)(0.9,0.1)(0.8,0)(0.7,-0.1)(0.6,0)
\pscurve(4,0)(4.1,0.1)(4.2,0)(4.3,-0.1)(4.4,0)
(4.5,0.1)(4.6,0)(4.7,-0.1)(4.8,0)(4.9,0.1)(5,0)
(5.1,-0.1)(5.2,0)(5.3,0.1)(5.4,0)(5.5,-0.1)(5.6,0)
(5.7,0.1)(5.8,0)(5.9,-0.1)(6.0,0)
\pscurve(9,0)(8.9,-0.1)(8.8,0)(8.7,0.1)(8.6,0)
(8.5,-0.1)(8.4,0)(8.3,0.1)(8.2,0)(8.1,-0.1)(8,0)
(7.9,0.1)(7.8,0)(7.7,-0.1)(7.6,0)
\pscurve(11,0)(11.1,0.1)(11.2,0)(11.3,-0.1)(11.4,0)
(11.5,0.1)(11.6,0)(11.7,-0.1)(11.8,0)(11.9,0.1)(12,0)
(12.1,-0.1)(12.2,0)(12.3,0.1)(12.4,0)(12.5,-0.1)
(12.6,0)(12.7,0.1)(12.8,0)(12.9,-0.1)(13,0)
\psset{linewidth=2.5pt}
\psline[arrows=->](5.2,1.0)(3.25,1.4) 
\psset{linewidth=2.5pt}
\psline[arrows=->](12.2,1.0)(10.25,1.4) 
\psset{linewidth=0.5pt}
%
\pscurve(2,0)(4,-1)(5,-2)
\pscurve(2,0)(1.5,1)(1,2)
\pscurve(2,0)(1.5,-1)(1,-2)
\pscurve(4,0)(2,1)(1,2)
\pscurve(4,0)(4.5,1)(5,2)
\pscurve(4,0)(4.5,-1)(5,-2)
\pscurve(9,0)(11,1)(12,2)
\pscurve(9,0)(8.5,1)(8,2)
\pscurve(9,0)(8.5,-1)(8,-2)
\pscurve(11,0)(9,-1)(8,-2)
\pscurve(11,0)(11.5,1)(12,2)
\pscurve(11,0)(11.5,-1)(12,-2)
\end{pspicture}
\end{center}
\caption{Two connection problems.} 
\label{fig:two-connection-problems}
\end{figure}

\begin{proof}
Let us consider two connection problems for the WKB solutions 
indicated in Figure \ref{fig:two-connection-problems} 
which depicts a part of the Stokes graph $G_{+\delta}$ 
and $G_{-\delta}$. The first one is the connection problem 
from  the Stokes region $D_1^{+}$ to $D_2^{+}$ in 
the Stokes graph $G_{+\delta}$, while the second one 
is the connection problem from  
the Stokes region $D_1^{-}$ to $D_2^{-}$ in 
the Stokes graph $G_{-\delta}$ along the thick paths 
depicted in Figure \ref{fig:two-connection-problems}. 

Take the WKB solutions 
\begin{equation} \label{eq:app-WKB-a1}
\psi_{\pm,a_1}(z,\eta) = 
\frac{1}{\sqrt{S_{\rm odd}(z,\eta)}}
\exp\left(\pm\int_{a_1}^{z}
S_{\rm odd}(z,\eta) dz\right).
\end{equation}
normalized at the turning point $a_1$ 
depicted in Figure \ref{fig:two-connection-problems}.
Since the saddle reductions $G_{\pm\delta}$ are saddle-free, 
Corollary \ref{cor:summability} ensures that 
the Borel sum of the WKB solutions are well-defined 
on each Stokes region of $G_{\pm\delta}$. 
We denote by $\Psi_{\pm,a_1}^{D_1^+}$ etc.\ 
the Borel sum of $\psi_{\pm,a_1}(z,\eta)$ in 
the Stokes region $D_1^+$ etc. 
Then, using Theorem \ref{thm:Voros-formula},  we have 
the following formula for the first connection problem:
\begin{eqnarray}
\begin{cases} \label{eq:conn-formula-plus}
\Psi^{D_1^+}_{+,a_1} = \Psi^{D_2^+}_{+,a_1} 
+ i \Psi^{D_2^+}_{-,a_1}, \\ 
\Psi^{D_1^+}_{-,a_1} = \Psi^{D_2^+}_{-,a_1}.
\end{cases}
\end{eqnarray}
On the other hand, in the second connection problem 
we have to cross two Stokes curves emanating from 
$a_1$ and $a_2$, respectively, as in 
Figure \ref{fig:two-connection-problems}.
In order to use Theorem \ref{thm:Voros-formula} 
on the Stokes curve emanating from $a_2$, 
we need to change the normalization of 
the WKB solutions from \eqref{eq:app-WKB-a1} to 
\begin{equation} \label{eq:app-WKB-a2}
\psi_{\pm,a_2}(z,\eta) = 
\frac{1}{\sqrt{S_{\rm odd}(z,\eta)}}
\exp\left(\pm\int_{a_2}^{z}
S_{\rm odd}(z,\eta)dz\right)
\end{equation}
which is normalized at $a_2$. Using the relation 
\begin{equation}
\psi_{\pm, a_{1}}(z,\eta)=
\exp\left(\pm\int_{a_{1}}^{a_{2}}
S_{\rm odd}(z,\eta)dz\right)\psi_{\pm, a_{2}}(z,\eta)  
= \exp\left(\pm \frac{1}{2}V_{\gamma_0}(\eta)\right) 
\psi_{\pm, a_{2}}(z,\eta)
\end{equation}
(here $\gamma_0$ is the cycle depicted in 
Figure \ref{fig:gamma0-and-alpha}) of the WKB solutions 
with different normalizations, we have the following 
formula for the second connection problem 
(see \cite[Section 3]{Kawai05}):
\begin{eqnarray}
\begin{cases} \label{eq:conn-formula-minus}
\Psi^{D_1^-}_{+,a_1} = \Psi^{D_2^-}_{+,a_1} 
+ i \hspace{+.1em} (1 + {\mathcal S}[e^{V_{\gamma_0}}]) 
\hspace{+.1em} \Psi^{D_2^-}_{-,a_1}, \\ 
\Psi^{D_1^-}_{-,a_1} = \Psi^{D_2^-}_{-,a_1}.
\end{cases}
\end{eqnarray}
Thus, we obtain the two formulas 
\eqref{eq:conn-formula-plus} and 
\eqref{eq:conn-formula-minus}.

Next, let us rewrite the formulas 
\eqref{eq:conn-formula-plus} and 
\eqref{eq:conn-formula-minus} to the formulas 
for the WKB solutions 
\begin{equation} \label{eq:app-WKB-p1}
\psi_{\pm,p_1}(z,\eta) = 
\frac{1}{\sqrt{S_{\rm odd}(z,\eta)}}
\exp\left\{ \pm\left(\eta\int_{a_1}^{z}
\sqrt{Q_{0}(z)}dz+\int_{p_1}^z
S_{\rm odd}^{\rm reg}(z,\eta)dz\right) \right\}
\end{equation}
normalized at $p_1 \in \hat{P}_{\infty}$,
 which  
is an end-point of a Stokes curve emanating from $a_1$ 
as depicted in Figure \ref{fig:two-connection-problems}.
The WKB solutions \eqref{eq:app-WKB-a1} and 
\eqref{eq:app-WKB-p1} are related as 
\begin{equation} \label{eq:change-of-normalization-app}
\psi_{\pm, a_{1}}(z,\eta) = 
\exp\left(\pm \int_{a_{1}}^{p_1}
S^{\rm reg}_{\rm odd}(z,\eta)dz \right)\psi_{\pm,p_1}
= \exp\left(\pm \frac{1}{2} W_{\beta_0}(\eta) \right)
\psi_{\pm,p_1}(z,\eta),
\end{equation}
where $\beta_0$ is the path designated in Figure 
\ref{fig:gamma0-and-alpha}. Therefore, it follows 
from \eqref{eq:change-of-normalization-app},
\eqref{eq:conn-formula-plus} and 
\eqref{eq:conn-formula-minus} that 
the following equalities hold:
\begin{eqnarray}
\begin{cases} \label{eq:conn-formula-p1-plus}
\Psi^{D_1^+}_{+,p_1} = \Psi^{D_2^+}_{+,p_1} + 
i~{\mathcal S}_{+}[e^{-{W}_{\beta_0}}]~
\Psi^{D_2^+}_{-,p_1}, \\[+.2em] 
\Psi^{D_1^{+}}_{-,p_1} = \Psi^{D_2^{+}}_{-,p_1}.
\end{cases} \\[+.2em]
\begin{cases} \label{eq:conn-formula-p1-minus}
\Psi^{D_1^-}_{+,p_1} = \Psi^{D_2^-}_{+,p_1} + i 
\left(1 + {\mathcal S}[e^{{V}_{\gamma_0}}]\right) 
{\mathcal S}_{-}[e^{-{W}_{\beta_0}}]~
\Psi^{D_2^-}_{-,p_1}, \\[+.2em] 
\Psi^{D_1^{-}}_{-,p_1} = \Psi^{D_2^{-}}_{-,p_1}.
\end{cases}
\end{eqnarray}
Taking $\delta>0$ sufficiently small, 
we may assume that, for a fixed $z_1 \in D_{1}^{\pm}$ 
(resp., $z_2 \in D_{2}^{\pm}$) the path from $p_1$ to $z$,
which normalize the WKB solution \eqref{eq:app-WKB-p1} 
when $z$ lies in a neighborhood of $z_1$ (resp., $z_2$), 
is admissible in any direction $\theta$ satisfying 
$-\delta \le \theta \le + \delta$. Therefore, Proposition 
\ref{prop:S1-action-and-summability} implies that  
\begin{equation} \label{eq:no-PSP-formula-app}
\Psi^{D_1^{+}}_{\pm,p_1} = 
\Psi^{D_1^{-}}_{\pm,p_1}, \quad 
\Psi^{D_2^{+}}_{\pm,p_1} = 
\Psi^{D_2^{-}}_{\pm,p_1}
\end{equation}
holds as analytic functions of both $z$ and 
$\eta$ for a sufficiently large $\eta \gg 1$. 
Therefore, comparing the connection multipliers 
in \eqref{eq:conn-formula-p1-plus}
and \eqref{eq:conn-formula-p1-minus} and 
using the equality \eqref{eq:no-PSP-formula-app}, 
we obtain \eqref{eq:pre-DDP-formula}.
\end{proof}

\begin{rem}
Since the WKB solutions \eqref{eq:app-WKB-a1} are 
normalized along a path which intersects with the 
saddle trajectory $\ell_0$, we can not expect 
similar equalities as \eqref{eq:no-PSP-formula-app} 
holds for the Borel sums of \eqref{eq:app-WKB-a1}. 
\end{rem}

\begin{figure}[h]
\begin{center}
\begin{pspicture}(0.5,-2.5)(13,2.5)
%
%
\psset{fillstyle=solid, fillcolor=black}
\psset{fillstyle=none}
\rput[c]{0}(5.08,2.2){$\oplus$} 
\rput[c]{0}(5.08,-2.15){$\oplus$}
\rput[c]{0}(0.9,2.15){$\ominus$} 
\rput[c]{0}(0.9,-2.15){$\ominus$}
\rput[c]{0}(2,0){\small $\times$}
\rput[c]{0}(4,0){\small $\times$}
\rput[c]{0}(3.2,-2.5){$G_{+\delta}$} 
\rput[c]{0}(5.6,2.0){$p_1$} 
\rput[c]{0}(1.7,2.15){$p_2$}
\rput[c]{0}(1.7,-2.0){$p_3$}  
\rput[c]{0}(5.6,-2.0){$p_4$} 
\rput[c]{0}(4.7,-0.5){$a_1$} 
\rput[c]{0}(1.4,-0.5){$a_2$} 
\rput[c]{0}(5.8,+0.7){$\beta_{p^\ast_1, p_4}$} 
\rput[c]{0}(3.1,-0.9){$\beta_{p_4, p_2}$} 
\rput[c]{0}(3.7,+1.8){$\beta_{p_2, p_1}$} 
\rput[c]{0}(12.08,2.2){$\oplus$} 
\rput[c]{0}(12.08,-2.15){$\oplus$}
\rput[c]{0}(7.9,2.15){$\ominus$} 
\rput[c]{0}(7.9,-2.15){$\ominus$}
\rput[c]{0}(9,0){\small $\times$}
\rput[c]{0}(11,0){\small $\times$}
\rput[c]{0}(10.2,-2.5){$G_{-\delta}$} 
\rput[c]{0}(12.6,2.0){$p_1$} 
\rput[c]{0}(8.7,2.15){$p_2$}
\rput[c]{0}(8.7,-2.0){$p_3$}  
\rput[c]{0}(12.6,-2.0){$p_4$} 
\rput[c]{0}(11.7,-0.5){$a_1$} 
\rput[c]{0}(8.4,-0.5){$a_2$} 
\rput[c]{0}(12.8,+0.7){$\beta_{p^\ast_1, p_4}$} 
\rput[c]{0}(10.1,-0.9){$\beta_{p_4, p_2}$} 
\rput[c]{0}(10.6,+1.8){$\beta_{p_2, p_1}$} 
\psset{linewidth=1pt,linestyle=solid}
\pscurve(2,0)(1.9,-0.1)(1.8,0)(1.7,0.1)
(1.6,0)(1.5,-0.1)(1.4,0)(1.3,0.1)(1.2,0)
(1.1,-0.1)(1,0)(0.9,0.1)(0.8,0)(0.7,-0.1)(0.6,0)
\pscurve(4,0)(4.1,0.1)(4.2,0)(4.3,-0.1)(4.4,0)
(4.5,0.1)(4.6,0)(4.7,-0.1)(4.8,0)(4.9,0.1)(5,0)
(5.1,-0.1)(5.2,0)(5.3,0.1)(5.4,0)(5.5,-0.1)(5.6,0)
(5.7,0.1)(5.8,0)(5.9,-0.1)(6.0,0)
\pscurve(9,0)(8.9,-0.1)(8.8,0)(8.7,0.1)(8.6,0)
(8.5,-0.1)(8.4,0)(8.3,0.1)(8.2,0)(8.1,-0.1)(8,0)
(7.9,0.1)(7.8,0)(7.7,-0.1)(7.6,0)
\pscurve(11,0)(11.1,0.1)(11.2,0)(11.3,-0.1)(11.4,0)
(11.5,0.1)(11.6,0)(11.7,-0.1)(11.8,0)(11.9,0.1)(12,0)
(12.1,-0.1)(12.2,0)(12.3,0.1)(12.4,0)(12.5,-0.1)
(12.6,0)(12.7,0.1)(12.8,0)(12.9,-0.1)(13,0)
\psset{linewidth=0.5pt}
\pscurve(2,0)(3.95,-1)(4.85,-2)
\pscurve(2,0)(1.5,1)(0.9,1.95)
\pscurve(2,0)(1.5,-1)(1,-2)
\pscurve(4,0)(2.05,1)(1.05,2)
\pscurve(4,0)(4.5,1)(5,2)
\pscurve(4,0)(4.5,-1)(5.00,-1.95)
\pscurve(9,0)(11,1)(12,2)
\pscurve(9,0)(8.5,1)(8,2)
\pscurve(9,0)(8.5,-1)(8,-2)
\pscurve(11,0)(9,-1)(8,-2)
\pscurve(11,0)(11.5,1)(12,2)
\pscurve(11,0)(11.5,-1)(12,-2)
\psset{linewidth=1pt}
\pscurve(5.1,0)(5.1,-1)(5.1,-1.95)
\pscurve(4.90,-1.9)(3.45,-0.2)(2.45,0.3)(1.02,1.95)
\pscurve(1.06,2.1)(3.5,1.0)(4.9,2.1)
\psline(5.1,-1.3)(5.3,-1.1) \psline(5.1,-1.3)(4.9,-1.1)
\psline(3,0.05)(3.25,0.15)\psline(2.95,0.05)(3.15,-0.25)
\psline(3.2,1)(3.0,1.25)\psline(3.2,1)(2.92,0.85)
\pscurve(12.1,0)(12.1,-1)(12.1,-1.95)
\pscurve(11.85,-2.1)(10.45,-0.2)(9.45,0.3)(8.05,2.0)
\pscurve(8.06,2.1)(10.0,1.0)(11.9,2.1)
\psline(12.1,-1.3)(12.3,-1.1) 
\psline(12.1,-1.3)(11.9,-1.1)
\psline(10,0.05)(10.25,0.15)
\psline(9.95,0.05)(10.15,-0.25)
\psline(10.0,1)(9.77,1.22)\psline(10.0,1)(9.77,0.82)
\psset{linewidth=1pt, linestyle=dashed}
\pscurve(5.1,2)(5.1,1)(5.1,0)
\pscurve(12.1,2)(12.1,1)(12.1,0)
\end{pspicture}
\end{center}
\caption{Decomposition of the path $\beta_0$.} 
\label{fig:decomposition-of-paths}
\end{figure}
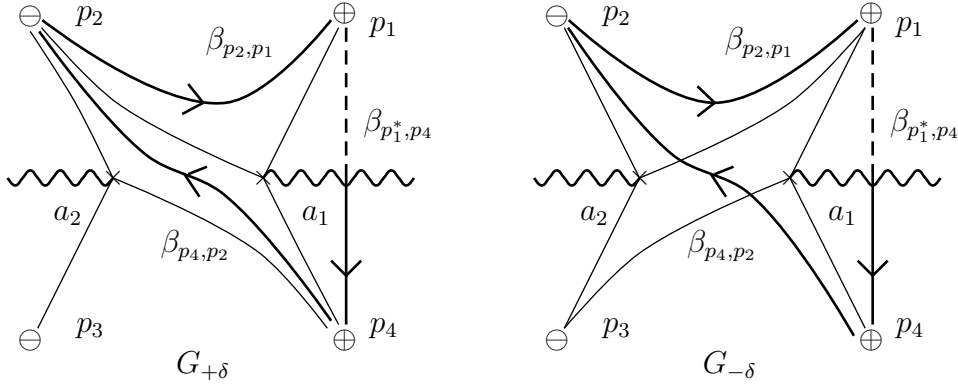

The equality \eqref{eq:pre-DDP-formula} is also 
one of the desired formula \eqref{eq:DDP-analytic-appendix} 
since the intersection number is 
$\langle\gamma_0,\beta_0\rangle=-1$ as 
depicted in Figure \ref{fig:gamma0-and-alpha}. 
From this relation we can derive 
\eqref{eq:DDP-analytic-appendix} for any path. 
The path $\beta_0$ in Figure \ref{fig:gamma0-and-alpha} 
has the following decomposition in 
$H_1(\hat{\Sigma}\setminus\hat{P}_0,\hat{P}_{\infty})$: 
\begin{equation} \label{eq:decomposition-of-beta}
\beta_0 = \beta_{p_1^{*},p_4} + \beta_{p_4,p_2} + \beta_{p_2,p_1} 
\end{equation}
as depicted in 
Figure \ref{fig:decomposition-of-paths}. 
Here $p_1^{\ast}$  represents the point on the second 
sheet of $\hat{\Sigma}$ corresponding to $p_1$,
and a dashed line is a path on the second sheet of $\hat{\Sigma}$.
In the decomposition \eqref{eq:decomposition-of-beta}
the path $\beta_{p_4,p_2}$ intersects 
with the saddle trajectory $\ell_0$ and, 
on the other hand, the paths $\beta_{p_1^{\ast},p_4}$ 
and $\beta_{p_2,p_1}$ never intersects with 
any saddle trajectories of $\phi$ 
since the Stokes graph $G_0$ does not 
has any saddle trajectory except for $\ell_0$. 
Thus, using the equality \eqref{eq:no-jump-formula}, 
the ratio of the Borel sums of 
the Voros symbol $e^{W_{\beta_0}}$ is given by
\begin{equation} 
\label{eq:pre-DDP-formula2}
\frac{{\mathcal S}_{-}[e^{{W}_{\beta_0}}]}
{{\mathcal S}_{+}[ e^{{W}_{\beta_0}}]} = 
\frac{{\mathcal S}[e^{{W}_{\beta_{p_1^*,p_4}}}]
{\mathcal S}_{-}[e^{{W}_{\beta_{p_4,p_2}}}]
{\mathcal S}[e^{{W}_{\beta_{p_2,p_1}}}]}
{{\mathcal S}[e^{{W}_{\beta_{p_1^*,p_4}}}]
{\mathcal S}_{+}[ e^{{W}_{\beta_{p_4,p_2}}}]
{\mathcal S}[e^{{W}_{\beta_{p_2,p_1}}}]} =
\frac{{\mathcal S}_{-}[e^{{W}_{\beta_{p_4,p_2}}}]}
{{\mathcal S}_{+}[ e^{{W}_{\beta_{p_4,p_2}}}]}.
\end{equation}
Together with \eqref{eq:pre-DDP-formula} 
we have the formula \eqref{eq:DDP-analytic-appendix} 
for $\beta_{p_4,p_2}$:
\begin{equation} \label{eq:pre-DDP-formula3}
{\mathcal S}_{-}[e^{{W}_{\beta_{p_4,p_2}}}] = 
{\mathcal S}_{+}[ e^{{W}_{\beta_{p_4,p_2}}}] 
(1 + {\mathcal S}[e^{{V}_{\gamma_0}}]).
\end{equation} 
Since any path and any cycle intersecting with $\ell_0$ 
can be expressed as a sum of 
$\pm \beta_{p_4,p_2}$
and some paths which never intersect with $\ell_0$.
Thus, \eqref{eq:DDP-analytic-appendix} 
holds for any path and any cycle.
Thus we have proved Theorem \ref{thm:DDP-analytic}.



\section{Proof of Theorem 
\ref{thm:loop-type-degeneration-analytic}}
\label{section:proof-of-AIT}
Here we give a proof of 
Theorem \ref{thm:loop-type-degeneration-analytic}. 
The proof presented here is different from 
that of \cite{Aoki14}, and a more sophisticated 
proof will be presented there.  

\subsection{Settings}
\label{section:setting-appB}

Here we recall the situation and 
explain the idea for the proof of Theorem 
\ref{thm:loop-type-degeneration-analytic}. 

We consider the case that the Stokes graph 
$G_0 = G(\phi)$ 
has a {\em unique} saddle trajectory $\ell_0$, 
and it is a {\em degenerate} saddle trajectory
around a double pole $p$. 
Denote by $D_0$ the degenerate ring domain 
whose boundary consists of $\ell_0$ and $p$. 
To specify the situation, in addition to 
Figure \ref{fig:Stokes-auto},
we take branch cuts and assign $\oplus$ and $\ominus$ 
as in Figure \ref{fig:gamma0-and-beta-B}. 
In this case we can not assign the sign for $p$ 
because there is no trajectory which flows to $p$ 
(recall that a degenerate ring domain is 
swept by closed trajectories).
Then, the saddle class $\gamma_0$ associated with 
$\ell_0$ has the orientation 
shown in Figure \ref{fig:gamma0-and-beta-B}.
As indicated in Figure \ref{fig:gamma0-and-beta-B}, 
$\gamma_0$ is $\ast$-equivalent to 
the sum of two closed cycles around 
the double pole $p$; one lies on the first sheet 
and the other lies on the second sheet. 
Therefore, we have
\begin{equation} \label{eq:Voros-at-double-pole-app}
V_{\gamma_0}(\eta) = \oint_{\gamma_0}
S_{\rm odd}(z,\eta)\,dz 
= - 4 \pi i \eta 
\operatornamewithlimits{Res}_{z=p} 
\sqrt{Q_0(z)}\,dz.
\end{equation}
(See \eqref{eq:Vint1}.)
Moreover, $\gamma_0$ is contained in the kernel
of the intersection form $(~,~)$ on 
$H_1(\hat{\Sigma}\setminus\hat{P})$.

\begin{figure}
\begin{center}
\begin{pspicture}(6.4,-8.5)(14.2,-3.8)
%
\psset{linewidth=0.5pt}
\psset{fillstyle=solid, fillcolor=black}
\pscircle(8,-5.5){0.08}
\pscircle(13,-5.5){0.08}
\psset{linewidth=0.5pt}
\psset{fillstyle=none}
\rput[c]{0}(8,-6.98){\small $\times$}
\rput[c]{0}(8,-4.3){\small $\ell_{0}$}
\rput[c]{0}(8,-8.5){\small $\ominus$}
\rput[c]{0}(8,-5.8){\small $p$} 
\rput[c]{0}(8,-4.9){\small $\gamma_0$}
\rput[c]{0}(10.3,-6.1){\small $\equiv$} 
\rput[c]{0}(13,-6.98){\small $\times$}
\rput[c]{0}(13,-8.5){\small $\ominus$}
%
\psset{linewidth=1pt,linestyle=solid}
\psset{linewidth=1pt}
\pscurve(8,-7)(7.9,-7.1)(7.8,-7)
(7.7,-6.9)(7.6,-7)(7.5,-7.1)(7.4,-7)
(7.3,-6.9)(7.2,-7)(7.1,-7.1)(7,-7)
(6.9,-6.9)(6.8,-7)(6.7,-7.1)(6.6,-7)
(6.5,-6.9)(6.4,-7)
\pscurve(13,-7)(12.9,-7.1)(12.8,-7)
(12.7,-6.9)(12.6,-7)(12.5,-7.1)(12.4,-7)
(12.3,-6.9)(12.2,-7)(12.1,-7.1)(12,-7)
(11.9,-6.9)(11.8,-7)(11.7,-7.1)(11.6,-7)
(11.5,-6.9)(11.4,-7)
\psset{linewidth=0.5pt}
\psline(8,-7.0)(8,-8.34)
\pscurve(8,-7.0)(9,-6.3)(9.2,-5.5)
(9,-4.8)(8,-4.3)(7,-4.8)
(6.8,-5.5)(7,-6.3)(8,-7.0)
\psline(13,-7.0)(13,-8.34)
\pscurve(13,-7.0)(14,-6.3)(14.2,-5.5)
(14,-4.8)(13,-4.3)(12,-4.8)
(11.8,-5.5)(12,-6.3)(13,-7.0)
%
\psset{linewidth=1pt, linestyle=dashed}
%
\pscurve(7.75,-6.87)(8.65,-6.2)(8.9,-5.5)
(8.7,-4.9)(8,-4.65)(7.3,-4.9)(7.1,-5.5)(7.3,-6)
(8,-6.5)(8.3,-6.9)(8.2,-7.4)(7.7,-7.4)(7.2,-6.9)
%
%
\pscircle(13,-5.5){0.5}
\psset{linewidth=1pt,linestyle=solid}
\pscurve(7.75,-6.87)(7.7,-7.1)(8,-7.2)(9.15,-6.35)
(9.4,-5.6)(9,-4.42)(8.4,-4.04)(8,-4)(7.6,-4.04)
(7,-4.42)(6.6,-5.60)(6.8,-6.4)(7.2,-6.9)
\psline(6.6,-5.6)(6.45,-5.9)
\psline(6.6,-5.6)(6.8,-5.85)
\psline(7.0,-5.7)(7.23,-5.95)
\psline(7.23,-5.9)(7.30,-5.62)
\pscircle(13,-5.5){0.8}
\psline(12.2,-5.5)(12.07,-5.75)
\psline(12.2,-5.5)(12.4,-5.7)
\psline(12.53,-5.47)(12.4,-5.23)
\psline(12.53,-5.47)(12.75,-5.28)
\end{pspicture}
\end{center}
\caption{The saddle class 
$\gamma_0\in H_1(\hat{\Sigma}\setminus\hat{P})$. 
The picture depicts a part of the Stokes graph 
$G_0$ near the saddle trajectory $\ell_0$.} 
\label{fig:gamma0-and-beta-B}
\end{figure}
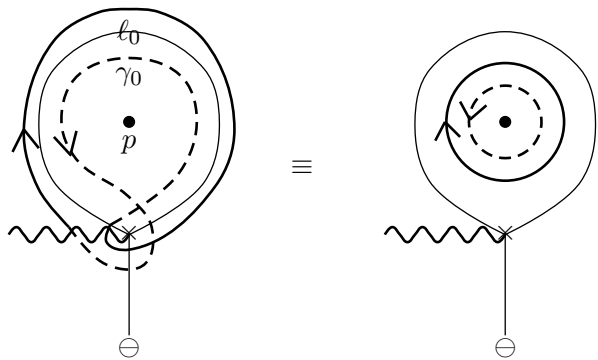

Fix a sufficiently small number 
$\delta_0>0$ and consider 
the saddle reductions
$G_{\pm\delta_0} = G(\phi_{\pm\delta_0})$ 
of $G_0$. Let  ${\mathcal S}_{\pm}[e^{V_{\gamma}}]$ 
(resp., ${\mathcal S}_{\pm}[e^{W_{\beta}}]$) 
be the Borel sum of the Voros symbol 
$e^{V_{\gamma}}$ (resp., $e^{W_{\beta}}$)
in the direction $\pm\delta_0$. 
Here $\gamma \in H_1(\hat{\Sigma}\setminus\hat{P})$ 
is any cycle and $\beta \in H_1(\hat{\Sigma}\setminus
\hat{P}_0,\hat{P}_{\infty})$ is any path. 
Now we will show the equalities 
\begin{equation} \label{eq:AIT-analytic-appendix}
{\mathcal S}_{-}[e^{{W}_{\beta}}] = 
{\mathcal S}_{+}[e^{{W}_{\beta}}]
(1 - {\mathcal S}_{+}[e^{{V}_{\gamma_0}}])
^{\langle\gamma_0,\beta\rangle},~~
{\mathcal S}_{-}[e^{{V}_{\gamma}}] = 
{\mathcal S}_{+}[e^{{V}_{\gamma}}].
\end{equation}
(i.e., the equality 
\eqref{eq:loop-type-degeneration-analytic}) 
on a domain containing 
$\{\eta \in {\mathbb R}~|~ \eta \gg 1 \}$. 
Note that the Borel sum of $e^{V_{\gamma_0}}$
coincides with itself; 
${\mathcal S}_{+}[e^{{V}_{\gamma_0}}] = e^{V_{\gamma_0}}$
since $V_{\gamma_0}$ is not a formal series but a scaler
in view of \eqref{eq:Voros-at-double-pole-app}.
By the same argument in the proof of
Lemma \ref{lem:no-jump-formula}, 
the equality \eqref{eq:AIT-analytic-appendix} 
holds for any path $\beta$ and any cycle $\gamma$
which never intersect with $\gamma_0$. 
Especially, since $(\gamma_0,\gamma) = 0$ 
for any cycle $\gamma$, \eqref{eq:AIT-analytic-appendix} 
holds for any cycle $\gamma$. 
Therefore, it suffices to show 
\eqref{eq:AIT-analytic-appendix} for any path 
$\beta \in H_1(\hat{\Sigma}\setminus
\hat{P}_0,\hat{P}_{\infty})$ intersecting with $\gamma_0$. 
The path $\beta_p$ depicted in Figure 
\ref{fig:path-beta-p} is a typical example 
of such a path. 

\begin{figure}
\begin{center}
\begin{pspicture}(0,-9.1)(16.7,-3.8)
%
\psset{linewidth=0.5pt}
\psset{fillstyle=solid, fillcolor=black}
\pscircle(8,-5.5){0.08}
\pscircle(8.7,-6.2){0.05}
\pscircle(3.7,-6.2){0.05}
\pscircle(13.7,-6.2){0.05}
\psset{linewidth=0.5pt}
\psset{fillstyle=none}
\rput[c]{0}(3,-6.98){\small $\times$}
\rput[c]{0}(3,-5.5){\small $\ominus$}
\rput[c]{0}(3,-8.5){\small $\ominus$}
\rput[c]{0}(3.1,-9.1){\small $G_{+\delta_0}$}
\rput[c]{0}(3.7,-5.8){\small $z_0$} 
\rput[c]{0}(8,-6.98){\small $\times$}
\rput[c]{0}(8.1,-9.1){\small $G_{0}$} 
\rput[c]{0}(8,-8.5){\small $\ominus$}
\rput[c]{0}(8.7,-5.8){\small $z_0$} 
\rput[c]{0}(8,-5.1){\small $p$} 
\rput[c]{0}(8.4,-7.3){\small $a$} 
\rput[c]{0}(7.5,-6.){\small $\beta_p$} 
\rput[c]{0}(13,-6.98){\small $\times$}
\rput[c]{0}(13,-5.5){\small $\oplus$}
\rput[c]{0}(13,-8.5){\small $\ominus$}
\rput[c]{0}(13.1,-9.1){\small $G_{-\delta_0}$}
\rput[c]{0}(13.7,-5.8){\small $z_0$}  
%
\psline(3,-7.0)(3,-8.34)
\pscurve(3,-7)(2.7,-6.5)(2.7,-6)(2.85,-5.6)
\psset{linewidth=0.5pt,linestyle=solid}
\pscurve(3,-7.0)(4,-6.3)(4.2,-5.5)(4,-4.8)(3,-4.4)
(2.0,-4.9)(1.8,-6.4)(2.0,-7.0)
\psset{linewidth=0.5pt,linestyle=solid}
\psecurve(1.8,-6.4)(2.0,-7.0)(2.9,-8.35)(2.9,-8.35)
\psset{linewidth=1pt,linestyle=solid}
\pscurve(3,-7)(2.9,-7.1)(2.8,-7)
(2.7,-6.9)(2.6,-7)(2.5,-7.1)(2.4,-7)
(2.3,-6.9)(2.2,-7)(2.1,-7.1)(2,-7)
(1.9,-6.9)(1.8,-7)(1.7,-7.1)(1.6,-7)
(1.5,-6.9)(1.4,-7)
\pscurve(8,-7)(7.9,-7.1)(7.8,-7)
(7.7,-6.9)(7.6,-7)(7.5,-7.1)(7.4,-7)
(7.3,-6.9)(7.2,-7)(7.1,-7.1)(7,-7)
(6.9,-6.9)(6.8,-7)(6.7,-7.1)(6.6,-7)
(6.5,-6.9)(6.4,-7)
\psset{linewidth=0.5pt}
\psline(8,-7.0)(8,-8.34)
\pscurve(8,-7.0)(9,-6.3)(9.2,-5.5)
(9,-4.8)(8,-4.3)(7,-4.8)
(6.8,-5.5)(7,-6.3)(8,-7.0)
\psset{linewidth=1pt}
\pscurve(13,-7)(12.9,-7.1)(12.8,-7)
(12.7,-6.9)(12.6,-7)(12.5,-7.1)(12.4,-7)
(12.3,-6.9)(12.2,-7)(12.1,-7.1)(12,-7)
(11.9,-6.9)(11.8,-7)(11.7,-7.1)(11.6,-7)
(11.5,-6.9)(11.4,-7)
\pscurve(8,-5.5)(8.25,-6.5)(8.2,-7.1)(8.1,-7.2)
(8,-7.25)(7.9,-7.2)(7.8,-7.1)(7.75,-6.9)
\psline(8.25,-6.5)(8.13,-6.65)
\psline(8.25,-6.5)(8.37,-6.62)
\psset{linewidth=0.5pt}
\psline(13,-7.0)(13,-8.34)
\pscurve(13,-7)(13.3,-6.5)(13.3,-6)(13.15,-5.6)
\pscurve
(13,-7.0)(12,-6.3)(11.8,-5.5)
(12,-4.8)(13,-4.4)
(14.0,-4.9)(14.2,-6.4)
(14.0,-7.0)(13.1,-8.35)
\psset{linewidth=0.5pt}
%
%
\psset{linewidth=1pt, linestyle=dashed}
\pscurve(8,-5.5)(7.85,-6)(7.75,-6.9)
\psset{linewidth=1pt,linestyle=solid}
\psset{linewidth=0.5pt}
%
%
%
\end{pspicture}
\end{center}
\caption{The saddle reduction of $G_0$. 
These pictures are schematic ones, and an 
actual saddle reduction is given in 
Figure \ref{fig:spirals}.} 
\label{fig:path-beta-p}
\end{figure}
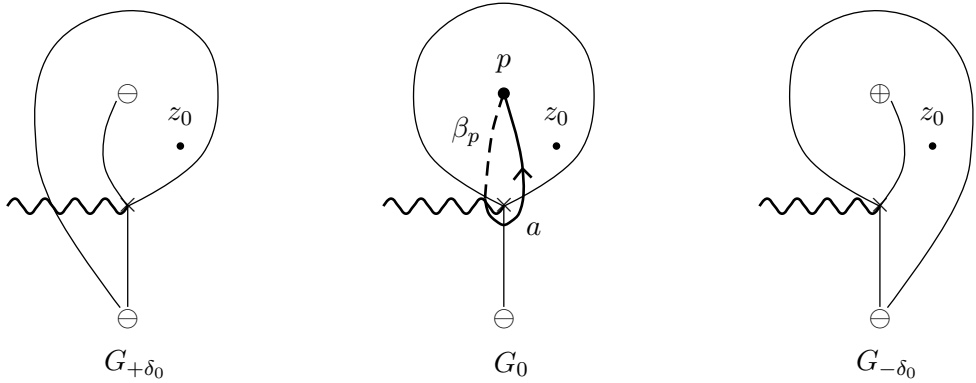

Before the derivation of the formula 
\eqref{eq:AIT-analytic-appendix}, 
we give a remark on the property of Stokes 
curves in the saddle reductions of a degenerate 
saddle trajectory. Figure \ref{fig:spirals} 
depicts examples of Stokes curves 
for some directions. 
For any direction $\theta \ne 0$ near $0$, 
there exists a Stokes curve emanating from 
$a$ forms a logarithmic spiral and flows into $p$. 
As we vary the direction $\theta$, the Stokes curve 
hits any point $z \in D_0$. 
Moreover, as $\theta$ tends to $0$, any point 
$z \in D_0$ is hit by the Stokes curve
{\em infinitely} many times, and the 
degenerate saddle trajectory appears as 
the limit $\theta\rightarrow0$ of the Stokes curve. 

\begin{figure}
  \begin{minipage}{0.33\hsize}
  \begin{center} 
  \includegraphics[width=50mm]
  {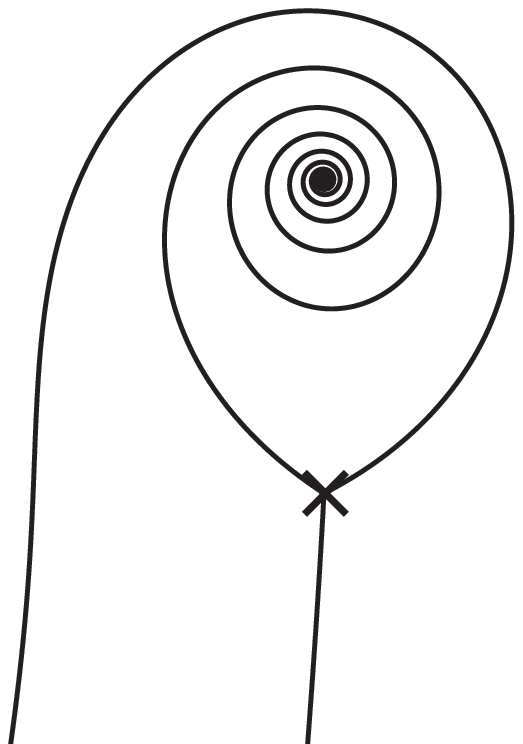} \\
  {$\theta=+\delta_0$.} 
  \end{center}
  \end{minipage} 
  \begin{minipage}{0.33\hsize}
  \begin{center}
  \includegraphics[width=50mm]
  {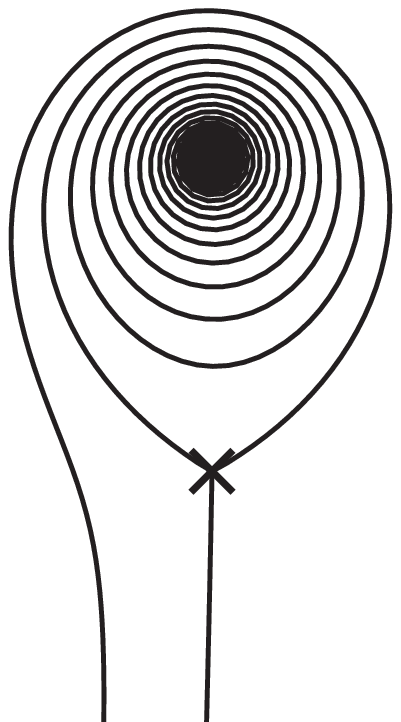} \\
  {$0<\theta<+\delta_0$.} 
  \end{center}
  \end{minipage}  
  \begin{minipage}{0.33\hsize}
  \begin{center}
  \includegraphics[width=50mm]
  {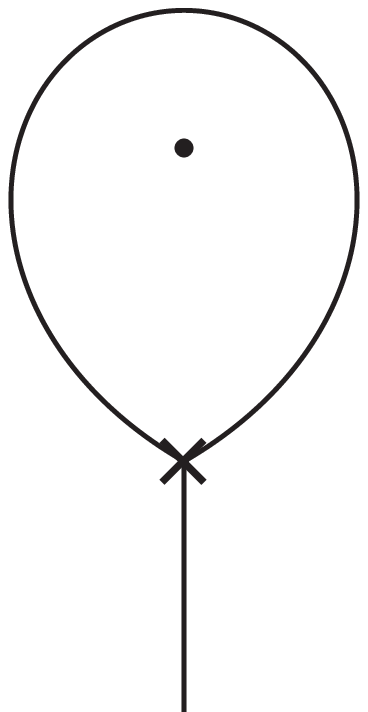} \\
  {$\theta=0$.} 
  \end{center}
  \end{minipage}  
  \caption{Stokes graphs in the limit 
  $\theta\rightarrow+0$. Any point in the degenerate 
  ring domain is hit by a Stokes curve infinitely 
  many times as $\theta$ tends to $0$. Similar pictures 
  appear in the opposite limit $\theta\rightarrow-0$.}
  \label{fig:spirals}
\end{figure}

Fix a point $z_0 \in D_0$ (``reference point")
and take a small disc $U$ containing $z_0$ and 
included in $D_0$. We assume that any point 
$U$ is not on the Stokes graphs $G_{\pm\delta_0}$.
We will take certain WKB solutions defined on $U$, 
and compare their Borel sums in the directions 
$+\delta_0$ and $-\delta_0$. 
The comparison will be used to 
derive the formula \eqref{eq:AIT-analytic-appendix}
similarly to Appendix \ref{section:proof-of-DDP}.

However, the situation here is more complicated 
than that of Appendix \ref{section:proof-of-DDP}. 
Recall that, when $z$ lies on a Stokes curve 
in a direction $\theta$, WKB solutions may not 
be Borel summable in the direction $\theta$. 
In other wards, the Stokes phenomenon occurs 
to WKB solutions in such a direction $\theta$.
Therefore, to compare the Borel 
sum in the directions $-\theta_0$ and $\theta_0$, 
we have to deal with {\em infinitely} many Stokes phenomena 
occurring to WKB solutions since the reference point 
$z_0$ is hit by Stokes curves infinitely many times. 
The situation is analyzed in 
Lemma \ref{lemma:no-jump-appB} below. 

In our computation, we choose the reference point $z_0$ 
near the Stokes curve which forms 
a boundary of the degenerate horizontal strip around 
$p$ in the Stokes graphs $G_{+\delta_0}$
(see Figure \ref{fig:path-beta-p}). 
We may also assume the following: 
Take a straight half line starting from 
$p$ and passing through $z_0$. The point $z_0$ divides 
the half line into two parts; one is the segment 
between $p$ and $z_0$ and the other is a half line 
starting from $z_0$. Then, we assume that latter part 
never intersects with the Stokes curve which forms 
the logarithmic spiral in $G_{+\delta_0}$ and 
in $G_{-\delta_0}$. 
This guarantees that, the Borel sum of 
WKB solutions defined near $z_0$ can be analytically 
continued to the Stokes region which is adjacent to 
the degenerate horizontal strip along a straight path
without crossing the logarithmic spirals 
(see Figure \ref{fig:two-connection-appB}). 
We will also compare the result of the analytic 
continuations of Borel sum of WKB solutions along 
two paths indicated in 
Figure \ref{fig:two-connection-appB} in Lemma 
\ref{lemma:two-connection-problems} below.  
Combining Lemma \ref{lemma:two-connection-problems}
and Lemma \ref{lemma:no-jump-appB}, finally we 
obtain the formula \eqref{eq:AIT-analytic-appendix}.

\subsection{Derivation of the formulas in 
Theorem \ref{thm:loop-type-degeneration-analytic}}

In the computation, we use the WKB solutions 
\begin{equation} \label{eq:appB-WKB-p}
\psi_{\pm,p}(z,\eta) = 
\frac{1}{\sqrt{S_{\rm odd}(z,\eta)}}
\exp\left\{ \pm\left(\eta\int_{a}^{z}
\sqrt{Q_{0}(z)}\,dz+\int_{p}^z
S_{\rm odd}^{\rm reg}(z,\eta)\,dz\right) \right\}
\end{equation}
defined on $U$, which are normalized at the double pole 
$p$ along a path from $p$ to $z$. 
Here $a$ is the turning point where $\ell_0$ emanates. 
For any $z \in U$, we can choose the path from $p$ 
to $z$ in \eqref{eq:appB-WKB-p} so that it never 
intersects with $\ell_0$, and hence $\psi_{\pm,p}$ 
is Borel summable in the direction 0. Denote by 
$\Psi_{\pm,p}^{(G_0)}$ the Borel sum defined on $U$. 
We will use $\Psi_{\pm,p}^{(G_0)}(z,\eta)$ 
in order to compare the Borel sums 
$\Psi_{\pm,p}^{(G_{+\delta_0})}(z,\eta)$
and $\Psi_{\pm,p}^{(G_{-\delta_0})}(z,\eta)$ 
of $\psi_{\pm,p}(z,\eta)$ in the direction $+\delta_0$ 
and $-\delta_0$ defined on $U$, respectively.

We also deal with another WKB solutions 
\begin{equation} \label{eq:appB-WKB-TP}
\psi_{\pm}(z,\eta) = 
\frac{1}{\sqrt{S_{\rm odd}(z,\eta)}}
\exp\left(\pm \int_{a}^{z}
S_{\rm odd}(z,\eta)\,dz \right)
\end{equation}
which are normalized at the turning point $a$ 
since the Stokes phenomena for $\psi_{\pm}$
are easy to describe. 
We fix the normalization of \eqref{eq:appB-WKB-TP}
so that the path from $a$ to $z$ is defined by
\begin{equation}
\int_{a}^{z}
S_{\rm odd}(z,\eta)\,dz = 
\int_{\gamma_z^{(1)}} S_{\rm odd}(z,\eta)\,dz,
\end{equation}
where the path $\gamma_{z}^{(1)}$ is depicted 
in Figure \ref{fig:normalizationss-appB}. 
Although we have illustrated the path $\gamma_{z}^{(1)}$ 
so that the one of its end point is the turning point $a$, 
the integral should be regarded as a contour integral 
(see Section \ref{subsec:WKB}). Then, the WKB solutions 
$\psi_{\pm}$ and $\psi_{\pm,p}$ satisfy 
\begin{equation} \label{eq:Voros-beta-p-appB}
\psi_{\pm}(z,\eta) = 
\exp\left(\pm \frac{1}{2}
W_{\beta_p}(\eta) \right) 
\psi_{\pm,p}(z,\eta),
\end{equation}
where $W_{\beta_p}$ is the Voros coefficient 
for the path $\beta_p$ depicted in 
Figure \ref{fig:path-beta-p}. 
Denote by $\Psi_{\pm}^{(G_{+\delta_0})}$ 
(resp., $\Psi_{\pm}^{(G_{-\delta_0})}$)
the Borel sum of $\psi_{\pm}$ 
in the direction $+\delta_0$ (resp., $-\delta_0$)
defined on $U$. Note that $\psi_{\pm}$ may not be 
Borel summable since the path $\gamma_z^{(1)}$ 
intersects with $\ell_0$. 

\begin{figure}
\begin{minipage}{0.47\hsize}
\begin{center}
\begin{pspicture}(10.5,-9.)(14.5,-3.5)
%
\psset{linewidth=0.5pt}
\psset{fillstyle=solid, fillcolor=black}
\pscircle(14,-6){0.05}
\psset{linewidth=0.5pt}
\psset{fillstyle=none}
%
\rput[c]{0}(13,-6.98){\small $\times$}
\rput[c]{0}(13,-5.5){\small $\oplus$}
\rput[c]{0}(13,-8.5){\small $\ominus$}
\rput[c]{0}(13.8,-7.){\small $\gamma_{z}^{(1)}$}
\rput[c]{0}(12.9,-4.95){\small $\gamma_{z}^{(2)}$}
\rput[c]{0}(14.25,-6.){\small $z$}
\rput[c]{0}(12.75,-7.2){\small $a$}
\psset{linewidth=1pt}
\pscurve(13,-7)(12.9,-7.1)(12.8,-7)
(12.7,-6.9)(12.6,-7)(12.5,-7.1)(12.4,-7)
(12.3,-6.9)(12.2,-7)(12.1,-7.1)(12,-7)
(11.9,-6.9)(11.8,-7)(11.7,-7.1)(11.6,-7)
(11.5,-6.9)(11.4,-7)
\pscurve(14,-6)(13.7,-6.7)(13,-7)
\psline(13.7,-6.7)(13.45,-6.7)
\psline(13.7,-6.7)(13.62,-6.93)
\pscurve(14,-6)(13.95,-5.4)(13.9,-5.25)(13.8,-5.05)
(13.4,-4.75)(13,-4.65)(12.6,-4.75)(12.2,-5.18)(12.1,-5.5)
(12.2,-6.2)(13,-6.98)
\psline(13.7,-4.95)(13.45,-4.94)
\psline(13.7,-4.95)(13.6,-4.7)
\psset{linewidth=0.5pt,linestyle=solid}
\psline(13,-7.0)(13,-8.34)
\pscurve(13,-7)(13.3,-6.5)(13.3,-6)(13.15,-5.6)
\pscurve(13,-7.0)(12,-6.3)(11.8,-5.5)
(12,-4.8)(13,-4.4)
(14.0,-4.9)(14.2,-6.4)
(14.0,-7.0)(13.1,-8.35)
\end{pspicture}
\caption{Normalization paths for WKB solutions 
(the picture depicts the Stokes graph in the direction 
$-\delta_0$).} 
\label{fig:normalizationss-appB}
\end{center}
\end{minipage} \hspace{-2.5em} 
\begin{minipage}{0.55\hsize}
\begin{center}
\begin{pspicture}(1,-9.1)(8,-3.8)
%
\psset{linewidth=0.5pt}
\psset{fillstyle=solid, fillcolor=black}
\pscircle(3.7,-6.2){0.05}
\pscircle(4.6,-6.2){0.05}
\psset{linewidth=0.5pt}
\psset{fillstyle=none}
\rput[c]{0}(3,-6.98){\small $\times$}
\rput[c]{0}(3,-5.5){\small $\ominus$}
\rput[c]{0}(3,-8.5){\small $\ominus$}
\rput[c]{0}(3.1,-9.1){\small $G_{+\delta_0}$}
\rput[c]{0}(3.7,-5.8){\small $z_0$} 
\rput[c]{0}(4.6,-5.8){\small $z_1$} 
%
\psline(3,-7.0)(3,-8.34)
\pscurve(3,-7)(2.7,-6.5)(2.7,-6)(2.85,-5.6)
\psset{linewidth=0.5pt,linestyle=solid}
\pscurve(3,-7.0)(4,-6.3)(4.2,-5.5)(4,-4.8)(3,-4.4)
(2.0,-4.9)(1.8,-6.4)(2.0,-7.0)
\psset{linewidth=0.5pt,linestyle=solid}
\psecurve(1.8,-6.4)(2.0,-7.0)(2.9,-8.35)(2.9,-8.35)
\psset{linewidth=1pt,linestyle=solid}
\pscurve(3,-7)(2.9,-7.1)(2.8,-7)
(2.7,-6.9)(2.6,-7)(2.5,-7.1)(2.4,-7)
(2.3,-6.9)(2.2,-7)(2.1,-7.1)(2,-7)
(1.9,-6.9)(1.8,-7)(1.7,-7.1)(1.6,-7)
(1.5,-6.9)(1.4,-7)
\psset{linewidth=2.5pt}
\psline[arrows=->](3.8,-6.2)(4.5,-6.2)  
\psset{linewidth=0.5pt}
\psset{fillstyle=solid, fillcolor=black}
\pscircle(7.7,-6.2){0.05}
\pscircle(8.6,-6.2){0.05}
\psset{linewidth=0.5pt}
\psset{fillstyle=none}
\rput[c]{0}(7,-6.98){\small $\times$}
\rput[c]{0}(7,-5.5){\small $\oplus$}
\rput[c]{0}(7,-8.5){\small $\ominus$}
\rput[c]{0}(7.1,-9.1){\small $G_{-\delta_0}$}
\rput[c]{0}(7.7,-5.8){\small $z_0$}  
\rput[c]{0}(8.6,-5.8){\small $z_1$}  
%
\psset{linewidth=1pt}
\pscurve(7,-7)(6.9,-7.1)(6.8,-7)
(6.7,-6.9)(6.6,-7)(6.5,-7.1)(6.4,-7)
(6.3,-6.9)(6.2,-7)(6.1,-7.1)(6,-7)
(5.9,-6.9)(5.8,-7)(5.7,-7.1)(5.6,-7)
(5.5,-6.9)(5.4,-7)
\psset{linewidth=0.5pt}
\psline(7,-7.0)(7,-8.34)
\pscurve(7,-7)(7.3,-6.5)(7.3,-6)(7.15,-5.6)
\pscurve
(7,-7.0)(6,-6.3)(5.8,-5.5)
(6,-4.8)(7,-4.4)
(8.0,-4.9)(8.2,-6.4)
(8.0,-7.0)(7.1,-8.35)
\psset{linewidth=2.5pt}
\psline[arrows=->](7.8,-6.2)(8.5,-6.2) 
\end{pspicture}
\end{center}
\caption{Two connection problems.} 
\label{fig:two-connection-appB} 
\end{minipage} 
\end{figure}

\begin{lem} \label{lemma:two-connection-problems}
The Borel sums $\Psi_{\pm}^{(G_{+\delta_0})}$ 
and $\Psi_{\pm}^{(G_{-\delta_0})}$ of $\psi_{\pm}$
are related as follows:
\begin{eqnarray}
\begin{cases} \label{eq:Psi-plus-minus-app} 
\Psi_{+}^{(G_{+\delta_0})}=(1-e^{V_{\gamma_0}})
\Psi_{+}^{(G_{-\delta_0})}
-i\Psi_{-}^{(G_{-\delta_0})}, \\[+.2em] 
\Psi_{-}^{(G_{+\delta_0})} = -ie^{V_{\gamma_0}}
\Psi_{+}^{(G_{-\delta_0})}+\Psi_{-}^{(G_{-\delta_0})}.
\end{cases} 
\end{eqnarray}
\end{lem}
\begin{proof}
Fix another reference point $z_1$ in the Stokes 
region adjacent to the Stokes region containing $z_0$ 
in $G_{\pm\delta_0}$, and let us consider 
two connection problems for the WKB solutions from 
$z_0$ to $z_1$ as indicated 
in Figure \ref{fig:two-connection-appB}. 
Taking $\delta_0$ sufficiently small, 
we may take $z_1$ so that it does not lie on 
any Stokes curves in any direction $\theta$ 
satisfying $-\delta_0 \le \theta \le +\delta_0$. 
Let $\Psi_{\pm,z_1}^{(G_{+\delta_0})}$ 
(resp., $\Psi_{\pm,z_1}^{(G_{-\delta_0})}$)
be the Borel sum of $\psi_{\pm}$ 
in the direction $+\delta_0$ (resp., $-\delta_0$)
defined in a neighborhood of $z_1$. 
Then, Theorem \ref{thm:Voros-formula} implies that, 
for the connection problem for $G_{+\delta_0}$, we have 
\begin{eqnarray}
\begin{cases} \label{eq:connection1-appB} 
\Psi_{+}^{(G_{+\delta_0})}=
\Psi_{+,z_1}^{(G_{+\delta_0})}
-i\Psi_{-,z_1}^{(G_{+\delta_0})}, \\[+.2em] 
\Psi_{-}^{(G_{+\delta_0})} = 
\Psi_{-,z_1}^{(G_{+\delta_0})}.
\end{cases} 
\end{eqnarray}
Next let us consider the connection problem for 
$G_{-\delta_0}$. In this case, in order to use 
Theorem \ref{thm:Voros-formula} on the Stokes curve 
in question, we have to change the normalization 
of the WKB solution; 
{Theorem \ref{thm:Voros-formula} is valid 
for the WKB solution normalized at $a$ 
along the Stokes curve in question.} 
Namely, the formula \eqref{eq:Voros-formula-2}
in Theorem \ref{thm:Voros-formula} holds for 
the Borel sum of the WKB solutions 
\begin{equation}
\phi_{\pm}(z,\eta) = 
\frac{1}{\sqrt{S_{\rm odd}(z,\eta)}}
\exp\left(\pm \int_{\gamma_z^{(2)}}
S_{\rm odd}(z,\eta)\,dz \right)
\end{equation}
whose integration path $\gamma_{z}^{(2)}$ 
is depicted in Figure 
\ref{fig:normalizationss-appB}. Thus we have
\begin{eqnarray}
\begin{cases} 
\Phi_{+}^{(G_{-\delta_0})}=
\Phi_{+,z_1}^{(G_{-\delta_0})}, \\[+.2em] 
\Phi_{-}^{(G_{-\delta_0})} = 
\Phi_{-,z_1}^{(G_{-\delta_0})}
+i\Phi_{+,z_1}^{(G_{-\delta_0})},
\end{cases} 
\end{eqnarray}
where $\Phi_{\pm}^{(G_{-\delta_0})}$ are the 
Borel sum of $\phi_{\pm}$. Since $\psi_{\pm}$ 
and $\phi_{\pm}$ are related as
\begin{equation}
\phi_{\pm} = \exp\left( \pm\frac{1}{2}
\int_{\gamma_0}S_{\rm odd}(z,\eta)\,dz \right)\psi_{\pm}
= \exp\left(\pm\frac{1}{2}V_{\gamma_0}\right)\psi_{\pm},
\end{equation}
we have the following formula for the Borel sum 
of $\psi_{\pm}$:
\begin{eqnarray}
\begin{cases} \label{eq:connection2-appB} 
\Psi_{+}^{(G_{-\delta_0})}=
\Psi_{+,z_1}^{(G_{-\delta_0})}, \\[+.2em] 
\Psi_{-}^{(G_{-\delta_0})} = 
\Psi_{-,z_1}^{(G_{-\delta_0})} 
+ ie^{V_{\gamma_0}}\Psi_{+,z_1}^{(G_{-\delta_0})}.
\end{cases} 
\end{eqnarray}

Let us compare the right-hand sides of 
\eqref{eq:connection1-appB} and 
\eqref{eq:connection2-appB}. 
Since we take $z_1$ outside of the saddle 
trajectory $\ell_0$, 
we can deform the path from $a$ to $z_1$ 
so that it never touches with $\ell_0$ 
(see Figure \ref{fig:path-deformation-appB}).
Therefore, when $z$ lies on a neighborhood of $z_1$, 
the WKB solutions are Borel summable 
in the direction $0$ even if the saddle trajectory 
$\ell_0$ appears. This implies that, as we vary the 
direction $\theta$ in $-\delta_0\le\theta\le+\delta_0$, 
no Stokes phenomenon occurs 
to $\psi_{\pm}$ near $z_1$. 
Thus we have 
\begin{equation}
\Psi_{\pm,z_1}^{(G_{+\delta_0})} = 
\Psi_{\pm,z_1}^{(G_{-\delta_0})}.
\end{equation}
Comparing \eqref{eq:connection1-appB} and 
\eqref{eq:connection2-appB}, we obtain 
\eqref{eq:Psi-plus-minus-app}.
\end{proof}

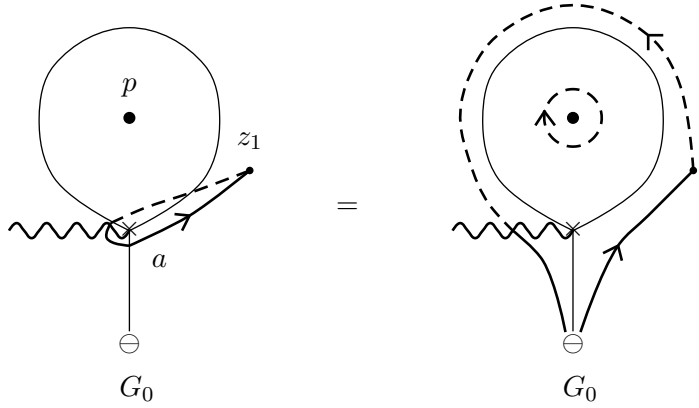
\begin{figure}
\begin{center}
\begin{pspicture}(6,-9.1)(10.,-3.8)
%
\psset{linewidth=0.5pt}
\psset{fillstyle=solid, fillcolor=black}
\pscircle(8,-5.5){0.08}
\pscircle(9.6,-6.2){0.05}
\psset{linewidth=0.5pt}
\psset{fillstyle=none}
\rput[c]{0}(8,-6.98){\small $\times$}
\rput[c]{0}(8.1,-9.1){\small $G_{0}$} 
\rput[c]{0}(8,-8.5){\small $\ominus$}
\rput[c]{0}(9.6,-5.8){\small $z_1$} 
\rput[c]{0}(8,-5.1){\small $p$} 
\rput[c]{0}(8.4,-7.4){\small $a$} 
\psset{linewidth=1pt,linestyle=solid}
\pscurve(8,-7)(7.9,-7.1)(7.8,-7)
(7.7,-6.9)(7.6,-7)(7.5,-7.1)(7.4,-7)
(7.3,-6.9)(7.2,-7)(7.1,-7.1)(7,-7)
(6.9,-6.9)(6.8,-7)(6.7,-7.1)(6.6,-7)
(6.5,-6.9)(6.4,-7)
\psset{linewidth=0.5pt}
\psline(8,-7.0)(8,-8.34)
\pscurve(8,-7.0)(9,-6.3)(9.2,-5.5)
(9,-4.8)(8,-4.3)(7,-4.8)
(6.8,-5.5)(7,-6.3)(8,-7.0)
\psset{linewidth=1pt}
\pscurve(9.6,-6.2)(8.8,-6.8)(8,-7.2)
\pscurve(8,-7.2)(7.7,-7.1)(7.75,-6.9)
\psline(8.8,-6.8)(8.6,-6.78)
\psline(8.8,-6.8)(8.7,-7)
\psset{linewidth=0.5pt}
\psset{linewidth=1pt, linestyle=dashed}
\pscurve(9.6,-6.2)(8.8,-6.5)(8.2,-6.7)(7.75,-6.9)
\psset{linewidth=1pt,linestyle=solid}
\psset{linewidth=0.5pt}
\end{pspicture}
\hskip50pt
\begin{pspicture}(6,-9.1)(10.,-3.8)
%
\psset{linewidth=0.5pt}
\psset{fillstyle=solid, fillcolor=black}
\pscircle(8,-5.5){0.08}
\pscircle(9.6,-6.2){0.05}
\psset{linewidth=0.5pt}
\psset{fillstyle=none}
\rput[c]{0}(5,-6.7){\small $=$}
\rput[c]{0}(8,-6.98){\small $\times$}
\rput[c]{0}(8.1,-9.1){\small $G_{0}$} 
\rput[c]{0}(8,-8.5){\small $\ominus$}
\psset{linewidth=1pt,linestyle=solid}
\pscurve(8,-7)(7.9,-7.1)(7.8,-7)
(7.7,-6.9)(7.6,-7)(7.5,-7.1)(7.4,-7)
(7.3,-6.9)(7.2,-7)(7.1,-7.1)(7,-7)
(6.9,-6.9)(6.8,-7)(6.7,-7.1)(6.6,-7)
(6.5,-6.9)(6.4,-7)
\psset{linewidth=0.5pt}
\psline(8,-7.0)(8,-8.34)
\pscurve(8,-7.0)(9,-6.3)(9.2,-5.5)
(9,-4.8)(8,-4.3)(7,-4.8)
(6.8,-5.5)(7,-6.3)(8,-7.0)
\psset{linewidth=1pt}
\pscurve(9.6,-6.2)(9,-6.8)(8.6,-7.2)(8.1,-8.35)
\pscurve(7.9,-8.35)(7.6,-7.4)(7.2,-7)
\psline(8.6,-7.2)(8.65,-7.4)
\psline(8.6,-7.2)(8.4,-7.3)
\psline(9,-4.4)(9.2,-4.4)
\psline(9,-4.4)(9.0,-4.6)
\psline(7.63,-5.4)(7.5,-5.6)
\psline(7.61,-5.4)(7.76,-5.6)
\psset{linewidth=0.5pt}
\psset{linewidth=1pt, linestyle=dashed}
\pscircle(8,-5.5){0.4}
\pscurve(9.6,-6.2)(9.4,-5)(9,-4.4)(8,-4)(7,-4.4)
(6.6,-5)(6.5,-5.5)(6.6,-6)(6.85,-6.6)(7.2,-7)
\psset{linewidth=1pt,linestyle=solid}
\psset{linewidth=0.5pt}
\end{pspicture}
\end{center}
\caption{Deformation of a path.} 
\label{fig:path-deformation-appB}
\end{figure}

It follows from the first equality of 
\eqref{eq:Psi-plus-minus-app} and 
\eqref{eq:Voros-beta-p-appB}, we have
\begin{equation} \label{eq:compare-WKB-slutions-appB}
{\cal S}_{+}[e^{W_{\beta_p}/2}]
\Psi_{+,p}^{(G_{+\delta_0})} = 
(1-e^{V_{\gamma_0}})
{\cal S}_{-}[e^{W_{\beta_p}/2}]
\Psi_{+,p}^{(G_{-\delta_0})}-
i{\cal S}_{-}[e^{-W_{\beta_p}/2}]
\Psi_{-,p}^{(G_{-\delta_0})}.
\end{equation}
The desired equality \eqref{eq:AIT-analytic-appendix}
will be follows from \eqref{eq:compare-WKB-slutions-appB}
and the following lemma.

\begin{lem} \label{lemma:no-jump-appB}
(i) The Borel sum 
${\cal S}_{\theta}[\psi_{+,p}]$ 
does not depend on the direction $\theta$ 
satisfying $0 \le \theta \le \delta_0$.  
Especially, we have 
\begin{equation} \label{eq:limit-formula-1}
\Psi_{+,p}^{(G_{+\delta_0})} = 
\Psi_{+,p}^{(G_0)}
\end{equation}
for $z \in U$ and 
$\eta \in \{\eta \in {\mathbb R}~|~ \eta \gg 1 \}$.
Similarly, we have
\begin{equation} \label{eq:limit-formula-2}
\Psi_{-,p}^{(G_{-\delta_0})} = 
\Psi_{-,p}^{(G_0)}
\end{equation}
for $z \in U$ and 
$\eta \in \{\eta \in {\mathbb R}~|~ \eta \gg 1 \}$.

\noindent
(ii) For the Borel sums $\Psi_{+,p}^{(G_{-\delta_0})}$
and $\Psi_{+,p}^{(G_{0})}$ of the WKB solution
$\psi_{+,p}$, the equality
\begin{equation} \label{eq:limit-formula-3}
\Psi_{+,p}^{(G_{-\delta_0})} = \Psi_{+,p}^{(G_{0})} 
+ i (1-e^{V_{\gamma_0}})^{-1}
{\cal S}_{-}[e^{-W_{\beta_p}}] \Psi_{-,p}^{(G_{0})}
\end{equation}
holds for $z \in U$ and 
$\eta \in \{\eta \in {\mathbb R}~|~ \eta \gg 1 \}$.
\end{lem}
\begin{proof}
(i).
As is explained in Subsection \ref{section:setting-appB}, 
when we vary the direction $\theta$, any point 
$z \in D_0$ is hit by a Stokes curve, and which 
may yield the Stokes phenomenon to WKB solutions. 
However, such Stokes phenomena 
do {\em not} occur to the $\psi_{+,p}$ 
as long as we vary the direction $\theta$ in
$0 < \theta \le \delta_0$ by the following reason. 
As we vary $\theta$ in $0 < \theta \le \delta_0$, 
the sign of $p$ never changes and is always $\ominus$. 
Hence, Stokes phenomena do {\em not} occur to 
$\psi_{+,p}$ even if $z_0$ lies on the Stokes curve. 
{(See \eqref{eq:Voros-formula-Stokes-2} in Remark 
\ref{remark:Voros-formula-is-Stokes}.) }
This guarantees that 
$\psi_{+,p}$ is Borel summable in any direction 
$\theta$ satisfying $0 < \theta \le \delta_0$. 
Furthermore, $\psi_{+,p}$ is also Borel 
summable in the direction $0$ as we have noted above. 
Therefore, by the same argument in the proof of 
Proposition \ref{prop:S1-action-and-summability}, 
we have proved the equality \eqref{eq:limit-formula-1}. 

On the other hand, since the sign of $p$ in the Stokes 
graph of direction $\theta$ satisfying 
$-\delta_0 \le \theta < 0$ is always $\oplus$, 
the Stokes phenomenon do not occur to 
$\psi_{-,p}$.
{(See \eqref{eq:Voros-formula-Stokes-1} in Remark 
\ref{remark:Voros-formula-is-Stokes}.)}
Thus the equality \eqref{eq:limit-formula-2} is 
derived in the same manner.
\par

(ii).
In the proof, we use a different expression of the 
Borel sum of WKB solutions. Let us write the WKB 
solution $\psi_{\pm,p}$ in the following form:
\begin{equation}
\psi_{\pm,p}(z,\eta) = \exp\left(\pm\eta s(z)\right)
\sum_{k=0}^{\infty}\eta^{-k-1/2}
\psi_{\pm,k}(z), \quad
s(z) = \int_{a}^z \sqrt{Q_0(z)}\,dz.
\end{equation}
Then, the Borel sum of $\psi_{\pm,p}$ 
in a direction $\theta$ is given by 
\begin{equation} \label{eq:Borel-sum-appB}
\int_{\mp s(z)}^{\infty e^{-i\theta}}
e^{-\eta y} \psi_{\pm,B}(z,y)\,dy,
\end{equation}
where 
\begin{equation}
\psi_{\pm,B}(z,y) = \sum_{k=0}^{\infty}
\frac{\psi_{\pm,k}(z)}{\Gamma(k+1/2)}
(y\pm s(z))^{k-1/2}
\end{equation}
is the Borel transform of $\psi_{\pm,p}$.
The expression \eqref{eq:Borel-sum-appB} coincides 
with the definition of the Borel sum 
(of formal series 
{\em with exponential factors})
adopted in Definition 
\ref{def:Borel-summability}.
The difference of the Borel sums 
$\Psi_{+,p}^{(G_{-\delta_0})}$ 
and $\Psi_{+,p}^{(G_{0})}$ in question 
is that of the path of Laplace integral 
in \eqref{eq:Borel-sum-appB}. 
That is, $\Psi_{+,p}^{(G_{-\delta_0})}$ 
is given by the Laplace transform of 
$\psi_{+,B}(z,y)$ along the half line 
$L_{+\delta_0} = 
\{-s(z)+re^{i\delta_0}\in{\mathbb C}~|~r\ge 0\}$, 
while $\Psi_{+,p}^{(G_{0})}$
is given by the Laplace transform of the same 
function $\psi_{+,B}(z,y)$ along the half line 
$L_{0} = 
\{-s(z)+r\in{\mathbb C}~|~r\ge 0\}$ 
in the $y$-plane. 

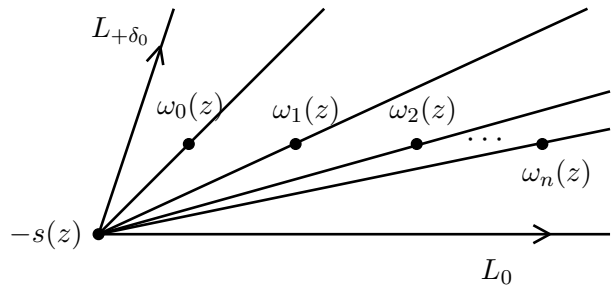
\begin{figure}
\begin{center}
\begin{pspicture}(2,-7.8)(10.,-3.8)
%
\psset{linewidth=0.5pt}
\psset{fillstyle=solid, fillcolor=black}
\pscircle(3,-7.5){0.08}
\pscircle(4.2,-6.3){0.08}
\pscircle(5.62,-6.3){0.08}
\pscircle(7.23,-6.3){0.08}
\pscircle(8.9,-6.3){0.08}
\psset{linewidth=0.5pt}
\psset{fillstyle=none}
\rput[c]{0}(2.3,-7.5){\small $-s(z)$}
\rput[c]{0}(4.23,-5.8){\small $\omega_0(z)$}
\rput[c]{0}(5.77,-5.85){\small $\omega_1(z)$}
\rput[c]{0}(7.32,-5.85){\small $\omega_2(z)$}
\rput[c]{0}(9.1,-6.7){\small $\omega_n(z)$}
\rput[c]{0}(8.1,-6.25){ $\cdots$} 
\rput[c]{0}(8.3,-8.0){\small $L_{0}$} 
\rput[c]{0}(3.3,-4.8){\small $L_{+\delta_0}$}
\psset{linewidth=1pt,linestyle=solid}
\psline(3,-7.5)(9.8,-7.5)
\psline(3,-7.5)(4,-4.5) 
\psline(9,-7.5)(8.78,-7.38)
\psline(9,-7.5)(8.78,-7.62)
\psline(3.82,-5)(3.64,-5.2) 
\psline(3.82,-5)(3.86,-5.25)
\psset{linewidth=1pt,linestyle=solid}
\psline(3,-7.5)(6,-4.5)
\psline(3,-7.5)(9.5,-4.5)
\psline(3,-7.5)(9.8,-5.6)
\psline(3,-7.5)(9.8,-6.1)
\end{pspicture}
\end{center}
\caption{Singular points of $\psi_{+,B}(z,y)$ 
in the $y$-plane.} 
\label{fig:Laplace-integral-path-appB}
\end{figure}

To compare the Borel sums $\Psi_{+,p}^{(G_{-\delta_0})}$ 
and $\Psi_{+,p}^{(G_{0})}$, we investigate the Stokes phenomena
occurring in directions in $0 < \theta \le \delta_0$. 
Contrary to the case of (i) of 
Lemma \ref{lemma:no-jump-appB}, when a Stokes curve 
in a direction $\theta$ satisfying $0 < \theta \le \delta_0$
hits a point $z$, 
the Stokes phenomenon occurs to $\psi_{+,p}$ 
since the sign at $p$ is $\oplus$ 
when $-\delta_0 \le \theta < 0$ 
(see \eqref{eq:Voros-formula-1}).  
In other wards, when a point $z$ lies on a Stokes 
curve in a direction $\theta$, a singular point 
of $\psi_{+,B}(z,y)$ lies on the half line 
$\{-s(z)+re^{-i\theta}\in{\mathbb C}~|~r\ge 0 \}$.
{(See \eqref{eq:Voros-formula-Stokes-1} in Remark 
\ref{remark:Voros-formula-is-Stokes}). }
As is explained in Subsection \ref{section:setting-appB}, 
a Stokes curve hits a point $z \in U$ infinitely 
many times when $\theta$ tends to $0$. Therefore, 
$\psi_{+,B}(z,y)$ has infinitely many 
singular points $\{\omega_n(z) \}_{n=0}^{\infty}$
in the $y$-plane (see Figure 
\ref{fig:Laplace-integral-path-appB}), 
and these singular points cause the Stokes phenomena 
for $\psi_{+,p}$. Hence, we have to deal with 
the infinitely many Stokes phenomena to compare 
the Borel sums $\Psi_{+,p}^{(G_{-\delta_0})}$ 
and $\Psi_{+,p}^{(G_{0})}$.

Let $\{\theta_n\}_{n=0}^{\infty}$ be the increasing 
sequence of directions in $-\delta_0 < \theta < 0$ 
such that the reference point $z_0$ lies on 
a Stokes curve in the directions. These directions 
depend on $z_0$, and are characterized 
by the following conditions:
\begin{itemize}
\item %
For any $n \ge 0$, the point $z_0$ lies on 
the Stokes curve in the direction $\theta_n$
emanating from $a$ and ending at $p$. The Stokes curve 
reaches at $z_0$ after turning around $p$ $n$-times 
(see Figure \ref{fig:spirals-theta-n}).
\item 
For any $n \ge 0$ and any $\theta$ satisfying 
$\theta_{n-1} < \theta < \theta_{n}$, $z_0$ does not 
lie on the Stokes curve in the direction $\theta$. 
Here we take $\theta_{-1} = -\delta_0$. 
\end{itemize}
The direction $\theta_n$ is nothing but 
$-\arg(s(z_0) + \omega_n(z_0))$ (see 
Figure \ref{fig:Laplace-integral-path-appB}). 

\begin{figure}
\begin{pspicture}(1,0)(15.,0)
\psset{fillstyle=solid, fillcolor=black}
\pscircle(4.59,-2.6){0.06}
\pscircle(9.75,-2.6){0.06}
\pscircle(15.0,-2.6){0.06}
\rput[c]{0}(4.9,-2.7){\small $z_0$}
\end{pspicture} \\
  \begin{minipage}{0.33\hsize}
  \begin{center} 
  \includegraphics[width=55mm]
  {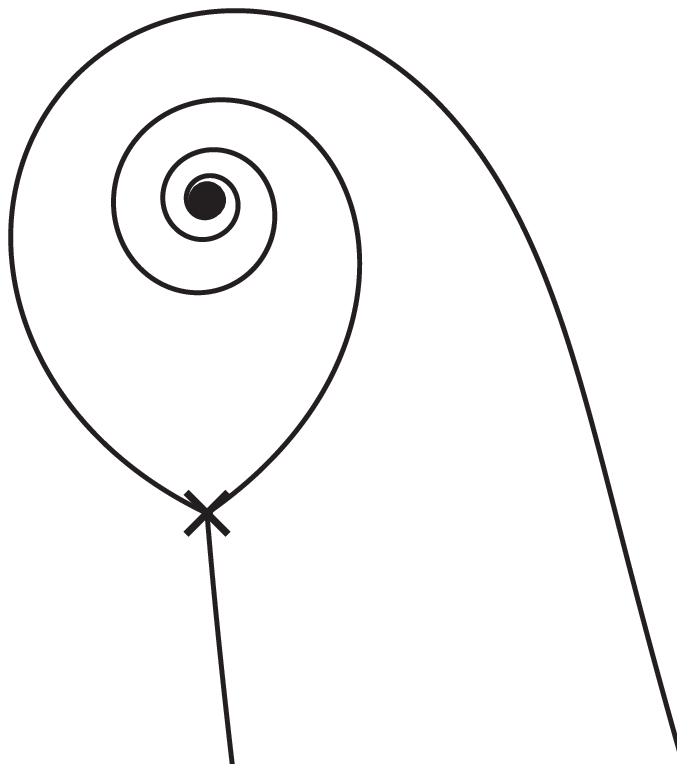} \\
  {$\theta=\theta_{n-1}$.} 
  \end{center}
  \end{minipage} 
  \begin{minipage}{0.33\hsize}
  \begin{center}
  \includegraphics[width=55mm]
  {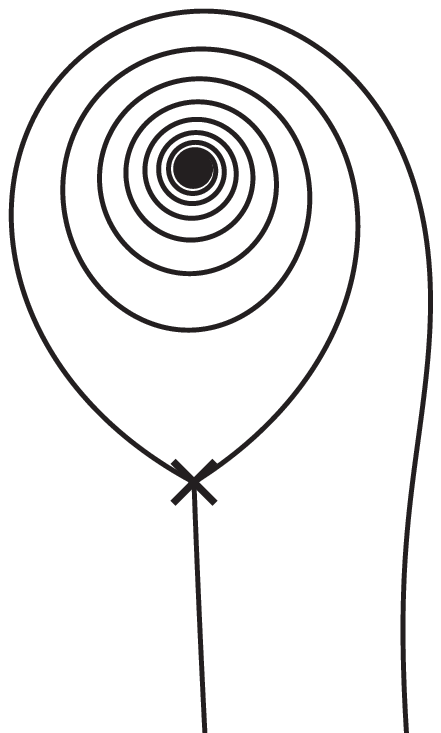} \\
  {$\theta_{n-1}<\theta<\theta_{n}$.} 
  \end{center}
  \end{minipage}  
  \begin{minipage}{0.33\hsize}
  \begin{center}
  \includegraphics[width=55mm]
  {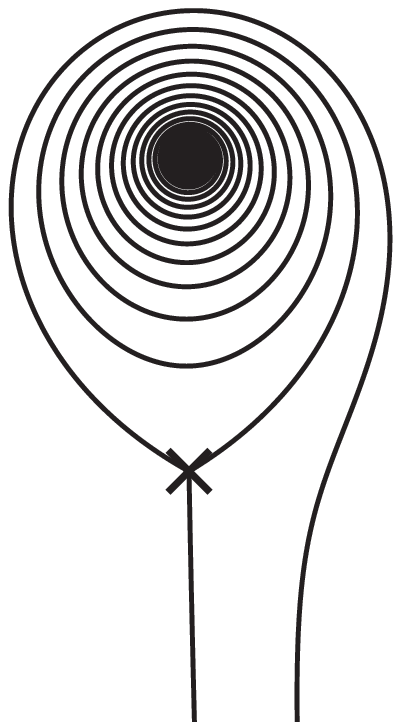} \\
  {$\theta=\theta_{n}$.} 
  \end{center}
  \end{minipage}  
  \caption{The directions $\theta_n$. 
  Here $\theta$ is the directions of Stokes curves.
  These pictures depict the case of $n=1$.}
  \label{fig:spirals-theta-n}
\end{figure}

For any $n \ge 0$, define the Borel sum 
$\Psi_{\pm}[n]$ of the WKB solutions $\psi_{\pm}$
defined on $U$ as follows. Fix a direction $\theta$ 
satisfying $\theta_{n-1}<\theta<\theta_{n}$. 
Since $z_0$ does not lie on Stokes curves 
in the direction $\theta$, 
there exists a small disc $U' \subset U$ containing 
$z_0$ and any point in $U'$ does not lie on Stokes 
curves in the direction $\theta$. Then, take the Borel 
sum $\Psi_{\pm}[n]$ in the direction $\theta$ defined 
on $U'$. Denote by the same letter $\Psi_{\pm}[n]$ 
the analytic continuation of the Borel sum to $U$.
Then, $\Psi_{\pm}[n]$ are independent of the choice of 
direction $\theta$ satisfying 
$\theta_{n-1}<\theta<\theta_{n}$ since they give 
the same germ of holomorphic function at $z_0$. 
Define the Borel sum 
$\Psi_{\pm,p}[n]$ of $\psi_{\pm,p}$ defined on $U$
by the same manner.
Especially, 
$\Psi_{\pm}[0] =\Psi_{\pm}^{(G_{-\delta_0})}$ and  
$\Psi_{\pm,p}[0] = \Psi_{\pm,p}^{(G_{-\delta_0})}$
hold.  

By a similar computation as 
\eqref{eq:connection2-appB}, 
the factor $e^{n V_{\gamma_0}}$ appears 
in the coefficient of the formula between 
$\Psi_{\pm}[n]$ and $\Psi_{\pm}[n+1]$ 
for each $n\ge0$ as follows:
\begin{eqnarray}
\begin{cases} 
\Psi_{+}[n] = \Psi_{+}[n+1] + 
i e^{n V_{\gamma_0}} \Psi_{-}[n+1], \\[+.2em] 
\Psi_{-}[n] = \Psi_{-}[n+1]~(=\Psi_{-}[0]).
\end{cases}
\end{eqnarray}
This formula is translated to the formula 
for $\Psi_{\pm,p}[n]$ as follows:
\begin{eqnarray}
\begin{cases} \label{eq:Psi-p-recursion} 
\Psi_{+,p}[n] = \Psi_{+,p}[n+1] + i e^{n V_{\gamma_0}} 
{\cal S}_{-}[e^{-W_{\beta_p}}]\Psi_{-,p}[n+1], \\[+.2em] 
\Psi_{-,p}[n] = \Psi_{-,p}[n+1]~(=\Psi_{-,p}[0]).
\end{cases} 
\end{eqnarray}
Here the factor ${\cal S}_{-}[e^{-W_{\beta_p}}]$ 
appears as the consequence of the difference 
\eqref{eq:Voros-beta-p-appB} of the 
normalization between $\psi_{\pm}$ and $\psi_{\pm,p}$. 
Since the difference of $\Psi_{+,p}[n]$ and 
$\Psi_{+,p}[n+1]$ are expressed as the integral 
along the contour $C_n$ encircling the singular 
point $\omega_n(z)$ 
(see Figure \ref{fig:contour-integrals-appB}), 
the equality \eqref{eq:Psi-p-recursion} implies that
\begin{equation}
\int_{C_n} e^{-\eta y}\psi_{+,B}(z,y)dy = 
i e^{n V_{\gamma_0}} 
{\cal S}_{-}[e^{-W_{\beta_p}}]\Psi_{-,p}^{(G_{0})}.
\end{equation}
Here we have used \eqref{eq:limit-formula-2}. 
Therefore, we have 
\begin{eqnarray}
\Psi_{+,p}^{(G_{-\delta_0})} - 
\Psi_{+,p}^{(G_{0})} & = & \sum_{n=0}^{\infty}
\int_{C_n} e^{-\eta y}\psi_{+,B}(z,y)dy \nonumber \\
& = & 
i\left(\sum_{n=0}^{\infty} e^{n V_{\gamma_0}}\right) 
{\cal S}_{-}[e^{-W_{\beta_p}}]\Psi_{-,p}^{(G_{0})}.
\end{eqnarray}
Thus, as the consequence of the infinitely 
many Stokes phenomena, the infinite sum 
$\sum_{n=0}^{\infty}e^{n V_{\gamma_0}}$ 
appears in the difference between 
$\Psi_{+,p}^{(G_{-\delta_0})}$ and 
$\Psi_{+,p}^{(G_0)}$. The infinite series converges 
and the factor $(1-e^{V_{\gamma_0}})^{-1}$ appears
since the real part of $\oint_{\gamma_0}\sqrt{Q_0(z)}dz$ 
is negative. Thus we obtain \eqref{eq:limit-formula-3}.
\end{proof}

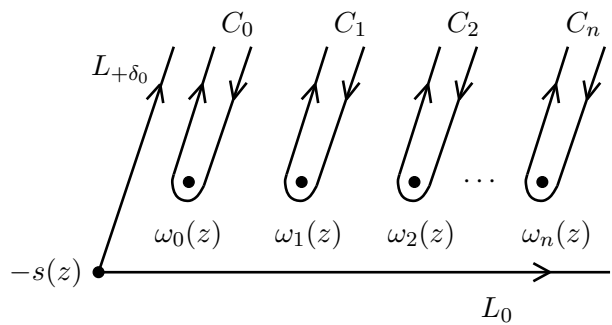
\begin{figure}
\begin{center}
\begin{pspicture}(2,-7.8)(10.,-3.8)
%
\psset{linewidth=0.5pt}
\psset{fillstyle=solid, fillcolor=black}
\pscircle(3,-7.5){0.08}
\pscircle(4.2,-6.3){0.08}
\pscircle(5.7,-6.3){0.08}
\pscircle(7.2,-6.3){0.08}
\pscircle(8.9,-6.3){0.08}
\psset{linewidth=0.5pt}
\psset{fillstyle=none}
\rput[c]{0}(2.3,-7.5){\small $-s(z)$}
\rput[c]{0}(4.2,-7.0){\small $\omega_0(z)$}
\rput[c]{0}(5.8,-7.0){\small $\omega_1(z)$}
\rput[c]{0}(7.3,-7.0){\small $\omega_2(z)$}
\rput[c]{0}(9.1,-7.0){\small $\omega_n(z)$}
\rput[c]{0}(8.1,-6.3){\small $\cdots$} 
\rput[c]{0}(8.3,-8.0){\small $L_{0}$} 
\rput[c]{0}(3.3,-4.8){\small $L_{+\delta_0}$}
\rput[c]{0}(4.85,-4.2){\small $C_0$} 
\rput[c]{0}(6.35,-4.2){\small $C_1$} 
\rput[c]{0}(7.85,-4.2){\small $C_2$} 
\rput[c]{0}(9.45,-4.2){\small $C_n$} 
\psset{linewidth=1pt,linestyle=solid}
\psline(3,-7.5)(9.8,-7.5)
\psline(3,-7.5)(4,-4.5) 
\psline(9,-7.5)(8.78,-7.38)
\psline(9,-7.5)(8.78,-7.62)
\psline(3.82,-5)(3.64,-5.2) 
\psline(3.82,-5)(3.86,-5.25)
\psline(4,-6.2)(4.54,-4.5)
\pscurve(4,-6.2)(3.98,-6.3)(4.05,-6.5)(4.2,-6.55)
(4.3,-6.5)(4.4,-6.35)
\psline(4.4,-6.35)(5.03,-4.5)%
\psline(4.38,-5)(4.2,-5.2)
\psline(4.38,-5)(4.42,-5.25)
\psline(4.795,-5.2)(4.75,-4.95)
\psline(4.795,-5.2)(4.97,-4.98)
\psline(5.5,-6.2)(6.04,-4.5)
\pscurve(5.5,-6.2)(5.48,-6.3)(5.55,-6.5)(5.7,-6.55)
(5.8,-6.5)(5.9,-6.35)
\psline(5.9,-6.35)(6.53,-4.5)%
\psline(5.88,-5)(5.7,-5.2)
\psline(5.88,-5)(5.92,-5.25)
\psline(6.295,-5.2)(6.25,-4.95)
\psline(6.295,-5.2)(6.47,-4.98)
\psline(7.0,-6.2)(7.54,-4.5)
\pscurve(7.0,-6.2)(6.98,-6.3)(7.05,-6.5)(7.2,-6.55)
(7.3,-6.5)(7.4,-6.35)
\psline(7.4,-6.35)(8.03,-4.5)%
\psline(7.38,-5)(7.2,-5.2)
\psline(7.38,-5)(7.42,-5.25)
\psline(7.795,-5.2)(7.75,-4.95)
\psline(7.795,-5.2)(7.97,-4.98)
\psline(8.7,-6.2)(9.24,-4.5)
\pscurve(8.7,-6.2)(8.68,-6.3)(8.75,-6.5)(8.9,-6.55)
(9.0,-6.5)(9.1,-6.35)
\psline(9.1,-6.35)(9.73,-4.5)%
\psline(9.08,-5)(8.9,-5.2)
\psline(9.08,-5)(9.12,-5.25)
\psline(9.495,-5.2)(9.45,-4.95)
\psline(9.495,-5.2)(9.67,-4.98)
\end{pspicture}
\end{center}
\caption{{Contour integrals representing the 
difference of $\Psi_{+,p}^{(G_{-\delta_0})}$
and $\Psi_{+,p}^{(G_{0})}$}.} 
\label{fig:contour-integrals-appB}
\end{figure}

Now we derive the desired formula 
\eqref{eq:AIT-analytic-appendix}
for the path $\beta_p$. 
Substituting \eqref{eq:limit-formula-1}, 
\eqref{eq:limit-formula-2} and 
\eqref{eq:limit-formula-3} into 
\eqref{eq:compare-WKB-slutions-appB}, we obtain
\begin{equation}
{\cal S}_{+}[e^{W_{\beta_p}/2}]
\Psi_{+,p}^{(G_{0})} = 
(1-e^{V_{\gamma_0}})
{\cal S}_{-}[e^{W_{\beta_p}/2}]
\Psi_{+,p}^{(G_{0})}. 
\end{equation}
Note that the coefficient of $\Psi_{-,p}^{(G_{0})}$
has been canceled out. Hence we have
\begin{equation} \label{eq:AIT-for-beta-p-appB}
{\cal S}_{-}[e^{W_{\beta_p}}] = 
{\cal S}_{+}[e^{W_{\beta_p}}]
(1-e^{V_{\gamma_0}})^{-2}.
\end{equation}
This is the desired formula 
\eqref{eq:AIT-analytic-appendix}
for the path $\beta_p$ since 
$\langle \gamma_0,\beta_p \rangle = -2$.

The formula \eqref{eq:AIT-analytic-appendix}
for general path $\beta$ is derived from 
\eqref{eq:AIT-for-beta-p-appB}. For example, 
let us take a path $\beta_a$ from $p$ to 
another pole in Figure \ref{fig:betap-and-any-path-appB}, 
which satisfies $\langle\gamma_0,\beta_a \rangle = -1$.
Note that, as indicated in Figure 
\ref{fig:betap-and-any-path-appB}, 
the path $\beta_p$ have the decomposition 
$\beta_p \equiv 2\beta_a + \beta_b$ 
under the $*$-equivalence. 
Here $\beta_b$ is a path which never intersects 
with $\ell_0$, and hence 
\begin{equation} 
{\cal S}_{-}[e^{W_{\beta_b}}] = 
{\cal S}_{+}[e^{W_{\beta_b}}].
\end{equation}
Then, it follows from 
\eqref{eq:AIT-for-beta-p-appB} that 
\begin{equation} 
{\cal S}_{-}[e^{2W_{\beta_a}}] = 
{\cal S}_{+}[e^{2W_{\beta_a}}]
(1-e^{V_{\gamma_0}})^{-2}
\end{equation}
holds. Taking the square root, we have 
\begin{equation} \label{eq:AIT-for-beta-a-appB}
{\cal S}_{-}[e^{W_{\beta_a}}] = 
\pm{\cal S}_{+}[e^{W_{\beta_a}}]
(1-e^{V_{\gamma_0}})^{-1}.
\end{equation}
We can conclude the sign $\pm$ in 
\eqref{eq:AIT-for-beta-a-appB} is $+$ 
because both of 
${\cal S}_{-}[e^{W_{\beta_a}}]$ and 
${\cal S}_{+}[e^{W_{\beta_a}}]$ have the same 
asymptotic expansion $e^{W_{\beta_a}}$
when $\eta \rightarrow +\infty$. 
Thus we have \eqref{eq:AIT-analytic-appendix} 
for the path $\beta_a$. 
Since any path intersecting with $\ell_0$ 
can be expressed as a sum of 
$\pm \beta_{a}$ and some paths which never 
intersect with $\ell_0$.
Thus, \eqref{eq:AIT-analytic-appendix} 
holds for any path, and we have proved Theorem 
\ref{thm:loop-type-degeneration-analytic}.

\begin{figure}
\begin{center}
\begin{pspicture}(6,-9.1)(10.,-3.8)
%
\psset{linewidth=0.5pt}
\psset{fillstyle=solid, fillcolor=black}
\pscircle(8,-5.5){0.08}
\psset{linewidth=0.5pt}
\psset{fillstyle=none}
\rput[c]{0}(8,-6.98){\small $\times$}
\rput[c]{0}(8.1,-9.1){\small $G_{0}$} 
\rput[c]{0}(8,-8.5){\small $\ominus$}
\rput[c]{0}(7.5,-6.){\small $\beta_p$} 
\psset{linewidth=1pt,linestyle=solid}
\pscurve(8,-7)(7.9,-7.1)(7.8,-7)
(7.7,-6.9)(7.6,-7)(7.5,-7.1)(7.4,-7)
(7.3,-6.9)(7.2,-7)(7.1,-7.1)(7,-7)
(6.9,-6.9)(6.8,-7)(6.7,-7.1)(6.6,-7)
(6.5,-6.9)(6.4,-7)
\psset{linewidth=0.5pt}
\psline(8,-7.0)(8,-8.34)
\pscurve(8,-7.0)(9,-6.3)(9.2,-5.5)
(9,-4.8)(8,-4.3)(7,-4.8)
(6.8,-5.5)(7,-6.3)(8,-7.0)
\psset{linewidth=1pt}
\pscurve(8,-5.5)(8.25,-6.5)(8.2,-7.1)(8.1,-7.2)
(8,-7.25)(7.9,-7.2)(7.8,-7.1)(7.75,-6.9)
\psline(8.25,-6.5)(8.13,-6.65)
\psline(8.25,-6.5)(8.37,-6.62)
\psset{linewidth=0.5pt}
\psset{linewidth=1pt, linestyle=dashed}
\pscurve(8,-5.5)(7.85,-6)(7.75,-6.9)
\psset{linewidth=1pt,linestyle=solid}
\psset{linewidth=0.5pt}
\end{pspicture}
\hskip50pt
\begin{pspicture}(6,-9.1)(10.,-3.8)
%
\psset{linewidth=0.5pt}
\psset{fillstyle=solid, fillcolor=black}
\pscircle(8,-5.5){0.08}
\psset{linewidth=0.5pt}
\psset{fillstyle=none}
\rput[c]{0}(8,-6.98){\small $\times$}
\rput[c]{0}(8.1,-9.1){\small $G_{0}$} 
\rput[c]{0}(8,-8.5){\small $\ominus$}
\rput[c]{0}(4.8,-6.7){\small $=$}
\rput[c]{0}(8.2,-6.3){\small $\beta_{a}$} 
\rput[c]{0}(7.5,-5.3){\small $-\beta_{a}^*$} 
\rput[c]{0}(9.5,-4.2){\small $\beta_{b}$} 
\psset{linewidth=1pt,linestyle=solid}
\pscurve(8,-7)(7.9,-7.1)(7.8,-7)
(7.7,-6.9)(7.6,-7)(7.5,-7.1)(7.4,-7)
(7.3,-6.9)(7.2,-7)(7.1,-7.1)(7,-7)
(6.9,-6.9)(6.8,-7)(6.7,-7.1)(6.6,-7)
(6.5,-6.9)(6.4,-7)
\psset{linewidth=0.5pt}
\psline(8,-7.0)(8,-8.34)
\pscurve(8,-7.0)(9,-6.3)(9.2,-5.5)
(9,-4.8)(8,-4.3)(7,-4.8)
(6.8,-5.5)(7,-6.3)(8,-7.0)
\psset{linewidth=1pt}
\pscurve(7.75,-6.9)(7.8,-7.5)(7.95,-8.35)
\pscurve(8.1,-8.35)(8.6,-7.2)(9.2,-6.5)(9.5,-5.5)
(9.2,-4.7)(8.8,-4.3)(8,-4.05)(7.2,-4.3)(6.8,-4.7)
(6.55,-5.5)(6.8,-6.5)(7.2,-7)
\pscurve(7.5,-7)(7.5,-6.4)(7.7,-5.7)(7.95,-5.5)
\psline(8.6,-7.2)(8.65,-7.4)
\psline(8.6,-7.2)(8.4,-7.3)
\psline(7.8,-7.4)(7.9,-7.2)
\psline(7.78,-7.4)(7.68,-7.19)
\psline(7.57,-6)(7.43,-6.14)
\psline(7.58,-5.99)(7.65,-6.18)
\psset{linewidth=0.5pt}
\psset{linewidth=1pt,linestyle=dashed}
\pscurve(8,-5.5)(7.85,-6)(7.75,-6.9)
\pscurve(7.2,-7)(7.4,-7.4)(7.8,-8.5)
\pscurve(7.85,-8.4)(7.65,-7.6)(7.5,-7)
\psset{linewidth=0.5pt}
\end{pspicture}
\end{center}
\caption{The path $\beta_p$ and a typical path $\beta_a$
which intersects $\ell_0$.} 
\label{fig:betap-and-any-path-appB}
\end{figure}
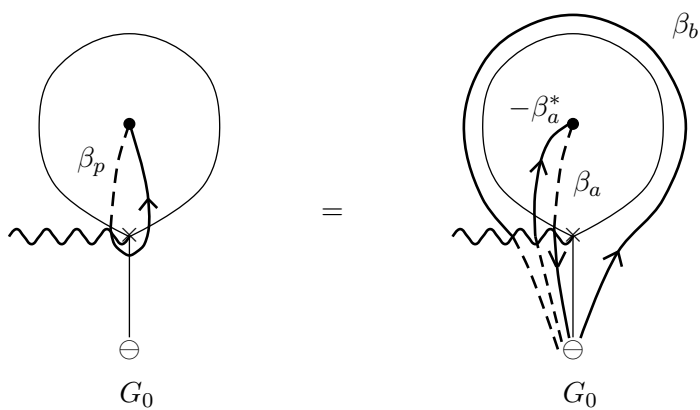


\section*{References}
\bibliography{../../biblist/biblist.bib}

\newcommand{\etalchar}[1]{$^{#1}$}
\providecommand{\bysame}{\leavevmode\hbox to3em{\hrulefill}\thinspace}
\providecommand{\MR}{\relax\ifhmode\unskip\space\fi MR }
\providecommand{\MRhref}[2]{%
  \href{http://www.ams.org/mathscinet-getitem?mr=#1}{#2}
}
\providecommand{\href}[2]{#2}
\begin{thebibliography}{IIK{\etalchar{+}}13b}

\bibitem[AIT]{Aoki14}
T.~Aoki, K.~Iwaki, and T.~Takahashi, \emph{Exact {WKB} analysis of
  {S}chr\"odinger equation with a {S}tokes curve of loop type}, in preparation.

\bibitem[AKT91]{Aoki91}
T.~Aoki, T.~Kawai, and Y.~Takei, \emph{The {B}ender-{W}u analysis and the
  {V}oros theory}, ICM-90 Satellite Conf. Proc. ``Special Functions",
  Springer-Verlag, 1991, pp.~1--29.

\bibitem[AT13]{Aoki12}
T.~Aoki and M.~Tanda, \emph{Borel sum of {V}oros coefficients of hypergeometric
  differential equation with a large parameter}, RIMS K\^oky\^uroku
  \textbf{1861} (2013), 17--24.

\bibitem[BFZ05]{Berenstein05}
A.~Berenstein, S.~Fomin, and A.~Zelevinsky, \emph{Cluster algebras {III}: upper
  bounds and double {B}ruhat cells}, Duke Math. J. \textbf{126} (2005), 1--52;
  arXiv:math/035434 [math.RT].

\bibitem[BS13]{Bridgeland13}
T.~Bridgeland and I.~Smith, \emph{Quadratic differentials as stability
  conditions}, 2013, arXiv:1302.7030 [math.AG].

\bibitem[Cir13]{Cirafici13}
M.~Cirafici, \emph{Line defects and (framed) {BPS} quivers}, JHEP \textbf{11}
  (2013), 141; arXiv:1307.713 [hep--th].

\bibitem[Cos08]{Costin08}
O.~Costin, \emph{Asymptotics and {B}orel {S}ummability}, Monographs and surveys
  in pure and applied mathematics, vol. 141, Chapmann and Hall/CRC, 2008.

\bibitem[DDP93]{Delabaere93}
E.~Delabaere, H.~Dillinger, and F.~Pham, \emph{R\'esurgence de {V}oros et
  p\'eriodes des courbes hyperelliptiques}, Ann. Inst. Fourier (Grenoble)
  \textbf{43} (1993), 163--199.

\bibitem[DP99]{Delabaere99}
E.~Delabaere and F.~Pham, \emph{Resurgent methods in semi-classical
  asymptotics}, Ann. Inst. Henri Poincar\'e \textbf{71} (1999), 1--94.

\bibitem[DWZ10]{Derksen10}
H.~Derksen, J.~Weyman, and A.~Zelevinsky, \emph{Quivers with potentials and
  their representations {II}: {A}pplications to cluster algebras}, J. Amer.
  Math. Soc. \textbf{23} (2010), 749--790; arXiv:0904.0676 [math.RA].

\bibitem[Eca84]{Ecalle84}
J.~Ecall\'e, \emph{Cinq applications des fonctions r\'esurgentes}, 1984,
  Preprint, Prepub. Math. d'Orsay, 84T62, 110 pp.

\bibitem[Fed93]{Fedoryuk93}
M.~V. Fedoryuk, \emph{Asymptotic analysis: linear ordinary differential
  equations}, Springer-Verlag, 1993.

\bibitem[FG06]{Fock03b}
V.~V. Fock and A.~B. Goncharov, \emph{Moduli spaces of local systems and higher
  {T}eichm\"uller theory}, Publ. Math. IHES \textbf{103} (2006), 1--211,
  arXiv:math/0311149 [math.AG].

\bibitem[FG07]{Fock05}
\bysame, \emph{Dual {T}eichm\"uller and lamination spaces}, Handbook of
  Teich\"uller theory, Vol. {I}, Eur. Math. Soc., 2007, pp.~647--684,
  arXiv:math/0510312 [math.DG].

\bibitem[FG09a]{Fock03}
\bysame, \emph{Cluster ensembles, quantization and the dilogarithm}, Annales
  Sci. de l'\'Ecole Norm. Sup. \textbf{42} (2009), 865--930; arXiv:math/0311245
  [math.AG].

\bibitem[FG09b]{Fock07b}
\bysame, \emph{Cluster ensembles, quantization and the dilogarithm {II}: {T}he
  intertwiner}, Prog. Math. \textbf{269} (2009), 655--673; arXiv:math.0702398
  [math.AG].

\bibitem[FK94]{Faddeev94}
L.~D. Faddeev and R.~M. Kashaev, \emph{Quantum dilogarithm}, Mod. Phys. Lett.
  \textbf{A9} (94), 427--434; arXiv:hep--th/9310070.

\bibitem[FST08]{Fomin08}
S.~Fomin, M.~Shapiro, and D.~Thurston, \emph{Cluster algebras and triangulated
  surfaces. {P}art {I}: {C}luster complexes}, Acta Math. \textbf{201} (2008),
  83--146; arXiv:math/0608367 [math.RA].

\bibitem[FT12]{Fomin08b}
S.~Fomin and D.~Thurston, \emph{Cluster algebras and triangulated surfaces.
  {P}art {II}: {L}ambda lengths}, 2012, arXiv:1210.5569.

\bibitem[FZ02]{Fomin02}
S.~Fomin and A.~Zelevinsky, \emph{Cluster algebras {I}. {F}oundations}, J.
  Amer. Math. Soc. \textbf{15} (2002), 497--529 (electronic);
  arXiv:math/0104151 [math.RT].

\bibitem[FZ03]{Fomin03a}
\bysame, \emph{Cluster algebras {II}. {F}inite type classification}, Invent.
  Math. \textbf{154} (2003), 63--121; arXiv:math/0208229 [math.RA].

\bibitem[FZ07]{Fomin07}
\bysame, \emph{Cluster algebras {IV}. {C}oefficients}, Compositio Mathematica
  \textbf{143} (2007), 112--164; arXiv:math/0602259 [math.RT].

\bibitem[Get09]{Getmanenko09}
A.~Getmanenko, \emph{Shatalov-{S}ternin's construction of complex {WKB}
  solutions and the associated {R}iemann surface}, 2009, arXiv:0907.2934
  [math.CA].

\bibitem[Get11]{Getmanenko11}
\bysame, \emph{Shatalov-{S}ternin's construction of complex {WKB} solutions and
  the choice of integration paths}, 2011, arXiv:1111.6325 [math.CA].

\bibitem[GMN13]{Gaiotto09}
D.~Gaiotto, G.~W. Moore, and A.~Neitzke, \emph{Wall-crossing, {H}itchin
  systems, and the {WKB} approximation}, Adv. in Math. \textbf{234} (2013),
  239--403; arXiv:0907.3987 [hep--th].

\bibitem[GSV05]{Gekhtman05}
M.~Gekhtman, M.~Shapiro, and A.~Vainshtein, \emph{Cluster algebras and
  {W}eil-{P}etersson forms}, Duke Math. J. \textbf{127} (2005), 291--311;
  arXiv:math/0309138 [math.QA].

\bibitem[GT11]{Getmanenko11b}
A.~Getmanenko and D.~Tamarkin, \emph{Microlocal properties of sheaves and
  complex {WKB}}, 2011, arXiv:1111.6325 [math-ph].

\bibitem[Hat91]{Hatcher91}
A.~Hatcher, \emph{On triangulations of surfaces}, Topololy Appl. \textbf{40}
  (1991), 189--194.

\bibitem[IIK{\etalchar{+}}13a]{Inoue10a}
R.~Inoue, O.~Iyama, B.~Keller, A.~Kuniba, and T.~Nakanishi, \emph{Periodicities
  of {T} and {Y}-systems, dilogarithm identities, and cluster algebras {I}:
  {T}ype {$B_r$}}, Publ. RIMS \textbf{49} (2013), 1--42; arXiv:1001.1880
  [math.QA].

\bibitem[IIK{\etalchar{+}}13b]{Inoue10b}
\bysame, \emph{Periodicities of {T} and {Y}-systems, dilogarithm identities,
  and cluster algebras {II}: {T}ypes {$C_r$}, {$F_4$}, and {$G_2$}}, Publ. RIMS
  \textbf{49} (2013), 43--85; arXiv:1001.1881 [math.QA].

\bibitem[Kel10]{Keller08}
B.~Keller, \emph{Cluster algebras, quiver representations and triangulated
  categories}, Triangulated categories (T.~Holm, P.~J{\o}rgensen, and
  R.~Rouquier, eds.), Lecture Note Series, vol. 375, London Mathematical
  Society, Cambridge University Press, 2010, pp.~76--160; arXiv:0807.1960
  [math.RT].

\bibitem[Kel11]{Keller11}
B.~Keller, \emph{On cluster theory and quantum dilogarithm identities},
  Representations of algebras and related topics (A.~Skowro\'nski and
  K.~Yamagata, eds.), EMS Series of Congress Reports, European Mathematical
  Society, 2011, pp.~85--116; arXiv:1102.4148 [math.RT].

\bibitem[KN11]{Kashaev11}
R.~M. Kashaev and T.~Nakanishi, \emph{Classical and quantum dilogarithm
  identities}, SIGMA \textbf{7} (2011), 102, 29 pages, arXiv:1104.4630
  [math.QA].

\bibitem[Koi00]{Koike00}
T.~Koike, \emph{On the exact {WKB} anlysis of second order linear ordinary
  differential equations with simple poles}, Publ. RIMS (2000), 297--319.

\bibitem[KS]{Koike13}
T.~Koike and R.~Sch\"afke, \emph{On the {B}orel summability of {WKB} solutions
  of {S}chr\"odinger equations with polynomial potentials and its application},
  in preparation; also Talk given by T. Koike in the RIMS workshop ``Exact WKB
  analysis -- Borel summability of WKB solutions", September, 2010.

\bibitem[KS08]{Kontsevich08}
M.~Kontsevich and Y.~Soibelman, \emph{Stability structures,
  {D}onaldson-{T}homas invariants and cluster transformations}, 2008,
  arXiv:0811.2435 [math.AG].

\bibitem[KS10]{Kontsevich09}
\bysame, \emph{Motivic {D}onaldson-{T}homas invariants: summary of results},
  Contemp. Math. \textbf{527} (2010), 55--89; arXiv:0910.4315 [math.AG].

\bibitem[KT05]{Kawai05}
T.~Kawai and Y.~Takei, \emph{Algebraic analysis of sigular perturbation
  theory}, Translations of mathematical monographs, no. 227, American
  Mathematical Society, 2005.

\bibitem[LF12]{Labardini12}
D.~Labardini-Fragoso, \emph{Quivers with potentials associated to triangulated
  surfaces, {P}art {IV}: {R}emoving boundary assumptions}, 2012,
  arXiv:1206.1798 [math.CO].

\bibitem[Nag11]{Nagao11b}
K.~Nagao, \emph{Quantumd dilogarithm identities}, RIMS K\^oky\^uroku Bessatsu
  \textbf{B28} (2011), 165--170.

\bibitem[Nag13]{Nagao10}
K.~Nagao, \emph{Donaldson-{T}homas theory and cluster algebras}, Duke Math. J.
  \textbf{7} (2013), 1313--1367; arXiv:1002.4884 [math.AG].

\bibitem[Nak12]{Nakanishi11c}
T.~Nakanishi, \emph{Tropicalization method in cluster algebras}, Contemp. Math.
  \textbf{580} (2012), 95--115; arXiv:1110.5472 [math.QA].

\bibitem[NS12]{Nakanishi12c}
T.~Nakanishi and S.~Stella, \emph{Wonder of sine-{G}ordon {$Y$}-systems}, 2012,
  preprint version in arXiv:1212.6853.

\bibitem[Pla11]{Plamondon10b}
P.~Plamondon, \emph{Cluster algebras via cluster categories with
  infinite-dimensional morphism spaces}, Compos. Math. \textbf{147} (2011),
  1921--1954; arXiv:1004.0830 [math.RT].

\bibitem[Qiu14]{Qiu14}
Y.~Qiu, \emph{On the spheciral twists on 3-{C}alabi-{Y}au categories from
  marked surfaces}, 2014, arXiv:1407.0806.

\bibitem[SS09]{Speyer04}
D.~Speyer and B.~Sturmfels, \emph{Tropical mathematics}, Mathematics Magazine
  \textbf{82} (2009), 163--173; arXiv:math/0408099 [math.CO].

\bibitem[Str84]{Strebel84}
K.~Strebel, \emph{Quadratic differentials}, Springer-Verlag, 1984.

\bibitem[Tak08]{Takei08}
Y.~Takei, \emph{Sato's conjecture for the {W}eber equation and transformation
  theory for {S}chr\"odinger equations with a merging pair of turning points},
  RIMS K\^oky\^uroku Bessatsu \textbf{B10} (2008), 205--224.

\bibitem[Vor83]{Voros83}
A.~Voros, \emph{The return of the quartic oscillator. the complex {WKB}
  method}, Ann. Inst. Henri Poincar\'e \textbf{39} (1983), 211--338.

\bibitem[Xie12]{Xie12}
D.~Xie, \emph{{BPS} spectrum, wall crossing and quantum dilogarithm identity},
  2012, arXiv:1211.707 [hep-th].

\end{thebibliography}

\end{document}